%% file: MCGHopf8.tex
\documentclass[11pt, a4paper]{article}

\usepackage[utf8]{inputenc}
\usepackage[T1]{fontenc}
\usepackage{lmodern}
\usepackage{microtype}
\usepackage[pdfborder={0 0 0}]{hyperref}

\usepackage{color}
\usepackage{paralist}
\usepackage{graphicx}
\usepackage[all]{xy}
\xyoption{rotate}
\usepackage{verbatim}
\usepackage{tikz,pgfplots}
\usepackage{xcolor}

\usepackage[tbtags]{amsmath} 
\usepackage{amssymb}
\usepackage{amsthm}
\usepackage{mathtools}

\allowdisplaybreaks[2]

\usepackage{marginnote}

\newcommand{\tComma}{\quad\scalebox{1.5}{,}\quad\quad}
\newcommand{\tDot}{\quad\scalebox{1.5}{.}}

\usepackage[text={16.5cm, 24.2cm}, centering]{geometry}
\newcommand{\F}{\mathbb F}
\newcommand{\Z}{\mathbb{Z}}
\newcommand{\C}{\mathbb{C}}
\newcommand{\mac}{\mathcal C}

\newcommand{\inv}[0]{{-1}}
\newcommand{\oo}[0]{\otimes}
\newcommand{\id}[0]{\mathrm{id}}

\newcommand{\low}[2]{{#1}_{({#2})}}

\newcommand{\vect}[0]{\mathrm{Vect}}
\newcommand{\st}[0]{{\bf s}}
\newcommand{\ta}[0]{{\bf t}}

\newtheorem{theorem}{Theorem}[section]
\newtheorem*{theorem*}{Theorem}
\newtheorem{example}[theorem]{Example}
\newtheorem{lemma}[theorem]{Lemma}
\newtheorem{proposition}[theorem]{Proposition}
\newtheorem{corollary}[theorem]{Corollary}
\newtheorem{definition}[theorem]{Definition}
\newtheorem{remark}[theorem]{Remark}

\newcommand*\stdthebibliography{}
\let\stdthebibliography\thebibliography
\renewcommand{\thebibliography}[1]{%
\stdthebibliography{#1}\setlength{\itemsep}{-2pt}}

\def\mytitle{Mapping class group actions from  Hopf monoids and ribbon graphs}
\def\myauthors{C.~Meusburger, T.~Vo\ss}
\hypersetup{pdftitle=\mytitle, pdfauthor=\myauthors}

\begin{document}

\begin{center}
  {\huge\mytitle}

  \vspace{2em}

  {\large
   Catherine Meusburger\footnote{{\tt catherine.meusburger@math.uni-erlangen.de}}\\
   Thomas Vo\ss\footnote{{\tt voss@math.fau.de}} 
 }

 Department Mathematik \\
  Friedrich-Alexander-Universit\"at Erlangen-N\"urnberg \\
  Cauerstra\ss e 11, 91058 Erlangen, Germany\\[+2ex]

{February 10, 2020}

  \begin{abstract}
We show that any pivotal Hopf monoid $H$ in a symmetric monoidal category $\mac$ gives rise to actions of mapping class groups of oriented surfaces of genus $g\geq 1$ with $n\geq 1$ boundary components. These mapping class group actions are given by group homomorphisms into the group of  automorphisms of certain Yetter-Drinfeld modules over $H$. They are 
associated with edge slides in embedded ribbon graphs that generalise chord slides in chord diagrams.  We give a concrete description of these mapping class group actions in terms of generating Dehn twists and defining relations. For the case where $\mac$ is finitely complete and cocomplete, we also obtain actions of mapping class groups of closed surfaces by imposing invariance and coinvariance under the Yetter-Drinfeld module structure.
  \end{abstract}
\end{center}

\section{Introduction}
\label{sec:intro}

Mapping class group representations  arising from Hopf algebras    have been investigated extensively  in the context of topological quantum field theories, Chern-Simons theories and conformal field theories and in  the quantisation of moduli spaces of flat connections.  

In  \cite{Ly95a,Ly95b, Ly96} Lyubashenko  constructed projective  representations of surface mapping class groups  from Hopf algebras in certain abelian ribbon categories. 
At the same time, projective  mapping class group representations  were obtained from  the quantisation of Chern-Simons gauge theories,  via Reshetikhin-Turaev TQFTs \cite{RT} and  via the combinatorial  quantisation formalism due to Alekseev, Grosse and Schomerus  \cite{AGS95,AGS96,AS} and Buffenoir and Roche \cite{BR95,BR96}. 
As shown in  \cite{AS},  the mapping class group representations from the latter  are equivalent to the ones obtained from  \cite{RT}. Recently, they were revisited and generalised by Faitg \cite{Fa18a,Fa18b}, who related them to Lyubashenko's  representations. Brochier, Ben-Zvi and Jordan placed them in the context of factorisation homology  \cite{BBJ18}. 

A construction based on similar representation theoretical data  was employed by Fuchs,  Schweigert and Stigner in \cite{FS, FSS12, FSS14} to construct mapping class group invariants  in the context of  conformal field theories. Recently,  Fuchs, Schaumann and Schweigert \cite{FSS} constructed a more general modular functor and associated mapping class group representations from bimodule categories over certain tensor categories.

The conditions on the Hopf algebras and the underlying tensor categories in these works are less restrictive than the ones  for  TQFTs. However,  the underlying monoidal category is still required to have  duals, to be linear over a field and  abelian. This excludes interesting examples of Hopf monoids such as group objects in cartesian monoidal categories (including groups, strict 2-groups and simplicial groups) or  combinatorial Hopf monoids such as the ones in \cite{AM}. It  also does not directly include geometric examples, such as mapping class group actions on the moduli space $\mathrm{Hom}(\pi_1(\Sigma), G)/G$ of flat $G$-connections on a surface $\Sigma$ or on Teichm\"uller space.

In this article, we construct mapping class group actions from more general  data, namely a pivotal Hopf monoid $H$ in a symmetric monoidal category $\mac$.
We neither require that  $\mac$ is linear over some field  nor  that it is abelian,  has a zero object or is equipped with  duals. 
In particular, the construction can be applied to cartesian monoidal categories such as $\mathrm{Set}$, $\mathrm{Top}$, $\mathrm{Cat}$ or the category $\mathrm{PSh}(\mathcal A)=\mathrm{Set}^{\mathcal A^{op}}$, where  Hopf monoids correspond to group objects.

We show that  pivotality, which generalises the notion of a pivotal Hopf algebra in $\mac=\vect_\F$, is sufficient to obtain actions of mapping class groups of  surfaces with boundary. Already in the case $\mac=\vect_\F$ this assumption is much weaker than  involutivity or semisimplicity and allows for  additional examples.  We also emphasise that pivotality is  a choice of structure. Even for involutive Hopf monoids, where the square of the antipode is the identity, one may choose a non-trivial pivotal structure and thus modify  the  mapping class group actions.  

If the symmetric monoidal category $\mac$ is finitely complete and cocomplete and the pivotal structure involutive, 
our construction also yields actions of mapping class groups of closed surfaces. These are  described concretely in terms of generating Dehn twists, 
 and the description is simple enough to verify its relations  by explicit computations.

More specifically, we consider the mapping class group $\mathrm{Map}(\Sigma)$ of an oriented surface $\Sigma$ of genus $g\geq 1$ with a finite number of boundary components. The ingredients in the construction  are
\begin{compactenum}[(i)]
 \item an embedded directed graph  $\Gamma$ in $\Sigma$ that satisfies certain additional conditions,
 \item  a pivotal Hopf monoid  $H$  in a symmetric monoidal category $\mac$. 
\end{compactenum}

To the  graph $\Gamma$ with edge set $E$, we associate the $E$-fold tensor product $H^{\oo E}$ in $\mac$.  Edge orientation reversal corresponds to an involution constructed from the pivotal structure and the antipode. Every edge of $\Gamma$ is associated with two $H$-module and two $H$-comodule structures on $H^{\oo E}$, and every pair of a vertex and an adjacent face in $\Gamma$ defines a $H$-Yetter-Drinfeld module structure on $H^{\oo E}$.

The mapping class group actions are obtained from edge slides that generalise the chord slides  in \cite{B,ABP}. They transform the graph $\Gamma$ by sliding an edge end in $\Gamma$ along an adjacent  edge. To each  edge slide we assign an automorphism of $H^{\oo E}$ constructed from the $H$-module and comodule structures for these edges. 
From the presentation of the mapping class group in \cite{B}  we then obtain a mapping class group action by automorphisms of Yetter-Drinfeld modules.

{\bf Theorem 1:} Let $\Sigma$ be an oriented surface of genus $g\geq 1$ with one boundary component. Then the edge slides induce a group homomorphism $\rho:\mathrm{Map}(\Sigma)\to \mathrm{Aut}_{YD}(H^{\oo 2g})$.

We then generalise this result to surfaces $\Sigma$ of genus $g\geq 1$ and with $n+1\geq 1$ boundary components. For this, we use the presentation  of the mapping class group $\mathrm{Map}(\Sigma)$ in terms of generating Dehn twists due to Gervais \cite{G01}. We work with a graph consisting of a set of generators of the fundamental group $\pi_1(\Sigma)$ and $n$ additional edges that extend to the boundary.
To each generating Dehn twist in \cite{G01} we assign an automorphism of the object $H^{\oo 2(n+g)}$. These are automorphisms of 
Yetter-Drinfeld modules with respect to a certain Yetter-Drinfeld module structure  on $H^{\oo 2(n+g)}$ and  satisfy the defining relations of the mapping class group $\mathrm{Map}(\Sigma)$.

{\bf Theorem 2:} Let $\Sigma$ be an oriented surface of genus $g\geq 1$ with $n+1\geq 1$ boundary components. Then the edge slides induce a group homomorphism 
$\rho: \mathrm{Map}(\Sigma)\to \mathrm{Aut}_{YD}(H^{\oo 2(n+g)})$. 

If  $\mac$ is finitely complete and cocomplete, one 
can  define (co)invariants of  $H$-(co)modules and $H$-modules as (co)equalisers  of their (co)action  morphisms and 
associate to each Yetter-Drinfeld module $M$ over $H$ an object $M_{inv}$ in $\mac$ that is both, invariant and coinvariant.  If the pivotal structure is involutive, this  yields  mapping class group actions for closed surfaces by automorphisms of $M_{inv}$.

{\bf Theorem 3:} Let $\mac$ be a symmetric monoidal category that is finitely complete and cocomplete and $H$ a pivotal Hopf monoid in $\mac$ with an involutive pivotal structure. Then the mapping class group action of Theorem 2  induces a group homomorphism $\rho: \mathrm{Map}(\Sigma')\to \mathrm{Aut}(H^{\oo 2(n+g)}_{inv})$, where  $\Sigma'$ is  obtained  by attaching a disc to a boundary component of $\Sigma$.

For $n=0$, this defines actions of mapping class groups of closed surfaces. 
In particular, we obtain the mapping class group action on the moduli space $\mathrm{Hom}(\pi_1(\Sigma), G)/G$ for a group $G$ and a surface $\Sigma$ of genus $g\geq 1$. We can also modify these actions with different (involutive) pivotal structures.

The article is structured as follows. Section \ref{sec:Hopf} contains the required background on Hopf monoids in symmetric monoidal categories and on their modules and comodules. Section \ref{sec:ribbon} summarises the background on ribbon graphs and Section \ref{sec:mapsec} the descriptions of mapping class groups from \cite{G01,B}.
In Section \ref{sec:hrib} we introduce the $H$-module, $H$-comodule and $H$-Yetter-Drinfeld module structures on $H^{\oo E}$ associated with an embedded graph $\Gamma$.  This is inspired by Kitaev's quantum double models \cite{Ki, BMCA}, but no background on these models is required.

Section \ref{sec:chordslideshopf} introduces graph transformations by edge slides and the associated automorphisms. It shows that these edge slides satisfy the defining relations in \cite{B} and thus define mapping class group actions for oriented surfaces of genus $g\geq 1$ with one boundary component (Theorem 1). Section \ref{sec:torus} gives an explicit description of these mapping class group actions for the simplest examples, the  torus and the one-holed torus. In particular, this defines an action of the braid group $B_3$ on $H^{ \oo 2}$ and an action of the modular group $\mathrm{SL}(2,\mathbb Z)$  on  $H^{ \oo 2}_{inv}$.

In Section \ref{sec:slidetwist}  we generalise edge slides  to slides along certain closed paths in the underlying graph and introduce Dehn twists. 
These are the technical prerequisites for Section \ref{sec:gervaismap}, where we consider a generating set of Dehn twists and prove that they satisfy the relations for the mapping class group of  a surface $\Sigma$ of genus $g\geq 1$ with $n+1\geq 1$ boundary components (Theorem 2). 
The remainder of Section \ref{sec:gervaismap}  requires a finitely complete and cocomplete symmetric monoidal category $\mac$ and an involutive pivotal structure and is dedicated to the proof of Theorem 3. We illustrate this with the example for the mapping class group of a surface of genus 2 with one or no boundary components.

\section{Background on Hopf monoids}
\label{sec:Hopf}

\subsection{Pivotal Hopf monoids}
Throughout the article we consider a symmetric monoidal category $\mathcal C$  with unit object $e$ and crossing or braiding morphisms $\tau_{X,Y}: X\oo Y\to Y\oo X$.
We denote by $1_{X}$ the identity morphism on an object $X$.
We use the usual diagrammatic notation for symmetric monoidal categories. Diagrams are read from  left to right for tensor products  and from top to bottom for the composition of morphisms. In formulas, we  suppress associators and left and right unit constraints of $\mac$. 
\begin{definition} Let $\mac$ be a symmetric monoidal category.
\begin{compactenum}
\item A 
{\bf monoid} in $\mac$ is a triple $(A,m,\eta)$ of an object $A$ in $\mac$ and  morphisms $m: A\oo A\to A$ and $\eta: e\to A$ such that, up to coherence data,
$$
m\circ (1_{A} \oo m)=m\circ (m\oo 1_{A})\qquad m\circ (1_A\oo\eta)=m\circ (\eta\oo 1_A)=1_A. 
$$
\item A 
{\bf comonoid} in $\mac$ is a triple $(C,\Delta,\epsilon)$ of an object $C$ in $\mac$ and  morphisms $\Delta: C\to C\oo C$ and $\epsilon: C\to e$ such that, up to coherence data,
$$
(\Delta\oo1_C)\circ \Delta = (1_C\oo \Delta)\circ\Delta \qquad (\epsilon\oo 1_C)\circ \Delta=(1_C\oo \epsilon)\circ\Delta=1_C.
$$
\item A  {\bf bimonoid} in $\mac$ is a pentuple $(H,m,\eta,\Delta,\epsilon)$ such that $(H,m,\eta)$ is a monoid, $(H,\Delta,\epsilon)$  a comonoid in $\mac$ and, up to coherence data
\begin{align*}
&\Delta\circ m=(m \oo m)\circ (1_H\oo \tau_{H,H}\oo 1_H)\circ (\Delta\oo \Delta)\qquad & &\Delta\circ\eta=\eta\oo \eta\\
&\epsilon\circ m=\epsilon\oo\epsilon & &\epsilon\circ\eta=1_e.
\end{align*}
\item A bimonoid $(H,m,\eta,\Delta,\epsilon)$ in $\mac$ is  a  {\bf Hopf monoid} if there is a morphism $S: H\to H$ with
$$
m\circ (S\oo 1_H)\circ\Delta=m\circ (1_H\oo S)\circ \Delta=\eta\circ\epsilon.
$$ 
\end{compactenum}
\end{definition}

If a bimonoid  in $\mac$ has an antipode, then this antipode is unique. The antipode is an anti-monoid and anti-comonoid morphism:
$$m\circ (S\oo S)\circ \tau_{H,H}=S\circ m,\quad  S\circ\eta=\eta,\quad
\tau_{H,H}\circ (S\oo S)\circ \Delta=\Delta\circ S, \quad \epsilon\circ S=\epsilon.$$

The multiplication $m:  H\oo H\to H$, unit $\eta: e\to H$, comultiplication $\Delta: H\to H\oo H$, counit  $\epsilon: H\to e$ and the antipode $S: H\to H$  are described by the following diagrams, respectively
\begin{align}\label{eq:hstructure}
\begin{tikzpicture}[scale=.25]
\begin{scope}
\draw[line width=1pt, color=black, draw opacity=1] plot [smooth, tension=0.6] coordinates 
      {(-2,2)(-1,0)(1,0)(2,2) };
 \draw[line width=1pt, color=black](0,-.2)--(0,-2);     
\end{scope}
\node at (2,-2) [anchor=north, color=black] {\scalebox{1.5}{,}};
\node at (6,-2) [anchor=north, color=black] {\scalebox{1.5}{,}};
\node at (13,-2) [anchor=north, color=black] {\scalebox{1.5}{,}};
\node at (16,-2) [anchor=north, color=black] {\scalebox{1.5}{,}};
\node at (21,-2) [anchor=north, color=black] {\scalebox{1.5}{.}};
\begin{scope}[shift={(5,0)}]
 \draw[line width=1pt, color=black](0,0)--(0,-2);  
 \draw[color=black, fill=white, line width=1pt] (0,0) circle (.2); 
\end{scope}
\begin{scope}[shift={(10,0)}]
\draw[line width=1pt, color=black, draw opacity=1] plot [smooth, tension=0.6] coordinates 
      {(-2,-2)(-1,0)(1,0)(2,-2) };
 \draw[line width=1pt, color=black](0,.2)--(0,2);     
\end{scope}
\begin{scope}[shift={(15,0)}]
 \draw[line width=1pt, color=black](0,2)--(0,0);  
 \draw[color=black, fill=white, line width=1pt] (0,0) circle (.2); 
\end{scope}
\begin{scope}[shift={(20,0)}]
 \draw[line width=1pt, color=black](0,2)--(0,-2);  
 \draw[color=black, fill=white, line width=1pt] (0,0) circle (.4); 
\end{scope}
\end{tikzpicture}
\end{align}
The condition that $H$ is a  monoid in $\mac$  reads
\begin{align}\label{eq:monoid}
&\begin{tikzpicture}[scale=.25]
\begin{scope} [shift={(-3.5,0)}]
\draw[line width=1pt, color=black, draw opacity=1] plot [smooth, tension=0.6] coordinates 
      {(-2,2)(-1,0)(1,0)(2.5,3) };
      \draw[line width=1pt, color=black, draw opacity=1] plot [smooth, tension=0.6] coordinates 
      {(-3,3)(-2.5,2.2)(-2,2)(-1.5,2.2)(-1,3) };
 \draw[line width=1pt, color=black](0,-.3)--(0,-2);     
\end{scope}
\node at (0,0){$=$};
\begin{scope} [shift={(3.5,0)}]
\draw[line width=1pt, color=black, draw opacity=1] plot [smooth, tension=0.6] coordinates 
      {(2,2)(1,0)(-1,0)(-2.5,3) };
      \draw[line width=1pt, color=black, draw opacity=1] plot [smooth, tension=0.6] coordinates 
      {(3,3)(2.5,2.2)(2,2)(1.5,2.2)(1,3) };
 \draw[line width=1pt, color=black](0,-.3)--(0,-2);     
\end{scope}
\end{tikzpicture}
\tComma
\begin{tikzpicture}[scale=.25]
\begin{scope} [shift={(-3.5,0)}]
  \draw[line width=1pt, color=black, draw opacity=1] plot [smooth, tension=0.6] coordinates 
      {(-2,2)(-1,0)(1,0)(2,2) };
 \draw[line width=1pt, color=black](0,-.2)--(0,-2);  
 \draw[color=black, fill=white, line width=1pt] (2,2) circle (.25);
\end{scope}
\node at (0,0){$=$};
\draw[line width=1pt, color=black] (1.5,2)--(1.5,-2);
\node at (3,0){$=$};
\begin{scope} [shift={(6.5,0)}]
    \draw[line width=1pt, color=black, draw opacity=1] plot [smooth, tension=0.6] coordinates 
      {(-2,2)(-1,0)(1,0)(2,2) };
 \draw[line width=1pt, color=black](0,-.2)--(0,-2);  
 \draw[color=black, fill=white, line width=1pt] (-2,2) circle (.25); 
\end{scope}
\end{tikzpicture}\tComma
\\\intertext{and the condition that $H$ is a comonoid in $\mac$  is}
\label{eq:comonoid}
&\begin{tikzpicture}[scale=.25]
\begin{scope} [shift={(-3.5,0)}]
\draw[line width=1pt, color=black, draw opacity=1] plot [smooth, tension=0.6] coordinates 
      {(-2,-2)(-1,0)(1,0)(2.5,-3) };
      \draw[line width=1pt, color=black, draw opacity=1] plot [smooth, tension=0.6] coordinates 
      {(-3,-3)(-2.5,-2.2)(-2,-2)(-1.5,-2.2)(-1,-3) };
 \draw[line width=1pt, color=black](0,.3)--(0,2);     
\end{scope}
\node at (0,0){$=$};
\begin{scope} [shift={(3.5,0)}]
\draw[line width=1pt, color=black, draw opacity=1] plot [smooth, tension=0.6] coordinates 
      {(2,-2)(1,0)(-1,0)(-2.5,-3) };
      \draw[line width=1pt, color=black, draw opacity=1] plot [smooth, tension=0.6] coordinates 
      {(3,-3)(2.5,-2.2)(2,-2)(1.5,-2.2)(1,-3) };
 \draw[line width=1pt, color=black](0,.3)--(0,2);     
\end{scope}
\end{tikzpicture}
\tComma
\begin{tikzpicture}[scale=.25]
\begin{scope} [shift={(-3.5,0)}]
  \draw[line width=1pt, color=black, draw opacity=1] plot [smooth, tension=0.6] coordinates 
      {(-2,-2)(-1,0)(1,0)(2,-2) };
 \draw[line width=1pt, color=black](0,.2)--(0,2);  
 \draw[color=black, fill=white, line width=1pt] (2,-2) circle (.25);
\end{scope}
\node at (0,0){$=$};
\draw[line width=1pt, color=black] (1.5,2)--(1.5,-2);
\node at (3,0){$=$};
\begin{scope} [shift={(6.5,0)}]
    \draw[line width=1pt, color=black, draw opacity=1] plot [smooth, tension=0.6] coordinates 
      {(-2,-2)(-1,0)(1,0)(2,-2) };
 \draw[line width=1pt, color=black](0,.2)--(0,2);  
 \draw[color=black, fill=white, line width=1pt] (-2,-2) circle (.25); 
\end{scope}
\end{tikzpicture}
\tDot
\end{align}
Using \eqref{eq:monoid} and \eqref{eq:comonoid}, we sometimes write $m^{n}: H^{\oo (n+1)}\to H$ and $\Delta^{n}: H\to H^{\oo(n+1)}$ for the $n$-fold  composites of $m$ and $\Delta$. We denote them by diagrams analogous to \eqref{eq:hstructure} but with $n+1$ lines meeting in a single point or on a horizontal line.

The compatibility conditions between the monoid and comonoid structure correspond to the diagrams
\begin{align}
	\nonumber
\begin{tikzpicture}[scale=.2]
\begin{scope}[shift={(-4,0)}]
  \draw[line width=1pt, color=black, draw opacity=1] plot [smooth, tension=0.6] coordinates 
      {(-2,2)(-1,0)(1,0)(2,2) };
       \draw[line width=1pt, color=black](0,-.2)--(0,-1.8); 
         \draw[line width=1pt, color=black, draw opacity=1] plot [smooth, tension=0.6] coordinates 
      {(-2,-4)(-1,-2)(1,-2)(2,-4) };
\end{scope}
\node at (-.5,-1){$=$};
\begin{scope}[shift={(4,0)}]
\draw[line width=1pt, color=black](-2,2)--(-2,1);
\draw[line width=1pt, color=black](2,2)--(2,1);
\draw[line width=1pt, color=black](-2,-4)--(-2,-3);
\draw[line width=1pt, color=black](2,-4)--(2,-3);
\draw[line width=1pt, color=black, draw opacity=1] plot [smooth cycle, tension=0.6] coordinates 
      {(-2,1)(-3, 0)(-3,-2)(-2,-3)(-.5,-2)(.5,0)(2,1)(3,0)(3,-2)(2,-3)(.5,-2)(-.5,0)};
 \draw[line width=1pt, color=black, draw opacity=1] plot [smooth, tension=0.6] coordinates 
      {};
\end{scope}
\end{tikzpicture}
\tComma
\begin{tikzpicture}[scale=.2]
\begin{scope}[shift={(12,0)}]
\draw[line width=1pt, color=black, draw opacity=1] plot [smooth, tension=0.6] coordinates 
      {(-2,2)(-1,0)(1,0)(2,2) };
 \draw[line width=1pt, color=black](0,-.2)--(0,-4);     
 \draw[color=black, fill=white, line width=.7pt] (0,-4) circle (.25);
 \node at (3,-1){$=$};
  \draw[line width=1pt, color=black](5,2)--(5,-4);
     \draw[color=black, fill=white, line width=.7pt] (5,-4) circle (.25);  
       \draw[line width=1pt, color=black](7,2)--(7,-4);
     \draw[color=black, fill=white, line width=.7pt] (7,-4) circle (.25);  
\end{scope}
\end{tikzpicture}
\tComma
\begin{tikzpicture}[scale=.2]
\begin{scope}[shift={(25,0)}]
\draw[line width=1pt, color=black, draw opacity=1] plot [smooth, tension=0.6] coordinates 
      {(-2,-4)(-1,-2)(1,-2)(2,-4) };
 \draw[line width=1pt, color=black](0,-1.8)--(0,2);     
 \draw[color=black, fill=white, line width=.7pt] (0,2) circle (.25);
 \node at (3,-1){$=$};
  \draw[line width=1pt, color=black](5,2)--(5,-4);
     \draw[color=black, fill=white, line width=.7pt] (5,2) circle (.25);  
       \draw[line width=1pt, color=black](7,2)--(7,-4);
     \draw[color=black, fill=white, line width=.7pt] (7,2) circle (.25);  
\end{scope}
\end{tikzpicture}
\tComma
\begin{tikzpicture}[scale=.2]
\begin{scope}[shift={(36,0)}]
  \draw[line width=1pt, color=black](0,2)--(0,-4);
   \draw[color=black, fill=white, line width=.7pt] (0,2) circle (.25);  
    \draw[color=black, fill=white, line width=.7pt] (0,-4) circle (.25);  
    \node at (1.5,-1){$=$};
	\draw[line width=1pt, color=black, dashed](3,2)--(3,-3.9);
    \end{scope}
\end{tikzpicture}\tComma
\end{align}
and the defining conditions on the antipode read
\begin{align}\label{eq:antdef}
\begin{tikzpicture}[scale=.2]
\begin{scope}[shift={(-3,0)}]
\draw[line width=1pt, color=black] (0,3)--(0,2);
\draw[line width=1pt, color=black] (0,-3)--(0,-2);
\draw[line width=1pt, color=black, draw opacity=1] plot [smooth cycle, tension=0.6] coordinates 
      {(0,2)(-1.2, 1)(-1.2,-1)(0,-2)(1.2,-1)(1.2,1)};
   \draw[color=black, fill=white, line width=1pt] (-1.3,0) circle (.4);     
\end{scope}
\node at (0,0){$=$};
\draw[line width=1pt, color=black] (1.5,3)--(1.5,1);
 \draw[color=black, fill=white, line width=.7pt] (1.5,1) circle (.25); 
\draw[line width=1pt, color=black] (1.5,-3)--(1.5,-1);
 \draw[color=black, fill=white, line width=.7pt] (1.5,-1) circle (.25); 
\node at (3,0){$=$};
\begin{scope}[shift={(6,0)}]
\draw[line width=1pt, color=black] (0,3)--(0,2);
\draw[line width=1pt, color=black] (0,-3)--(0,-2);
\draw[line width=1pt, color=black, draw opacity=1] plot [smooth cycle, tension=0.6] coordinates 
      {(0,2)(-1.2, 1)(-1.2,-1)(0,-2)(1.2,-1)(1.2,1)};
   \draw[color=black, fill=white, line width=1pt] (1.3,0) circle (.4);  
\end{scope}
\end{tikzpicture}
\tDot
\end{align}
That the antipode is an anti-monoid and anti-comonoid morphism is expressed in  the diagrams
\begin{align}\label{eq:anthom}
&\begin{tikzpicture}[scale=.2]
\begin{scope}[shift={(-3.5,0)}]
\draw[line width=1pt, color=black] (0,3)--(0,0);
\draw[line width=1pt, color=black, draw opacity=1] plot [smooth, tension=0.6] coordinates 
      {(-2,-2)(-1.5,-.5)(0,0)(1.5,-.5)(2,-2)};
   \draw[color=black, fill=white, line width=1pt] (0,1.5) circle (.4);     
\end{scope}
\node at (0,0){$=$};
\begin{scope}[shift={(3.6,0)}]
\draw[line width=1pt, color=black] (0,3)--(0,2);
\draw[line width=1pt, color=black, draw opacity=1] plot [smooth, tension=0.6] coordinates 
      {(2,-2)(-1,1)(0,2)(1,1)(-2,-2)};
   \draw[color=black, fill=white, line width=1pt] (1.5,-1.5) circle (.4); 
      \draw[color=black, fill=white, line width=1pt] (-1.5,-1.5) circle (.4);  
\end{scope}
\end{tikzpicture}
\tComma
\begin{tikzpicture}[scale=.2]
\begin{scope}[shift={(-3.5,0)}]
\draw[line width=1pt, color=black] (0,-3)--(0,0);
\draw[line width=1pt, color=black, draw opacity=1] plot [smooth, tension=0.6] coordinates 
      {(-2,2)(-1.5,.5)(0,0)(1.5,.5)(2,2)};
   \draw[color=black, fill=white, line width=1pt] (0,-1.5) circle (.4);     
\end{scope}
\node at (0,0){$=$};
\begin{scope}[shift={(3.5,0)}]
\draw[line width=1pt, color=black] (0,-3)--(0,-2);
\draw[line width=1pt, color=black, draw opacity=1] plot [smooth, tension=0.6] coordinates 
      {(2,2)(-1,-1)(0,-2)(1,-1)(-2,2)};
   \draw[color=black, fill=white, line width=1pt] (1.5,1.5) circle (.4); 
      \draw[color=black, fill=white, line width=1pt] (-1.5,1.5) circle (.4);  
\end{scope}
\end{tikzpicture}
\tComma
\begin{tikzpicture}[scale=.2]
\draw[line width=1pt, color=black] (-.5,2)--(-.5,-2);
   \draw[color=black, fill=white, line width=1pt] (-.5,0) circle (.4); 
      \draw[color=black, fill=white, line width=.7pt] (-.5,2) circle (.25);     
\node at (1.2,0){$=$};
\draw[line width=1pt, color=black] (2.5,2)--(2.5,-2);
      \draw[color=black, fill=white, line width=.7pt] (2.5,2) circle (.25);  
\end{tikzpicture}
\tComma
\begin{tikzpicture}[scale=.2]
\draw[line width=1pt, color=black] (-.5,2)--(-.5,-2);
   \draw[color=black, fill=white, line width=1pt] (-.5,0) circle (.4); 
      \draw[color=black, fill=white, line width=.7pt] (-.5,-2) circle (.25);     
\node at (1.2,0){$=$};
\draw[line width=1pt, color=black] (2.5,2)--(2.5,-2);
      \draw[color=black, fill=white, line width=.7pt] (2.5,-2) circle (.25);  
\end{tikzpicture}
\tDot
\end{align}

Where this simplifies the presentation, we also use  Sweedler notation. For  $\mac=\vect_\F$ this is the usual Sweedler notation for a Hopf algebra over a field $\mathbb F$. For a general symmetric monoidal category $\mac$ it is to be interpreted as a shorthand notation for a diagram in $\mac$.

In the following, we consider Hopf monoids  with additional structure, which  generalise  {\em pivotal} Hopf algebras over a field $\mathbb F$.
The concept of a pivotal Hopf algebra
in  $\mac=\vect_\F$ was introduced in \cite[Def 3.1]{BW99} as a Hopf algebra $H$ over $\F$ together with a choice of a grouplike element $p\in H$, the {\em pivotal element}\footnote{The term {\em pivotal Hopf} algebra was subsequently adopted  in other publications. The grouplike element $p$ in this article corresponds to the {\em inverse} of the grouplike element in \cite[Def 3.1]{BW99}.}, such that $h=p S^2(h) p^\inv$ for all $h\in H$. 
It is clear from its definition that the pivotal element $p$ is unique up to multiplication with central grouplike elements in $H$.

This definition easily generalises to a Hopf monoid in a symmetric monoidal category $\mac$. In this case, a grouplike element $p\in H$ is replaced by a morphism $p: e\to H$ satisfying $\Delta\circ p=p\oo p$ and $\epsilon\circ p=1_e$. The second condition generalises the requirement $p\neq 0$ for a grouplike element $p\in H$.
Note also that these conditions imply together with \eqref{eq:antdef} 
$$m\circ (S\oo 1_H)\circ (p\oo p)=\eta=m\circ (1_H\oo S)\circ (p\oo p),$$ which replaces the identity $S(p)=p^\inv$ for grouplike elements. The condition $h=pS^2(h) p^\inv$ for all $h\in H$  can then  be formulated as a condition on the morphism $p: e\to H$. 
This yields the following definition, which reduces to the notion of pivotality in \cite{BW99} for $\mac=\vect_\F$.

\begin{definition}\label{def:pivotal}  Let $\mac$ be a symmetric monoidal category. A {\bf pivotal Hopf monoid} in $\mac$ is a pair $(H,p)$ of a Hopf monoid $H$ in $\mac$ and a morphism $p: e\to H$
satisfying the  identities
\begin{align*}
&\Delta\circ p=p\oo p\quad & &\epsilon\circ p=1_e & &m^{2}\circ (1_H\oo S^2\oo S )\circ (p\oo 1_H\oo p)=1_H.
\end{align*}
A Hopf monoid $H$ in $\mac$ is called {\bf involutive} if $(H,\eta)$ is a pivotal Hopf monoid.
\end{definition}

We denote the morphism $p: e\to H$ by a diagram similar to the one for the unit morphism, but with a small black circle labeled $p$ instead of a small  white circle. The conditions in Definition \ref{def:pivotal} then correspond to the  diagrams
\begin{align}\label{eq:pivdiags}
\begin{tikzpicture}[scale=.2]
\begin{scope}[shift={(0,0)}]
\draw[line width=1pt, color=black, draw opacity=1] plot [smooth, tension=0.6] coordinates 
      {(-2,-4)(-1,-2)(1,-2)(2,-4) };
 \draw[line width=1pt, color=black](0,-1.8)--(0,1);     
 \draw[color=black, fill=black, line width=.7pt] (0,1) circle (.25);
 \node at (0,1)[anchor=east]{$p$};
 \node at (3,-1){$=$};
  \draw[line width=1pt, color=black](5.5,1)--(5.5,-4);
     \draw[color=black, fill=black, line width=.7pt] (5.5,1) circle (.25); 
      \node at (5.5,1)[anchor=east]{$p$}; 
       \draw[line width=1pt, color=black](7.5,1)--(7.5,-4);
     \draw[color=black, fill=black, line width=.7pt] (7.5,1) circle (.25);  
      \node at (7.5,1)[anchor=west]{$p$};
\end{scope}\end{tikzpicture}
\tComma
\begin{tikzpicture}[scale=.2]
\begin{scope}[shift={(16,0)}]
 \draw[line width=1pt, color=black](0,-4)--(0,1);     
 \draw[color=black, fill=black, line width=.7pt] (0,1) circle (.25);
  \node at (0,1)[anchor=east]{$p$};
  \draw[color=black, fill=white, line width=.7pt] (0,-4) circle (.25);
 \node at (2,-1.5){$=$};
  \draw[line width=1pt, color=black,dashed](4,-3.9)--(4,1);  
\end{scope}\end{tikzpicture}
\tComma
\begin{tikzpicture}[scale=.2]
\begin{scope}[shift={(32,0)}]
\draw[line width=1pt, color=black, draw opacity=1] plot [smooth, tension=0.6] coordinates 
      {(-2,1)(-1,-2)(1,-2)(2,1) };
  \draw[line width=1pt, color=black](0,1.5)--(0,-4); 
     \draw[color=black, fill=white, line width=1pt] (0,0) circle (.4);  
        \draw[color=black, fill=white, line width=1pt] (0,-1) circle (.4);  
                \draw[color=black, fill=white, line width=1pt] (1.5,-1) circle (.4);  
                 \draw[color=black, fill=black, line width=.7pt] (-2,1) circle (.25);  
                   \node at (-2,1)[anchor=east]{$p$};  
                  \draw[color=black, fill=black, line width=.7pt] (2,1) circle (.25);
                    \node at (2,1)[anchor=west]{$p$};
                   \node at (4.5,-1.5){$=$};
                     \draw[line width=1pt, color=black](6.5,1.5)--(6.5,-4); 
\end{scope}
\end{tikzpicture}
\tDot
\end{align}
Note that the third condition implies that the antipode  $S: H\to H$ is an isomorphism with inverse $S^\inv=m^{2}\circ (1_H\oo S\oo S )\circ (p\oo 1_H\oo p):H\to H$. 
In diagrams we denote $S^\inv$ by a grey circle.  Identities \eqref{eq:antdef} and \eqref{eq:anthom} then imply  identities  analogous to \eqref{eq:anthom} and the following counterpart of \eqref{eq:antdef}
\begin{align}
	\nonumber
\begin{tikzpicture}[scale=.25]
\begin{scope}[shift={(-3,0)}]
\draw[line width=1pt, color=black] (0,3)--(0,2);
\draw[line width=1pt, color=black] (0,-3)--(0,-2);
\draw[line width=1pt, color=black, draw opacity=1] plot [smooth cycle, tension=0.6] coordinates 
      {(0,2)(-1.2, 1)(1.2,-1)(0,-2)(-1.2,-1)(1.2,1)};
   \draw[color=black, fill=gray, line width=1pt] (-1.2,1.3) circle (.4);     
\end{scope}
\node at (0,0){$=$};
\draw[line width=1pt, color=black] (1.5,3)--(1.5,1);
 \draw[color=black, fill=white, line width=1pt] (1.5,1) circle (.2); 
\draw[line width=1pt, color=black] (1.5,-3)--(1.5,-1);
 \draw[color=black, fill=white, line width=1pt] (1.5,-1) circle (.2); 
\node at (3,0){$=$};
\begin{scope}[shift={(6,0)}]
\draw[line width=1pt, color=black] (0,3)--(0,2);
\draw[line width=1pt, color=black] (0,-3)--(0,-2);
\draw[line width=1pt, color=black, draw opacity=1] plot [smooth cycle, tension=0.6] coordinates 
      {(0,2)(-1.2, 1)(1.2,-1)(0,-2)(-1.2,-1)(1.2,1)};
   \draw[color=black, fill=gray, line width=1pt] (1.2,1.3) circle (.4);    
\end{scope}
\end{tikzpicture}
\tDot
\end{align}

The condition that a Hopf monoid $H$ in $\mac$ has a pivotal structure is much less restrictive than involutivity. Many standard examples of Hopf algebras are pivotal Hopf monoids in $\mac=\vect_\F$.

\begin{example} \label{ex:hopfmon}$\quad$
\begin{compactenum}
\item If $H$  is an involutive Hopf algebra over a field $\F$, it is trivially pivotal with  $p=1\in H$. 
Pivotal structures on $H$ are  in bijection with central grouplike elements $p\in H$.

\item  If $H$ is a ribbon Hopf algebra over $\F$ with Drinfeld element $u= S(\low R 2)\low R 1$ and ribbon element $\nu$ satisfying $uS(u)=\nu^2$ and $\Delta(\nu)=(\nu\oo \nu)(R_{21}R)^\inv$, then $H$ is pivotal with $p=u^\inv\nu$.

\item 	If $H$ is a Hopf algebra over $\F$ with an invertible antipode, one can adjoin a grouplike element to form a pivotal Hopf algebra as follows. Let $G\cong \Z$ be the free group generated by the single element $p$. Then the tensor product $H \otimes \F[G]$ is a Hopf algebra with 
	\begin{align*}
		&( h \oo p^{m} ) \cdot ( k \oo p^{n} ) 
		= h S^{-2m}(k) \oo p^{m+n}
		&&\eta (1)
		= 1 \oo 1
		\\
		&\Delta ( h \oo p^{m} ) 
		= ( h_{(1)} \oo p^{m} ) \oo ( h_{(2)} \oo p^{m} )
		&&\epsilon ( h \oo p^{m} ) 
		= \epsilon(h)
		\\
		&S( h  \oo p^{m} ) =  S^{1+2m} (h) \oo p^{-m},
	\end{align*}
	and the element $1 \oo p\in H\oo \F[G]$ is pivotal.
	This shows that any Hopf algebra with invertible antipode is embedded in a pivotal Hopf algebra. 
	If $H$ is finite-dimensional with $\dim_\F(H)=d$, then adding the additional relation $p^{2d} = 1$ yields a finite-dimensional pivotal Hopf algebra containing $H$. 
	This finite-dimensional example is from~\cite{So}.
\item 
	
	For any complex semisimple finite-dimensional Lie algebra $\mathfrak{g}$, the Drinfeld-Jimbo algebra $U_{q}(\mathfrak{g})$ generated by elements $E_{i},F_{i}, K_{i}, K_{i}^{-1}$ as in~\cite{KlS} is a pivotal Hopf algebra. The pivotal element is the inverse of the grouplike element $K_{2\rho}$, where $\rho$ is the half-sum of the positive roots. A proof is given in~\cite[Ch.~6, Prop.~6]{KlS}.

\item It is clear from Example  4.~that $U_q(\mathfrak g)$ and $U_q(\mathfrak g)^{res}$ for $q$ an $n$th root of unity with  $q^4\neq 1$ are also pivotal Hopf algebras. 
The same holds for their Hopf-subalgebras  $U_q(\mathfrak b_+)$ and $U_q(\mathfrak b_-)$.

\item In a cartesian monoidal category $\mac$, where the tensor product is given by a product, Hopf monoids $H$ in  $\mac$ are precisely the group objects in $\mac$. This includes groups for $\mac=\mathrm{Set}$, topological groups for $\mac=\mathrm{Top}$,    abelian groups for $\mac=\mathrm{Grp}$ or $\mac=\mathrm{Ab}$ and strict 2-groups for $\mac=\mathrm{Cat}$ or $\mac=\mathrm{Grpd}$. 
If $\mac=G\mathrm{-Set}$ for a group $G$, group objects in $\mac$ are semidirect products $H\rtimes_\phi G$. If $\mac$ is concrete,  pivotal structures on $H$ correspond to central elements of $H$.

\item If $\mathcal A$ is a small category and $\mathcal B$ a symmetric monoidal category, then the category $\mathcal B^{\mathcal A}$ of functors from $\mathcal A$ to $\mathcal B$ and natural transformations between them is a symmetric monoidal category. In particular, this applies to  the cartesian monoidal category
$\mathcal B=\mathrm{Set}$ and the category $\mac=\mathrm{PSh}(\mathcal A)=\mathrm{Set}^{\mathcal A^{op}}$ of presheaves on $\mathcal A$. Hopf monoids in $\mathrm{PSh}(\mathcal A)$ correspond to functors $F:\mathcal A^{op}\to \mathrm{Grp}$, and these are involutive. In particular, Hopf monoids in  the category $\mac=\mathrm{sSet}=\mathrm{PSh}(\Delta)$ of simplicial sets are simplicial groups.

\end{compactenum}
\end{example}

The reason why we require a pivotal structure on a Hopf monoid $H$ is that it defines  an involutive automorphism of $H$.  This automorphism is  denoted by a double circle in diagrams and given by
\begin{align}\label{eq:tplusdef}
&T=m \circ (p\oo S)=m\circ (S^\inv\oo p): H\to H,\\
&\begin{tikzpicture}[scale=.2]
\draw[line width=1pt, color=black] (-1,3)--(-1,-3);
   \draw[color=black, fill=white, line width=.7pt] (-1,0) circle (.4); 
      \draw[color=black, fill=white, line width=.7pt] (-1,0) circle (.2);  
   \node at (1,0){$:=$};
   \begin{scope}[shift={(0.5,0)}]
   \draw[line width=1pt, color=black, draw opacity=1] plot [smooth, tension=0.6] coordinates 
      {(2,3)(2.5,.5)(4,0)(5.5,.5)(6,3)};
      \draw[line width=1pt, color=black] (4,-3)--(4,0);
         \draw[color=black, fill=black, line width=.7pt] (2,3) circle (.25); 
         \node at (2,3)[anchor=west]{$p$}; 
            \draw[color=black, fill=white, line width=1pt] (5.7,1.5) circle (.4);  
            \node at (7.5,0){$=$};
    \end{scope}
            \begin{scope}[shift={(8,0)}]
               \draw[line width=1pt, color=black, draw opacity=1] plot [smooth, tension=0.6] coordinates 
      {(2,3)(2.5,.5)(4,0)(5.5,.5)(6,3)};
            \draw[line width=1pt, color=black] (4,-3)--(4,0);
         \draw[color=black, fill=black, line width=.7pt] (6,3) circle (.25); 
         \node at (6,3)[anchor=west]{$p$}; 
            \draw[color=black, fill=gray, line width=1pt] (2.3,1.5) circle (.4);  
            \end{scope}
\end{tikzpicture}\nonumber\tDot
\end{align}

\begin{lemma}\label{lem:tprops} Let $(H,p)$ be a pivotal Hopf monoid in a symmetric monoidal category $\mac$. 
\begin{compactenum}
\item The morphism $T$ from \eqref{eq:tplusdef} is an involution: $T\circ T=1_H$.\\[-2ex]

\item It is an anti-comonoid morphism
\begin{align}\label{eq:tplusprop1}
&\begin{tikzpicture}[scale=.25]
\begin{scope}[shift={(-3,0)}]
\draw[line width=1pt, color=black] (0,3)--(0,0);
\draw[line width=1pt, color=black, draw opacity=1] plot [smooth, tension=0.6] coordinates 
      {(-2,-2)(-1.5,-.5)(0,0)(1.5,-.5)(2,-2)};
   \draw[color=black, fill=white, line width=1pt] (0,1.5) circle (.4);   
      \draw[color=black, fill=white, line width=1pt] (0,1.5) circle (.2);     
\end{scope}
\node at (0,0){$=$};
\begin{scope}[shift={(3,0)}]
\draw[line width=1pt, color=black] (0,3)--(0,2);
\draw[line width=1pt, color=black, draw opacity=1] plot [smooth, tension=0.6] coordinates 
      {(2,-2)(-1,1)(0,2)(1,1)(-2,-2)};
   \draw[color=black, fill=white, line width=1pt] (1.5,-1.5) circle (.4); 
      \draw[color=black, fill=white, line width=1pt] (1.5,-1.5) circle (.2); 
      \draw[color=black, fill=white, line width=1pt] (-1.5,-1.5) circle (.4);  
            \draw[color=black, fill=white, line width=1pt] (-1.5,-1.5) circle (.2);  
\end{scope}
\end{tikzpicture}
\tComma
\begin{tikzpicture}[scale=.25]
\draw[line width=1pt, color=black] (0,2)--(0,-3);
   \draw[color=black, fill=white, line width=1pt] (0,0) circle (.4);
      \draw[color=black, fill=white, line width=1pt] (0,0) circle (.2);
      \draw[color=black, fill=white, line width=1pt] (0,-3) circle (.2); 
      \node at (1.5,0){$=$};    
      \draw[line width=1pt, color=black] (2.5,2)--(2.5,-3);
      \draw[color=black, fill=white, line width=1pt] (2.5,-3) circle (.2); 
\end{tikzpicture}\tDot
\end{align}

\item It satisfies the diagrammatic identities
\begin{align}\label{eq:tplusprop2}
\begin{tikzpicture}[scale=.25]
\begin{scope}[shift={(-3,0)}]
\draw[line width=1pt, color=black] (0,-3)--(0,0);
\draw[line width=1pt, color=black, draw opacity=1] plot [smooth, tension=0.6] coordinates 
      {(-2,2)(-1.5,.5)(0,0)(1.5,.5)(2,2)};
   \draw[color=black, fill=white, line width=1pt] (0,-1.5) circle (.4);   
      \draw[color=black, fill=white, line width=1pt] (0,-1.5) circle (.2);     
\end{scope}
\node at (0,0){$=$};
\begin{scope}[shift={(3,0)}]
\draw[line width=1pt, color=black] (0,-3)--(0,-2);
\draw[line width=1pt, color=black, draw opacity=1] plot [smooth, tension=0.6] coordinates 
      {(2,2)(-1,-1)(0,-2)(1,-1)(-2,2)};
   \draw[color=black, fill=gray, line width=1pt] (1.5,1.5) circle (.4); 
      \draw[color=black, fill=white, line width=1pt] (-1.5,1.5) circle (.4);
            \draw[color=black, fill=white, line width=1pt] (-1.5,1.5) circle (.2);  
\end{scope}
\node at (6,0){$=$};
\begin{scope}[shift={(9,0)}]
\draw[line width=1pt, color=black] (0,-3)--(0,-2);
\draw[line width=1pt, color=black, draw opacity=1] plot [smooth, tension=0.6] coordinates 
      {(2,2)(-1,-1)(0,-2)(1,-1)(-2,2)};
   \draw[color=black, fill=white, line width=1pt] (1.5,1.5) circle (.4);
      \draw[color=black, fill=white, line width=1pt] (1.5,1.5) circle (.2); 
      \draw[color=black, fill=white, line width=1pt] (-1.5,1.5) circle (.4);  
\end{scope}
\end{tikzpicture}
\tComma
\begin{tikzpicture}[scale=.25]
\draw[line width=1pt, color=black] (0,2)--(0,-3);
   \draw[color=black, fill=white, line width=1pt] (0,0) circle (.4);
      \draw[color=black, fill=white, line width=1pt] (0,0) circle (.2);
      \draw[color=black, fill=white, line width=1pt] (0,2) circle (.2); 
      \node at (1.5,0){$=$};    
      \draw[line width=1pt, color=black] (2.5,2)--(2.5,-3);
      \draw[color=black, fill=black, line width=1pt] (2.5,2) circle (.2); 
      \node at (2.5,2)[anchor=west]{$p$};
\end{tikzpicture}
\tComma
\begin{tikzpicture}[scale=.25]
\begin{scope}[shift={(-3,0)}]
\draw[line width=1pt, color=black] (0,3)--(0,2);
\draw[line width=1pt, color=black] (0,-3)--(0,-2);
\draw[line width=1pt, color=black, draw opacity=1] plot [smooth cycle, tension=0.6] coordinates 
      {(0,2)(-1.2, 1)(-1.2,-1)(0,-2)(1.2,-1)(1.2,1)};
   \draw[color=black, fill=white, line width=1pt] (-1.3,0) circle (.4);   
      \draw[color=black, fill=white, line width=1pt] (-1.3,0) circle (.2);     
\end{scope}
\node at (0,0){$=$};
\begin{scope}[shift={(0.5,0)}]
\draw[line width=1pt, color=black] (1.5,3)--(1.5,1);
 \draw[color=black, fill=white, line width=1pt] (1.5,1) circle (.2); 
\draw[line width=1pt, color=black] (1.5,-3)--(1.5,-1);
 \draw[color=black, fill=black, line width=1pt] (1.5,-1) circle (.2); 
 \node at (1.5,-1) [anchor=west]{$p$};
 \end{scope}
\node at (4,0){$=$};
\begin{scope}[shift={(7,0)}] 
\draw[line width=1pt, color=black] (0,3)--(0,2);
\draw[line width=1pt, color=black] (0,-3)--(0,-2);
\draw[line width=1pt, color=black, draw opacity=1] plot [smooth cycle, tension=0.6] coordinates 
      {(0,2)(-1.2, 1)(1.2,-1)(0,-2)(-1.2,-1)(1.2,1)};
   \draw[color=black, fill=white, line width=1pt] (-1.2,1.3) circle (.4);  
      \draw[color=black, fill=white, line width=1pt] (-1.2,1.3) circle (.2);   
\end{scope}
\end{tikzpicture}
\tDot
\end{align}
\end{compactenum}
\end{lemma}

\begin{proof} The identities follow by direct diagrammatic computations. That $T$ is an involution is obtained  from the defining diagram for  $T$ in \eqref{eq:tplusdef}, the second identity in \eqref{eq:anthom} and
 the last identity in \eqref{eq:pivdiags}. The identities in \eqref{eq:tplusprop1} follow from the defining diagram for  $T$ in \eqref{eq:tplusdef}, the first and last  identities  in \eqref{eq:anthom} and the  first diagram in \eqref{eq:pivdiags}. The first two identities in \eqref{eq:tplusprop2} follow from \eqref{eq:tplusdef}, the second and third identity in \eqref{eq:anthom} and the associativity of the multiplication. The last two identities in \eqref{eq:tplusprop2}  follows from  \eqref{eq:tplusdef}, the associativity of the multiplication and the defining diagrams \eqref{eq:antdef} for the antipode.
\end{proof}

\subsection{Modules and comodules}

In this section, we summarise  background and notation for (co)modules over (co)monoids in a symmetric monoidal category.

\begin{definition} Let $\mac$ be a symmetric monoidal category, $(A, m, \eta)$ a monoid in $\mac$ and $(C,\Delta,\epsilon)$ a comonoid in $\mac$. 
\begin{compactenum}
\item A {\bf (left) module} over $A$ is an object $M\in \mathrm{Ob}\,\mac$ together with a morphism $\rhd: A\oo M\to M$ such that
$\rhd\circ (m\oo 1_M)=\rhd\circ (1_A\oo \rhd)$ and $\rhd\circ (\eta\oo 1_M)=1_M$.\\[-1ex]

\item A {\bf morphism of $A$-modules} from $(M,\rhd)$ to $(M',\rhd')$ is a morphism $f: M\to M'$ such that $\rhd'\circ (1_A\oo f)=f\circ \rhd$.\\[-1ex]

\item A {\bf (left) comodule} over $C$ is an object $M\in \mathrm{Ob}\,\mac$ together with a morphism $\delta: M\to C\oo M$ such that
$(\Delta\oo 1_M)\circ \delta=(1_C\oo\delta)\circ\delta$ and $(\epsilon\oo 1_M)\circ\delta=1_M$.\\[-1ex]

\item A {\bf morphism of $C$-comodules} from $(M,\delta)$ to $(M',\delta')$ is a morphism $f: M\to M'$ with $\delta'\circ f=(1_C\oo f)\circ \delta$.
\end{compactenum}
\end{definition}

 In the diagrams we denote (co)modules  by fat  coloured vertical lines and the action and coaction morphisms by vertices on these lines.  The defining conditions for left modules and left comodules  then  take the form
\begin{align}\label{eq:module}
\begin{tikzpicture}[scale=.3]
\begin{scope}[shift={(-3,0)}]
\draw[line width=1.5pt, color=blue](0,2)--(0,-2.5);
\draw[line width=1pt, color=black, draw opacity=1] plot [smooth, tension=0.6] coordinates 
      {(-1,2)(-1.5,1)(-2.5,1)(-3,2) };
\draw[line width=1pt, color=black, draw opacity=1] plot [smooth, tension=0.6] coordinates 
      {(-2,.9)(-2,0)(-1.5,-.5)(0,-1) };
\end{scope}
\node at (-1.5,0){$=$};
\begin{scope}[shift={(2,0)}]
\draw[line width=1.5pt, color=blue](0,2)--(0,-2.5);
\draw[line width=1pt, color=black, draw opacity=1] plot [smooth, tension=0.6] coordinates 
      {(-1,2)(-.8,1)(0,.5) };
      \draw[line width=1pt, color=black, draw opacity=1] plot [smooth, tension=0.6] coordinates 
      {(-3,2)(-2.7,.5)(0,-1) };
\end{scope}
\end{tikzpicture}
\tComma
\begin{tikzpicture}[scale=.3]
\begin{scope}[shift={(-3,0)}]
\draw[line width=1.5pt, color=blue](0,2)--(0,-2.5);
\draw[line width=1pt, color=black, draw opacity=1] plot [smooth, tension=0.6] coordinates 
      {(-2,2)(-1.8,.5)(0,-.5) };
      \draw [color=black, fill=white, line width=1pt] (-2,1.8) circle (.2); 
\end{scope}
\node at (-1.5,0){$=$};
\begin{scope}[shift={(0,0)}]
\draw[line width=1.5pt, color=blue](0,2)--(0,-2.5);
\end{scope}
\end{tikzpicture}
\tComma\qquad
\begin{tikzpicture}[scale=.3]
\begin{scope}[shift={(-3,0)}]
\draw[line width=1.5pt, color=blue](0,-2)--(0,2.5);
\draw[line width=1pt, color=black, draw opacity=1] plot [smooth, tension=0.6] coordinates 
      {(-1,-2)(-1.5,-1)(-2.5,-1)(-3,-2) };
\draw[line width=1pt, color=black, draw opacity=1] plot [smooth, tension=0.6] coordinates 
      {(-2,-.9)(-2,-0)(-1.5,.5)(0,1) };
\end{scope}
\node at (-1.5,0){$=$};
\begin{scope}[shift={(2,0)}]
\draw[line width=1.5pt, color=blue](0,-2)--(0,2.5);
\draw[line width=1pt, color=black, draw opacity=1] plot [smooth, tension=0.6] coordinates 
      {(-1,-2)(-.8,-1)(0,-.5) };
      \draw[line width=1pt, color=black, draw opacity=1] plot [smooth, tension=0.6] coordinates 
      {(-3,-2)(-2.7,-.5)(0,1) };
\end{scope}
\end{tikzpicture}\tComma
\begin{tikzpicture}[scale=.3]
\begin{scope}[shift={(-3,0)}]
\draw[line width=1.5pt, color=blue](0,-2)--(0,2.5);
\draw[line width=1pt, color=black, draw opacity=1] plot [smooth, tension=0.6] coordinates 
      {(-2,-2)(-1.8,-.5)(0,.5) };
      \draw [color=black, fill=white, line width=1pt] (-2,-1.8) circle (.2); 
\end{scope}
\node at (-1.5,0){$=$};
\begin{scope}[shift={(0,0)}]
\draw[line width=1.5pt, color=blue](0,-2)--(0,2.5);
\end{scope}
\end{tikzpicture}
\tDot
\end{align}
The defining properties of  left module and  left comodule morphisms
are given by the diagrams
\begin{align}\label{eq:modulemorph}
\begin{tikzpicture}[scale=.3]
\begin{scope}
\draw[line width=1.5pt, color=blue](0,2)--(0,-1.5);
\draw[line width=1.5pt, color=red](0,-1.5)--(0,-3.5);
\node at (0,2)[anchor=west] {$M$};
\node at (0,-3.5)[anchor=west] {$M'$};
\draw[line width=1pt, color=black, draw opacity=1] plot [smooth, tension=0.6] coordinates 
      {(-2,2)(-1.5,.5)(0,0) };
 \draw[color=violet, fill=violet, line width=1pt] (0,-1.5) circle (.4); 
 \node at (.4,-1.5)[anchor=west]{$f$};
 \end{scope}
 \node at (2,0){$=$};
 \begin{scope}[shift={(5,0)}]
 \draw[line width=1.5pt, color=blue](0,2)--(0,0);
\draw[line width=1.5pt, color=red](0,0)--(0,-3.5);
\node at (0,2)[anchor=west] {$M$};
\node at (0,-3.5)[anchor=west] {$M'$};
\draw[line width=1pt, color=black, draw opacity=1] plot [smooth, tension=0.6] coordinates 
      {(-2,2)(-2,0)(-1.5,-1)(0,-1.5) };
 \draw[color=violet, fill=violet, line width=1pt] (0,0) circle (.4); 
 \node at (.4,0)[anchor=west]{$f$};
 \end{scope}
\end{tikzpicture}
\tComma
\begin{tikzpicture}[scale=.3]
\begin{scope}
\draw[line width=1.5pt, color=blue](0,2)--(0,-1.5);
\draw[line width=1.5pt, color=red](0,-1.5)--(0,-3.5);
\node at (0,2)[anchor=west] {$M$};
\node at (0,-3.5)[anchor=west] {$M'$};
\draw[line width=1pt, color=black, draw opacity=1] plot [smooth, tension=0.6] coordinates 
      {(-2,-3.5)(-1.5,-1)(0,0) };
 \draw[color=violet, fill=violet, line width=1pt] (0,-1.5) circle (.4); 
 \node at (.4,-1.5)[anchor=west]{$f$};
 \end{scope}
 \node at (2,0){$=$};
 \begin{scope}[shift={(5,0)}]
 \draw[line width=1.5pt, color=blue](0,2)--(0,0);
\draw[line width=1.5pt, color=red](0,0)--(0,-3.5);
\node at (0,2)[anchor=west] {$M$};
\node at (0,-3.5)[anchor=west] {$M'$};
\draw[line width=1pt, color=black, draw opacity=1] plot [smooth, tension=0.6] coordinates 
      {(-2,-3.5)(-1.5,-2.5)(0,-1.5) };
 \draw[color=violet, fill=violet, line width=1pt] (0,0) circle (.4); 
 \node at (.4,0)[anchor=west]{$f$};
 \end{scope}
\end{tikzpicture}
\tDot
\end{align}

Right (co)modules over a (co)monoid in $\mac$, morphisms of right (co)modules and their diagrams are defined analogously.   
 
 Note also that if $H$ is a pivotal Hopf monoid in $\mac$,  the involution $T: H\to H$ from \eqref{eq:tplusdef} relates $H$-left and right comodule structures.
 As $T$ is an anti-comonoid morphism by Lemma \ref{lem:tprops}, 
 for any $H$-left comodule $(M,\delta)$, we obtain a $H$-right comodule $(M, \delta')$ by setting
 \begin{align}\label{eq:rightdef}
 \delta'=\tau_{H,M}\circ (T\oo 1_M)\circ \delta: M\to M\oo H.
 \end{align}

 Bimodules over a monoid $A$  are defined as triples $(M,\rhd, \lhd)$ such that $(M,\rhd)$ is a left module, $(M,\lhd)$ a right module over $A$ and $\lhd\circ (\rhd\oo 1_A)=\rhd\circ  (1_A\oo \lhd)$. Bicomodules over a comonoid $C$ as triples $(M,\delta_L,\delta_R)$ such that $(M,\delta_L)$ is a left comodule over $C$, $(M,\delta_R)$ a right comodule over $C$ and $(\delta_L\oo 1_C)\circ\delta_R=(1_C\oo\delta_R)\circ\delta_L$. Morphisms of bimodules or of bicomodules are morphisms that are both left and right module morphisms or left and right comodule morphisms.
The compatibility conditions between left and right actions and left and right coactions are given by the diagrams
\begin{align}
	\nonumber
\begin{tikzpicture}[scale=.25]
\begin{scope}
\draw[line width=1.5pt, color=blue](0,2.5)--(0,-2.5);
\draw[line width=1pt, color=black, draw opacity=1] plot [smooth, tension=0.6] coordinates 
      {(-2,2.5)(-1.5,1)(0,.5) };
      \draw[line width=1pt, color=black, draw opacity=1] plot [smooth, tension=0.6] coordinates 
      {(2,2.5)(1.5,-.5)(0,-1)};
      \end{scope}
      \node at (3,0){$=$};
      \begin{scope}[shift={(6,0)}]
\draw[line width=1.5pt, color=blue](0,2.5)--(0,-2.5);
\draw[line width=1pt, color=black, draw opacity=1] plot [smooth, tension=0.6] coordinates 
      {(-2,2.5)(-1.5,-.5)(0,-1) };
      \draw[line width=1pt, color=black, draw opacity=1] plot [smooth, tension=0.6] coordinates 
      {(2,2.5)(1.5,1)(0,.5)};
      \end{scope}
\end{tikzpicture}
\tComma
\begin{tikzpicture}[scale=.25]
\begin{scope}
\draw[line width=1.5pt, color=blue](0,2.5)--(0,-2.5);
\draw[line width=1pt, color=black, draw opacity=1] plot [smooth, tension=0.6] coordinates 
      {(-2,-2.5)(-1.5,-1)(0,-.5) };
      \draw[line width=1pt, color=black, draw opacity=1] plot [smooth, tension=0.6] coordinates 
      {(2,-2.5)(1.5,.5)(0,1)};
      \end{scope}
      \node at (3,0){$=$};
      \begin{scope}[shift={(6,0)}]
\draw[line width=1.5pt, color=blue](0,2.5)--(0,-2.5);
\draw[line width=1pt, color=black, draw opacity=1] plot [smooth, tension=0.6] coordinates 
      {(-2,-2.5)(-1.5,.5)(0,1) };
      \draw[line width=1pt, color=black, draw opacity=1] plot [smooth, tension=0.6] coordinates 
      {(2,-2.5)(1.5,-1)(0,-.5)};
      \end{scope}
\end{tikzpicture}
\tDot
\end{align}

In the following, we will need to consider  (co)invariants of (co)modules over  Hopf monoids in $\mac$.
As we do not restrict attention to  abelian categories or even categories with  zero objects, we impose that the category $\mac$ has all coequalisers and  equalisers and define them 
 as coequalisers  and equalisers. 

\begin{definition}  \label{def:invcoinv} Let $\mac$ be a symmetric monoidal category that has all equalisers and coequalisers and $H$ a Hopf monoid in $\mac$.
\begin{compactenum}
\item The {\bf invariants} of an $H$-left module $(M,\rhd)$ are the coequaliser $(M^{H},\pi)$ of $\rhd$ and $\epsilon\circ 1_{M}$: 
\begin{align}
	\nonumber
\xymatrix{  H\oo M  \ar@<-.5ex>[r]_{\quad\epsilon\oo 1_M} \ar@<.5ex>[r]^{\quad\rhd} & M \ar[r]^{\pi} & M^H.
}
\end{align} 

	\item The {\bf coinvariants} of an $H$-left comodule $(M,\delta)$ are the equaliser $(M^{coH},\iota)$ of $\delta$ and $\eta \oo 1_{M}$:
\begin{align}
	\nonumber
\xymatrix{ M^{coH} \ar[r]^{\iota} & M  \ar@<-.5ex>[r]_{\eta\oo 1_M\quad} \ar@<.5ex>[r]^{\delta\quad} & H\oo M.
}
\end{align}

\end{compactenum}
\end{definition}

(Co)invariants of $H$-right (co)modules are defined analogously.  
Note that for $\mac=\vect_\F$ Definition~\ref{def:invcoinv} of coinvariants coincides with the usual definition as the subset $M^{coH}=\{m\in M\mid \delta(m)=1\oo m\}$. This is not the case for the definition of invariants in terms of the coequaliser, which yields a quotient of the $H$-module $M$, not a linear subspace.

It follows directly from Definition \ref{def:invcoinv}  that morphisms of $H$-(co)modules  induce morphisms between their (co)invariants. 
In the following we say that two morphisms $f: M\to M'$ of $H$-(co)modules {\em agree on the (co)invariants} if the induced morphisms between the (co)invariants agree.

\begin{lemma} \label{lem:inducedinvcoinv}Let $\mac$ be a symmetric monoidal category that has all equalisers and coequalisers and  $H$ a Hopf monoid in $\mac$.
\begin{compactenum}
\item For every $H$-module morphism $f: (M,\rhd)\to (M',\rhd')$ there is  a unique morphism\linebreak  $f^{H}: M^{H}\to M'^{H}$ with $\ f^{H}\circ \pi=\pi'\circ f$.
\item For every $H$-comodule morphism $f: (M,\delta)\to (M',\delta')$ there is a unique morphism\linebreak $f^{coH}: M^{coH}\to M'^{coH}$ with $\iota'\circ f^{coH}=f\circ \iota$.
\end{compactenum}
\end{lemma}

\begin{proof} If $f: (M,\delta)\to (M',\delta')$ is a morphism of $H$-comodules, then $\delta'\circ f=(1_H\oo f)\circ \delta$ and
\begin{align*}
\delta'\circ (f\circ \iota)=(1_H\oo f)\circ \delta\circ \iota=(1_H\oo f)\circ (\eta\oo 1_M)\circ \iota=(\eta\oo 1_M)\circ (f\circ \iota).
\end{align*}
By the universal property of the equaliser $M'^{coH}$, there is a unique morphism $f^{coH}: M^{coH}\to M'^{coH}$ with $\iota'\circ f^{coH}=f\circ \iota$.
The proof for invariants is analogous.
\end{proof}

In the following, we also consider morphisms  that are both, module and comodule morphisms, and  impose both, invariance and coinvariance.  As we do not  restrict attention to categories with a zero object, we implement this with the (non-abelian) notion of an image, see for instance  \cite[I.10]{Mi}. 

An  {\em image} of a morphism $f: C\to C'$ in $\mac$ is a pair $(P,I)$ of a monomorphism $I: \mathrm{im}(f)\to C'$ and a morphism $P: C\to \mathrm{im}(f)$ with $I\circ P=f$ such that for any pair $(Q,J)$ of a monomorphism $J: X\to C'$ and a morphism $Q: C\to X$ with $J\circ Q=f$ there is a unique morphism $v:\mathrm{im}(f)\to X$ with $I=J\circ v$. Note that if $\mac$ has all equalisers, then the morphism $P:C\to \mathrm{im}(f)$ is an epimorphism  \cite[Sec I.10]{Mi}. 
If $\mac$ is finitely complete and cocomplete, images  are obtained as equalisers of cokernel pairs
\cite[Def 5.1.1]{KaS}, and their existence
is guaranteed. In particular, this holds for the categories $\mathrm{Set}$, $\mathrm{Top}$, $\mathrm{Ab}$, $\mathrm{Grp}$, $\mathrm{Cat}$, the category  $G\mathrm{-Set}$ for a group $G$ and the category $\mathrm{PSh}(\mathcal A)$ for a small category $\mathcal A$ from Example \ref{ex:hopfmon}, 6.~and 7.
Note also  that if $\mac$ is abelian, this  coincides with the usual notion of an image in an abelian category.

\begin{definition} \label{def:biinv}Let $\mac$ be a symmetric monoidal category that is finitely complete and cocomplete and $H$ a Hopf monoid in $\mac$. 
 The {\bf biinvariants} of an object $M$ in $\mac$  that is both an $H$-left module and  $H$-left comodule are the image  of the morphism $\pi\circ\iota: M^{coH}\to M^H$  
\begin{align}\label{eq:imagedef}
\xymatrix{ M^{coH} \ar[rd]_P\ar[r]^\iota & M \ar[r]^\pi & M^H\\
& M_{inv}:=\mathrm{im}(\pi\circ \iota). \ar[ru]_{I}
}
\end{align}
\end{definition}

\begin{remark} Analogously, one could define biinvariants of  $M$ as the {\em coimage} of the morphism $\pi\circ \iota: M^{coH}\to M^H$, defined as the image of the morphism $\iota\circ \pi: M^H\to M^{coH}$ in the opposite category $\mac^{op}$ in \cite[Sec I.10]{Mi}. This amounts to replacing $H$ by the corresponding Hopf monoid 
 $H^*$ in $\mac^{op}$ that is the same object, but  with  multiplication and comultiplication and unit and counit exchanged.  The
 invariants and coinvariants of an $H$-module and $H$-comodule  $M$  correspond to the coinvariants and invariants of the associated $H^*$-comodule and $H^*$-module $M$ in $\mac^{op}$. Hence, the coimage of  $\pi\circ \iota: M^{coH}\to M^H$ in $\mac$ coincides with the biinvariants of $M$ in $\mac^{op}$.
\end{remark}

By Lemma \ref{lem:inducedinvcoinv} morphisms of $H$-(co)modules induce morphisms between their (co)invariants. A similar statement holds for   isomorphisms that respect both,  the  $H$-module and $H$-comodule structure, and for  the associated biinvariants.

\begin{lemma}\label{lem:maclemma} Let $H$ be a Hopf monoid in a   finitely complete and cocomplete symmetric monoidal category $\mac$. Then for any  isomorphism $\phi: M\to M'$  of $H$-modules and $H$-comodules there is a unique morphism  $\phi_{inv}: M_{inv}\to M'_{inv}$ with $\pi'\circ \phi\circ \iota=I'\circ \phi_{inv}\circ P$, and $\phi_{inv}$ is an isomorphism. 
\end{lemma}

\begin{proof}
By Lemma \ref{lem:inducedinvcoinv} the isomorphism $\phi: M\to M'$ induces isomorphisms
$\phi^{coH}: M^{coH}\to M'^{coH}$ and  $\phi^{H}: M^H\to M'^H$ such that the following diagram commutes 
\begin{align}\label{eq:invcoinvcomm}
\xymatrix{ M^{coH}\ar[d]_{\phi^{coH}} \ar[r]^\iota & M \ar[d]^{\phi}\ar[r]^\pi & M^H \ar[d]^{\phi^H}\\
M'^{coH} \ar[r]_{\iota'} & M' \ar[r]_{\pi'} & M'^H.
}
\end{align}
Let $(M_{inv}, I, P)$ and $(M'_{inv}, I',P')$  be the biinvariants for $M$ and $M'$ from \eqref{eq:imagedef}.
Then  by \eqref{eq:imagedef} and \eqref{eq:invcoinvcomm}  the morphisms $ J=(\phi^H)^\inv\circ I': M'_{inv}\to M^H$ and $Q=P'\circ \phi^{coH}: M^{coH}\to M'_{inv}$ satisfy
$J \circ Q=\pi\circ \iota$, and there is a unique morphism $\phi_{inv}: M_{inv}\to M'_{inv}$ with
$I=J \circ \phi_{inv}=(\phi^H)^\inv\circ I'\circ \phi_{inv}$.
 This condition implies 
$I'\circ \phi_{inv}\circ P=\pi'\circ \phi\circ \iota$ by    \eqref{eq:imagedef} and \eqref{eq:invcoinvcomm}. 
As $P$ is an epimorphism and $I'$ a monomorphism, the last equation determines $\phi_{inv}$ uniquely. By applying this to the inverse of $\phi$, one finds that $\phi_{inv}$ is an isomorphism.
\end{proof}

\begin{example} $\quad$
\begin{compactenum}
\item A Hopf monoid $H$  in $\mathrm{Set}$ is a group,  $H$-modules are $H$-sets and an $H$-comodule is a set $M$ together with a map $f: M\to H$, whose graph is the morphism $\delta: M\to H\times M$.
\begin{compactitem}
\item The invariants of an $H$-set $M$ are the orbit space $M^H=M/H$  with the canonical surjection $\pi: M\to  M/H$. 
\item The coinvariants of an $H$-comodule $(M,f)$ are the subset $M^{coH}=\{m\in M\mid f(m)=1\}$ with the inclusion $\iota: M^{coH}\to M$.  
\item The biinvariants of an $H$-module and comodule $M$ 
 are the set $M_{inv}=\pi(M^{coH})$ together with the inclusion $I: M_{inv}\to M/H$ and the surjection $P: M^{coH}\to M_{inv}$, $m\mapsto \pi(m)$.
 \end{compactitem}
 \item If $\mac$ is abelian,  the invariants of an $H$-module $(M,\rhd)$ are the 
 cokernel of the morphism $\rhd-\epsilon\oo 1_M$, the coinvariants of an $H$-comodule $(M,\delta)$ are the kernel of the morphism $\delta-\eta\oo 1_M$, and the biinvariants are the  image of $\pi\circ \iota: M^{coH}\to M^H$ in  $\mac$.
 
 \item If $H$ is a finite-dimensional semisimple Hopf algebra in $\mac=\vect_\F$ with $\mathrm{char}(\F)=0$, then 
 \begin{compactitem}
 \item  the invariants of an $H$-module $(M,\rhd)$ are  $M^H=\{m\in M\mid h\rhd m=\epsilon(h) m\;\forall h\in H\}$ with  $\pi: M\to M^{H}$, $m\mapsto \ell \rhd m$, where $\ell$ is the normalised Haar integral of $H$,
\item  the coinvariants of an $H$-comodule $(M,\delta)$ are $M^{coH}=\{m\in M\mid \delta(m)=1\oo m\}$ with the inclusion $\iota: M^{coH}\to M$, 
\item the biinvariants of an $H$-module and $H$-comodule $M$ are  $M_{inv}=M^H\cap M^{coH}$ with the inclusion $I: M_{inv}\to M^{H}$ and the surjection $P: M^{coH}\to M^{inv}$, $m\mapsto \ell\rhd m$.
 \end{compactitem}
\end{compactenum}
\end{example}

\subsection{Hopf modules and Yetter-Drinfeld modules}

For a bimonoid $H$ in $\mac$ an object $M$ in $\mac$ that has both a $H$-module structure $\rhd: H\oo M\to M$ and a $H$-comodule structure $\delta: M\to H\oo M$ one can impose certain compatibility conditions between $\rhd$ and $\delta$. This leads to the notions of Hopf modules~\cite{LS} and Yetter-Drinfeld modules~\cite{Y}. 

\begin{definition} Let $H$ be a bimonoid in a symmetric monoidal category $\mac$. A {\bf left-left Hopf module} over $H$ is a triple $(M,\rhd,\delta)$  such that $(M,\rhd)$ is a left module over $H$, $(M,\delta)$ is a left comodule over $H$ and  the following compatibility condition is satisfied
\begin{align}
	\label{eq:lefthopf}
\begin{tikzpicture}[scale=.3]
\begin{scope}[shift={(-8,0)}]
\draw[line width=1.5pt, color=blue](0,2)--(0,-2.5);
\draw[line width=1pt, color=black, draw opacity=1] plot [smooth, tension=0.6] coordinates 
      {(-2,2)(-1.5,.5)(0,0) };
      \draw[line width=1pt, color=black, draw opacity=1] plot [smooth, tension=0.6] coordinates 
      {(0,-.5)(-1.5,-1)(-2,-2.5) };
\end{scope}
\node at (-7,0){$=$};
\begin{scope}[shift={(-3,0)}]
\draw[line width=1.5pt, color=blue](0,2)--(0,-2.5);
\draw[line width=1pt, color=black, draw opacity=1] plot [smooth, tension=0.6] coordinates 
      {(0,-1) (-1.3,0)(-2,.5)(-2.7,0)(-3,-.5)(-2.7,-1) (-2,-1.5)(-1.3, -1) (0,1) };
      \draw[line width=1pt](-2,2)--(-2,.5);
            \draw[line width=1pt](-2,-1.5)--(-2,-2.5);
\end{scope}
\end{tikzpicture}
\end{align}
A morphism of left-left $H$-Hopf modules from $(M,\rhd,\delta)$ to $(M',\rhd',\delta')$ 
is  a morphism $f: M\to M$ that is a morphism of $H$-left modules and of $H$-left comodules.
\end{definition}

There are analogous definitions of left-right $H$-Hopf modules,  right-left $H$-Hopf modules and  right-right $H$-Hopf modules, where the left action  of $H$ is replaced by a right action   or the left coaction 
by a right coaction. In these cases the  diagrams that replace \eqref{eq:lefthopf}  read
\begin{align}
	\nonumber
&\begin{tikzpicture}[scale=.3]
\begin{scope}[shift={(-3,0)}]
\draw[line width=1.5pt, color=blue](0,2)--(0,-2.5);
\draw[line width=1pt, color=black, draw opacity=1] plot [smooth, tension=0.6] coordinates 
      {(-2,2)(-1.5,.5)(0,0) };
      \draw[line width=1pt, color=black, draw opacity=1] plot [smooth, tension=0.6] coordinates 
      {(0,-.5)(1.5,-1)(2,-2.5) };
\end{scope}
\node at (0,0){$=$};
\begin{scope}[shift={(4,0)}]
\draw[line width=1.5pt, color=blue](0,2)--(0,-2.5);
\draw[line width=1pt, color=black, draw opacity=1] plot [smooth, tension=0.6] coordinates 
      {(1.5,-1.5)(1.5,-2.5)};
      \draw[line width=1pt, color=black, draw opacity=1] plot [smooth, tension=0.6] coordinates 
      {(-2,2)(-2,1)};
      \draw[line width=1pt, color=black, draw opacity=1] plot [smooth, tension=0.6] coordinates 
      {(0,1)(2,-.5)(1.5,-1.5)(-2,1)(-2.5,-.5)(0,-2) };
\end{scope}
\end{tikzpicture}
\tComma
\begin{tikzpicture}[scale=.3]
\begin{scope}[shift={(-3,0)}]
\draw[line width=1.5pt, color=blue](0,2)--(0,-2.5);
\draw[line width=1pt, color=black, draw opacity=1] plot [smooth, tension=0.6] coordinates 
      {(2,2)(1.5,.5)(0,0) };
      \draw[line width=1pt, color=black, draw opacity=1] plot [smooth, tension=0.6] coordinates 
      {(0,-.5)(-1.5,-1)(-2,-2.5) };
\end{scope}
\node at (0,0){$=$};
\begin{scope}[shift={(3,0)}]
\draw[line width=1.5pt, color=blue](0,2)--(0,-2.5);
\draw[line width=1pt, color=black, draw opacity=1] plot [smooth, tension=0.6] coordinates 
      {(-1.5,-1.5)(-1.5,-2.5)};
      \draw[line width=1pt, color=black, draw opacity=1] plot [smooth, tension=0.6] coordinates 
      {(2,2)(2,1)};
      \draw[line width=1pt, color=black, draw opacity=1] plot [smooth, tension=0.6] coordinates 
      {(0,1)(-2,-.5)(-1.5,-1.5)(2,1)(2.5,-.5)(0,-2) };
\end{scope}
\end{tikzpicture}
\tComma
\begin{tikzpicture}[scale=.3]
\begin{scope}[shift={(-3,0)}]
\draw[line width=1.5pt, color=blue](0,2)--(0,-2.5);
\draw[line width=1pt, color=black, draw opacity=1] plot [smooth, tension=0.6] coordinates 
      {(2,2)(1.5,.5)(0,0) };
      \draw[line width=1pt, color=black, draw opacity=1] plot [smooth, tension=0.6] coordinates 
      {(0,-.5)(1.5,-1)(2,-2.5) };
\end{scope}
\node at (0,0){$=$};
\begin{scope}[shift={(1,0)}]
\draw[line width=1.5pt, color=blue](0,2)--(0,-2.5);
\draw[line width=1pt, color=black, draw opacity=1] plot [smooth, tension=0.6] coordinates 
      {(0,-1) (1.3,0)(2,.5)(2.7,0)(3,-.5)(2.7,-1) (2,-1.5)(1.3, -1) (0,1) };
      \draw[line width=1pt](2,2)--(2,.5);
            \draw[line width=1pt](2,-1.5)--(2,-2.5);
\end{scope}
\end{tikzpicture}
\tDot
\end{align}
Analogously, a Hopf bimodule over a Hopf monoid $H$ in $\mac$ is defined as a pentuple $(M,\rhd,\lhd,\delta_L,\delta_R)$ such that
$(M,\rhd,\lhd)$ is a $H$-bimodule, $(M,\delta_L,\delta_R)$ is a $H$-bicomodule and $(M,\rhd,\delta_L)$, $(M,\rhd,\delta_R)$, $(M,\lhd,\delta_L)$, $(M,\lhd,\delta_R)$
are left-left, left-right, right-left and right-right $H$-Hopf modules, respectively.

\begin{example} \label{ex:hopfbimodule} Let $H$ be a Hopf monoid in $\mac$.
\begin{compactenum}
\item  $H$  is a Hopf bimodule over itself with
$\rhd=\lhd=m: H\oo H\to H $, $\delta_L=\delta_R=\Delta: H\to H\oo H$.

\item For any object $M$ in $\mac$, the object $H\oo M$ is a Hopf bimodule with 
\begin{align*}
&\rhd=m\oo 1_M: H\oo H\oo M\to H\oo M & &\lhd=(m\oo 1_M)\circ (1_H\oo \tau) : H\oo M\oo H\to H\oo M\\
&\delta_L=\Delta\oo 1_M: H\oo M\to H\oo H\oo M, & &\delta_R=(1_H\oo \tau)\circ(\Delta\oo 1_M):H\oo M\to H\oo  M\oo H.
\end{align*}
Hopf bimodules and the associated Hopf modules of this type are called trivial. 

\end{compactenum}
\end{example}

The main examples of Hopf (bi)modules  considered in this article are the trivial ones from Example \ref{ex:hopfbimodule}. By the fundamental theorem of Hopf modules \cite[Prop.~1]{LS}, any Hopf module $M$ over a Hopf monoid $H$ in $\mac=\vect_\F$ is isomorphic to the trivial Hopf module $H\oo M^{coH}$.
For a Hopf algebra in an  abelian rigid braided monoidal category an analogous result was shown in \cite[Thm~1.2]{Ly95b}.
 
 Note also that for finite-dimensional Hopf algebras in $\mac=\vect_\F$, the notion of a Hopf module coincides with the notion of a module over the {\em Heisenberg double} of $H$, which exists in different versions corresponding to left-left, left-right, right-left and right-right Hopf modules.

Just as Hopf modules can be viewed as categorical analogues of modules over Heisenberg doubles, there is also an analogue of modules over the Drinfeld double $D(H)$, namely Yetter-Drinfeld modules. Just as Hopf modules, they  come in four variants, for left and right module and comodule structures. We restrict attention to left module and left comodule structures.

\begin{definition} Let $H$ be a bimonoid in a symmetric monoidal category $\mac$.  A {\bf left-left Yetter-Drinfeld module} over $H$ is a triple $(M,\rhd,\delta)$ such that $(M,\rhd)$ is a left module over $H$, $(M,\delta)$ a left comodule over $H$ and the following compatibility condition is satisfied
\begin{align}\label{eq:ydcomp}
&\begin{tikzpicture}[scale=.3]
\begin{scope}[shift={(-3,0)}]
\draw[line width=1.5pt, color=blue](0,2.5)--(0,-2.5);
\draw[line width=1pt, color=black, draw opacity=1] plot [smooth, tension=0.6] coordinates 
      {(0,-.5) (-1.3,.5)(-2,1)(-2.7,.5)(-3,0)(-2.7,-.5) (-2,-1)(-1.3, -.5) (0,1.5) };
      \draw[line width=1pt](-2,2.5)--(-2,1);
            \draw[line width=1pt](-2,-1)--(-2,-2.5);
\end{scope}
\node at (-1,0){$=$};
\begin{scope}[shift={(3,0)}]
\draw[line width=1.5pt, color=blue](0,2.5)--(0,-2.5);
\draw[line width=1pt, color=black, draw opacity=1] plot [smooth, tension=0.6] coordinates 
      {(0,.5) (-2.5,1)(-2,2)(-1.5,1)(-1.5,-1)(-2,-2) (-2.5,-1) (0,-.5) };
      \draw[line width=1pt](-2,2)--(-2,2.5);
            \draw[line width=1pt](-2,-2)--(-2,-2.5);
\end{scope}
\end{tikzpicture}
\tDot
\end{align}
A {\bf morphism of $H$-left-left Yetter-Drinfeld modules} from $(M,\rhd,\delta)$ to $(M',\rhd',\delta')$ is a morphism $f: M\to M'$ that is a morphism of $H$-left modules and of $H$-left comodules.
\end{definition}

\section{Background on ribbon graphs}
\label{sec:ribbon}

In this section, we summarise the required background on {\em embedded graphs} or {\em ribbon graphs}. For more details,  see for instance \cite{LZ}. All graphs considered in this article are finite, but we allow  loops, multiple edges and univalent vertices.

\subsection{Paths}
\label{subsec:paths}
Paths in a graph are most easily described by orienting the edges of $\Gamma$ and considering the free groupoid generated by the resulting directed graph. Note that different choices of orientation yield isomorphic groupoids. 

 \begin{definition} \label{def:pathgroupoid}The {\bf path groupoid} $\mathcal G_\Gamma$ of a graph $\Gamma$ is the free groupoid generated by $\Gamma$. A {\bf path} in $\Gamma$ from a vertex $v$ to a vertex $w$ is a morphism $\gamma: v\to w$ in $\mathcal G_\Gamma$.
 \end{definition}

 The objects of $\mathcal G_\Gamma$ are the vertices of $\Gamma$. A morphism from $v$ to $w$ in $\mathcal G_\Gamma$ is a finite sequence $p=\alpha_1^{\epsilon_1}\circ ...\circ \alpha_n^{\epsilon_n}$, $\epsilon_i\in\{\pm 1\}$ of oriented edges $\alpha_i$ and their inverses such that the starting vertex of the first edge 
$\alpha_{n}^{\epsilon_{n}}$
 is $v$, the target vertex of the last edge 
$\alpha_{1}^{\epsilon_{1}}$
 is $w$, and the starting vertex of each edge in the sequence is the target vertex of the preceding edge. These sequences are taken with the relations  $\alpha\circ \alpha^\inv=1_{t(\alpha)}$ and $\alpha^\inv\circ \alpha=1_{s(e)}$, where $\alpha^\inv$ denotes the edge $\alpha$ with the reversed orientation, $s(\alpha)$ the starting and $t(\alpha)$ the target vertex of $\alpha$ and we set $s(\alpha^{\pm 1})=t(\alpha^{\mp 1})$. 
 
 An edge $\alpha\in E$ with  $s(\alpha)=t(\alpha)$ is called a {\em loop}, and a path  $\rho\in \mathcal G_\Gamma$  is called {\em closed} if it is an automorphism of a vertex.  
 We call a path $\rho\in\mathcal G_\Gamma$ a {\em subpath} of a path $\gamma \in \mathcal G_\Gamma$ if  the expression for $\gamma$ as a reduced word in $E$  is of the form $\gamma=\gamma_1\circ\rho\circ\gamma_2$ with (possibly empty) reduced words $\gamma_1,\gamma_2$. We call it  a {\em proper subpath} of $\gamma$ if $\gamma_1,\gamma_2$ are not both empty. We say that two paths $\rho,\gamma\in \mathcal G_\Gamma$ {\em overlap} if there is an edge in $\Gamma$ that is traversed by both $\rho$ and $\gamma$, and by both  in the same direction.

\subsection{Ribbon graphs}
The graphs we consider have additional structure. They are called
 {\em ribbon graphs}, \emph{fat graphs} or \emph{embedded graphs} and give a combinatorial description of oriented surfaces with or without boundary. 

 \begin{definition} A {\bf ribbon graph} is a graph together with a cyclic ordering of the edge ends at each vertex. 
 A {\bf ciliated vertex} in a ribbon graph is a vertex together with  a linear ordering of the incident  edge ends  that is compatible with their cyclic ordering.
 \end{definition}
 
A ciliated vertex in a ribbon graph is obtained  by selecting one of its incident edge ends  as the starting end of the linear ordering. We indicate 
this in figures by assuming the counterclockwise cyclic ordering in the plane and inserting a line, the {\em cilium},  that separates the edges of minimal and maximal order, as shown in Figure \ref{fig:ciliated}. We say that an edge end $\beta$ at a ciliated vertex $v$ is {\em between} two edge ends $\alpha$ and $\gamma$ incident at $v$ if $\alpha<\beta<\gamma$ or $\gamma<\beta<\alpha$. We denote by $\st(\alpha)$ the starting end and by $\ta(\alpha)$ the target end of a directed edge $\alpha$.

The cyclic ordering of the edge ends at each vertex  allows  one to thicken the edges of a ribbon graph to  strips or ribbons and its vertices to polygons.
 It also equips the ribbon graph with the notion of a face. One says that a path in a ribbon graph $\Gamma$ turns {\em maximally left} at a vertex $v$ if it
 enters $v$ through an edge end $\alpha$ and leaves it through an edge end $\beta$ that comes directly before $\alpha$ with respect to 
  the cyclic ordering at $v$. 

 \begin{definition}\label{def:face} Let $\Gamma$ be a ribbon graph. 
 \begin{compactenum}
 \item A {\bf face path} in $\Gamma$ is a path  that turns maximally left at each  vertex in the path and traverses each edge at most once in each direction. 
 \item A {\bf ciliated face} in $\Gamma$  is a closed face path whose cyclic permutations are also face paths.
 \item A {\bf face} of $\Gamma$ is an equivalence class of ciliated faces under  cyclic permutations. 
 \end{compactenum}
 \end{definition}

 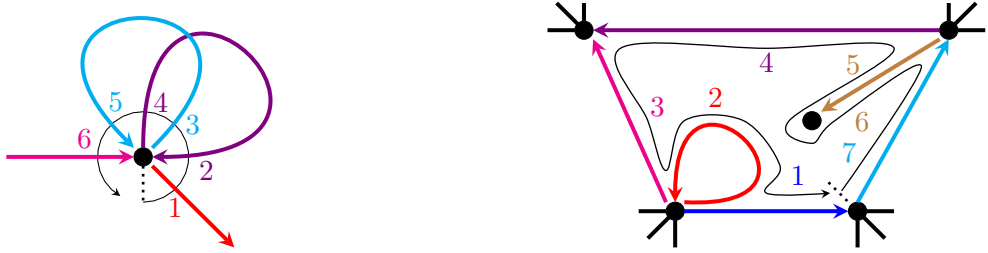
\begin{figure}
	\centering
		\begin{tikzpicture}[scale=.6]
			\draw [color=black, fill=black] (0,0) circle (.2); 
			\draw[color=black, style=dotted, line width=1pt] (0,-.2)--(0,-1);
			\draw [black,->,>=stealth,domain=-90:240] plot ({cos(\x)}, {sin(\x)});
			\draw[color=red, line width=1.5pt, ->,>=stealth] (.2,-.2)--(2,-2);
			\draw[color=violet, line width=1.5pt, <-,>=stealth] (.2,0)..controls (6,0) and (0,6).. (0,.2);
			\draw[color=cyan, line width=1.5pt, ->,>=stealth] (.2,.2)..controls (4,4) and (-4,4).. (-.2,.2);
			\draw[color=magenta, line width=1.5pt, <-,>=stealth] (-.2,0)--(-3,0);
			\node at (.7,-.7)[color=red, anchor=north]{${1}$};
			\node at (1,-.3)[color=violet, anchor=west]{${2}$};
			\node at (.7,.7)[color=cyan, anchor=west]{$3$};
			\node at (0,1.2)[color=violet, anchor=west]{${4}$};
			\node at (-.6,.8)[color=cyan, anchor=south]{$5$};
			\node at (-.9,.4)[color=magenta, anchor=east]{$6$};
		\end{tikzpicture}
		\qquad\qquad
				\begin{tikzpicture}[scale=.6]
		\begin{scope}[shift={(-6,0)}]
		\draw [color=black, fill=black] (-2,0) circle (.2); 
		\draw [color=black, fill=black] (2,0) circle (.2); 
		\draw [color=black, fill=black] (1,2) circle (.2); 
		\draw [color=black, fill=black] (4,4) circle (.2); 
		\draw [color=black, fill=black] (-4,4) circle (.2); 
		\draw[color=black, line width=1pt, style=dotted](1.9,.1)--(1.3,.7);
		\draw[color=blue, line width=1.5pt, ->,>=stealth] (-1.8,0)--(1.8,0);
		\node at (.7,1.2)[color=blue, anchor=north]{$1$};
		\draw[color=red, line width=1.5pt, ->,>=stealth] (-1.8,.2).. controls (2,0) and (-2,4).. (-2,.2);
		\node at (-1.5,2.5)[color=red, anchor=west]{$2$};
		\draw[color=magenta, line width=1.5pt, ->,>=stealth] (-2.2,.2)--(-3.8,3.8);
		\node at (-2,2.3)[color=magenta, anchor=east]{$3$};
		\draw[color=cyan, line width=1.5pt, ->,>=stealth] (2,.2)--(4,3.8);
		\node at (2.2,1.3)[color=cyan, anchor=east]{$7$};
		\draw[color=violet, line width=1.5pt, ->,>=stealth] (3.8,4)--(-3.8,4);
		\node at (0,3.7)[color=violet, anchor=north]{${4}$};
		\draw[color=brown, line width=1.5pt, ->,>=stealth] (3.8,3.8)--(1.2,2.2);
		\node at (2.1,2.4)[color=brown, anchor=north]{$6$};
		\node at (1.9,2.8)[color=brown, anchor=south]{$5$};
		\draw[color=black, line width=1.5pt] (-2.2,0)--(-2.8,0);
		\draw[color=black, line width=1.5pt ] (-2,-.2)--(-2,-.8);
		\draw[color=black, line width=1.5pt ] (-2.1,-.1)--(-2.6,-.6);
		\draw[color=black, line width=1.5pt] (2.2,0)--(2.8,0);
		\draw[color=black, line width=1.5pt ] (2,-.2)--(2,-.8);
		\draw[color=black, line width=1.5pt ] (2.1,-.1)--(2.6,-.6);
		\draw[color=black, line width=1.5pt] (-4,4.2)--(-4,4.6);
		\draw[color=black, line width=1.5pt ] (-4.2,4)--(-4.8,4);
		\draw[color=black, line width=1.5pt ] (-4.1,4.1)--(-4.6,4.6);
		\draw[color=black, line width=1.5pt] (4,4.2)--(4,4.6);
		\draw[color=black, line width=1.5pt ] (4.2,4)--(4.8,4);
		\draw[color=black, line width=1.5pt ] (4.1,4.1)--(4.6,4.6);
		\draw[line width=.5pt, color=black,->,>=stealth] plot [smooth, tension=0.6] coordinates 
      {(1.65,.45)(3.4,3.2)(1.5,2)(1,1.5)(.5,2)(2.8,3.6)(0,3.6)(-3.2,3.6)(-2.6,2) (-2.2,.9)(-1.8,2)(-.5,2)(.3,1)(0,.4)(1.4,.4)};
		\end{scope}
		\end{tikzpicture}
\caption{A ciliated vertex and a ciliated face in a directed ribbon graph}
\label{fig:ciliated}
\end{figure}

Examples of  faces and face paths are shown in Figures \ref{fig:ciliated} and \ref{figure:facePathChord}.
Each  face defines a cyclic  ordering of the edges and their inverses in  the face. A ciliated face  is a face together with the choice of a starting vertex and induces a linear ordering of these edges. These orderings are taken counterclockwise, as shown in Figure \ref{fig:ciliated}. 
 
 A graph $\Gamma$ embedded into an oriented surface $\Sigma$ inherits a cyclic ordering of the edge ends at each vertex from the orientation of $\Sigma$ and hence a ribbon graph structure. Conversely, every ribbon graph $\Gamma$ defines a compact oriented surface $\Sigma_\Gamma$ that is obtained by attaching a disc at each face.  For a graph $\Gamma$ embedded into an oriented surface  $\Sigma$, the surfaces $\Sigma_\Gamma$ and $\Sigma$ are homeomorphic if and only if $\Sigma\setminus\Gamma$ is a disjoint union of discs.  
  An oriented surface with a boundary is obtained from a ribbon graph $\Gamma$ by attaching annuli instead of discs to some of its faces.

  \subsection{Chord diagrams}
  In the following, we sometimes restrict attention to ribbon graphs with a single vertex. Such a ribbon graph can be described   equivalently by a circular chord diagram. This chord diagram (with the chords pointing outwards from the circle)  is obtained  by thickening the vertex of the ribbon graph to a circle.  Similarly, a  ribbon graph with a single \emph{ciliated} vertex corresponds to a \emph{linear} chord diagram  that  is obtained from the circular chord diagram by cutting the circle at the position of the cilium,  as shown in Figure \ref{fig:chord}.

 We call the resulting line segment the {\em baseline} of the chord diagram.    The edges  of the ribbon graph then correspond to the {\em chords} of the  diagram. The  ribbon graph structure is given by the ordering of their ends on the circle or baseline, as shown in Figure \ref{fig:chord}. 
 Two chords are called {\em disjoint}, if they do not intersect, and two chord ends are called {\em disjoint} if they do not belong to the same chord. 
 In addition to the chords, we sometimes also admit additional edges that start or end at the baseline  and have a univalent vertex at their other end.

For a ribbon graph with a single vertex,  the path groupoid from Definition \ref{def:pathgroupoid} becomes a group. 
In the associated chord diagram, we represent its elements by paths that start and end below the baseline and are composed of segments along the chords and horizontal segments along the baseline, as shown in Figures \ref{figure:facePathChord} and \ref{figure:NonFaceChord}. 

\begin{figure}
\begin{center}
\begin{tikzpicture}[scale=.6, baseline=(current bounding box.center)]
\draw[line width=1pt, color=red, ->,>=stealth] (-.14,.14) .. controls (-3,3) and (-3,-3) .. (-.17,-.17);
\draw[line width=1pt, color=cyan, <-,>=stealth] (.22,0) .. controls (3,1.5) and (3,-1.5) .. (.14,-.14);
\draw[line width=1pt, color=black, ->,>=stealth] (0,-.14) .. controls (3.5,-4) and (3.5,4) .. (0,.22);
\draw[line width=1pt, color=blue, ->,>=stealth] (.14,.14) .. controls (3.5,2) and (-2.5,2.5) .. (-.22,-.1);
\draw[fill=black, color=black] (0,0) circle (.2);
\draw[line width=1pt, style=dotted] (0,-1)--(0,0);
\end{tikzpicture}
\qquad\qquad
\begin{tikzpicture}[scale=.4, baseline=(current bounding box.center)]
	\draw[line width=1pt, color=black](-3,0)--(10,0);
			\draw[color=red, line width=1.5pt, <-,>=stealth] (-2,0).. controls (-2,2) and (2,2)..(2,0);
			\draw[color=blue, line width=1.5pt, <-,>=stealth] (0,0).. controls (0,2) and (4,2)..(4,0);
			\draw[color=cyan, line width=1.5pt, <-,>=stealth] (6,0).. controls (6,2) and (8,2)..(8,0);   
			\draw[color=black, line width=1.5pt, <-,>=stealth] (3,0).. controls (3,5) and (9,5)..(9,0);    
			\node at (9,-0.1)[anchor=north] {1};
			\node at (8,-0.1)[anchor=north] {2};
			\node at (6,-0.1)[anchor=north] {3};
			\node at (4,-0.1)[anchor=north] {4};
			\node at (3,-0.1)[anchor=north] {5};
			\node at (2,-0.1)[anchor=north] {6};
			\node at (0,-0.1)[anchor=north] {7};
			\node at (-2,-0.1)[anchor=north] {8};
			\end{tikzpicture}
\end{center}

\vspace{-1cm}
\caption{Ribbon graph with a single ciliated vertex and the associated chord diagram.}
\label{fig:chord}
\end{figure}
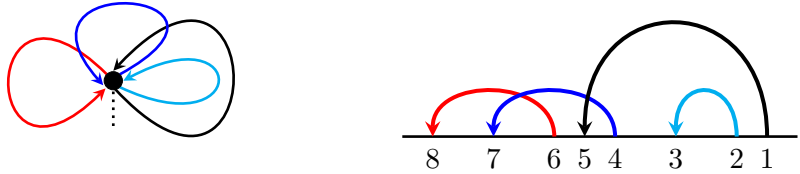

\begin{figure}
	\minipage[t]{0.45\textwidth}
	\begin{tikzpicture}[scale=.5]
		\draw[line width=1pt, color=black](-3,0)--(12,0);
		\draw[color=red, line width=1.5pt, <-,>=stealth] (-2,0).. controls (-2,2) and (2,2)..(2,0);
		\draw[color=blue, line width=1.5pt, <-,>=stealth] (0,0).. controls (0,2) and (4,2)..(4,0);
		\draw[line width=.5pt, color=black,->,>=stealth] (3.5,-.5)--(3.5,.5);
		\draw[line width=.5pt, color=black,->,>=stealth] (3.5,.5).. controls (3.5,1.2) and (.6,1.5) .. (.6,.5);
		\draw[line width=.5pt, color=black,->,>=stealth] (.6,.5)--(1.5,.5);
		\draw[line width=.5pt, color=black,->,>=stealth] (1.5,.5).. controls (1.5,1.2) and (-1.4,1.5) .. (-1.4,.5);
		\draw[line width=.5pt, color=black,->,>=stealth] (-1.4,.5)--(-.5,.5);
		\draw[line width=.5pt, color=black,->,>=stealth] (-.5,.5).. controls (-.5,2.5) and (4.5,2.5) .. (4.5,.5);
		\draw[line width=.5pt, color=black,->,>=stealth] (4.5,.5).. controls (4.5,2.5) and (7.5,2.5) .. (7.5,.5);
		\draw[line width=.5pt, color=black,->,>=stealth] (7.5,.5)--(7.5,-.5);
		\draw[color=violet, line width=1.5pt, <-,>=stealth] (5,0).. controls (5,2) and (7,2)..(7,0);
		\draw[color=cyan, line width=1.5pt, <-,>=stealth] (8,0).. controls (8,2) and (10,2)..(10,0);   
		\draw[color=black, line width=1.5pt, <-,>=stealth] (3,0).. controls (3,5) and (11,5)..(11,0);    
		\node at (0,1.8) [anchor=south, color=red] {$\alpha_{5}$};
		\node at (2,1.8) [anchor=south, color=blue] {$\alpha_{4}$};
		\node at (6,1.8) [anchor=south, color=violet] {$\alpha_{3}$};
		\node at (9,1.5) [anchor=south, color=cyan] {$\alpha_{2}$};
		\node at (7,3.8) [anchor=south, color=black] {$\alpha_{1}$};
		\node at (11,-0.1)[anchor=north] {1};
		\node at (10,-0.1)[anchor=north] {2};
		\node at (8,-0.1)[anchor=north] {3};
		\node at (7,-0.1)[anchor=north] {4};
		\node at (5,-0.1)[anchor=north] {5};
		\node at (4,-0.1)[anchor=north] {6};
		\node at (3,-0.1)[anchor=north] {7};
		\node at (2,-0.1)[anchor=north] {8};
		\node at (0,-0.1)[anchor=north] {9};
		\node at (-2,-0.1)[anchor=north] {10};
	\end{tikzpicture}
	\caption{The  path $\alpha_{3}^{-1}\alpha_{4}^{-1}\alpha_{5}\alpha_{4}$ is a face path, but not a ciliated face.}
	\label{figure:facePathChord}
	\endminipage 
$\qquad$
\minipage[t]{0.45\textwidth}
\begin{tikzpicture}[scale=.5]
\draw[line width=1pt, color=black](-3,0)--(12,0);
\draw[color=red, line width=1.5pt, <-,>=stealth] (-2,0).. controls (-2,2) and (2,2)..(2,0);
\draw[color=blue, line width=1.5pt, <-,>=stealth] (0,0).. controls (0,2) and (4,2)..(4,0);
\draw[line width=.5pt, color=black,->,>=stealth] (3.5,-.5)--(3.5,.5);
\draw[line width=.5pt, color=black,->,>=stealth] (3.5,.5).. controls (3.5,1.2) and (.6,1.5) .. (.6,.5);
\draw[line width=.5pt, color=black,->,>=stealth] (.6,.5)--(1.5,.5);
\draw[line width=.5pt, color=black,->,>=stealth] (1.5,.5).. controls (1.5,1.2) and (-1.4,1.5) .. (-1.4,.5);
\draw[line width=.5pt, color=black,->,>=stealth] (-1.4,.5)--(-.5,.5);
\draw[line width=.5pt, color=black,->,>=stealth] (-.5,.5).. controls (-.5,2.5) and (4.5,2.5) .. (4.5,.5);
\draw[line width=.5pt, color=black,->,>=stealth] (4.5,.5)--(7.5,.5);
\draw[line width=.5pt, color=black,->,>=stealth] (7.5,.5).. controls (7.5,2.5) and (10.5,2.5) .. (10.5,.5);
\draw[line width=.5pt, color=black,->,>=stealth] (10.5,.5)--(10.5,-.5);
      \draw[color=violet, line width=1.5pt, <-,>=stealth] (5,0).. controls (5,2) and (7,2)..(7,0);
            \draw[color=cyan, line width=1.5pt, <-,>=stealth] (8,0).. controls (8,2) and (10,2)..(10,0);   
\draw[color=black, line width=1.5pt, <-,>=stealth] (3,0).. controls (3,5) and (11,5)..(11,0);    
	\node at (0,1.8) [anchor=south, color=red] {$\alpha_{5}$};
	\node at (2,1.8) [anchor=south, color=blue] {$\alpha_{4}$};
	\node at (6,1.5) [anchor=south, color=violet] {$\alpha_{3}$};
	\node at (9,1.9) [anchor=south, color=cyan] {$\alpha_{2}$};
	\node at (7,3.8) [anchor=south, color=black] {$\alpha_{1}$};	
	\node at (11,-0.1)[anchor=north] {1};
	\node at (10,-0.1)[anchor=north] {2};
	\node at (8,-0.1)[anchor=north] {3};
	\node at (7,-0.1)[anchor=north] {4};
	\node at (5,-0.1)[anchor=north] {5};
	\node at (4,-0.1)[anchor=north] {6};
	\node at (3,-0.1)[anchor=north] {7};
	\node at (2,-0.1)[anchor=north] {8};
	\node at (0,-0.1)[anchor=north] {9};
	\node at (-2,-0.1)[anchor=north] {10};
\end{tikzpicture}
	\caption{The  path $\alpha_{2}^{-1} \alpha_{4}^{-1}\alpha_{5}\alpha_{4}$ is not a face path.
	}
	\label{figure:NonFaceChord}
	\endminipage
\end{figure}

\section{Mapping class groups}
\label{sec:mapsec}
In this section, we review the  background on mapping class groups for oriented surfaces with or without boundaries. Where no other references are given, we follow the presentation in \cite{FM}. As in \cite[Sec.~2.1]{FM}, we define the mapping class group of  an oriented surface $\Sigma$ with boundary $\partial \Sigma$ as 
\begin{align}
	\nonumber
\mathrm{Map}(\Sigma)=\frac{\mathrm{Homeo}^+(\Sigma, \partial \Sigma)}{\mathrm{Homeo}_0 (\Sigma,\partial \Sigma)}.
\end{align}
where $\text{Homeo}^+(\Sigma,\partial \Sigma)$ is the group of orientation preserving homeomorphisms that restrict to the identity on the boundary and  $\text{Homeo}_0(\Sigma,\partial \Sigma)$ is the subgroup of  orientation preserving homeomorphisms that fix the boundary and are homotopic to the identity.  

This is  not to be confused with the mapping class group of a surface with marked points or punctures. For the relation between these mapping class groups see for instance \cite{Bi} or \cite[Sec.~2.1, 3.6]{FM} and the references therein. In particular,  Dehn twists along circles around the boundary components are not necessarily trivial.

\subsection{Presentation in terms of Dehn twists}
A simple finite presentation of the mapping class group of a surface of genus $g\geq 1$ with $n\geq 0$ boundary components was first given by Wajnryb \cite{W,BW94}. We  work with the presentation of  Gervais derived in \cite{G01} from Wajnryb's presentation and with a different method by Hirose \cite{H}.

Gervais' presentation describes the mapping class group $\mathrm{Map}(\Sigma)$ in terms of Dehn twists around a set of closed simple curves on $\Sigma$,  depicted in Figures \ref{fig:gervais} and \ref{fig:gervais2}. The relations depend on their essential intersection numbers. We write $|\alpha\cap\beta|=0$ if the homotopy classes of the closed simple curves $\alpha,\beta$ have representatives that do not intersect and $|\alpha\cap\beta|=1$ if they have representatives with a single intersection point.

\begin{theorem}[\cite{G01}, Th.~1] \label{th:gervais}Let $\Sigma$ be an oriented surface of genus $g\geq 1$ with $n\geq 0$ boundary components such that $g+n>1$. Then the mapping class group $\mathrm{Map}(\Sigma)$  is generated by the Dehn twists along the closed simple curves in the set $$G=\big \{\alpha_i\mid i=1,...,g\big\}\cup \big\{\delta_j\mid j=1,...,n+2g-2\big\}\cup \{\gamma_{k,l}\mid k,l=1,...,n+2g-2, k\neq l\big \}$$  in Figures \ref{fig:gervais} and \ref{fig:gervais2}, subject to the relations
\begin{compactenum}[(i)]
\item $D_{\gamma_{n+2j+1, n+2j}}=D_{\gamma_{n+2j, n+2j-1}}$ for $j=1,...,g-2$ and\footnote{The last relation   is not stated explicitly in \cite[Th.~1]{G01}
but implied by its proof and its verbal description. Note also that our notation differs from the one in \cite{G01}. Our curves $\alpha_i$ are called $\beta_i$, our curves $\delta_i$ are called $\alpha_i$, and our $\alpha_g$ is called $\beta$ in \cite{G01}. We reserve the letters $\alpha_i$ and $\beta_i$ for certain generators of the fundamental group $\pi_1(\Sigma)$.
} $D_{\gamma_{1, n+2g-2}}=D_{\gamma_{n+2g-2, n+2g-3}}$,

\item  $D_\alpha\circ D_\beta=D_\beta\circ D_\alpha$ for $\alpha,\beta\in G$ with $|\alpha\cap\beta|=0$, 

\item  $D_\alpha\circ D_\beta\circ D_\alpha=D_\beta\circ D_\alpha\circ D_\beta$ for $\alpha,\beta\in G$ with $|\alpha\cap \beta|=1$,
\item $(D_{\delta_k}\circ D_{\delta_i}\circ D_{\delta_j}\circ D_{\alpha_g})^3=D_{\gamma_{i,j}}\circ D_{\gamma_{j,k}}\circ D_{\gamma_{k,i}}$ for $i,j,k$ not all equal and  $j\leq i\leq k$ or $k\leq j\leq i$ or $i\leq k\leq j$ with $D_{\gamma_{i,i}}=\id$ for $1\leq i\leq n+2g-2$. 
 \end{compactenum}
\end{theorem}

Relation (i) expresses the fact that  the  loops  around the two legs of the $j$th handle for $j=1,...,g-1$ are freely homotopic, such that   their Dehn twists coincide. 
 The curves  in relation (iv) are depicted in Figure \ref{fig:gervais2}. 
Note that the Dehn twists along $\gamma_{i,j}$, $\gamma_{j,k}$ and $\gamma_{k,i}$ and along $\delta_i$, $\delta_j$ and $\delta_k$  commute by (ii), so that their relative order in (iv) is irrelevant.

\begin{figure} 
\centering
\def\svgwidth{.95\columnwidth}
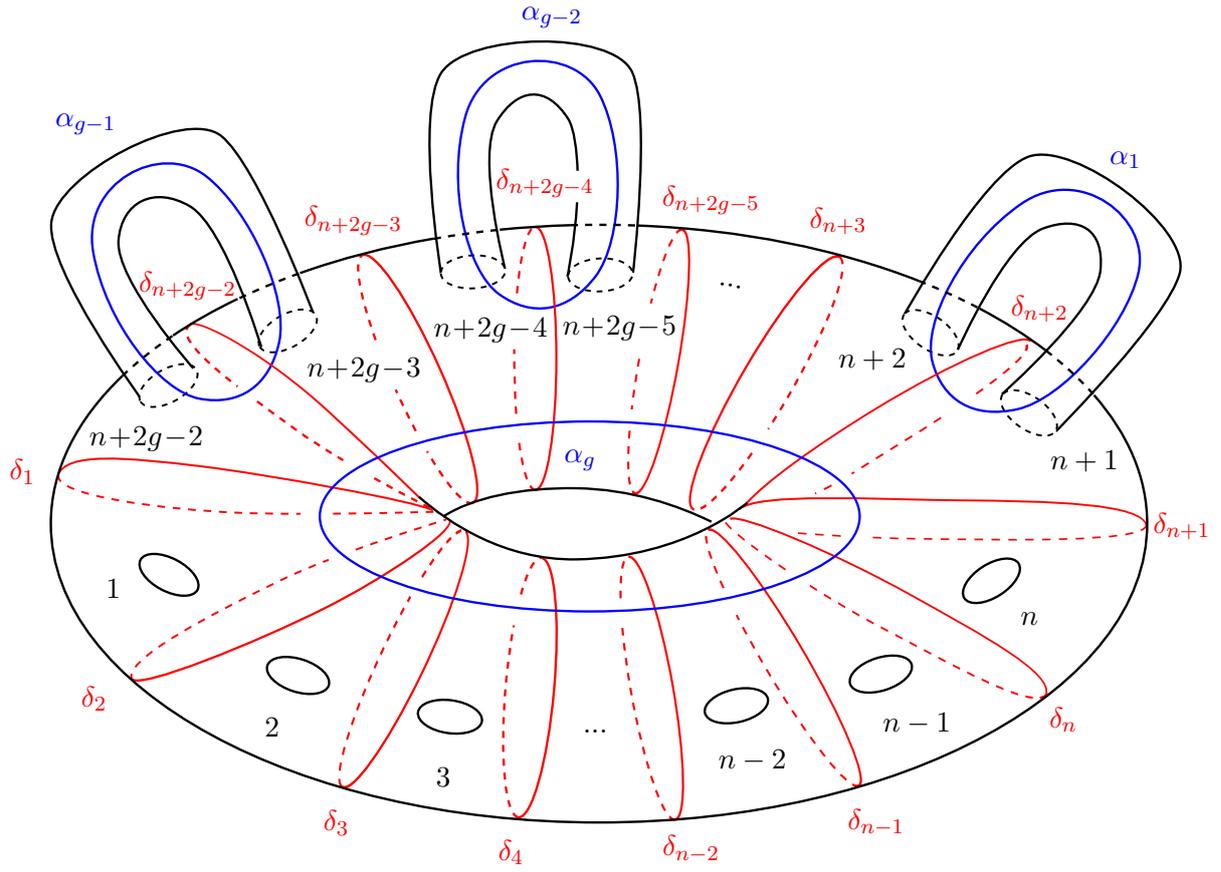
\caption{The curves $\alpha_i$ and $\delta_j$}
\label{fig:gervais}
\end{figure}

\begin{figure}
\centering
\def\svgwidth{.9\columnwidth}
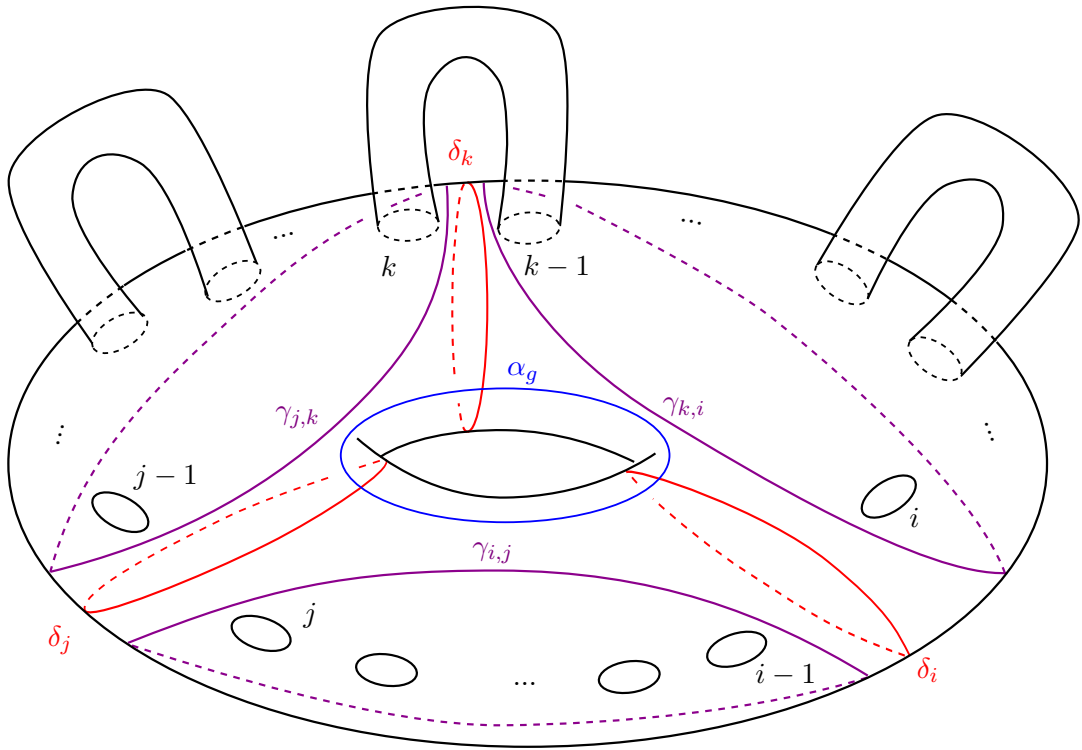
\caption{The curves $\delta_i$, $\delta_j$, $\delta_k$ and $\gamma_{i,j}$, $\gamma_{j,k}$, $\gamma_{k,i}$ for $j<i<k$. }
\label{fig:gervais2}
\end{figure}

The essential intersection numbers of the curves in Figures \ref{fig:gervais} and \ref{fig:gervais2} are given by
\begin{align}\label{eq:intnumbers}
		| \alpha_{i} \cap \alpha_{j} | &= |\delta_{i} \cap \delta_{j} | =0 \qquad\text{for all }i \neq j,
		\\
		| \alpha_{i} \cap \delta_{j} | &= 
			\begin{cases}
				1 \quad \text{if $i= g$ or $j= n+2i$} \\
				0 \quad \text{ otherwise,}
			\end{cases}\nonumber
			\\
		| \alpha_{i} \cap \gamma_{j,k} | &= 
			\begin{cases}
				1 \quad \text{if $j= n+2i$ or $k=n+2i$}\\
				0 \quad \text{otherwise,}
			\end{cases}\nonumber
		\\
		| \delta_{i} \cap \gamma_{j, k} |& =
			\begin{cases}
				2 \quad \text{if $k< i < j$ up to cyclic permutations}
					\\
				0 \quad \text{otherwise,}
			\end{cases}\nonumber
		\\
		| \gamma_{i,j} \cap \gamma_{k,l} | &=
			\begin{cases}
				4 \quad \text{if $j<k< l< i$ up to cyclic permutations}
				\\
				0 \quad \text{if either $j \leq l \leq k \leq i, l \leq j \leq i \leq k$ or $j \leq i \leq l \leq k$ up to cyclic permutations}\\
				2 \quad \text{otherwise.}
			\end{cases}\nonumber
	\end{align}

\subsection{Presentation in terms of chord slides}
\label{sec:chordslides}

For an oriented surface $\Sigma$ of genus $g\geq 1$ with one boundary component, we also
 consider a different presentation of the mapping class group, derived by Bene \cite{B} and also considered in \cite{ABP}. It describes the mapping class group $\mathrm{Map}(\Sigma)$ in terms of chord slides in an associated ribbon graph $\Gamma$ with a single ciliated vertex and face
 or the corresponding linear chord diagram. The surface $\Sigma$ is obtained by attaching an annulus to the face.

The chord slide of an end  of a chord $\beta$ along a chord $\alpha$ is defined if two ends of $\alpha$ and $\beta$ are adjacent with respect to the linear ordering of the chord ends at the  baseline. It is  given by the following diagram, where the holes in the  circle indicate  possible positions of  other chord ends or the cilium. 
\begin{align}
	\nonumber
&\begin{tikzpicture}[scale=.4]
\begin{scope}[shift={(-5,0)}]
\draw[color=blue, line width=1.5pt, ] (2,.2)--(0,2);
\draw[color=red, line width=1.5pt, ] (-2,0)--(2,0);
\node at (0,0)[anchor=north, color=red]{$\alpha$};
\node at (1,1)[anchor=south west, color=blue]{$\beta$};
\draw [ line width=1pt, domain=-20:20] plot ({2*cos(\x)}, {2*sin(\x)});
\draw [ line width=1pt, domain=160:200] plot ({2*cos(\x)}, {2*sin(\x)});
\draw [ line width=1pt, domain=70:110] plot ({2*cos(\x)}, {2*sin(\x)});
\end{scope}
\draw[color=black,->,>=stealth] (-1,0)--(1,0);
\begin{scope}[shift={(5,0)}]
\draw[color=blue, line width=1.5pt, ] (-2,.2)--(0,2);
\draw[color=red, line width=1.5pt, ] (-2,0)--(2,0);
\node at (0,0)[anchor=north, color=red]{$\alpha$};
\node at (-1,1)[anchor=south east, color=blue]{$\beta$};
\draw [ line width=1pt, domain=-20:20] plot ({2*cos(\x)}, {2*sin(\x)});
\draw [ line width=1pt, domain=160:200] plot ({2*cos(\x)}, {2*sin(\x)});
\draw [ line width=1pt, domain=70:110] plot ({2*cos(\x)}, {2*sin(\x)});
\end{scope}
\end{tikzpicture}
\end{align}
Building on the work of Penner \cite{P87}, which describes the mapping class group $\mathrm{Map}(\Sigma)$ in terms  of {\em flip moves} or {\em Whitehead moves} on a triangulation, Bene \cite{B} gives a simple presentation in terms of chord slides.

\begin{theorem}[\cite{B}, Th.~7.1]\label{th:bene} Let $\Sigma$ be a surface of genus $g\geq 1$ with one boundary component and $\Gamma$ an embedded  ribbon graph  with a single  ciliated vertex and face such that $\Sigma\setminus\Gamma$ is an annulus. 
Then  the mapping class group $\mathrm{Map}(\Sigma)$  is presented  by finite sequences of chord slides that preserve $\Gamma$ up to the cilium
 subject to  the  relations
\begin{compactenum}
\item {\bf Involutivity:} the following composite of slides is the identity,
\begin{align}\label{eq:slidegerinv}
&\begin{tikzpicture}[scale=.4]
\begin{scope}[shift={(-15,0)}]
\draw[color=blue, line width=1.5pt, ] (2,.2)--(0,2);
\draw[color=red, line width=1.5pt, ] (-2,0)--(2,0);
\node at (0,0)[anchor=north, color=red]{$\alpha$};
\node at (1,1)[anchor=south west, color=blue]{$\beta$};
\draw [ line width=1pt, domain=-20:20] plot ({2*cos(\x)}, {2*sin(\x)});
\draw [ line width=1pt, domain=160:200] plot ({2*cos(\x)}, {2*sin(\x)});
\draw [ line width=1pt, domain=70:110] plot ({2*cos(\x)}, {2*sin(\x)});
\end{scope}
\draw[color=black,->,>=stealth] (-11,0)--(-9,0);
\begin{scope}[shift={(-5,0)}]
\draw[color=blue, line width=1.5pt, ] (-2,.2)--(0,2);
\draw[color=red, line width=1.5pt, ] (-2,0)--(2,0);
\node at (0,0)[anchor=north, color=red]{$\alpha$};
\node at (-1,1)[anchor=south east, color=blue]{$\beta$};
\draw [ line width=1pt, domain=-20:20] plot ({2*cos(\x)}, {2*sin(\x)});
\draw [ line width=1pt, domain=160:200] plot ({2*cos(\x)}, {2*sin(\x)});
\draw [ line width=1pt, domain=70:110] plot ({2*cos(\x)}, {2*sin(\x)});
\end{scope}
\draw[color=black,->,>=stealth] (-1,0)--(1,0);
\begin{scope}[shift={(5,0)}]
\draw[color=blue, line width=1.5pt, ] (2,.2)--(0,2);
\draw[color=red, line width=1.5pt, ] (-2,0)--(2,0);
\node at (0,0)[anchor=north, color=red]{$\alpha$};
\node at (1,1)[anchor=south west, color=blue]{$\beta$};
\draw [ line width=1pt, domain=-20:20] plot ({2*cos(\x)}, {2*sin(\x)});
\draw [ line width=1pt, domain=160:200] plot ({2*cos(\x)}, {2*sin(\x)});
\draw [ line width=1pt, domain=70:110] plot ({2*cos(\x)}, {2*sin(\x)});
\end{scope}
\end{tikzpicture}
\end{align}

\item {\bf Commutativity:} chord slides of disjoint ends along distinct chords commute,
\begin{align}\label{eq:slidegercomm}
&\begin{tikzpicture}[scale=.4]
\begin{scope}[shift={(-4,0)}]
\draw[color=red, line width=1.5pt, ] (10:2)--(160:2);
\draw[color=blue, line width=1.5pt, ] (90:2)--(15:2);
\draw[color=magenta, line width=1.5pt, ] (-20:2)--(200:2);
\draw[color=cyan, line width=1.5pt, ] (215:2)--(290:2);
\node at (40:1.5)[anchor=south west, color=blue]{$\beta$};
\node at (-.5,.5)[anchor=south, color=red]{$\alpha$};
\node at (-.5,-2.8)[anchor=south, color=cyan]{$\delta$};
\node at (-.5,-.7)[anchor=south, color=magenta]{$\gamma$};
\draw [ line width=1pt, domain=0:30] plot ({2*cos(\x)}, {2*sin(\x)});
\draw [ line width=1pt, domain=140:170] plot ({2*cos(\x)}, {2*sin(\x)});
\draw [ line width=1pt, domain=190:230] plot ({2*cos(\x)}, {2*sin(\x)});
\draw [ line width=1pt, domain=280:300] plot ({2*cos(\x)}, {2*sin(\x)});
\draw [ line width=1pt, domain=-10:-40] plot ({2*cos(\x)}, {2*sin(\x)});
\draw [ line width=1pt, domain=80:100] plot ({2*cos(\x)}, {2*sin(\x)});
\end{scope}
\draw[color=black,->,>=stealth] (-1,0)--(1,0);
\begin{scope}[shift={(4,0)}]
\draw[color=red, line width=1.5pt, ] (10:2)--(160:2);
\draw[color=blue, line width=1.5pt, ] (90:2)--(155:2);
\draw[color=magenta, line width=1.5pt, ] (-20:2)--(200:2);
\draw[color=cyan, line width=1.5pt, ] (215:2)--(290:2);
\node at (125:1.5)[anchor=south east, color=blue]{$\beta$};
\node at (-.5,.5)[anchor=south, color=red]{$\alpha$};
\node at (-.5,-2.8)[anchor=south, color=cyan]{$\delta$};
\node at (-.5,-.7)[anchor=south, color=magenta]{$\gamma$};
\draw [ line width=1pt, domain=0:30] plot ({2*cos(\x)}, {2*sin(\x)});
\draw [ line width=1pt, domain=140:170] plot ({2*cos(\x)}, {2*sin(\x)});
\draw [ line width=1pt, domain=190:230] plot ({2*cos(\x)}, {2*sin(\x)});
\draw [ line width=1pt, domain=280:300] plot ({2*cos(\x)}, {2*sin(\x)});
\draw [ line width=1pt, domain=-10:-40] plot ({2*cos(\x)}, {2*sin(\x)});
\draw [ line width=1pt, domain=80:100] plot ({2*cos(\x)}, {2*sin(\x)});
\end{scope}
\draw[color=black,->,>=stealth] (-4,-3)--(-4,-4);
\draw[color=black,->,>=stealth] (4,-3)--(4,-4);
\begin{scope}[shift={(-4,-7)}]
\draw[color=red, line width=1.5pt, ] (10:2)--(160:2);
\draw[color=blue, line width=1.5pt, ] (90:2)--(15:2);
\draw[color=magenta, line width=1.5pt, ] (-20:2)--(200:2);
\draw[color=cyan, line width=1.5pt, ] (-25:2)--(290:2);
\node at (40:1.5)[anchor=south west, color=blue]{$\beta$};
\node at (-.5,.5)[anchor=south, color=red]{$\alpha$};
\node at (2,-1.5)[anchor=north east, color=cyan]{$\delta$};
\node at (-.5,-.7)[anchor=south, color=magenta]{$\gamma$};
\draw [ line width=1pt, domain=0:30] plot ({2*cos(\x)}, {2*sin(\x)});
\draw [ line width=1pt, domain=140:170] plot ({2*cos(\x)}, {2*sin(\x)});
\draw [ line width=1pt, domain=190:230] plot ({2*cos(\x)}, {2*sin(\x)});
\draw [ line width=1pt, domain=280:300] plot ({2*cos(\x)}, {2*sin(\x)});
\draw [ line width=1pt, domain=-10:-40] plot ({2*cos(\x)}, {2*sin(\x)});
\draw [ line width=1pt, domain=80:100] plot ({2*cos(\x)}, {2*sin(\x)});
\end{scope}
\draw[color=black,->,>=stealth] (-1,-7)--(1,-7);
\begin{scope}[shift={(4,-7)}]
\draw[color=red, line width=1.5pt, ] (10:2)--(160:2);
\draw[color=blue, line width=1.5pt, ] (90:2)--(155:2);
\draw[color=magenta, line width=1.5pt, ] (-20:2)--(200:2);
\draw[color=cyan, line width=1.5pt, ] (-25:2)--(290:2);
\node at (125:1.5)[anchor=south east, color=blue]{$\beta$};
\node at (-.5,.5)[anchor=south, color=red]{$\alpha$};
\node at (2,-1.5)[anchor=north east, color=cyan]{$\delta$};
\node at (-.5,-.7)[anchor=south, color=magenta]{$\gamma$};
\draw [ line width=1pt, domain=0:30] plot ({2*cos(\x)}, {2*sin(\x)});
\draw [ line width=1pt, domain=140:170] plot ({2*cos(\x)}, {2*sin(\x)});
\draw [ line width=1pt, domain=190:230] plot ({2*cos(\x)}, {2*sin(\x)});
\draw [ line width=1pt, domain=280:300] plot ({2*cos(\x)}, {2*sin(\x)});
\draw [ line width=1pt, domain=-10:-40] plot ({2*cos(\x)}, {2*sin(\x)});
\draw [ line width=1pt, domain=80:100] plot ({2*cos(\x)}, {2*sin(\x)});
\end{scope}
\end{tikzpicture}
\end{align}

\item {\bf Triangle relation:}
\begin{align}
\label{eq:triangleger}
\begin{tikzpicture}[scale=.4]
\begin{scope}[shift={(-4,0)}]
\draw[color=red, line width=1.5pt, ] (-90:2)--(27:2);
\draw[color=blue, line width=1.5pt, ] (33:2)--(150:2);
\draw [ line width=1pt, domain=-70:-110] plot ({2*cos(\x)}, {2*sin(\x)});
\draw [ line width=1pt, domain=10:50] plot ({2*cos(\x)}, {2*sin(\x)});
\draw [ line width=1pt, domain=130:170] plot ({2*cos(\x)}, {2*sin(\x)});
\node at (-45:1.3)[anchor=west, color=red]{$\alpha$};
\node at (90:1)[anchor=south, color=blue]{$\beta$};
\end{scope}
\draw[color=black,->,>=stealth] (-1,0)--(1,0);
\begin{scope}[shift={(4,0)}]
\draw[color=red, line width=1.5pt, ] (-87:2)--(30:2);
\draw[color=blue, line width=1.5pt, ] (-93:2)--(150:2);
\draw [ line width=1pt, domain=-70:-110] plot ({2*cos(\x)}, {2*sin(\x)});
\draw [ line width=1pt, domain=10:50] plot ({2*cos(\x)}, {2*sin(\x)});
\draw [ line width=1pt, domain=130:170] plot ({2*cos(\x)}, {2*sin(\x)});
\node at (-45:1.3)[anchor=west, color=red]{$\alpha$};
\node at (225:1.3)[anchor=east, color=blue]{$\beta$};
\end{scope}
\draw[color=black,->,>=stealth] (4,-3)--(4,-4);
\begin{scope}[shift={(4,-7)}]
\draw[color=red, line width=1.5pt, ] (147:2)--(30:2);
\draw[color=blue, line width=1.5pt, ] (-90:2)--(153:2);
\draw [ line width=1pt, domain=-70:-110] plot ({2*cos(\x)}, {2*sin(\x)});
\draw [ line width=1pt, domain=10:50] plot ({2*cos(\x)}, {2*sin(\x)});
\draw [ line width=1pt, domain=130:170] plot ({2*cos(\x)}, {2*sin(\x)});
\node at (90:1)[anchor=south, color=red]{$\alpha$};
\node at (225:1.3)[anchor=east, color=blue]{$\beta$};
\end{scope}
\draw[color=black,->,>=stealth] (1,-7)--(-1,-7);
\begin{scope}[shift={(-4,-7)}]
\draw[color=blue, line width=1.5pt, ] (-90:2)--(27:2);
\draw[color=red, line width=1.5pt, ] (33:2)--(150:2);
\draw [ line width=1pt, domain=-70:-110] plot ({2*cos(\x)}, {2*sin(\x)});
\draw [ line width=1pt, domain=10:50] plot ({2*cos(\x)}, {2*sin(\x)});
\draw [ line width=1pt, domain=130:170] plot ({2*cos(\x)}, {2*sin(\x)});
\node at (90:1)[anchor=south, color=red]{$\alpha$};
\node at (-45:1.3)[anchor=west, color=blue]{$\beta$};
\end{scope}
\draw[color=black,<->,>=stealth] (-4,-4)--(-4,-3);
\node at (-4,-3.5) [anchor=east, color=black] {$\alpha\leftrightarrow\beta$};
\end{tikzpicture}
\end{align}

\item {\bf Left pentagon relation:}
\begin{align}\label{eq:leftpentger}
\begin{tikzpicture}[scale=.4]
\begin{scope}[shift={(-8,0)}]
\draw[color=red, line width=1.5pt, ] (-80:2)--(0:2);
\node at (1,-1) [anchor=north west, color=red]{$\alpha$};
\draw[color=violet, line width=1.5pt, ] (-90:2)--(90:2);
\node at (0,0)[anchor=west, color=violet]{$\gamma$};
\draw[color=blue, line width=1.5pt, ] (-100:2)--(180:2);
\node at (-1,-1)[anchor=north east, color=blue]{$\beta$};
\draw [ line width=1pt, domain=-25:25] plot ({2*cos(\x)}, {2*sin(\x)});
\draw [ line width=1pt, domain=65:115] plot ({2*cos(\x)}, {2*sin(\x)});
\draw [ line width=1pt, domain=155:205] plot ({2*cos(\x)}, {2*sin(\x)});
\draw [ line width=1pt, domain=245:295] plot ({2*cos(\x)}, {2*sin(\x)});
\end{scope}
\draw[color=black,->,>=stealth] (-4.5,0)--(-2.5,0);
\begin{scope}[shift={(1,0)}]
\draw[color=red, line width=1.5pt, ] (80:2)--(0:2);
\node at (1,1) [anchor=south west, color=red]{$\alpha$};
\draw[color=violet, line width=1.5pt, ] (-90:2)--(90:2);
\node at (0,0)[anchor=west, color=violet]{$\gamma$};
\draw[color=blue, line width=1.5pt, ] (-100:2)--(180:2);
\node at (-1,-1)[anchor=north east, color=blue]{$\beta$};
\draw [ line width=1pt, domain=-25:25] plot ({2*cos(\x)}, {2*sin(\x)});
\draw [ line width=1pt, domain=65:115] plot ({2*cos(\x)}, {2*sin(\x)});
\draw [ line width=1pt, domain=155:205] plot ({2*cos(\x)}, {2*sin(\x)});
\draw [ line width=1pt, domain=245:295] plot ({2*cos(\x)}, {2*sin(\x)});
\end{scope}
\draw[color=black,->,>=stealth] (4.5,0)--(6.5,0);
\begin{scope}[shift={(10,0)}]
\draw[color=red, line width=1.5pt, ] (80:2)--(0:2);
\node at (1,1) [anchor=south west, color=red]{$\alpha$};
\draw[color=violet, line width=1.5pt, ] (170:2)--(90:2);
\node at (-1,1)[anchor=south east, color=violet]{$\gamma$};
\draw[color=blue, line width=1.5pt, ] (-100:2)--(190:2);
\node at (-1,-1)[anchor=north east, color=blue]{$\beta$};
\draw [ line width=1pt, domain=-25:25] plot ({2*cos(\x)}, {2*sin(\x)});
\draw [ line width=1pt, domain=65:115] plot ({2*cos(\x)}, {2*sin(\x)});
\draw [ line width=1pt, domain=155:205] plot ({2*cos(\x)}, {2*sin(\x)});
\draw [ line width=1pt, domain=245:295] plot ({2*cos(\x)}, {2*sin(\x)});
\end{scope}
\draw[color=black,->,>=stealth] (-7,-3)--(-5,-5);
\draw[color=black,->,>=stealth] (5,-5)--(7,-3);
\begin{scope}[shift={(5,-8)}]
\draw[color=red, line width=1.5pt, ] (180:2)--(0:2);
\node at (0,0) [anchor=south, color=red]{$\alpha$};
\draw[color=violet, line width=1.5pt, ] (170:2)--(90:2);
\node at (-1,1)[anchor=south east, color=violet]{$\gamma$};
\draw[color=blue, line width=1.5pt, ] (-100:2)--(190:2);
\node at (-1,-1)[anchor=north east, color=blue]{$\beta$};
\draw [ line width=1pt, domain=-25:25] plot ({2*cos(\x)}, {2*sin(\x)});
\draw [ line width=1pt, domain=65:115] plot ({2*cos(\x)}, {2*sin(\x)});
\draw [ line width=1pt, domain=155:205] plot ({2*cos(\x)}, {2*sin(\x)});
\draw [ line width=1pt, domain=245:295] plot ({2*cos(\x)}, {2*sin(\x)});
\end{scope}
\draw[color=black,->,>=stealth] (-1,-8)--(1,-8);
\begin{scope}[shift={(-5,-8)}]
\draw[color=red, line width=1.5pt, ] (-80:2)--(0:2);
\node at (1,-1) [anchor=north west, color=red]{$\alpha$};
\draw[color=violet, line width=1.5pt, ] (170:2)--(90:2);
\node at (-1,1)[anchor=south east, color=violet]{$\gamma$};
\draw[color=blue, line width=1.5pt, ] (-100:2)--(190:2);
\node at (-1,-1)[anchor=north east, color=blue]{$\beta$};
\draw [ line width=1pt, domain=-25:25] plot ({2*cos(\x)}, {2*sin(\x)});
\draw [ line width=1pt, domain=65:115] plot ({2*cos(\x)}, {2*sin(\x)});
\draw [ line width=1pt, domain=155:205] plot ({2*cos(\x)}, {2*sin(\x)});
\draw [ line width=1pt, domain=245:295] plot ({2*cos(\x)}, {2*sin(\x)});
\end{scope}
\end{tikzpicture}
\end{align}

\item {\bf Right pentagon relation:}
\begin{align}\label{eq:rightpentger}
\begin{tikzpicture}[scale=.4]
\begin{scope}[shift={(-8,0)}]
\draw[color=red, line width=1.5pt, ] (-80:2)--(0:2);
\node at (1,-1) [anchor=north west, color=red]{$\alpha$};
\draw[color=violet, line width=1.5pt, ] (-90:2)--(90:2);
\node at (0,0)[anchor=west, color=violet]{$\gamma$};
\draw[color=blue, line width=1.5pt, ] (-100:2)--(180:2);
\node at (-1,-1)[anchor=north east, color=blue]{$\beta$};
\draw [ line width=1pt, domain=-25:25] plot ({2*cos(\x)}, {2*sin(\x)});
\draw [ line width=1pt, domain=65:115] plot ({2*cos(\x)}, {2*sin(\x)});
\draw [ line width=1pt, domain=155:205] plot ({2*cos(\x)}, {2*sin(\x)});
\draw [ line width=1pt, domain=245:295] plot ({2*cos(\x)}, {2*sin(\x)});
\end{scope}
\draw[color=black,<-,>=stealth] (-4.5,0)--(-2.5,0);
\begin{scope}[shift={(1,0)}]
\draw[color=red, line width=1.5pt, ] (-80:2)--(0:2);
\node at (1,-1) [anchor=north west, color=red]{$\alpha$};
\draw[color=violet, line width=1.5pt, ] (-90:2)--(90:2);
\node at (0,0)[anchor=west, color=violet]{$\gamma$};
\draw[color=blue, line width=1.5pt, ] (100:2)--(180:2);
\node at (-1,1)[anchor=south east, color=blue]{$\beta$};
\draw [ line width=1pt, domain=-25:25] plot ({2*cos(\x)}, {2*sin(\x)});
\draw [ line width=1pt, domain=65:115] plot ({2*cos(\x)}, {2*sin(\x)});
\draw [ line width=1pt, domain=155:205] plot ({2*cos(\x)}, {2*sin(\x)});
\draw [ line width=1pt, domain=245:295] plot ({2*cos(\x)}, {2*sin(\x)});
\end{scope}
\draw[color=black,<-,>=stealth] (4.5,0)--(6.5,0);
\begin{scope}[shift={(10,0)}]
\draw[color=red, line width=1.5pt, ] (-80:2)--(0:2);
\node at (1,-1) [anchor=north west, color=red]{$\alpha$};
\draw[color=violet, line width=1.5pt, ] (10:2)--(90:2);
\node at (1,1)[anchor=south west, color=violet]{$\gamma$};
\draw[color=blue, line width=1.5pt, ] (100:2)--(180:2);
\node at (-1,1)[anchor=south east, color=blue]{$\beta$};
\draw [ line width=1pt, domain=-25:25] plot ({2*cos(\x)}, {2*sin(\x)});
\draw [ line width=1pt, domain=65:115] plot ({2*cos(\x)}, {2*sin(\x)});
\draw [ line width=1pt, domain=155:205] plot ({2*cos(\x)}, {2*sin(\x)});
\draw [ line width=1pt, domain=245:295] plot ({2*cos(\x)}, {2*sin(\x)});
\end{scope}
\draw[color=black,<-,>=stealth] (-7,-3)--(-5,-5);
\draw[color=black,<-,>=stealth] (5,-5)--(7,-3);
\begin{scope}[shift={(5,-8)}]
\draw[color=red, line width=1.5pt, ] (-80:2)--(0:2);
\node at (1,-1) [anchor=north west, color=red]{$\alpha$};
\draw[color=violet, line width=1.5pt, ] (10:2)--(90:2);
\node at (1,1)[anchor=south west, color=violet]{$\gamma$};
\draw[color=blue, line width=1.5pt, ] (0:2)--(180:2);
\node at (0,0)[anchor=south, color=blue]{$\beta$};
\draw [ line width=1pt, domain=-25:25] plot ({2*cos(\x)}, {2*sin(\x)});
\draw [ line width=1pt, domain=65:115] plot ({2*cos(\x)}, {2*sin(\x)});
\draw [ line width=1pt, domain=155:205] plot ({2*cos(\x)}, {2*sin(\x)});
\draw [ line width=1pt, domain=245:295] plot ({2*cos(\x)}, {2*sin(\x)});
\end{scope}
\draw[color=black,<-,>=stealth] (-1,-8)--(1,-8);
\begin{scope}[shift={(-5,-8)}]
\draw[color=red, line width=1.5pt, ] (-80:2)--(0:2);
\node at (1,-1) [anchor=north west, color=red]{$\alpha$};
\draw[color=violet, line width=1.5pt, ] (10:2)--(90:2);
\node at (1,1)[anchor=south west, color=violet]{$\gamma$};
\draw[color=blue, line width=1.5pt, ] (-100:2)--(180:2);
\node at (-1,-1)[anchor=north east, color=blue]{$\beta$};
\draw [ line width=1pt, domain=-25:25] plot ({2*cos(\x)}, {2*sin(\x)});
\draw [ line width=1pt, domain=65:115] plot ({2*cos(\x)}, {2*sin(\x)});
\draw [ line width=1pt, domain=155:205] plot ({2*cos(\x)}, {2*sin(\x)});
\draw [ line width=1pt, domain=245:295] plot ({2*cos(\x)}, {2*sin(\x)});
\end{scope}
\end{tikzpicture}
\end{align}
\end{compactenum}
\end{theorem}

Bene \cite{B} also defines analogous chord slides and  relations  for {\em oriented} chords. In this case, the chord slides preserve the orientation, and each relation holds for all possible orientation of the chords. The triangle relation  then requires an  additional orientation reversal. 
It is also shown  in \cite{B} that the relations in Theorem \ref{th:bene} imply two further relations.

\begin{lemma} [\cite{B}, Lemma 6.1]\label{lem:bene}
The relations in Theorem \ref{th:bene} imply 
\begin{compactenum}
\item {\bf Opposite end commutativity:} chord slides of different ends of a chord  commute.
\item {\bf Adjacent commutativity:} chord slides along  different  sides of a chord commute.

\begin{align}\label{eq:oppendger}
\begin{tikzpicture}[scale=.4]
\begin{scope}[shift={(-4,0)}]
\draw[color=red, line width=1.5pt, ] (80:2)--(10:2);
\node at (1,1) [anchor=south west, color=red]{$\alpha$};
\draw[color=violet, line width=1.5pt, ] (-90:2)--(90:2);
\node at (0,0)[anchor=west, color=violet]{$\gamma$};
\draw[color=blue, line width=1.5pt, ] (-100:2)--(190:2);
\node at (-1,-1)[anchor=north east, color=blue]{$\beta$};
\draw [ line width=1pt, domain=-25:25] plot ({2*cos(\x)}, {2*sin(\x)});
\draw [ line width=1pt, domain=65:115] plot ({2*cos(\x)}, {2*sin(\x)});
\draw [ line width=1pt, domain=155:205] plot ({2*cos(\x)}, {2*sin(\x)});
\draw [ line width=1pt, domain=245:295] plot ({2*cos(\x)}, {2*sin(\x)});
\end{scope}
\draw[color=black,->,>=stealth] (-1,0)--(1,0);
\begin{scope}[shift={(4,0)}]
\draw[color=red, line width=1.5pt, ] (80:2)--(10:2);
\node at (1,1) [anchor=south west, color=red]{$\alpha$};
\draw[color=violet, line width=1.5pt, ] (-90:2)--(-10:2);
\node at (1,-1)[anchor=north west, color=violet]{$\gamma$};
\draw[color=blue, line width=1.5pt, ] (-100:2)--(190:2);
\node at (-1,-1)[anchor=north east, color=blue]{$\beta$};
\draw [ line width=1pt, domain=-25:25] plot ({2*cos(\x)}, {2*sin(\x)});
\draw [ line width=1pt, domain=65:115] plot ({2*cos(\x)}, {2*sin(\x)});
\draw [ line width=1pt, domain=155:205] plot ({2*cos(\x)}, {2*sin(\x)});
\draw [ line width=1pt, domain=245:295] plot ({2*cos(\x)}, {2*sin(\x)});
\end{scope}
\draw[color=black,->,>=stealth] (4,-3)--(4,-4);
\draw[color=black,->,>=stealth] (-4,-3)--(-4,-4);
\begin{scope}[shift={(4,-7)}]
\draw[color=red, line width=1.5pt, ] (80:2)--(10:2);
\node at (1,1) [anchor=south west, color=red]{$\alpha$};
\draw[color=violet, line width=1.5pt, ] (180:2)--(0:2);
\node at (0,0)[anchor=south, color=violet]{$\gamma$};
\draw[color=blue, line width=1.5pt, ] (-100:2)--(190:2);
\node at (-1,-1)[anchor=north east, color=blue]{$\beta$};
\draw [ line width=1pt, domain=-25:25] plot ({2*cos(\x)}, {2*sin(\x)});
\draw [ line width=1pt, domain=65:115] plot ({2*cos(\x)}, {2*sin(\x)});
\draw [ line width=1pt, domain=155:205] plot ({2*cos(\x)}, {2*sin(\x)});
\draw [ line width=1pt, domain=245:295] plot ({2*cos(\x)}, {2*sin(\x)});
\end{scope}
\draw[color=black,->,>=stealth] (-1,-7)--(1,-7);
\begin{scope}[shift={(-4,-7)}]
\draw[color=red, line width=1.5pt, ] (80:2)--(10:2);
\node at (1,1) [anchor=south west, color=red]{$\alpha$};
\draw[color=violet, line width=1.5pt, ] (170:2)--(90:2);
\node at (-1,1)[anchor=south east, color=violet]{$\gamma$};
\draw[color=blue, line width=1.5pt, ] (-100:2)--(190:2);
\node at (-1,-1)[anchor=north east, color=blue]{$\beta$};
\draw [ line width=1pt, domain=-25:25] plot ({2*cos(\x)}, {2*sin(\x)});
\draw [ line width=1pt, domain=65:115] plot ({2*cos(\x)}, {2*sin(\x)});
\draw [ line width=1pt, domain=155:205] plot ({2*cos(\x)}, {2*sin(\x)});
\draw [ line width=1pt, domain=245:295] plot ({2*cos(\x)}, {2*sin(\x)});
\end{scope}
\end{tikzpicture}
\qquad
\qquad
\qquad
\begin{tikzpicture}[scale=.4]
\begin{scope}[shift={(-4,0)}]
\draw[color=red, line width=1.5pt, ] (-80:2)--(-10:2);
\node at (1,-1) [anchor=north west, color=red]{$\alpha$};
\draw[color=violet, line width=1.5pt, ] (-90:2)--(90:2);
\node at (0,0)[anchor=west, color=violet]{$\gamma$};
\draw[color=blue, line width=1.5pt, ] (-100:2)--(190:2);
\node at (-1,-1)[anchor=north east, color=blue]{$\beta$};
\draw [ line width=1pt, domain=-25:25] plot ({2*cos(\x)}, {2*sin(\x)});
\draw [ line width=1pt, domain=65:115] plot ({2*cos(\x)}, {2*sin(\x)});
\draw [ line width=1pt, domain=155:205] plot ({2*cos(\x)}, {2*sin(\x)});
\draw [ line width=1pt, domain=245:295] plot ({2*cos(\x)}, {2*sin(\x)});
\end{scope}
\draw[color=black,->,>=stealth] (-1,0)--(1,0);
\begin{scope}[shift={(4,0)}]
\draw[color=red, line width=1.5pt, ] (-80:2)--(-10:2);
\node at (1,-1) [anchor=north west, color=red]{$\alpha$};
\draw[color=violet, line width=1.5pt, ] (-90:2)--(90:2);
\node at (0,0)[anchor=west, color=violet]{$\gamma$};
\draw[color=blue, line width=1.5pt, ] (100:2)--(170:2);
\node at (-1,1)[anchor=south east, color=blue]{$\beta$};
\draw [ line width=1pt, domain=-25:25] plot ({2*cos(\x)}, {2*sin(\x)});
\draw [ line width=1pt, domain=65:115] plot ({2*cos(\x)}, {2*sin(\x)});
\draw [ line width=1pt, domain=155:205] plot ({2*cos(\x)}, {2*sin(\x)});
\draw [ line width=1pt, domain=245:295] plot ({2*cos(\x)}, {2*sin(\x)});
\end{scope}
\draw[color=black,->,>=stealth] (4,-3)--(4,-4);
\draw[color=black,->,>=stealth] (-4,-3)--(-4,-4);
\begin{scope}[shift={(4,-7)}]
\draw[color=red, line width=1.5pt, ] (80:2)--(-10:2);
\node at (1,1) [anchor=south west, color=red]{$\alpha$};
\draw[color=violet, line width=1.5pt, ] (-90:2)--(90:2);
\node at (0,0)[anchor=west, color=violet]{$\gamma$};
\draw[color=blue, line width=1.5pt, ] (100:2)--(170:2);
\node at (-1,1)[anchor=south east, color=blue]{$\beta$};
\draw [ line width=1pt, domain=-25:25] plot ({2*cos(\x)}, {2*sin(\x)});
\draw [ line width=1pt, domain=65:115] plot ({2*cos(\x)}, {2*sin(\x)});
\draw [ line width=1pt, domain=155:205] plot ({2*cos(\x)}, {2*sin(\x)});
\draw [ line width=1pt, domain=245:295] plot ({2*cos(\x)}, {2*sin(\x)});
\end{scope}
\draw[color=black,->,>=stealth] (-1,-7)--(1,-7);
\begin{scope}[shift={(-4,-7)}]
\draw[color=red, line width=1.5pt, ] (80:2)--(-10:2);
\node at (1,1) [anchor=south west, color=red]{$\alpha$};
\draw[color=violet, line width=1.5pt, ] (-90:2)--(90:2);
\node at (0,0)[anchor=west, color=violet]{$\gamma$};
\draw[color=blue, line width=1.5pt, ] (-100:2)--(190:2);
\node at (-1,-1)[anchor=north east, color=blue]{$\beta$};
\draw [ line width=1pt, domain=-25:25] plot ({2*cos(\x)}, {2*sin(\x)});
\draw [ line width=1pt, domain=65:115] plot ({2*cos(\x)}, {2*sin(\x)});
\draw [ line width=1pt, domain=155:205] plot ({2*cos(\x)}, {2*sin(\x)});
\draw [ line width=1pt, domain=245:295] plot ({2*cos(\x)}, {2*sin(\x)});
\end{scope}
\end{tikzpicture}
\end{align}
\end{compactenum}
\end{lemma}

\section{Hopf monoid labelled ribbon graphs}
\label{sec:hrib}

The mapping class group actions we derive in this article are inspired by Kitaev's {\em lattice models} or {\em quantum double models}, which were introduced by Kitaev in \cite{Ki} for group algebras of a finite group and generalised in \cite{BMCA} to finite-dimensional semisimple Hopf-*-algebras over $\C$.  We generalise some aspects of these models to  pivotal Hopf monoids in a symmetric monoidal category $\mac$. For the case where $\mac=\vect_\C$ and $H$ is a finite-dimensional semisimple Hopf $*$-algebra, our model is equivalent to the one in \cite{BMCA}, up to minor changes in conventions. 

Throughout this section, let $H$ be a pivotal Hopf monoid in a symmetric monoidal category $\mac$ with pivotal morphism $p: e\to H$ and  
$\Gamma$ a directed ribbon graph  with vertex set $V$ and edge set $E$.
We associate to $\Gamma$ the object $H^{\oo |E|}$ in $\mac$ with different copies of $H$  assigned to the oriented edges of $\Gamma$. To emphasise this assignment, we use the notation $H^{\oo E}$. This is a direct analogue of the construction of the extended Hilbert space in \cite{Ki, BMCA}.

Edge orientation is reversed with the involutive morphism $T: H\to H$ from \eqref{eq:tplusdef}. If $\Gamma'$ is obtained from $\Gamma$ by reversing the orientation of an edge $\alpha$ and
 $f: A\oo H^{\oo E}\oo B\to C\oo H^{\oo E}\oo D$ a morphism in $\mac$  associated with  $\Gamma$, we define the associated morphism for  $\Gamma'$  as
 $T_\alpha\circ f\circ T_\alpha$, where $T_\alpha$  applies $T$ to the copy of $H$ for $\alpha$ and is the identity morphism on all other factors of the tensor product.  
 Note that for an involutive Hopf monoid we can choose $T=S$ and $p=\eta$.

To each oriented edge $\alpha$ of $\Gamma$ we assign two  $H$-left module  structures $\rhd_{\alpha\pm}: H\oo H^{\oo E}\to H^{\oo E}$ and two  $H$-left comodule structures $\delta_{\alpha\pm}: H^{\oo E}\to H\oo H^{\oo E}$. The $H$-left module structures are associated with the edge ends of $\alpha$ and the $H$-left comodule structures with the edge sides of $\alpha$ as shown below. They  are related by the involution $T: H\to H$ from \eqref{eq:tplusdef}, which corresponds to reversing the orientation of $\alpha$.

 \begin{center}
 \begin{tikzpicture}[scale=.4]
 \draw[line width=1.5pt, ->,>=stealth, color=black] (-3.7,0)--(3.7,0);
  \node at (4.7,0)[anchor=west] {$\rhd_{\alpha+}$};
    \node at (-4.7,0)[anchor=east] {$\rhd_{\alpha-}$};
    \draw[line width=.5, <-, >=stealth, color=black] (4.5, -.5)--(4.5, .5);
        \draw[line width=.5, ->, >=stealth, color=black] (-4.5, -.5)--(-4.5, .5);
      \node at (0,1.3)[anchor=south] {$\delta_{\alpha+}$};
            \node at (0,-1.4)[anchor=north] {$\delta_{\alpha-}$};
                \draw[line width=.5, ->, >=stealth, color=black] (-.5, 1.3)--(.5, 1.3);
                                \draw[line width=.5, ->, >=stealth, color=black] (.5, -1.3)--(-.5, -1.3);
            \node at (0,.2)[anchor=south]{$\alpha$};
                \draw [color=black, fill=black] (-4,0) circle (.2); 
  \draw [color=black, fill=black] (4,0) circle (.2); 
 \end{tikzpicture}
 \end{center}

\begin{definition}\label{def:triangmodcom} Let $\Gamma$ be a directed ribbon graph  and $H$ a pivotal Hopf monoid in $\mac$. The $H$-left module structures $\rhd_{\alpha\pm}: H\oo H^{\oo E}\to H^{\oo E}$ and $H$-left comodule structures $\delta_{\alpha\pm}:  H^{\oo E}\to H\oo H^{\oo E}$ assigned to an oriented edge $\alpha$ are the following morphisms in $\mac$
\begin{align}
	\nonumber
&\begin{tikzpicture}[scale=.3]
\begin{scope}[shift={(-8,0)}]
\draw[line width=1.5pt, color=red](0,2.5)--(0,-2.5);
\draw[line width=1.5pt, color=blue](2,2.5)--(2,-2.5);
\draw[line width=1.5pt, color=violet](4,2.5)--(4,-2.5);
\node at (2,2.7)[anchor=south, color=blue]{$\alpha$};
\draw[line width=1pt, color=black, draw opacity=1] plot [smooth, tension=0.6] coordinates 
      {(-2,2.5)(0,.5)(2,0) };
\end{scope}
\end{tikzpicture}\qquad\qquad
&
&
\begin{tikzpicture}[scale=.3]
\begin{scope}[shift={(-6,0)}]
\draw[line width=1.5pt, color=red](0,2.5)--(0,-2.5);
\draw[line width=1.5pt, color=blue](2,2.5)--(2,-2.5);
\draw[line width=1.5pt, color=violet](4,2.5)--(4,-2.5);
\node at (2,2.7)[anchor=south, color=blue]{$\alpha$};
\draw[color=black,line width=1pt, fill=white] (2,1) circle (.4);
\draw[color=black,line width=1pt, fill=white] (2,1) circle (.2);
\draw[line width=1pt, color=black, draw opacity=1] plot [smooth, tension=0.6] coordinates 
      {(-2,2.5)(0,.5)(2,0) };
      \draw[color=black,line width=1pt, fill=white] (2,-1) circle (.4);
\draw[color=black,line width=1pt, fill=white] (2,-1) circle (.2);
\end{scope}
\node at (0,0){$=$};
\begin{scope}[shift={(3,0)}]
\draw[line width=1.5pt, color=red](0,2.5)--(0,-2.5);
\draw[line width=1.5pt, color=blue](2,2.5)--(2,-2.5);
\draw[line width=1.5pt, color=violet](4,2.5)--(4,-2.5);
\node at (2,2.7)[anchor=south, color=blue]{$\alpha$};
\draw[line width=1pt, color=black, draw opacity=1] plot [smooth, tension=0.6] coordinates 
      {(-2,2.5)(0,.5)(3.5,-.5)(2,-1.5) };
      \draw[color=black,line width=1pt, fill=white] (3.3,-.6) circle (.4);
\end{scope}
\end{tikzpicture}\\
&\rhd_{\alpha+}: H\oo H^{\oo E}\to H^{\oo E}, & &\quad\rhd_{\alpha-}: H\oo H^{\oo E}\to H^{\oo E},\nonumber
\\[+2ex]
	\nonumber
&\begin{tikzpicture}[scale=.3]
\begin{scope}[shift={(-8,0)}]
\draw[line width=1.5pt, color=red](0,2.5)--(0,-2.5);
\draw[line width=1.5pt, color=blue](2,2.5)--(2,-2.5);
\draw[line width=1.5pt, color=violet](4,2.5)--(4,-2.5);
\node at (2,2.7)[anchor=south, color=blue]{$\alpha$};
\draw[line width=1pt, color=black, draw opacity=1] plot [smooth, tension=0.6] coordinates 
      {(-2,-2.5)(0,-.5)(2,0) };
\end{scope}
\end{tikzpicture}
&
&\begin{tikzpicture}[scale=.3]
\begin{scope}[shift={(-6,0)}]
\draw[line width=1.5pt, color=red](0,2.5)--(0,-2.5);
\draw[line width=1.5pt, color=blue](2,2.5)--(2,-2.5);
\draw[line width=1.5pt, color=violet](4,2.5)--(4,-2.5);
\node at (2,2.7)[anchor=south, color=blue]{$\alpha$};
\draw[color=black,line width=1pt, fill=white] (2,1) circle (.4);
\draw[color=black,line width=1pt, fill=white] (2,1) circle (.2);
\draw[line width=1pt, color=black, draw opacity=1] plot [smooth, tension=0.6] coordinates 
      {(-2,-2.5)(0,-.5)(2,0) };
      \draw[color=black,line width=1pt, fill=white] (2,-1) circle (.4);
\draw[color=black,line width=1pt, fill=white] (2,-1) circle (.2);
\end{scope}
\node at (0,0){$=$};
\begin{scope}[shift={(3,0)}]
\draw[line width=1.5pt, color=red](0,2.5)--(0,-2.5);
\draw[line width=1.5pt, color=blue](2,2.5)--(2,-2.5);
\draw[line width=1.5pt, color=violet](4,2.5)--(4,-2.5);
\node at (2,2.7)[anchor=south, color=blue]{$\alpha$};
\draw[line width=1pt, color=black, draw opacity=1] plot [smooth, tension=0.6] coordinates 
      {(-2,-2.5)(0,-.5)(3.5,.5)(2,1.5) };
      \draw[color=black,line width=1pt, fill=white] (3.3,.6) circle (.4);
            \draw[color=black,line width=1pt, fill=white] (3.3,.6) circle (.2);
\end{scope}
\end{tikzpicture}\\
&\delta_{\alpha+}: H^{\oo E}\to H\oo H^{\oo E}, & &\quad\delta_{\alpha-}: H^{\oo E}\to H\oo H^{\oo E} ,
\nonumber
\end{align}
where the $H$-left and $H$-right (co)module structures on $H$ in the diagrams are given by left and right (co)multiplication, as in Example \ref{ex:hopfbimodule}.
\end{definition}

These $H$-left module and comodule  structures correspond to the  {\em triangle operators} in Kitaev's lattice models.  If $H$ is a finite-dimensional semisimple Hopf algebra over  $\C$ with $p=\eta$, then the
former correspond to the operators $L^h_{\alpha\pm}: H^{\oo E}\to H^{\oo E}$ for $h\in H$ in \cite{BMCA}.
 The morphisms $\delta_{\alpha\pm}$ induce two $H^*$-right module structures on $H^{\oo E}$.
 Up to antipodes that transform right $H^*$- into left $H^*$-modules, these define  the operators $T^\beta_{\alpha\pm}: H^{\oo E}\to H^{\oo E}$ in \cite{BMCA}.
The following lemma generalises the  relations between these operators from \cite{BMCA}. 

\begin{lemma} \label{lem:triangleops} Let $\alpha$ be an edge in a directed ribbon graph $\Gamma$ and $H$ a pivotal Hopf monoid in $\mac$. 
\begin{compactenum}
\item The triples $(H^{\oo E}, \rhd_{\alpha+}, \delta_{\alpha+})$ and $(H^{\oo E}, \rhd_{\alpha-},\delta_{\alpha-})$ are $H$-left-left Hopf modules in $\mac$,
\item The two $H$-left module structures $\rhd_{\alpha\pm}$ and $H$-left comodule structures $\delta_{\alpha\pm}$  commute:
\begin{align*}
&\rhd_{\alpha-}\circ (1_H\oo \rhd_{\alpha+})= \rhd_{\alpha+} \circ (1_H\oo \rhd_{\alpha-})\circ (\tau_{H,H}\oo 1_{H^{\oo E}})\\
&(1_H\oo \delta_{\alpha-})\circ \delta_{\alpha+}=(\tau_{H,H}\oo 1_{H^{\oo E}})\circ (1_H\oo\delta_{\alpha+})\circ \delta_{\alpha-}.
\end{align*}
\end{compactenum}
\end{lemma}

\begin{proof} The first claim for  $(H^{\oo E}, \rhd_{\alpha+}, \delta_{\alpha+})$ follows directly from the definition of $\rhd_{\alpha\pm}$ and Example \ref{ex:hopfbimodule}, and the first claim for $(H^{\oo E}, \rhd_{\alpha-},\delta_{\alpha-})$ follows from the one for $(H^{\oo E}, \rhd_{\alpha+}, \delta_{\alpha+})$ and the fact that $T$ is an involution. The second claim is a direct consequence of (co)associativity for $H$.
\end{proof}

To every ciliated vertex $v$ and ciliated face $f$ in  $\Gamma$ we assign, respectively, an $H$-left module and an $H$-left comodule structure on $H^{\oo E}$. These are constructed from  the $H$-left
module structures for the edges incident at $v$ and the $H$-left comodule structures for the edges traversed by $f$. 

The $H$-left module structure for a ciliated vertex $v$ with $n$ incident edge ends  is obtained by applying the 
comultiplication  $\Delta^{n-1}: H\to H^{\oo n}$.  The different factors in the tensor product then act on the copies of $H$ associated to the edge ends at $v$, according to their  ordering. 
 The action on an edge end is $\rhd_{\alpha+}$ if the end of $\alpha$ is incoming and $\rhd_{\alpha-}$ if it is outgoing at $v$.

The $H$-left comodule structure for a ciliated face $f$ that traverses $n$ edges is obtained  from the comodule structures at each edge in $f$, where  one takes the left $H$-coaction $\delta_{\alpha+}$ if the edge $\alpha$ is traversed parallel to its orientation and $\delta_{\alpha-}$ if it is traversed against its orientation. One then applies the 
multiplication  $m^{n-1}: H^{\oo n}\to H$,  according to the order of the edges in $f$.

\begin{definition} \label{def:vertexface} Let $\Gamma$ be a directed ribbon graph and $H$ a pivotal Hopf monoid in $\mac$.
\begin{compactenum}
\item The $H$-left module structure $\rhd_v: H\oo H^{\oo E}\to H^{\oo E}$ for a ciliated vertex $v$ with incident edge ends $\alpha_1<\alpha_2<...<\alpha_n$ is 
$$
\rhd_v= \rhd_{\alpha_1}\circ (1_H\oo \rhd_{\alpha_2}) ...\circ (1_{H^{\oo (n-1)}}\oo \rhd_{\alpha_n})\circ (\Delta^{n-1}\oo 1_{H^{\oo E}}),
$$
where  $\rhd_{\alpha_i}=\rhd_{e(\alpha_i)+}$ and  $\rhd_{\alpha_i}=\rhd_{e(\alpha_i)-}$, respectively,  if the edge end $\alpha_i$ is incoming and outgoing  at $v$ and  $e(\alpha_i)$ denotes the edge for  the edge end $\alpha_i$. 

\item The $H$-left comodule structure $\delta_f:  H^{\oo E}\to H\oo H^{\oo E}$ for a ciliated face $f=\alpha_1^{\epsilon_1}\circ \ldots \circ \alpha_r^{\epsilon_r}$ 
 is 
$$
\delta_f= m^{r-1}\circ 
 (1_{H^{\oo(r-1)}}\oo\delta_{\alpha_r\epsilon_r})\circ \ldots\circ(1_H\oo \delta_{\alpha_2\epsilon_2})\circ \delta_{\alpha_1\epsilon_1}.
$$
\end{compactenum}
\end{definition}

In the following we will sometimes describe the module or comodule structure associated with a vertex or face with Sweedler notation on the edge labels. This  is to be understood as a shorthand notation for a diagram that describes a morphism in $\mac$, as illustrated in the following example.

\begin{example}\label{def:vertexaction}  The $H$-left module  and comodule structure for the ciliated vertex and face from Figure \ref{fig:ciliated} are given as follows

\begin{align*}
&\begin{tikzpicture}[scale=.5, baseline=(current bounding box.center)]
\begin{scope}[shift={(-6,0)}]
\draw [color=black, fill=black] (0,0) circle (.2); 
\draw[color=black, style=dotted, line width=1pt] (0,-.2)--(0,-1);
\draw [black,->,>=stealth,domain=-90:240] plot ({cos(\x)}, {sin(\x)});
\draw[color=red, line width=1.5pt, <-,>=stealth] (.2,-.2)--(2,-2);
\draw[color=violet, line width=1.5pt, <-,>=stealth] (.2,0)..controls (6,0) and (0,6).. (0,.2);
\draw[color=cyan, line width=1.5pt, ->,>=stealth] (.2,.2)..controls (4,4) and (-4,4).. (-.2,.2);
\draw[color=magenta, line width=1.5pt, <-,>=stealth] (-.2,0)--(-3,0);
\node at (2,-2)[color=red, anchor=north west]{$a$};
\node at (2.5,2.5)[color=violet, anchor=west]{$b$};
\node at (0,3.2)[color=cyan, anchor=south]{$c$};
\node at (-3,0)[color=magenta, anchor=east]{$d$};
\node at (0,-1) [color=black, anchor=north] {$h$};
\end{scope}
\draw[line width=1pt, color=black, ->,>=stealth](-3,0)--(-2,0);
\begin{scope}[shift={(3.5,0)}]
\draw [color=black, fill=black] (0,0) circle (.2); 
\draw[color=black, style=dotted, line width=1pt] (0,-.2)--(0,-1);
\draw[color=red, line width=1.5pt, <-,>=stealth] (.2,-.2)--(2,-2);
\draw[color=violet, line width=1.5pt, <-,>=stealth] (.2,0)..controls (6,0) and (0,6).. (0,.2);
\draw[color=cyan, line width=1.5pt, ->,>=stealth] (.2,.2)..controls (4,4) and (-4,4).. (-.2,.2);
\draw[color=magenta, line width=1.5pt, <-,>=stealth] (-.2,0)--(-3,0);
\node at (2,-2)[color=red, anchor=north west]{$\low h 1 a$};
\node at (2.5,2.5)[color=violet, anchor=west]{$\low h 2 b  S(\low h 4)$};
\node at (0,3.2)[color=cyan, anchor=south]{$\low h 5  c   S(\low h 3)$};
\node at (-3,0)[color=magenta, anchor=east]{$\low h 6  d$};
\end{scope}
\end{tikzpicture}
\qquad\qquad
\begin{tikzpicture}[scale=.3, baseline=(current bounding box.center)]
\draw[line width=1.5pt, color=red] (0,3)--(0,-5.5);
\draw[line width=1.5pt, color=violet] (2,3)--(2,-5.5);
\draw[line width=1.5pt, color=cyan] (4,3)--(4,-5.5);
\draw[line width=1.5pt, color=magenta] (6,3)--(6,-5.5);
\draw[color=black, line width=1pt] (-4.5,3)--(-4.5,2);
\draw[color=black, line width=1pt] (-7,2)--(-2,2);
\node at (0,3)[color=red, anchor=south]{$a$};
\node at (2,3)[color=violet, anchor=south]{$b$};
\node at (4,3)[color=cyan, anchor=south]{$c$};
\node at (6,3)[color=magenta, anchor=south]{$d$};
\draw[line width=1pt, color=black] plot [smooth, tension=0.6] coordinates 
      {(-2,2)(-1,1)(6,.5)};
 \draw[line width=1pt, color=black] plot [smooth, tension=0.6] coordinates 
      {(-3,2)(-1,0)(4,-.5)};
      \draw[line width=1pt, color=black] plot [smooth, tension=0.6] coordinates 
      {(-4,2)(-1.5,-.5)(3,-1.5)(2,-2)};
      \draw[line width=1pt, color=black] plot [smooth, tension=0.6] coordinates 
      {(-5,2)(-1,-2)(5,-3)(4,-3.5)};
      \draw[line width=1pt, color=black] plot [smooth, tension=0.6] coordinates 
      {(-6,2)(-2,-2.5)(2,-4)};
      \draw[line width=1pt, color=black] plot [smooth, tension=0.6] coordinates 
      {(-7,2)(-2,-3.5)(0,-4.5)};
      \draw[line width=1pt, color=black, fill=white] (3,-1.6)circle (.4);
            \draw[line width=1pt, color=black, fill=white] (5,-3.1)circle (.4);
\end{tikzpicture}
\end{align*}

\begin{align*}
\begin{tikzpicture}[scale=.5, baseline=(current bounding box.center)]
\begin{scope}[shift={(-6,0)}]
\draw [color=black, fill=black] (-2,0) circle (.2); 
\draw [color=black, fill=black] (2,0) circle (.2); 
\draw [color=black, fill=black] (1,2) circle (.2); 
\draw [color=black, fill=black] (4,4) circle (.2); 
\draw [color=black, fill=black] (-4,4) circle (.2); 
\draw[color=black, line width=1pt, style=dotted](1.9,.1)--(1.3,.7);
\draw[color=blue, line width=1.5pt, ->,>=stealth] (-1.8,0)--(1.8,0);
\node at (0,0)[color=blue, anchor=north]{$a$};
\draw[color=red, line width=1.5pt, ->,>=stealth] (-1.8,.2).. controls (2,0) and (-2,4).. (-2,.2);
\node at (-1.2,1)[color=red, anchor=west]{$f$};
\draw[color=magenta, line width=1.5pt, ->,>=stealth] (-2.2,.2)--(-3.8,3.8);
\node at (-3,2)[color=magenta, anchor=east]{$e$};
\draw[color=cyan, line width=1.5pt, ->,>=stealth] (2,.2)--(4,3.8);
\node at (3,2)[color=cyan, anchor=west]{$b$};
\draw[color=violet, line width=1.5pt, ->,>=stealth] (3.8,4)--(-3.8,4);
\node at (0,4)[color=violet, anchor=south]{$c$};
\draw[color=brown, line width=1.5pt, ->,>=stealth] (3.8,3.8)--(1.2,2.2);
\node at (1.8,2.8)[color=brown, anchor=south]{$d$};
\draw[color=black, line width=1.5pt] (-2.2,0)--(-2.8,0);
\draw[color=black, line width=1.5pt ] (-2,-.2)--(-2,-.8);
\draw[color=black, line width=1.5pt ] (-2.1,-.1)--(-2.6,-.6);
\draw[color=black, line width=1.5pt] (2.2,0)--(2.8,0);
\draw[color=black, line width=1.5pt ] (2,-.2)--(2,-.8);
\draw[color=black, line width=1.5pt ] (2.1,-.1)--(2.6,-.6);
\draw[color=black, line width=1.5pt] (-4,4.2)--(-4,4.6);
\draw[color=black, line width=1.5pt ] (-4.2,4)--(-4.8,4);
\draw[color=black, line width=1.5pt ] (-4.1,4.1)--(-4.6,4.6);
\draw[color=black, line width=1.5pt] (4,4.2)--(4,4.6);
\draw[color=black, line width=1.5pt ] (4.2,4)--(4.8,4);
\draw[color=black, line width=1.5pt ] (4.1,4.1)--(4.6,4.6);
\draw[line width=.5pt, color=black,->,>=stealth] plot [smooth, tension=0.6] coordinates 
      {(1.6,.4)(3.4,3.2)(1.5,2)(1,1.5)(.5,2)(2.8,3.6)(0,3.6)(-3.2,3.6)(-2.6,2) (-2.2,.9)(-1.8,2)(-.5,2)(.3,1)(0,.4)(1.4,.4)};
\end{scope}
\draw[color=black, line width=1pt,->,>=stealth](-2,2)--(-1,2);
\begin{scope}[shift={(4,0)}]           
      \draw [color=black, fill=black] (-2,0) circle (.2); 
\draw [color=black, fill=black] (2,0) circle (.2); 
\draw [color=black, fill=black] (1,2) circle (.2); 
\draw [color=black, fill=black] (4,4) circle (.2); 
\draw [color=black, fill=black] (-4,4) circle (.2); 
\draw[color=blue, line width=1.5pt, ->,>=stealth] (-1.8,0)--(1.8,0);
\node at (0,0)[color=blue, anchor=north]{$\low a 2$};
\draw[color=red, line width=1.5pt, ->,>=stealth] (-1.8,.2).. controls (2,0) and (-2,4).. (-2,.2);
\node at (-1.8,1)[color=red, anchor=west]{$\low f 1$};
\draw[color=magenta, line width=1.5pt, ->,>=stealth] (-2.2,.2)--(-3.8,3.8);
\node at (-3,2)[color=magenta, anchor=east]{$\low e 1$};
\draw[color=cyan, line width=1.5pt, ->,>=stealth] (2,.2)--(4,3.8);
\node at (3,1.9)[color=cyan, anchor=west]{$\low b 2$};
\draw[color=violet, line width=1.5pt, ->,>=stealth] (3.8,4)--(-3.8,4);
\node at (0,4)[color=violet, anchor=south]{$\low c 2$};
\draw[color=brown, line width=1.5pt, ->,>=stealth] (3.8,3.8)--(1.2,2.2);
\node at (1.8,2.65)[color=brown, anchor=south]{$\low d 2$};
\draw[color=black, line width=1.5pt] (-2.2,0)--(-2.8,0);
\draw[color=black, line width=1.5pt ] (-2,-.2)--(-2,-.8);
\draw[color=black, line width=1.5pt ] (-2.1,-.1)--(-2.6,-.6);
\draw[color=black, line width=1.5pt] (2.2,0)--(2.8,0);
\draw[color=black, line width=1.5pt ] (2,-.2)--(2,-.8);
\draw[color=black, line width=1.5pt ] (2.1,-.1)--(2.6,-.6);
\draw[color=black, line width=1.5pt] (-4,4.2)--(-4,4.6);
\draw[color=black, line width=1.5pt ] (-4.2,4)--(-4.8,4);
\draw[color=black, line width=1.5pt ] (-4.1,4.1)--(-4.6,4.6);
\draw[color=black, line width=1.5pt] (4,4.2)--(4,4.6);
\draw[color=black, line width=1.5pt ] (4.2,4)--(4.8,4);
\draw[color=black, line width=1.5pt ] (4.1,4.1)--(4.6,4.6);
\node at (-8,-2) [anchor=west, color=black]{$\low a 1 T(\low f 2) T(\low e 2)\low c {1} T(\low d 3)\low d {1} \low b {1} 
$};
\end{scope}
\end{tikzpicture}
\qquad 
\begin{tikzpicture}[scale=.25, baseline=(current bounding box.center)]
\draw[line width=1.5pt, color=blue] (0,3)--(0,-8);
\draw[line width=1.5pt, color=red] (2,3)--(2,-8);
\draw[line width=1.5pt, color=magenta] (4,3)--(4,-8);
\draw[line width=1.5pt, color=violet] (6,3)--(6,-8);
\draw[line width=1.5pt, color=brown] (8,3)--(8,-8);
\draw[line width=1.5pt, color=cyan] (10,3)--(10,-8);
\node at (0,3)[color=blue, anchor=south]{$a$};
\node at (2,3)[color=red, anchor=south]{$f$};
\node at (4,3)[color=magenta, anchor=south]{$e$};
\node at (6,3)[color=violet, anchor=south]{$c$};
\node at (8,3)[color=brown, anchor=south]{$d$};
\node at (10,3)[color=cyan, anchor=south]{$b$};
 \draw[line width=1pt, color=black] plot [smooth, tension=0.6] coordinates 
      {(0,2)(-2,1)(-8,-6)};
 \draw[line width=1pt, color=black] plot [smooth, tension=0.6] coordinates 
      {(2,2)(3,1)(-1.5,0)(-7,-6)};
       \draw[line width=1pt, color=black] plot [smooth, tension=0.6] coordinates 
      {(4,1) (5,0)(-1,-1)(-6,-6)};
             \draw[line width=1pt, color=black] plot [smooth, tension=0.6] coordinates 
      {(6,-.7)(-1,-2)(-5,-6)};
                   \draw[line width=1pt, color=black] plot [smooth, tension=0.6] coordinates 
      {(8,-3)(-1,-4)(-3,-6)};
                         \draw[line width=1pt, color=black] plot [smooth, tension=0.6] coordinates 
      {(8,-1)(9,-2)(-1,-3)(-4,-6)};
                         \draw[line width=1pt, color=black] plot [smooth, tension=0.4] coordinates 
      {(10,-4)(-1,-5)(-2,-6)};
      \draw[line width=1pt, color=black](-8,-6)--(-2,-6);
      \draw[line width=1pt, color=black](-5,-6)--(-5,-8); 
       \draw[color=black, line width=1pt, fill=white] (3,1) circle (.4); 
            \draw[color=black, line width=1pt, fill=white] (3,1) circle (.2);  
        \draw[color=black, line width=1pt, fill=white] (5,0) circle (.4); 
         \draw[color=black, line width=1pt, fill=white] (5,0) circle (.2); 
              \draw[color=black, line width=1pt, fill=white] (9.3,-1.8) circle (.4); 
                 \draw[color=black, line width=1pt, fill=white] (9.3,-1.8) circle (.2);     
\end{tikzpicture}
\end{align*}
\end{example}

The $H$-left (co)module  structures for ciliated vertices and  faces correspond to the {\em vertex} and {\em face operators} in Kitaev's lattice models. For a finite-dimensional semisimple Hopf algebra $H$ in $\mac=\vect_\C$, the vertex operator $A_v^h$ for an element $h\in H$ from \cite{BMCA} is  $A^h_v=h\rhd_v-: H^{\oo E}\to H^{\oo E}$ and the face operator for  an element $\beta\in H^*$ by $B_f^\beta=( S(\beta) \oo 1_{H^{\oo E}})\circ \delta_f: H^{\oo E}\to H^{\oo E}$. The $H$-left (co)module structures   thus have properties similar to the vertex and face operators.

\begin{lemma}\label{lem:facevertex} Let $\Gamma$ be a directed ribbon graph and $H$ a pivotal Hopf monoid in $\mac$. The $H$-left module and comodule structures associated with ciliated vertices and faces of $\Gamma$ satisfy:
\begin{compactenum}
\item All $H$-left module structures for distinct  vertices $v,v'\in V$  commute:  
\begin{align*}
&\rhd_{v'}\circ (1_H\oo \rhd_v)= \rhd_v \circ (1_H\oo \rhd_{v'})\circ(\tau_{H,H}\oo 1_{H^{\oo E}}).
\end{align*}
\item If two ciliated faces $f,f'$ are not related by cyclic permutations, the associated $H$-comodule structures commute:
\begin{align*}
&(1_H\oo \delta_{f'})\circ \delta_f=(\tau_{H,H}\oo 1_{H^{\oo E}})\circ (1_H\oo \delta_f)\circ \delta_{f'}.
\end{align*}
\item If two cilia are at distinct vertices and  distinct faces, the $H$-module structure for one of them commutes with the $H$-comodule structure for the other:
\begin{align*}
\delta_f\circ \rhd_v=(1_H\oo \rhd_v)\circ (\tau_{H,H}\oo 1_{H^{\oo E}})\circ (1_H\oo \delta_f)
\end{align*}
\item If a ciliated vertex $v$ and a ciliated face $f$ share a cilium, the associated $H$-left module and $H$-left comodule structure form a $H$-left-left  Yetter-Drinfeld module:
\begin{align*}
&(m\circ \tau_{H,H}\oo 1_{H^{\oo E}})\circ(1_H\oo \delta_f)\circ(1_H\oo \rhd_v)\circ (\tau_{H,H}\circ \Delta\oo 1_{H^{\oo E}})\\
=&(m\oo \rhd_v) \circ (1_H\oo \tau_{H,H}\oo 1_{H^{\oo E}}) \circ (\Delta\oo \delta_f)
\end{align*}
\end{compactenum}
\end{lemma}

\begin{proof} The proof of these statements is analogous to the one in \cite{BMCA}, see also \cite{Me}. The first two  follow directly from the definition of the vertex and face operators and from Lemma \ref{lem:triangleops}, 2. The third is obvious if the vertex and face do not share an edge. If they share an edge, but not a cilium, if is sufficient to consider the constellation in case (a), and the graphical proof  in Figure \ref{fig:facevertex} (a). The other cases follow by applying the involution $T$ from \eqref{eq:tplusdef} to reverse edge orientation.
\begin{align*}
&\text{(a)}\quad
\begin{tikzpicture}[scale=.6, baseline=(current bounding box.center)]
\draw [color=black, fill=black] (2,0) circle (.2); 
\draw[line width=1.5pt, color=red,->, >=stealth](-2,0)--(1.8,0);
\draw[line width=1.5pt, color=blue,->, >=stealth](-1.5,1.8)--(1.8,.2);
\draw[color=black, style=dotted, line width=1pt] (2.2,0)--(3.2,0);
\draw [black,->,>=stealth,domain=5:345] plot ({2+cos(\x)}, {sin(\x)});
\draw[line width=1.5 pt, color=black] (2,0)--(2,-.8);
\draw[line width=1.5 pt, color=black] (2,0)--(2,.8);
\draw[line width=1.5 pt, color=black] (2,0)--(2.5,.5);
\draw[line width=1.5 pt, color=black] (2,0)--(2.5,-.5);
\draw[line width=.5pt, color=black,->,>=stealth] plot [smooth, tension=0.6] coordinates 
{(-2,.5)(.5,.5)(-1.5,1.3)};
\end{tikzpicture}
\qquad\qquad
\text{(b)}\quad
\begin{tikzpicture}[scale=.6, baseline=(current bounding box.center)]
\draw[line width=1.5pt, color=blue,->, >=stealth](-2,0)--(1.8,0);
\draw[line width=1.5pt, color=red,->, >=stealth](2,2.5)--(2,.2);
\draw[line width=1.5pt, color=violet,->, >=stealth](-2,2.5)--(-2,.2);
\draw[line width=1.5pt, color=violet,->, >=stealth](0,2.5)--(-1.8,2.5);
\draw[line width=1.5pt, color=violet](2,2.5)--(1.5,2.5);
\node at (1,2.5){$\ldots$};
\draw[line width=1.5pt, color=violet,->, >=stealth](.8,2.5)--(.2,2.5);
\draw [color=black, fill=black] (2,0) circle (.2); 
\draw [color=black, fill=black] (2,2.5) circle (.2); 
\draw [color=black, fill=black] (-2,0) circle (.2); 
\draw [color=black, fill=black] (-2,2.5) circle (.2); 
\draw [color=black, fill=black] (0,2.5) circle (.2); 
\draw[color=black, style=dotted, line width=1pt] (2,0)--(1,1);
\draw [black,->,>=stealth,domain=135: 475] plot ({2+cos(\x)}, {sin(\x)});
\draw[line width=1.5 pt, <-, >=stealth,color=magenta] (2.2,0)--(3.2,0);
\draw[line width=1.5 pt, <-, >=stealth,color=magenta] (2.2,.2)--(2.8,.8);
\draw[line width=1.5 pt, <-, >=stealth,color=magenta] (2.2,-.2)--(2.8,-.8);
\draw[line width=.5pt, color=black,->,>=stealth] plot [smooth, tension=0.6] coordinates 
{(1.5,.8)(1.5,2)(-1.5,2)(-1.5,.5)(1,.5)};
\end{tikzpicture}
\end{align*}
To prove the last claim, it is sufficient to consider the constellation in (b) and to note that  condition \eqref{eq:ydcomp} for a left-left Yetter-Drinfeld module is equivalent to the condition
\begin{align}\nonumber
\begin{tikzpicture}[scale=.25]
\begin{scope}[shift={(-2.5,0)}]
\draw[line width=1.5pt, color=blue] (1,2)--(1,-5);
 \draw[line width=1pt, color=black] plot [smooth , tension=0.6] coordinates 
            {(-2.5,2) (-1.5,.5)(1,-.5)};
             \draw[line width=1pt, color=black] plot [smooth , tension=0.6] coordinates 
            {(-2.5,-5) (-1.5,-3.5)(1,-2.5)};
\end{scope}
\node at (0,-1.5){$=$};
\begin{scope}[shift={(5.5,0)}]
\draw[line width=1.5pt, color=blue] (1,2)--(1,-5);
\draw[line width=1pt, color=black] (-2.5,2)--(-2.5,1.2);
\draw[line width=1pt, color=black] (-2.5,-4.2)--(-2.5,-5);
             \draw[line width=1pt, color=black] plot [smooth , tension=0.6] coordinates 
            {(-.9,0) (-2,1) (-3,1)  (-4,0) (-4, -3) (-3,-4)  (-2,-4)   (-1,-3.1) };
 \draw[line width=1pt, color=black] plot [smooth , tension=0.6] coordinates 
            {(1,1)(-2,-2)(-1.5,-3) (-.5,-3) (0,-2) (0,-1) (-.5,0) (-1.5,0) (-2,-1)(1,-3)};    
            \draw[color=black, line width=1pt, fill=white] (0,-1.3) circle (.4);   
            \end{scope}    
\end{tikzpicture}
\tDot
\end{align}
The graphical proof  is given in  Figure \ref{fig:facevertex} (b), and the identity for other edge orientations again follows by applying the involution $T$.
\end{proof}

\begin{figure}
\centering
\def\svgwidth{.9\columnwidth}
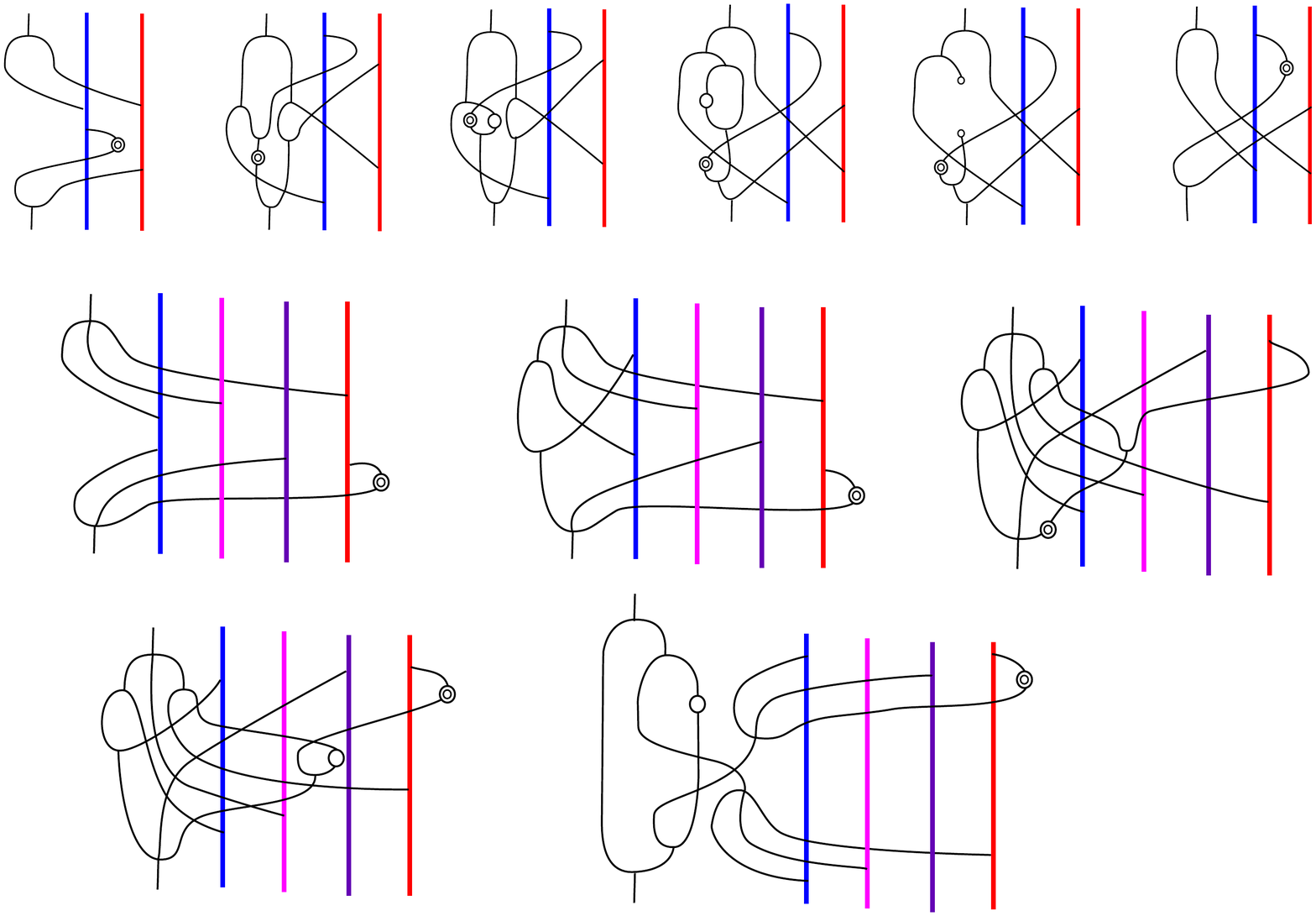
\caption{Diagrammatic proof of  Lemma \ref{lem:facevertex}}
\label{fig:facevertex}
\end{figure}

\section{Mapping class group actions  by edge slides}
\label{sec:chordslideshopf}

In this section, we consider a generalisation of the chord slides from Section \ref{sec:chordslides}.  Throughout this section, let  $\Gamma$ be a directed ribbon graph with edge set $E$. 

The slide is defined in analogy to the chord slides in Section \ref{sec:chordslides}, but not restricted to ribbon graphs with a single vertex. It slides  an end of  an oriented edge $\beta$  along another oriented edge $\alpha$ that
shares a vertex with this end of $\beta$ and is adjacent to it with respect to the cyclic ordering. 
 If some or all  of the vertices of the ribbon graph are ciliated, we sometimes impose that edge ends do not slide over cilia. This means that the relevant end of $\alpha$ is required to be adjacent to $\beta$ with respect to the {\em linear} ordering of the edges at this  vertex. 

As $\alpha$ is an oriented edge, we distinguish slides to the start and to the target of $\alpha$ and slides on the right and on the left of $\alpha$, viewed in the direction of its orientation. 
This yields the four different cases below. We denote by  $S_{\alpha^L}$ and $S_{\alpha^R}$ the slides  to the {\em target end} of $\alpha$, on the {\em left} and {\em right} of $\alpha$, respectively, and by $S_{\alpha^{-L}}$ and $S_{\alpha^{-R}}$ their inverses. In particular, we set
 $S_{(\alpha^\inv)^L}=S_{\alpha^{-R}}$, where $\alpha^\inv$ denotes the edge $\alpha$ with the reversed orientation. 
\begin{align}\label{eq:slideedge}
&\begin{tikzpicture}[scale=.4]
\begin{scope}[shift={(-5,0)}]
\draw [color=orange, fill=orange] (-2,0) circle (.2); 
\draw [color=violet, fill=violet] (2,0) circle (.2); 
\draw[color=blue, line width=1.5pt, ->,>=stealth] (-.2,1.8)--(-1.8,.2);
\draw[color=red, line width=1.5pt, ->,>=stealth] (-1.8,0)--(1.8,0);
\draw[color=black, line width=1pt] (-2.2,0)--(-2.8,0);
\draw[color=black, line width=1pt] (-2.2,.2)--(-2.6,.6);
\draw[color=black, line width=1pt] (-2.2,-.2)--(-2.6,-.6);
\draw[color=black, line width=1pt] (2.2,0)--(2.8,0);
\draw[color=black, line width=1pt] (2.2,.2)--(2.6,.6);
\draw[color=black, line width=1pt] (2.2,-.2)--(2.6,-.6);
\node at (0,0)[anchor=north, color=red]{$\alpha$};
\node at (-1,1)[anchor=south east, color=blue]{$\beta$};
\end{scope}
\draw[color=black,->,>=stealth] (-1,.5)--(1,.5);
\node at (0,0.5) [anchor=south, color=black] {$S_{\alpha^L}$};
\draw[color=black,<-,>=stealth] (-1,-.5)--(1,-.5);
\node at (0,-0.5) [anchor=north, color=black] {$S_{\alpha^{-L}}$};
\begin{scope}[shift={(5,0)}]
\draw [color=orange, fill=orange] (-2,0) circle (.2); 
\draw [color=violet, fill=violet] (2,0) circle (.2); 
\draw[color=blue, line width=1.5pt, ->,>=stealth] (.2,1.8)--(1.8,.2);
\draw[color=red, line width=1.5pt, ->,>=stealth] (-1.8,0)--(1.8,0);
\draw[color=black, line width=1pt] (-2.2,0)--(-2.8,0);
\draw[color=black, line width=1pt] (-2.2,.2)--(-2.6,.6);
\draw[color=black, line width=1pt] (-2.2,-.2)--(-2.6,-.6);
\draw[color=black, line width=1pt] (2.2,0)--(2.8,0);
\draw[color=black, line width=1pt] (2.2,.2)--(2.6,.6);
\draw[color=black, line width=1pt] (2.2,-.2)--(2.6,-.6);
\node at (0,0)[anchor=north, color=red]{$\alpha$};
\node at (1,1)[anchor=south west, color=blue]{$ \beta$};
\end{scope}
\end{tikzpicture}\qquad
&\begin{tikzpicture}[scale=.4]
\begin{scope}[shift={(-5,0)}]
\draw [color=orange, fill=orange] (-2,0) circle (.2); 
\draw [color=violet, fill=violet] (2,0) circle (.2); 
\draw[color=blue, line width=1.5pt, ->,>=stealth] (-.2,-1.8)--(-1.8,-.2);
\draw[color=red, line width=1.5pt, ->,>=stealth] (-1.8,0)--(1.8,0);
\draw[color=black, line width=1pt] (-2.2,0)--(-2.8,0);
\draw[color=black, line width=1pt] (-2.2,.2)--(-2.6,.6);
\draw[color=black, line width=1pt] (-2.2,-.2)--(-2.6,-.6);
\draw[color=black, line width=1pt] (2.2,0)--(2.8,0);
\draw[color=black, line width=1pt] (2.2,.2)--(2.6,.6);
\draw[color=black, line width=1pt] (2.2,-.2)--(2.6,-.6);
\node at (0,0)[anchor=north, color=red]{$\alpha$};
\node at (-1,-1)[anchor=north east, color=blue]{$\beta$};
\end{scope}
\draw[color=black,->,>=stealth] (-1,.5)--(1,.5);
\node at (0,0.5) [anchor=south, color=black] {$S_{\alpha^R}$};
\draw[color=black,<-,>=stealth] (-1,-.5)--(1,-.5);
\node at (0,-0.5) [anchor=north, color=black] {$S_{\alpha^{-R}}$};
\begin{scope}[shift={(5,0)}]
\draw [color=orange, fill=orange] (-2,0) circle (.2); 
\draw [color=violet, fill=violet] (2,0) circle (.2); 
\draw[color=blue, line width=1.5pt, ->,>=stealth] (-.2,-1.8)--(1.8,-.2);
\draw[color=red, line width=1.5pt, ->,>=stealth] (-1.8,0)--(1.8,0);
\draw[color=black, line width=1pt] (-2.2,0)--(-2.8,0);
\draw[color=black, line width=1pt] (-2.2,.2)--(-2.6,.6);
\draw[color=black, line width=1pt] (-2.2,-.2)--(-2.6,-.6);
\draw[color=black, line width=1pt] (2.2,0)--(2.8,0);
\draw[color=black, line width=1pt] (2.2,.2)--(2.6,.6);
\draw[color=black, line width=1pt] (2.2,-.2)--(2.6,-.6);
\node at (0,0)[anchor=north, color=red]{$\alpha$};
\node at (1,-1)[anchor=north west, color=blue]{$ \beta$};
\end{scope}
\end{tikzpicture}
\end{align}

The edge slides for the other orientation of $\beta$  are defined analogously. 
The short black lines in  \eqref{eq:slideedge}   indicate other edge ends or cilia at the vertices. Their  position remains unchanged under the edge slide, just as the position of the starting end of $\beta$.  Note also  that the two vertices in  \eqref{eq:slideedge}  may coincide and that the other end of $\beta$ may be incident at the same vertex. The only conditions are
\begin{compactenum}[(i)]
 	\item The sliding end of $\beta$  shares a vertex with the starting or target  end of  $\alpha$.
		\label{cond:ShareVertex}
	\item It comes directly after the starting end  (before the target end) of $\alpha$ with respect to the  ordering at their common vertex  for the slide $S_{\alpha^L}$ (for the slide $S_{\alpha^{-L}}$) and directly before the starting end (after the target end) of $\alpha$   for the slide $S_{\alpha^R}$ (for the slide $S_{\alpha^{-R}}$).	
	\item The edge $\beta$ is not identical to the edge $\alpha$. 
		\label{cond:NotIdenticalEdge}
 \end{compactenum}
 
 If $\Gamma'$ is obtained from a directed ribbon graph $\Gamma$ by an edge slide $S$, one has an isomorphism of groupoids  $S: \mathcal G_\Gamma\to \mathcal G_{\Gamma'}$. On the vertices, it is the canonical identification of the vertices of $\Gamma$ and $\Gamma'$.  On the morphisms, it is given by the images of the edges of $\Gamma$.
 If $S=S_{\alpha^{\pm L}}$ or $S=S_{\alpha^{\pm R}}$  slides the {\em target end} of an edge $\beta$ along $\alpha$, one has $S(\gamma)=\gamma$ for 
 edges
$\gamma\neq \beta$ and $S(\beta)=\alpha^{\mp 1} \circ\beta$, as shown in Figure \ref{figure:SlideGroupoidIsomorphism}.  If $S=S_{\alpha^{\pm L}}$ or $S=S_{\alpha^{\pm R}}$  slides the {\em starting end} of  an edge $\beta$ along $\alpha$, it is given by $S(\gamma)=\gamma$ for $\gamma\neq \beta$ and $S(\beta)=\beta\circ\alpha^{\pm 1}$.

\begin{figure} 
	\centering
	\begin{tikzpicture}[scale=.6]
\begin{scope}[shift={(-5,0)}]
\draw [color=orange, fill=orange] (-2,0) circle (.2); 
\draw [color=violet, fill=violet] (2,0) circle (.2); 
\draw[color=blue, line width=1.5pt, ->,>=stealth] (-.2,1.8)--(-1.8,.2);
\draw[color=red, line width=1.5pt, ->,>=stealth] (-1.8,0)--(1.8,0);
\draw[color=black, line width=1pt] (-2.2,0)--(-2.8,0);
\draw[color=black, line width=1pt] (-2.2,.2)--(-2.6,.6);
\draw[color=black, line width=1pt] (-2.2,-.2)--(-2.6,-.6);
\draw[color=black, line width=1pt] (2.2,0)--(2.8,0);
\draw[color=black, line width=1pt] (2.2,.2)--(2.6,.6);
\draw[color=black, line width=1pt] (2.2,-.2)--(2.6,-.6);
\node at (0,0)[anchor=north, color=red]{$\alpha$};
\node at (-1,1)[anchor=south east, color=blue]{$\beta$};
\draw[line width=.5pt, color=black,->,>=stealth] plot [smooth , tension=0.6] coordinates {(-.2, 1.5) (-1.4,.3)};
\end{scope}
\draw[color=black,->,>=stealth] (-1,.5)--(1,.5);
\node at (0,0.5) [anchor=south, color=black] {$S_{\alpha^L}$};
\draw[color=black,<-,>=stealth] (-1,-.5)--(1,-.5);
\node at (0,-0.5) [anchor=north, color=black] {$S_{\alpha^{-L}}$};
\begin{scope}[shift={(5,0)}]
\draw [color=orange, fill=orange] (-2,0) circle (.2); 
\draw [color=violet, fill=violet] (2,0) circle (.2); 
\draw[color=blue, line width=1.5pt, ->,>=stealth] (.2,1.8)--(1.8,.2);
\draw[color=red, line width=1.5pt, ->,>=stealth] (-1.8,0)--(1.8,0);
\draw[color=black, line width=1pt] (-2.2,0)--(-2.8,0);
\draw[color=black, line width=1pt] (-2.2,.2)--(-2.6,.6);
\draw[color=black, line width=1pt] (-2.2,-.2)--(-2.6,-.6);
\draw[color=black, line width=1pt] (2.2,0)--(2.8,0);
\draw[color=black, line width=1pt] (2.2,.2)--(2.6,.6);
\draw[color=black, line width=1pt] (2.2,-.2)--(2.6,-.6);
\draw[line width=.5pt, color=black,->,>=stealth] plot [smooth , tension=0.4] coordinates {(.2, 1.5) (1.3,.2)(-1.5,.2)};
\node at (0,0)[anchor=north, color=red]{$\alpha$};
\node at (1,1)[anchor=south west, color=blue]{$ \beta$};
\node at (-.5,.2)[anchor=south, color=black]{$\alpha^\inv\beta$};
\end{scope}
\end{tikzpicture}	\caption{Isomorphism of path groupoids induced by sliding the target  end of $\beta$ along the left of  $\alpha$.}
	\label{figure:SlideGroupoidIsomorphism}
\end{figure}
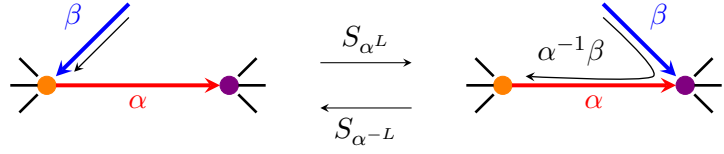

Given a Hopf monoid $H$  in a symmetric monoidal category $\mac$ and  a directed ribbon graph $\Gamma$, we associate to 
 each edge slide of $\Gamma$  an automorphism of the object $H^{\oo E}$ in $\mac$. It  is obtained by combining one of the $H$-left comodule structures  for  $\alpha$ and one of the $H$-left module structures for  $\beta$ from Definition \ref{def:triangmodcom}. 
The choice of the module and comodule structures depends  on the orientation of the edges and the direction of the slide.

\begin{definition}\label{def:edge slide} Let $\Gamma$ be a directed ribbon graph and $H$ a pivotal Hopf monoid in $\mac$.
\begin{compactenum}
\item If  $\alpha,\beta$ are oriented edges of $\Gamma$ as in \eqref{eq:slideedge}  the {\bf left edge slide}   of the target  end of  $\beta$  {\bf to the target end} of $\alpha$ is the isomorphism $S_{\alpha^L}=\rhd_{\beta+}\circ\delta_{\alpha+}: H^{\oo E}\to H^{\oo E}$
\begin{align}\label{eq:slidedef2a}
&\begin{tikzpicture}[scale=.6]
\begin{scope}[shift={(-5,0)}]
\draw [color=orange, fill=orange] (-2,0) circle (.2); 
\draw [color=violet, fill=violet] (2,0) circle (.2); 
\draw[color=blue, line width=1.5pt, ->,>=stealth] (-.2,1.8)--(-1.8,.2);
\draw[color=red, line width=1.5pt, ->,>=stealth] (-1.8,0)--(1.8,0);
\draw[color=black, line width=.5pt, ->,>=stealth] (-1.5,.5)--(-.5,.5);
\draw[color=black, line width=1pt] (-2.2,0)--(-2.8,0);
\draw[color=black, line width=1pt] (-2.2,.2)--(-2.6,.6);
\draw[color=black, line width=1pt] (-2.2,-.2)--(-2.6,-.6);
\draw[color=black, line width=1pt] (2.2,0)--(2.8,0);
\draw[color=black, line width=1pt] (2.2,.2)--(2.6,.6);
\draw[color=black, line width=1pt] (2.2,-.2)--(2.6,-.6);
\node at (0,0)[anchor=north, color=red]{$a$};
\node at (-1,1)[anchor=south east, color=blue]{$b$};
\end{scope}
\draw[color=black,->,>=stealth] (-1,0)--(1,0);
\node at (0,0.2) [anchor=south, color=black] {$S_{\alpha^{L}}$};
\begin{scope}[shift={(5,0)}]
\draw [color=orange, fill=orange] (-2,0) circle (.2); 
\draw [color=violet, fill=violet] (2,0) circle (.2); 
\draw[color=blue, line width=1.5pt, ->,>=stealth] (.2,1.8)--(1.8,.2);
\draw[color=red, line width=1.5pt, ->,>=stealth] (-1.8,0)--(1.8,0);
\draw[color=black, line width=1pt] (-2.2,0)--(-2.8,0);
\draw[color=black, line width=1pt] (-2.2,.2)--(-2.6,.6);
\draw[color=black, line width=1pt] (-2.2,-.2)--(-2.6,-.6);
\draw[color=black, line width=1pt] (2.2,0)--(2.8,0);
\draw[color=black, line width=1pt] (2.2,.2)--(2.6,.6);
\draw[color=black, line width=1pt] (2.2,-.2)--(2.6,-.6);
\node at (0,0)[anchor=north, color=red]{$\low a 2$};
\node at (1,1)[anchor=south west, color=blue]{$ \low a {1} b$};
\end{scope}
\end{tikzpicture}\qquad
&\begin{tikzpicture}[scale=0.3]
\draw[line width=1.5pt, color=blue](-2,3)--(-2,-3);
\draw[line width=1.5pt, color=red](2,3)--(2,-3);
\draw[line width=1pt, color=black] plot [smooth, tension=0.6] coordinates 
      {(2,2)(1,1.5)(-3,-1.5)(-2,-2)};
      \node at (-2,3)[anchor=south, color=blue]{$\beta$};
      \node at (2,3) [anchor=south, color=red]{$\alpha$};
\end{tikzpicture}
\end{align}
with inverse $S_{\alpha^{-L}}=\rhd_{\beta+}\circ (S^\inv\oo 1_{H^{\oo E}})\circ\delta_{\alpha+}: H^{\oo E}\to H^{\oo E}$
\begin{align}
\label{eq:slidedef2b}
&\begin{tikzpicture}[scale=.6]
\begin{scope}[shift={(-5,0)}]
\draw [color=orange, fill=orange] (-2,0) circle (.2); 
\draw [color=violet, fill=violet] (2,0) circle (.2); 
\draw[color=blue, line width=1.5pt, ->,>=stealth] (.2,1.8)--(1.8,.2);
\draw[color=red, line width=1.5pt, ->,>=stealth] (-1.8,0)--(1.8,0);
\draw[color=black, line width=.5pt, ->,>=stealth] (1.5,.5)--(.5,.5);
\draw[color=black, line width=1pt] (-2.2,0)--(-2.8,0);
\draw[color=black, line width=1pt] (-2.2,.2)--(-2.6,.6);
\draw[color=black, line width=1pt] (-2.2,-.2)--(-2.6,-.6);
\draw[color=black, line width=1pt] (2.2,0)--(2.8,0);
\draw[color=black, line width=1pt] (2.2,.2)--(2.6,.6);
\draw[color=black, line width=1pt] (2.2,-.2)--(2.6,-.6);
\node at (0,0)[anchor=north, color=red]{$a$};
\node at (1,1)[anchor=south west, color=blue]{$b$};
\end{scope}
\draw[color=black,->,>=stealth] (-1,0)--(1,0);
\node at (0,0.2) [anchor=south, color=black] {$S_{\alpha^{-L}}$};
\begin{scope}[shift={(5,0)}]
\draw [color=orange, fill=orange] (-2,0) circle (.2); 
\draw [color=violet, fill=violet] (2,0) circle (.2); 
\draw[color=blue, line width=1.5pt, ->,>=stealth] (-.2,1.8)--(-1.8,.2);
\draw[color=red, line width=1.5pt, ->,>=stealth] (-1.8,0)--(1.8,0);
\draw[color=black, line width=1pt] (-2.2,0)--(-2.8,0);
\draw[color=black, line width=1pt] (-2.2,.2)--(-2.6,.6);
\draw[color=black, line width=1pt] (-2.2,-.2)--(-2.6,-.6);
\draw[color=black, line width=1pt] (2.2,0)--(2.8,0);
\draw[color=black, line width=1pt] (2.2,.2)--(2.6,.6);
\draw[color=black, line width=1pt] (2.2,-.2)--(2.6,-.6);
\node at (0,0)[anchor=north, color=red]{$\low a 2$};
\node at (-.5,1.2)[anchor=south east, color=blue]{$S^\inv(\low a 1) b$};
\end{scope}
\end{tikzpicture}\qquad
&\begin{tikzpicture}[scale=0.3]
\draw[line width=1.5pt, color=blue](-2,3)--(-2,-3);
\draw[line width=1.5pt, color=red](2,3)--(2,-3);
\draw[line width=1pt, color=black] plot [smooth, tension=0.6] coordinates 
      {(2,2)(1,1.5)(-3,-1.5)(-2,-2)};
      \draw[line width=1pt, color=black, fill=gray] (-1,0) circle (.3);
      \node at (-2,3)[anchor=south, color=blue]{$\beta$};
      \node at (2,3) [anchor=south, color=red]{$\alpha$};
\end{tikzpicture}
\end{align}

\item  The edge slides for other edge orientations and their inverses are defined by 1.~by reversing edge orientation with  the involution $T: H\to H$ from \eqref{eq:tplusdef}.
\end{compactenum}
\end{definition}

That the morphism $S_{\alpha^{-L}}$ in  \eqref{eq:slidedef2b} is indeed the inverse of $S_{\alpha^L}$ in  \eqref{eq:slidedef2a} follows by a direct computation in Sweedler notation or, equivalently,   from the following diagrammatic computation
\begin{align*}
\begin{tikzpicture}[scale=.33]
\begin{scope}[shift={(-7,0)}]
\draw[line width=1.5pt, color=blue](-2,2.5)--(-2,-3.5);
\draw[line width=1.5pt, color=red](2,2.5)--(2,-3.5);
\draw[line width=1pt, color=black] plot [smooth, tension=0.6] coordinates 
      {(2,1.5)(1,1)(-3,-.5)(-2,-1.5)};
\draw[line width=1pt, color=black] plot [smooth, tension=0.6] coordinates 
      {(2,.5)(1,0)(-3,-1.5)(-2,-2.5)};
            \draw[line width=1pt, color=black, fill=gray] (-1,-.7) circle (.3);
\end{scope}
\node at (-3.5,-.5){$=$};
\begin{scope}[shift={(0,0)}]
\draw[line width=1.5pt, color=blue](-2,2.5)--(-2,-3.5);
\draw[line width=1.5pt, color=red](2,2.5)--(2,-3.5);
\draw[line width=1pt, color=black] plot [smooth, tension=0.6] coordinates 
{(2,2)(.5,1.5)(0,1)};
\draw[line width=1pt, color=black] plot [smooth cycle, tension=0.6] coordinates 
      { (0,1)(1,0)(-1,-.5)(0,-1.5) (1,-.5)(-1,0)};
\draw[line width=1pt, color=black] plot [smooth, tension=0.6] coordinates 
      {(0,-1.5)(0,-2)(-3,-2.5)(-2,-3)};
            \draw[line width=1pt, color=black, fill=gray] (-1,-.7) circle (.3);
\end{scope}
\node at (3.5,-.5){$=$};
\begin{scope}[shift={(7,0)}]
\draw[line width=1.5pt, color=blue](-2,2.5)--(-2,-3.5);
\draw[line width=1.5pt, color=red](2,2.5)--(2,-3.5);
\draw[line width=1pt, color=black] plot [smooth, tension=0.6] coordinates 
{(2,2)(.5,1.5)(0,1)};
\draw[line width=1pt, color=black, fill=white] (0,1) circle (.15);
\draw[line width=1pt, color=black, fill=white] (0,-1.5) circle (.15);
\draw[line width=1pt, color=black] plot [smooth, tension=0.6] coordinates 
      {(0,-1.5)(0,-2)(-3,-2.5)(-2,-3)};
\end{scope}
\node at (10.5,-.5){$=$};
\begin{scope}[shift={(14,0)}]
\draw[line width=1.5pt, color=blue](-2,2.5)--(-2,-3.5);
\draw[line width=1.5pt, color=red](2,2.5)--(2,-3.5);
\end{scope}
\node at (17.5,-.5){$=$};
\begin{scope}[shift={(21,0)}]
\draw[line width=1.5pt, color=blue](-2,2.5)--(-2,-3.5);
\draw[line width=1.5pt, color=red](2,2.5)--(2,-3.5);
\draw[line width=1pt, color=black] plot [smooth, tension=0.6] coordinates 
{(2,2)(.5,1.5)(0,1)};
\draw[line width=1pt, color=black, fill=white] (0,1) circle (.15);
\draw[line width=1pt, color=black, fill=white] (0,-1.5) circle (.15);
\draw[line width=1pt, color=black] plot [smooth, tension=0.6] coordinates 
      {(0,-1.5)(0,-2)(-3,-2.5)(-2,-3)};
\end{scope}
\node at (24.5,-.5){$=$};
\begin{scope}[shift={(28,0)}]
\draw[line width=1.5pt, color=blue](-2,2.5)--(-2,-3.5);
\draw[line width=1.5pt, color=red](2,2.5)--(2,-3.5);
\draw[line width=1pt, color=black] plot [smooth, tension=0.6] coordinates 
{(2,2)(.5,1.5)(0,1)};
\draw[line width=1pt, color=black] plot [smooth cycle, tension=0.6] coordinates 
      { (0,1)(1,0)(-1,-.5)(0,-1.5) (1,-.5)(-1,0)};
\draw[line width=1pt, color=black] plot [smooth, tension=0.6] coordinates 
      {(0,-1.5)(0,-2)(-3,-2.5)(-2,-3)};
            \draw[line width=1pt, color=black, fill=gray] (1,-.7) circle (.3);
\end{scope}
\node at (31.5,-.5){$=$};
\begin{scope}[shift={(36,0)}]
\draw[line width=1.5pt, color=blue](-2,2.5)--(-2,-3.5);
\draw[line width=1.5pt, color=red](2,2.5)--(2,-3.5);
\draw[line width=1pt, color=black] plot [smooth, tension=0.6] coordinates 
      {(2,1.5)(1,1)(-3,-.5)(-2,-1.5)};
\draw[line width=1pt, color=black] plot [smooth, tension=0.6] coordinates 
      {(2,.5)(1,0)(-3,-1.5)(-2,-2.5)};
            \draw[line width=1pt, color=black, fill=gray] (-1,.3) circle (.3);
\end{scope}
\end{tikzpicture}
\tDot
\end{align*}

Explicit expressions for the edge slides for other  edge orientations and their inverses are obtained from Definition \ref{def:triangmodcom} and  \eqref{eq:tplusdef} and given as follows.

\begin{align}
\label{eq:slidedef1b}
&\begin{tikzpicture}[scale=.6]
\begin{scope}[shift={(-5,0)}]
\draw [color=orange, fill=orange] (-2,0) circle (.2); 
\draw [color=violet, fill=violet] (2,0) circle (.2); 
\draw[color=blue, line width=1.5pt, ->,>=stealth] (-1.8,.2)--(-.2,1.8);
\draw[color=red, line width=1.5pt, ->,>=stealth] (-1.8,0)--(1.8,0);
\draw[color=black, line width=.5pt, ->,>=stealth] (-1.5,.5)--(-.5,.5);
\draw[color=black, line width=1pt] (-2.2,0)--(-2.8,0);
\draw[color=black, line width=1pt] (-2.2,.2)--(-2.6,.6);
\draw[color=black, line width=1pt] (-2.2,-.2)--(-2.6,-.6);
\draw[color=black, line width=1pt] (2.2,0)--(2.8,0);
\draw[color=black, line width=1pt] (2.2,.2)--(2.6,.6);
\draw[color=black, line width=1pt] (2.2,-.2)--(2.6,-.6);
\node at (0,0)[anchor=north, color=red]{$a$};
\node at (-1,1)[anchor=south east, color=blue]{$b$};
\end{scope}
\draw[color=black,->,>=stealth] (-1,0)--(1,0);
\node at (0,0.2) [anchor=south, color=black] {$S_{\alpha^{L}}$};
\begin{scope}[shift={(5,0)}]
\draw [color=orange, fill=orange] (-2,0) circle (.2); 
\draw [color=violet, fill=violet] (2,0) circle (.2); 
\draw[color=blue, line width=1.5pt, ->,>=stealth] (1.8,.2)--(.2,1.8);
\draw[color=red, line width=1.5pt, ->,>=stealth] (-1.8,0)--(1.8,0);
\draw[color=black, line width=1pt] (-2.2,0)--(-2.8,0);
\draw[color=black, line width=1pt] (-2.2,.2)--(-2.6,.6);
\draw[color=black, line width=1pt] (-2.2,-.2)--(-2.6,-.6);
\draw[color=black, line width=1pt] (2.2,0)--(2.8,0);
\draw[color=black, line width=1pt] (2.2,.2)--(2.6,.6);
\draw[color=black, line width=1pt] (2.2,-.2)--(2.6,-.6);
\node at (0,0)[anchor=north, color=red]{$\low a 2$};
\node at (1,1)[anchor=south west, color=blue]{$bS(\low a 1)$};
\end{scope}
\end{tikzpicture} 
&\begin{tikzpicture}[scale=0.35]
\draw[line width=1.5pt, color=blue](-2,3)--(-2,-3);
\draw[line width=1.5pt, color=red](2,3)--(2,-3);
\node at (-2,3)[anchor=south, color=blue]{$\beta$};
\node at (2,3)[anchor=south, color=red]{$\alpha$};
\draw[line width=1pt, color=black] plot [smooth, tension=0.6] coordinates 
      {(2,2)(1,1.5)(-1,-1.5)(-2,-2)};
      \draw[line width=1pt, color=black, fill=white] (0,0) circle (.3);
\end{tikzpicture}
\intertext{}
\label{eq:slidedef1a}
&\begin{tikzpicture}[scale=.6]
\begin{scope}[shift={(-5,0)}]
\draw [color=orange, fill=orange] (-2,0) circle (.2); 
\draw [color=violet, fill=violet] (2,0) circle (.2); 
\draw[color=blue, line width=1.5pt, ->,>=stealth] (1.8,.2)--(.2,1.8);
\draw[color=red, line width=1.5pt, ->,>=stealth] (-1.8,0)--(1.8,0);
\draw[color=black, line width=.5pt, ->,>=stealth] (1.5,.5)--(.5,.5);
\draw[color=black, line width=1pt] (-2.2,0)--(-2.8,0);
\draw[color=black, line width=1pt] (-2.2,.2)--(-2.6,.6);
\draw[color=black, line width=1pt] (-2.2,-.2)--(-2.6,-.6);
\draw[color=black, line width=1pt] (2.2,0)--(2.8,0);
\draw[color=black, line width=1pt] (2.2,.2)--(2.6,.6);
\draw[color=black, line width=1pt] (2.2,-.2)--(2.6,-.6);
\node at (0,0)[anchor=north, color=red]{$a$};
\node at (1,1)[anchor=south west, color=blue]{$b$};
\end{scope}
\draw[color=black,->,>=stealth] (-1,0)--(1,0);
\node at (0,0.2) [anchor=south, color=black] {$S_{\alpha^{-L}}$};
\begin{scope}[shift={(5,0)}]
\draw [color=orange, fill=orange] (-2,0) circle (.2); 
\draw [color=violet, fill=violet] (2,0) circle (.2); 
\draw[color=blue, line width=1.5pt, ->,>=stealth] (-1.8,.2)--(-.2,1.8);
\draw[color=red, line width=1.5pt, ->,>=stealth] (-1.8,0)--(1.8,0);
\draw[color=black, line width=1pt] (-2.2,0)--(-2.8,0);
\draw[color=black, line width=1pt] (-2.2,.2)--(-2.6,.6);
\draw[color=black, line width=1pt] (-2.2,-.2)--(-2.6,-.6);
\draw[color=black, line width=1pt] (2.2,0)--(2.8,0);
\draw[color=black, line width=1pt] (2.2,.2)--(2.6,.6);
\draw[color=black, line width=1pt] (2.2,-.2)--(2.6,-.6);
\node at (0,0)[anchor=north, color=red]{$\low a 2$};
\node at (-1,1)[anchor=south east, color=blue]{$b \low a {1}$};
\end{scope}
\end{tikzpicture}\qquad
&\begin{tikzpicture}[scale=0.35]
\draw[line width=1.5pt, color=blue](-2,3)--(-2,-3);
\draw[line width=1.5pt, color=red](2,3)--(2,-3);
\node at (-2,3)[anchor=south, color=blue]{$\beta$};
\node at (2,3)[anchor=south, color=red]{$\alpha$};
\draw[line width=1pt, color=black] plot [smooth, tension=0.6] coordinates 
      {(2,2)(1,1.5)(-1,-1.5)(-2,-2)};   
\end{tikzpicture}
\intertext{}
\label{eq:slidedef3b}
&\begin{tikzpicture}[scale=.6]
\begin{scope}[shift={(-5,0)}]
\draw [color=orange, fill=orange] (-2,0) circle (.2); 
\draw [color=violet, fill=violet] (2,0) circle (.2); 
\draw[color=blue, line width=1.5pt, <-,>=stealth] (-1.8,-.2)--(-.2,-1.8);
\draw[color=red, line width=1.5pt, ->,>=stealth] (-1.8,0)--(1.8,0);
\draw[color=black, line width=.5pt, ->,>=stealth] (-1.5,-.5)--(-.5,-.5);
\draw[color=black, line width=1pt] (-2.2,0)--(-2.8,0);
\draw[color=black, line width=1pt] (-2.2,.2)--(-2.6,.6);
\draw[color=black, line width=1pt] (-2.2,-.2)--(-2.6,-.6);
\draw[color=black, line width=1pt] (2.2,0)--(2.8,0);
\draw[color=black, line width=1pt] (2.2,.2)--(2.6,.6);
\draw[color=black, line width=1pt] (2.2,-.2)--(2.6,-.6);
\node at (0,0)[anchor=south, color=red]{$a$};
\node at (-1,-1)[anchor=north east, color=blue]{$b$};
\end{scope}
\draw[color=black,->,>=stealth] (-1,0)--(1,0);
\node at (0,0.2) [anchor=south, color=black] {$S_{\alpha^{R}}$};
\begin{scope}[shift={(5,0)}]
\draw [color=orange, fill=orange] (-2,0) circle (.2); 
\draw [color=violet, fill=violet] (2,0) circle (.2); 
\draw[color=blue, line width=1.5pt, <-,>=stealth] (1.8,-.2)--(.2,-1.8);
\draw[color=red, line width=1.5pt, ->,>=stealth] (-1.8,0)--(1.8,0);
\draw[color=black, line width=1pt] (-2.2,0)--(-2.8,0);
\draw[color=black, line width=1pt] (-2.2,.2)--(-2.6,.6);
\draw[color=black, line width=1pt] (-2.2,-.2)--(-2.6,-.6);
\draw[color=black, line width=1pt] (2.2,0)--(2.8,0);
\draw[color=black, line width=1pt] (2.2,.2)--(2.6,.6);
\draw[color=black, line width=1pt] (2.2,-.2)--(2.6,-.6);
\node at (0,0)[anchor=south, color=red]{$\low a 1$};
\node at (1,-1)[anchor=north west, color=blue]{$ S^\inv T(\low a {2})   b$};
\end{scope}
\end{tikzpicture}\qquad
&\begin{tikzpicture}[scale=0.35]
\draw[line width=1.5pt, color=blue](-2,3)--(-2,-3);
\draw[line width=1.5pt, color=red](2,3)--(2,-3);
\node at (-2,3)[anchor=south, color=blue]{$\beta$};
\node at (2,3)[anchor=south, color=red]{$\alpha$};
\draw[line width=1pt, color=black] plot [smooth, tension=0.6] coordinates 
      {(2,2)(2.7,1.5)(2.7,.5)(-2.7,-.5)(-2.7,-1.5)(-2,-2)};
      \draw[color=black, line width=1pt, fill=white] (2.9,1) circle (.3);
            \draw[color=black, line width=1pt, fill=white] (2.9,1) circle (.15);
      \draw[color=black, line width=1pt, fill=gray] (0,0) circle (.3);            
\end{tikzpicture}
\intertext{}
\label{eq:slidedef3a}
&\begin{tikzpicture}[scale=.6]
\begin{scope}[shift={(-5,0)}]
\draw [color=orange, fill=orange] (-2,0) circle (.2); 
\draw [color=violet, fill=violet] (2,0) circle (.2); 
\draw[color=blue, line width=1.5pt, <-,>=stealth] (1.8,-.2)--(.2,-1.8);
\draw[color=red, line width=1.5pt, ->,>=stealth] (-1.8,0)--(1.8,0);
\draw[color=black, line width=.5pt, ->,>=stealth] (1.5,-.5)--(.5,-.5);
\draw[color=black, line width=1pt] (-2.2,0)--(-2.8,0);
\draw[color=black, line width=1pt] (-2.2,.2)--(-2.6,.6);
\draw[color=black, line width=1pt] (-2.2,-.2)--(-2.6,-.6);
\draw[color=black, line width=1pt] (2.2,0)--(2.8,0);
\draw[color=black, line width=1pt] (2.2,.2)--(2.6,.6);
\draw[color=black, line width=1pt] (2.2,-.2)--(2.6,-.6);
\node at (0,0)[anchor=south, color=red]{$a$};
\node at (1,-1)[anchor=north west, color=blue]{$b$};
\end{scope}
\draw[color=black,->,>=stealth] (-1,0)--(1,0);
\node at (0,0.2) [anchor=south, color=black] {$S_{\alpha^{-R}}$};
\begin{scope}[shift={(5,0)}]
\draw [color=orange, fill=orange] (-2,0) circle (.2); 
\draw [color=violet, fill=violet] (2,0) circle (.2); 
\draw[color=blue, line width=1.5pt, <-,>=stealth] (-1.8,-.2)--(-.2,-1.8);
\draw[color=red, line width=1.5pt, ->,>=stealth] (-1.8,0)--(1.8,0);
\draw[color=black, line width=1pt] (-2.2,0)--(-2.8,0);
\draw[color=black, line width=1pt] (-2.2,.2)--(-2.6,.6);
\draw[color=black, line width=1pt] (-2.2,-.2)--(-2.6,-.6);
\draw[color=black, line width=1pt] (2.2,0)--(2.8,0);
\draw[color=black, line width=1pt] (2.2,.2)--(2.6,.6);
\draw[color=black, line width=1pt] (2.2,-.2)--(2.6,-.6);
\node at (0,0)[anchor=south, color=red]{$\low a 1$};
\node at (-1,-1)[anchor=north east, color=blue]{$T( \low a {2}) b$};
\end{scope}
\end{tikzpicture}\qquad
&\qquad\begin{tikzpicture}[scale=0.35]
\draw[line width=1.5pt, color=blue](-2,3)--(-2,-3);
\draw[line width=1.5pt, color=red](2,3)--(2,-3);
\node at (-2,3)[anchor=south, color=blue]{$\beta$};
\node at (2,3)[anchor=south, color=red]{$\alpha$};
\draw[line width=1pt, color=black] plot [smooth, tension=0.6] coordinates 
      {(2,2)(2.7,1.5)(2.7,.5)(-2.7,-.5)(-2.7,-1.5)(-2,-2)};
      \draw[color=black, line width=1pt, fill=white] (2.9,1) circle (.3);
            \draw[color=black, line width=1pt, fill=white] (2.9,1) circle (.15);
\end{tikzpicture}
\intertext{}
\label{eq:slidedef4a}
&\begin{tikzpicture}[scale=.6]
\begin{scope}[shift={(-5,0)}]
\draw [color=orange, fill=orange] (-2,0) circle (.2); 
\draw [color=violet, fill=violet] (2,0) circle (.2); 
\draw[color=blue, line width=1.5pt, ->,>=stealth] (-1.8,-.2)--(-.2,-1.8);
\draw[color=red, line width=1.5pt, ->,>=stealth] (-1.8,0)--(1.8,0);
\draw[color=black, line width=.5pt, ->,>=stealth] (-1.5,-.5)--(-.5,-.5);
\draw[color=black, line width=1pt] (-2.2,0)--(-2.8,0);
\draw[color=black, line width=1pt] (-2.2,.2)--(-2.6,.6);
\draw[color=black, line width=1pt] (-2.2,-.2)--(-2.6,-.6);
\draw[color=black, line width=1pt] (2.2,0)--(2.8,0);
\draw[color=black, line width=1pt] (2.2,.2)--(2.6,.6);
\draw[color=black, line width=1pt] (2.2,-.2)--(2.6,-.6);
\node at (0,0)[anchor=south, color=red]{$a$};
\node at (-1,-1)[anchor=north east, color=blue]{$b$};
\end{scope}
\draw[color=black,->,>=stealth] (-1,0)--(1,0);
\node at (0,0.2) [anchor=south, color=black] {$S_{\alpha^{R}}$};
\begin{scope}[shift={(5,0)}]
\draw [color=orange, fill=orange] (-2,0) circle (.2); 
\draw [color=violet, fill=violet] (2,0) circle (.2); 
\draw[color=blue, line width=1.5pt, ->,>=stealth] (1.8,-.2)--(.2,-1.8);
\draw[color=red, line width=1.5pt, ->,>=stealth] (-1.8,0)--(1.8,0);
\draw[color=black, line width=1pt] (-2.2,0)--(-2.8,0);
\draw[color=black, line width=1pt] (-2.2,.2)--(-2.6,.6);
\draw[color=black, line width=1pt] (-2.2,-.2)--(-2.6,-.6);
\draw[color=black, line width=1pt] (2.2,0)--(2.8,0);
\draw[color=black, line width=1pt] (2.2,.2)--(2.6,.6);
\draw[color=black, line width=1pt] (2.2,-.2)--(2.6,-.6);
\node at (0,0)[anchor=south, color=red]{$\low a 1$};
\node at (1,-1)[anchor=north west, color=blue]{$b  T( \low a {2})$};
\end{scope}
\end{tikzpicture}\qquad
&\begin{tikzpicture}[scale=0.35]
\draw[line width=1.5pt, color=blue](-2,3)--(-2,-3);
\draw[line width=1.5pt, color=red](2,3)--(2,-3);
\node at (-2,3)[anchor=south, color=blue]{$\beta$};
\node at (2,3)[anchor=south, color=red]{$\alpha$};
\draw[line width=1pt, color=black] plot [smooth, tension=0.6] coordinates 
      {(2,2)(2.7,1.5)(2.7,.5)(-1,-1.5)(-2,-2)};
      \draw[color=black, line width=1pt, fill=white] (2.9,1) circle (.3);
            \draw[color=black, line width=1pt, fill=white] (2.9,1) circle (.15);
\end{tikzpicture}
\intertext{}
\label{eq:slidedef4b}
&\begin{tikzpicture}[scale=.6]
\begin{scope}[shift={(-5,0)}]
\draw [color=orange, fill=orange] (-2,0) circle (.2); 
\draw [color=violet, fill=violet] (2,0) circle (.2); 
\draw[color=blue, line width=1.5pt, ->,>=stealth] (1.8,-.2)--(.2,-1.8);
\draw[color=red, line width=1.5pt, ->,>=stealth] (-1.8,0)--(1.8,0);
\draw[color=black, line width=.5pt, ->,>=stealth] (1.5,-.5)--(.5,-.5);
\draw[color=black, line width=1pt] (-2.2,0)--(-2.8,0);
\draw[color=black, line width=1pt] (-2.2,.2)--(-2.6,.6);
\draw[color=black, line width=1pt] (-2.2,-.2)--(-2.6,-.6);
\draw[color=black, line width=1pt] (2.2,0)--(2.8,0);
\draw[color=black, line width=1pt] (2.2,.2)--(2.6,.6);
\draw[color=black, line width=1pt] (2.2,-.2)--(2.6,-.6);
\node at (0,0)[anchor=south, color=red]{$a$};
\node at (1,-1)[anchor=north west, color=blue]{$b$};
\end{scope}
\draw[color=black,->,>=stealth] (-1,0)--(1,0);
\node at (0,0.2) [anchor=south, color=black] {$S_{\alpha^{-R}}$};
\begin{scope}[shift={(5,0)}]
\draw [color=orange, fill=orange] (-2,0) circle (.2); 
\draw [color=violet, fill=violet] (2,0) circle (.2); 
\draw[color=blue, line width=1.5pt, ->,>=stealth] (-1.8,-.2)--(-.2,-1.8);
\draw[color=red, line width=1.5pt, ->,>=stealth] (-1.8,0)--(1.8,0);
\draw[color=black, line width=1pt] (-2.2,0)--(-2.8,0);
\draw[color=black, line width=1pt] (-2.2,.2)--(-2.6,.6);
\draw[color=black, line width=1pt] (-2.2,-.2)--(-2.6,-.6);
\draw[color=black, line width=1pt] (2.2,0)--(2.8,0);
\draw[color=black, line width=1pt] (2.2,.2)--(2.6,.6);
\draw[color=black, line width=1pt] (2.2,-.2)--(2.6,-.6);
\node at (0,0)[anchor=south, color=red]{$\low a 1$};
\node at (-1,-1)[anchor=north east, color=blue]{$b ST( \low a {2})$};
\end{scope}
\end{tikzpicture}\qquad
&\begin{tikzpicture}[scale=0.35]
\draw[line width=1.5pt, color=blue](-2,3)--(-2,-3);
\draw[line width=1.5pt, color=red](2,3)--(2,-3);
\node at (-2,3)[anchor=south, color=blue]{$\beta$};
\node at (2,3)[anchor=south, color=red]{$\alpha$};
\draw[line width=1pt, color=black] plot [smooth, tension=0.6] coordinates 
      {(2,2)(2.7,1.5)(2.7,.5)(-1,-1.5)(-2,-2)};
      \draw[color=black, line width=1pt, fill=white] (2.9,1) circle (.3);
            \draw[color=black, line width=1pt, fill=white] (2.9,1) circle (.15);
             \draw[color=black, line width=1pt, fill=white] (.8,-.6) circle (.3);
\end{tikzpicture}
\end{align}

In the remainder of this section, we investigate the  properties of these  edge slides. We start by considering their interaction with the $H$-left module and  comodule structures for ciliated vertices and faces of the ribbon graph $\Gamma$.

\begin{proposition} \label{prop:vertfacecomp} Let $v$ be a ciliated vertex,  $f$ a ciliated face of $\Gamma$ and $\rhd_v: H\oo H^{\oo E}\to H^{\oo E}$ and $\delta_f: H^{\oo E}\to H\oo H^{\oo E}$ the associated $H$-left module and  comodule structures from Definition \ref{def:vertexface}. 
Then any  edge slide that does not slide edge ends over their cilia is an isomorphism of $H$-left modules and $H$-left comodules.  
\end{proposition}

\begin{proof}  That the edge slide  from \eqref{eq:slidedef2a} is an automorphism of the object $H^{\oo E}$ was already shown  above. The proof for the case where both vertices coincide is analogous, and so are the proofs that  the  slides \eqref{eq:slidedef1b} and  \eqref{eq:slidedef1a},  \eqref{eq:slidedef3b} and \eqref{eq:slidedef3a} and     \eqref{eq:slidedef4a} and \eqref{eq:slidedef4b} are inverse to each other. 

It remains to check compatibility with the $H$-module structures for the  vertices and with the $H$-comodule structures for the  faces. 
If the two edges involved in the slide are not incident at $v$ or not traversed by the face $f$, this follows directly from the definition of the slide and of the $H$-module and comodule structure.  We prove the remaining cases for the slide \eqref{eq:slidedef1a} under the assumption that the starting and target vertex of $\alpha$  are different and do not coincide with the target vertex of $\beta$. The proofs for the cases, where some of these vertices coincide are analogous. The proofs for other edge orientations follow by considering the inverses and reversing edge orientations with the involution $T: H\to H$ from \eqref{eq:tplusdef}.

The statement that \eqref{eq:slidedef1a} is an isomorphism of $H$-left modules and $H$-left comodules then corresponds to the equality of the following diagrams that arise from the $H$-left module and -comodule structures associated with the three vertices and three faces adjacent to the edges. 
\begin{align}\label{eq:vert1}
&\begin{tikzpicture}[scale=0.3]
\begin{scope}[shift={(-6,0)}]
\draw[line width=1.5pt, color=blue](-2,3)--(-2,-3);
\draw[line width=1.5pt, color=red](2,3)--(2,-3);
\node at (-2,3)[color=blue, anchor=south]{$\beta$};
\node at (2,3)[color=red, anchor=south]{$\alpha$};
\draw[line width=1pt, color=black] plot [smooth, tension=0.6] coordinates 
      {(2,0)(1,-.3)(-1,-1.7)(-2,-2)};
      \draw[line width=1pt, color=black] plot [smooth, tension=0.6] coordinates 
      {(-2,-1)(-1,0)(-5,1)(-4,2)(-2,1.5)(2,1)};
        \draw [color=black, fill=white, line width=1pt] (-1,-.25) circle (.3); 
      \draw[line width=1pt, color=black](-4,2)--(-4,3); 
      \end{scope}
      \node at (-1.5,0){$=$};
      \begin{scope}[shift={(4,0)}]
\draw[line width=1.5pt, color=blue](-2,3)--(-2,-3);
\draw[line width=1.5pt, color=red](2,3)--(2,-3);
\node at (-2,3)[color=blue, anchor=south]{$\beta$};
\node at (2,3)[color=red, anchor=south]{$\alpha$};
\draw[line width=1pt, color=black] plot [smooth, tension=0.6] coordinates 
      {(2,2)(1,1.7)(-1,.3)(-2,0)};
      \draw[line width=1pt, color=black] plot [smooth, tension=0.6] coordinates 
      {(-4,3)(-3,-1)(2,-2)};
      \end{scope}
\end{tikzpicture}
&
&\begin{tikzpicture}[scale=.6]
\begin{scope}[shift={(-3,0)}]
\draw [color=orange, fill=orange] (-2,0) circle (.2); 
\draw [color=cyan, fill=cyan] (0,2) circle (.2); 
\draw [color=violet, fill=violet] (2,0) circle (.2); 
\draw [color=black] (2,0) circle (.4); 
\draw[color=blue, line width=1.5pt, ->,>=stealth] (1.8,.2)--(.2,1.8);
\draw[color=red, line width=1.5pt, ->,>=stealth] (-1.8,0)--(1.8,0);
\node at (0,0) [anchor=north, color=red] {$h_{(2)} a$};
\node at (0.8,.8) [anchor=east, color=blue] {$bS(h_{(1)})$};
\end{scope}
\draw[color=black,->,>=stealth] (-1,1)--(1,1);
\begin{scope}[shift={(3,0)}]
\draw [color=orange, fill=orange] (-2,0) circle (.2); 
\draw [color=cyan, fill=cyan] (0,2) circle (.2); 
\draw [color=violet, fill=violet] (2,0) circle (.2); 
\draw [color=black] (2,0) circle (.4); 
\draw[color=blue, line width=1.5pt, ->,>=stealth] (-1.8,.2)--(-.2,1.8);
\draw[color=red, line width=1.5pt, ->,>=stealth] (-1.8,0)--(1.8,0);
\node at (0,0) [anchor=north, color=red] {$h a_{(2)}$};
\node at (-.8,.8) [anchor=west, color=blue] {$b a_{(1)}$};
\end{scope}
\end{tikzpicture}\\
\intertext{}
\label{eq:vert2}
&\begin{tikzpicture}[scale=0.3]
\begin{scope}[shift={(-6,0)}]
\draw[line width=1.5pt, color=blue](-2,3)--(-2,-3);
\draw[line width=1.5pt, color=red](2,3)--(2,-3);
\node at (-2,3)[color=blue, anchor=south]{$\beta$};
\node at (2,3)[color=red, anchor=south]{$\alpha$};
\draw[line width=1pt, color=black] plot [smooth, tension=0.6] coordinates 
      {(2,-.5)(1,-.8)(-1,-2.2)(-2,-2.5)};
      \draw[line width=1pt, color=black] plot [smooth, tension=0.6] coordinates 
      {(-3,3)(-3,2)(3,1.5)(3,1)(2,.5)};
        \draw [color=black, fill=white, line width=1pt] (3.3,1.25) circle (.3); 
      \end{scope}
      \node at (-2,0){$=$};
      \begin{scope}[shift={(4,0)}]
\draw[line width=1.5pt, color=blue](-2,3)--(-2,-3);
\draw[line width=1.5pt, color=red](2,3)--(2,-3);
\node at (-2,3)[color=blue, anchor=south]{$\beta$};
\node at (2,3)[color=red, anchor=south]{$\alpha$};
\draw[line width=1pt, color=black] plot [smooth, tension=0.6] coordinates 
      {(2,2.5)(1,2.2)(-1,.8)(-2,.5)};
      \draw[line width=1pt, color=black] plot [smooth, tension=0.6] coordinates 
      {(-2,-2.5)(-1,-2)(-1,-1.5)(-2.5, -1)(-3,.5)(-4,.5)(-5,0)(-4,-.7)(3,-1.5)(3,-2)(2,-2.5)};
              \draw [color=black, fill=white, line width=1pt] (3.3,-1.75) circle (.3); 
                            \draw [color=black, fill=white, line width=1pt] (-1,-1.75) circle (.3); 
                            \draw[color=black, line width=1pt] (-3.2,3)--(-3.2,.6);
      \end{scope}
\end{tikzpicture}\qquad
& 
&\begin{tikzpicture}[scale=.6]
\begin{scope}[shift={(-3,0)}]
\draw [color=orange, fill=orange] (-2,0) circle (.2); 
\draw [color=cyan, fill=cyan] (0,2) circle (.2); 
\draw [color=violet, fill=violet] (2,0) circle (.2); 
\draw [color=black] (-2,0) circle (.4); 
\draw[color=blue, line width=1.5pt, ->,>=stealth] (1.8,.2)--(.2,1.8);
\draw[color=red, line width=1.5pt, ->,>=stealth] (-1.8,0)--(1.8,0);
\node at (0,0) [anchor=north, color=red] {$aS(h)$};
\node at (0.8,.8) [anchor=east, color=blue] {$b$};
\end{scope}
\draw[color=black,->,>=stealth] (-1,1)--(1,1);
\begin{scope}[shift={(3,0)}]
\draw [color=orange, fill=orange] (-2,0) circle (.2); 
\draw [color=cyan, fill=cyan] (0,2) circle (.2); 
\node at (0,0) [anchor=north, color=red] {$ a_{(2)} S(h_{(1)})$};
\draw [color=violet, fill=violet] (2,0) circle (.2); 
\draw [color=black] (-2,0) circle (.4); 
\draw[color=blue, line width=1.5pt, ->,>=stealth] (-1.8,.2)--(-.2,1.8);
\draw[color=red, line width=1.5pt, ->,>=stealth] (-1.8,0)--(1.8,0);
\node at (-.8,.8) [anchor=west, color=blue] {$b a_{(1)}S(h_{(2)})$};
\end{scope}
\end{tikzpicture}
\\
\intertext{}
\label{eq:vert3}
&\begin{tikzpicture}[scale=0.3]
\begin{scope}[shift={(-6,0)}]
\draw[line width=1.5pt, color=blue](-2,3)--(-2,-3);
\draw[line width=1.5pt, color=red](2,3)--(2,-3);
\node at (-2,3)[color=blue, anchor=south]{$\beta$};
\node at (2,3)[color=red, anchor=south]{$\alpha$};
\draw[line width=1pt, color=black] plot [smooth, tension=0.6] coordinates 
      {(2,-.5)(1,-.8)(-1,-2.2)(-2,-2.5)};
      \draw[line width=1pt, color=black] plot [smooth, tension=0.6] coordinates 
      {(-3,3)(-3,.5)(-2,-.5)};
      \end{scope}
      \node at (-2,0){$=$};
      \begin{scope}[shift={(3,0)}]
\draw[line width=1.5pt, color=blue](-2,3)--(-2,-3);
\draw[line width=1.5pt, color=red](2,3)--(2,-3);
\node at (-2,3)[color=blue, anchor=south]{$\beta$};
\node at (2,3)[color=red, anchor=south]{$\alpha$};
\draw[line width=1pt, color=black] plot [smooth, tension=0.6] coordinates 
      {(2,2.5)(1,2.2)(-1,.8)(-2,.5)};
      \draw[line width=1pt, color=black] plot [smooth, tension=0.6] coordinates 
      {(-3,3)(-3,-1.5)(-2,-2)};
      \end{scope}
\end{tikzpicture}\qquad
&
&\begin{tikzpicture}[scale=.6]
\begin{scope}[shift={(-4,0)}]
\draw [color=orange, fill=orange] (-2,0) circle (.2); 
\draw [color=cyan, fill=cyan] (0,2) circle (.2); 
\draw [color=black, line width=.5] (0,2) circle (.4); 
\draw [color=violet, fill=violet] (2,0) circle (.2); 
\draw[color=blue, line width=1.5pt, ->,>=stealth] (1.8,.2)--(.2,1.8);
\draw[color=red, line width=1.5pt, ->,>=stealth] (-1.8,0)--(1.8,0);
\node at (0,0) [anchor=north, color=red] {$a$};
\node at (1,.8) [anchor=east, color=blue] {$h b$};
\end{scope}
\draw[color=black,->,>=stealth] (-1.5,1)--(.5,1);
\begin{scope}[shift={(3,0)}]
\draw [color=orange, fill=orange] (-2,0) circle (.2); 
\draw [color=cyan, fill=cyan] (0,2) circle (.2); 
\draw [color=black, line width=.5] (0,2) circle (.4); 
\draw [color=violet, fill=violet] (2,0) circle (.2); 
\draw[color=blue, line width=1.5pt, ->,>=stealth] (-1.8,.2)--(-.2,1.8);
\draw[color=red, line width=1.5pt, ->,>=stealth] (-1.8,0)--(1.8,0);
\node at (0,0) [anchor=north, color=red] {$a_{(2)}$};
\node at (-.8,.8) [anchor=west, color=blue] {$h  b a_{(1)}$};
\end{scope}
\end{tikzpicture}
\\
\intertext{}
\label{eq:vert4}
&\begin{tikzpicture}[scale=0.3]
\begin{scope}[shift={(-6,0)}]
\draw[line width=1.5pt, color=blue](-2,3)--(-2,-3);
\draw[line width=1.5pt, color=red](2,3)--(2,-3);
\node at (-2,3)[color=blue, anchor=south]{$\beta$};
\node at (2,3)[color=red, anchor=south]{$\alpha$};
\draw[line width=1pt, color=black] plot [smooth, tension=0.6] coordinates 
      {(2,-.5)(1,-.8)(-1,-2.2)(-2,-2.5)};
      \draw[line width=1pt, color=black] plot [smooth, tension=0.6] coordinates 
      {(-2,2.5)(-4, 1)(-3, -1)(2,2.5)};
      \draw[color=black, line width=1pt] (-3,-1)--(-3,-3);
      \end{scope}
      \node at (-1,0){$=$};
      \begin{scope}[shift={(3,0)}]
\draw[line width=1.5pt, color=blue](-2,3)--(-2,-3);
\draw[line width=1.5pt, color=red](2,3)--(2,-3);
\node at (-2,3)[color=blue, anchor=south]{$\beta$};
\node at (2,3)[color=red, anchor=south]{$\alpha$};
\draw[line width=1pt, color=black] plot [smooth, tension=0.6] coordinates 
      {(2,2.5)(1,2.2)(-1,1.3)(-2,1)};
      \draw[line width=1pt, color=black] plot [smooth, tension=0.4] coordinates 
      {(-2, -.5)(-3,-1)(-3,-3)};
      \end{scope}
\end{tikzpicture}\qquad
&
&\begin{tikzpicture}[scale=.6]
\begin{scope}[shift={(-3,0)}]
\draw [color=orange, fill=orange] (-2,0) circle (.2); 
\draw [color=cyan, fill=cyan] (0,2) circle (.2); 
\draw [color=violet, fill=violet] (2,0) circle (.2); 
\draw[color=blue, line width=1.5pt, ->,>=stealth] (1.8,.2)--(.2,1.8);
\draw[color=red, line width=1.5pt, ->,>=stealth] (-1.8,0)--(1.8,0);
\draw[line width=.5pt, color=black,->,>=stealth] plot [smooth, tension=0.6] coordinates 
      {(-1.6,.5)(.8, .5)(-.4,1.6)};
      \node at (0,0) [anchor=north, color=red] {$a_{(2)}$};
\node at (0.7,1.5) [anchor=west, color=blue] {$b_{(2)}$};
      \node at (-1.8,-1.5) [anchor=west, color=black] {$\cdots b_{(1)}a_{(1)}\cdots$};
\end{scope}
\draw[color=black,->,>=stealth] (-1,1)--(1,1);
\begin{scope}[shift={(4,0)}]
\draw [color=orange, fill=orange] (-2,0) circle (.2); 
\draw [color=cyan, fill=cyan] (0,2) circle (.2); 
\draw [color=violet, fill=violet] (2,0) circle (.2); 
\draw[color=blue, line width=1.5pt, ->,>=stealth] (-1.8,.2)--(-.2,1.8);
\draw[color=red, line width=1.5pt, ->,>=stealth] (-1.8,0)--(1.8,0);
\draw[line width=.5pt, color=black,->,>=stealth] plot [smooth, tension=0.6] coordinates 
      {(-2.3,.5)(-.8,2)};
\node at (0,0) [anchor=north, color=red] {$a_{(3)}$};
\node at (-.7,.8) [anchor=west, color=blue] {$b_{(2)} a_{(2)}$};
      \node at (-1.8,-1.5) [anchor=west, color=black] {$\cdots b_{(1)}a_{(1)}\cdots$};
\end{scope}
\end{tikzpicture}
\\
\intertext{}
\label{eq:vert5}
&\begin{tikzpicture}[scale=0.3]
\begin{scope}[shift={(-6,0)}]
\draw[line width=1.5pt, color=blue](-2,3)--(-2,-3);
\draw[line width=1.5pt, color=red](2,3)--(2,-3);
\node at (-2,3)[color=blue, anchor=south]{$\beta$};
\node at (2,3)[color=red, anchor=south]{$\alpha$};
\draw[line width=1pt, color=black] plot [smooth, tension=0.6] coordinates 
      {(2,-.5)(1,-.8)(-1,-2.2)(-2,-2.5)};
      \draw[line width=1pt, color=black] plot [smooth, tension=0.6] coordinates 
      {(-2,2)(-1,1.5)(-1,.5)(-3,-1)(-3,-3)};
       \draw [color=black, fill=white, line width=1pt] (-1,1) circle (.3); 
         \draw [color=black, fill=white, line width=1pt] (-1,1) circle (.15); 
      \end{scope}
      \node at (-1,0){$=$};
      \begin{scope}[shift={(3,0)}]
\draw[line width=1.5pt, color=blue](-2,3)--(-2,-3);
\draw[line width=1.5pt, color=red](2,3)--(2,-3);
\node at (-2,3)[color=blue, anchor=south]{$\beta$};
\node at (2,3)[color=red, anchor=south]{$\alpha$};
\draw[line width=1pt, color=black] plot [smooth, tension=0.6] coordinates 
      {(2,2.5)(1,2.2)(-1,1.3)(-2,1)};
      \draw[line width=1pt, color=black] plot [smooth, tension=0.6] coordinates 
      {(-2,.5)(-1, 0)(-3,-2)(-4,-1.5)(2,0)};
      \draw[line width=1pt, color=black] (-3,-2)--(-3,-3);
                    \draw [color=black, fill=white, line width=1pt] (-1,.2) circle (.3); 
                    \draw [color=black, fill=white, line width=1pt] (-1,.2) circle (.15);
      \end{scope}
\end{tikzpicture}\qquad
&
&\begin{tikzpicture}[scale=.6]
\begin{scope}[shift={(-4,0)}]
\draw [color=orange, fill=orange] (-2,0) circle (.2); 
\draw [color=cyan, fill=cyan] (0,2) circle (.2); 
\draw [color=violet, fill=violet] (2,0) circle (.2); 
\draw[color=blue, line width=1.5pt, ->,>=stealth] (1.8,.2)--(.2,1.8);
\draw[color=red, line width=1.5pt, ->,>=stealth] (-1.8,0)--(1.8,0);
\draw[line width=.5pt, color=black,<-,>=stealth] plot [smooth, tension=0.6] coordinates 
      {(2.2,.5)(.7,2)};
      \node at (0,0) [anchor=north, color=red] {$a$};
      \node at (1,1) [anchor=east, color=blue] {$b_{(1)}$};
      \node at (-1.8,-1.5) [anchor=west, color=black] {$\cdots T(b_{(2)})\cdots$};
\end{scope}
\draw[color=black,->,>=stealth] (-1.5,1)--(.5,1);
\begin{scope}[shift={(3.5,0)}]
\draw [color=orange, fill=orange] (-2,0) circle (.2); 
\draw [color=cyan, fill=cyan] (0,2) circle (.2); 
\draw [color=violet, fill=violet] (2,0) circle (.2); 
\draw[color=blue, line width=1.5pt, ->,>=stealth] (-1.8,.2)--(-.2,1.8);
\draw[color=red, line width=1.5pt, ->,>=stealth] (-1.8,0)--(1.8,0);
\draw[line width=.5pt, color=black,->,>=stealth] plot [smooth, tension=0.6] coordinates 
      {(.8,2)(-.8,.5)(1.4,.5)};
\node at (0,0) [anchor=north, color=red] {$a_{(2)}$};
\node at (-.8,1.5) [anchor=east, color=blue] {$b_{(1)} a_{(1)}$};
\node at (-1.8,-1.5) [anchor=west, color=black] {$\cdots T(b_{(2)})\cdots$};
\end{scope}
\end{tikzpicture}
\\
\intertext{}
\label{eq:vert6}
&\begin{tikzpicture}[scale=0.3]
\begin{scope}[shift={(-6,0)}]
\draw[line width=1.5pt, color=blue](-2,3)--(-2,-3);
\draw[line width=1.5pt, color=red](2,3)--(2,-3);
\node at (-2,3)[color=blue, anchor=south]{$\beta$};
\node at (2,3)[color=red, anchor=south]{$\alpha$};
\draw[line width=1pt, color=black] plot [smooth, tension=0.6] coordinates 
      {(2,-.5)(1,-.8)(-1,-2.2)(-2,-2.5)};
      \draw[line width=1pt, color=black] plot [smooth, tension=0.4] coordinates 
      {(2,2.5)(3,2)(3,1)(-2,.5)(-3,-3)};
       \draw [color=black, fill=white, line width=1pt] (3.1,1.5) circle (.3); 
            \draw [color=black, fill=white, line width=1pt] (3.1,1.5) circle (.15); 
      \end{scope}
      \node at (-1,0){$=$};
      \begin{scope}[shift={(4,0)}]
\draw[line width=1.5pt, color=blue](-2,3)--(-2,-3);
\draw[line width=1.5pt, color=red](2,3)--(2,-3);
\node at (-2,3)[color=blue, anchor=south]{$\beta$};
\node at (2,3)[color=red, anchor=south]{$\alpha$};
\draw[line width=1pt, color=black] plot [smooth, tension=0.6] coordinates 
      {(2,2.5)(1,2.2)(-1,.8)(-2,.5)};
      \draw[line width=1pt, color=black] plot [smooth, tension=0.4] coordinates 
      {(2,1)(3,.5)(3,0)(-2,-1)(-3,-1.5)(-3,-3)};
             \draw [color=black, fill=white, line width=1pt] (3.1,.2) circle (.3); 
               \draw [color=black, fill=white, line width=1pt] (3.1,.2) circle (.15); 
      \end{scope}
\end{tikzpicture}\qquad
&
&\begin{tikzpicture}[scale=.6]
\begin{scope}[shift={(-4,0)}]
\draw [color=orange, fill=orange] (-2,0) circle (.2); 
\draw [color=cyan, fill=cyan] (0,2) circle (.2); 
\draw [color=violet, fill=violet] (2,0) circle (.2); 
\draw[color=blue, line width=1.5pt, ->,>=stealth] (1.8,.2)--(.2,1.8);
\draw[color=red, line width=1.5pt, ->,>=stealth] (-1.8,0)--(1.8,0);
\draw[line width=.5,<-,>=stealth](-1.6,-.5)--(1.6, -.5);
\node at (0,0) [anchor=south, color=red] {$a_{(1)}$};
\node at (0.8,1.5) [anchor=west, color=blue] {$b$};
\node at (-1.8,-1.5) [anchor=west, color=black] {$\cdots T(a_{(2)})\cdots$};
\end{scope}
\draw[color=black,->,>=stealth] (-1.5,1)--(.5,1);
\begin{scope}[shift={(3.5,0)}]
\draw [color=orange, fill=orange] (-2,0) circle (.2); 
\draw [color=cyan, fill=cyan] (0,2) circle (.2); 
\draw [color=violet, fill=violet] (2,0) circle (.2); 
\draw[color=blue, line width=1.5pt, ->,>=stealth] (-1.8,.2)--(-.2,1.8);
\draw[color=red, line width=1.5pt, ->,>=stealth] (-1.8,0)--(1.8,0);
\draw[line width=.5,<-,>=stealth](-1.6,-.5)--(1.6, -.5);
\node at (0,0) [anchor=south, color=red] {$a_{(2)}$};
\node at (-.8,2) [anchor=north east, color=blue] {$b a_{(1)}$};
\node at (-1.8,-1.5) [anchor=west, color=black] {$\cdots T(a_{(3)})\cdots$};
\end{scope}
\end{tikzpicture}
\end{align}
Identities \eqref{eq:vert3} and \eqref{eq:vert6} follow directly from the fact that the morphisms $\rhd_{\beta+}$, $\rhd_{\beta-}$  and the morphisms $\delta_{\alpha+}$, $\delta_{\alpha-}$ commute by Lemma \ref{lem:triangleops}.
The proof of the other identities involves the compatibility conditions for $H$-left module and $H$-left comodule structures in a Hopf module, the properties of the antipode and the properties \eqref{eq:tplusprop1} and \eqref{eq:tplusprop2} of the involution \eqref{eq:tplusdef}.
A diagrammatic proof is given in Figure \ref{fig:compvertface}. Alternatively, this follows by a direct computation in Sweedler notation.
\end{proof}

We now show that the edge slides from Definition \ref{def:edge slide} satisfy generalisations of Bene's  relations for slides in \eqref{eq:slidegerinv} to \eqref{eq:rightpentger}.
These relations are obtained by replacing each connected component of the outer circle in  diagrams in \eqref{eq:slidegerinv} to \eqref{eq:rightpentger} with a vertex and orienting the edges between these vertices that replace the chords in  \eqref{eq:slidegerinv} to \eqref{eq:rightpentger}.

\begin{proposition} \label{prop: relationsbene} The edge slides from Definition \ref{def:edge slide} satisfy the  involutivity relation, the commutativity relation, triangle relation and the left and right pentagon relation from Theorem \ref{th:bene}. 
\end{proposition}

\begin{proof}
The involutivity relations are satisfied trivially, because all edge slides have inverses. The  commutativity relations follow directly from the fact that each edge slide affects only the copies of $H$ in $H^{\oo E}$ for edges involved in the slide.  

The triangle relation follows by a direct diagrammatic computation in Sweedler notation 
\begin{align}
\begin{tikzpicture}[scale=.5]
\begin{scope}[shift={(-8,0)}]
\draw [color=black, fill=black] (-2,0) circle (.2); 
\draw [color=black, fill=black] (0,2) circle (.2); 
\draw [color=black, fill=black] (2,0) circle (.2); 
\draw[color=red, line width=1.5pt, ->,>=stealth] (-1.8,0)--(1.8, 0);
\draw[color=blue, line width=1.5pt, ->,>=stealth] (-.2, 1.8)--(-1.8, .2);
\node at (0,0)[color=red, anchor=north]{$a$};
\node at (-1,1)[color=blue, anchor=south east]{$b$};
\end{scope}
\draw[line width=.5pt, ->,>=stealth](-5,1)--(-4,1);
\node at (-4.5,1)[anchor=south]{$S_{\alpha^L}$};
\begin{scope}[shift={(-1,0)}]
\draw [color=black, fill=black] (-2,0) circle (.2); 
\draw [color=black, fill=black] (0,2) circle (.2); 
\draw [color=black, fill=black] (2,0) circle (.2); 
\draw[color=red, line width=1.5pt, ->,>=stealth] (-1.8,0)--(1.8, 0);
\draw[color=blue, line width=1.5pt, ->,>=stealth] (.2, 1.8)--(1.8, .2);
\node at (0,0)[color=red, anchor=north]{$\low a 2$};
\node at (1.4,1)[color=blue, anchor=south]{$\low a 1 b$};
\end{scope}
\draw[line width=.5pt, ->,>=stealth](3,1)--(4,1);
\node at (3.5,1)[anchor=south]{$S_{\beta^{-R}}$};
\begin{scope}[shift={(8,0)}]
\draw [color=black, fill=black] (-2,0) circle (.2); 
\draw [color=black, fill=black] (0,2) circle (.2); 
\draw [color=black, fill=black] (2,0) circle (.2); 
\draw[color=red, line width=1.5pt, ->,>=stealth] (-1.8,.2)--(-.2, 1.8);
\draw[color=blue, line width=1.5pt, ->,>=stealth] (.2, 1.8)--(1.8, .2);
\node at (-1.9,1)[color=red, anchor=south]{$T(\low b 2)$};
\node at (1.4,1)[color=blue, anchor=south]{$ a  \low b 1$};
\end{scope}
\draw[line width=.5pt, ->,>=stealth](12,1)--(13,1);
\node at (12.5,1)[anchor=south]{$S_{\alpha^{-R}}$};
\begin{scope}[shift={(17,0)}]
\draw [color=black, fill=black] (-2,0) circle (.2); 
\draw [color=black, fill=black] (0,2) circle (.2); 
\draw [color=black, fill=black] (2,0) circle (.2); 
\draw[color=red, line width=1.5pt, ->,>=stealth] (-1.8,.2)--(-.2, 1.8);
\draw[color=blue, line width=1.5pt, ->,>=stealth] (-1.8,0)--(1.8, 0);
\node at (-1.6,1)[color=red, anchor=south]{$T(b)$};
\node at (0,0)[color=blue, anchor=north]{$a$};
\end{scope}
\end{tikzpicture}
\end{align}
The proofs for other edge orientations follow by taking inverses and reversing edge orientation with the involution $T$. The proof  for the cases where some vertices coincide is analogous.

The two pentagon relations follow directly from the compatibility condition between the $H$-left module and $H$-left comodule structures for a Hopf module and from the properties of the involution $T$. A diagrammatic proof for the edge orientations in Figure \ref{fig:sliderel5} is given in Figure \ref{fig:pent}. The proof for other edge orientations follow by taking inverses and reversing edge orientation with the involution $T$. The proofs  for the cases where some vertices coincide are analogous.
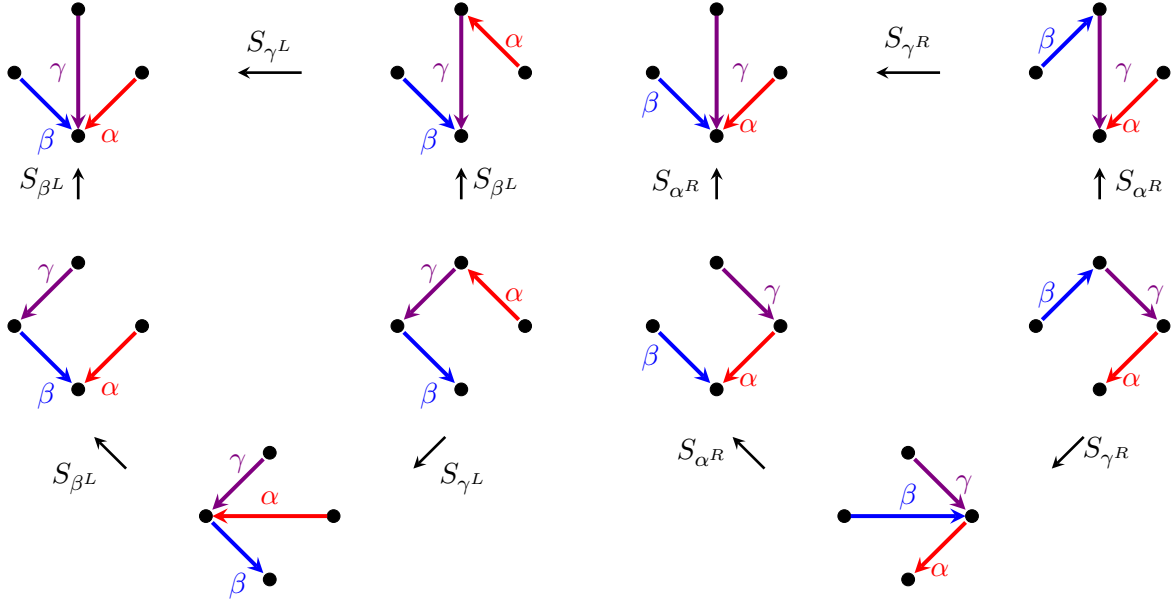
\begin{figure}[h]
	\centering
\begin{tikzpicture}[scale=.42]
\begin{scope}[shift={(-6,0)}]
\draw [color=black, fill=black] (-2,0) circle (.2); 
\draw [color=black, fill=black] (0,2) circle (.2); 
\draw [color=black, fill=black] (2,0) circle (.2); 
\draw [color=black, fill=black] (0,-2) circle (.2); 
\draw[color=blue, line width=1.5pt, <-,>=stealth] (-.2,-1.8)-- node[below=5pt] {$\beta$}(-1.8,-.2);
\node at (-1,-1) [anchor=north east] {};
\draw[color=violet, line width=1.5pt, <-,>=stealth] (0,-1.8)-- node[left]{$\gamma$} (0,1.8);
\node at (1,-1) [anchor=north west] {};
\draw[color=red, line width=1.5pt, <-,>=stealth] (.2,-1.8)--node[below=5pt]{$\alpha$}(1.8,-.2);
\node at (0,1) [anchor=west] {};
\end{scope}
\draw[color=black, line width=1pt, <-,>=stealth] (-1,0)--node[above]{$S_{\gamma^{L}}$}(1,0);
\begin{scope}[shift={(6,0)}]
\draw [color=black, fill=black] (-2,0) circle (.2); 
\draw [color=black, fill=black] (0,2) circle (.2); 
\draw [color=black, fill=black] (2,0) circle (.2); 
\draw [color=black, fill=black] (0,-2) circle (.2); 
\draw[color=blue, line width=1.5pt, <-,>=stealth] (-.2,-1.8)--node[below=5pt] {$\beta$}(-1.8,-.2);
\node at (-1,-1) [anchor=north east] {};
\draw[color=violet, line width=1.5pt, <-,>=stealth] (0,-1.8)--node[left]{$\gamma$}(0,1.8);
\node at (0,0) [anchor=west] {};
\draw[color=red, line width=1.5pt, <-,>=stealth] (.2,1.8)--node[right]{$\alpha$}(1.8,.2);
\node at (1,1)[anchor=south west] {};
\end{scope}
\draw[color=black, line width=1pt, <-,>=stealth] (6,-3)--node[right]{$S_{\beta^{L}}$}(6,-4);
\begin{scope}[shift={(6,-8)}]
\draw [color=black, fill=black] (-2,0) circle (.2); 
\draw [color=black, fill=black] (0,2) circle (.2); 
\draw [color=black, fill=black] (2,0) circle (.2); 
\draw [color=black, fill=black] (0,-2) circle (.2); 
\draw[color=blue, line width=1.5pt, <-,>=stealth] (-.2,-1.8)--node[below=5pt] {$\beta$}(-1.8,-.2);
\node at (-1,-1) [anchor=north east] {};
\draw[color=violet, line width=1.5pt, <-,>=stealth] (-1.8,.2)--node[left, above]{$\gamma$}(-.2,1.8);
\node at (-1,1)[anchor=south east] {};
\draw[color=red, line width=1.5pt, <-,>=stealth] (.2,1.8)--node[right]{$\alpha$}(1.8,.2);
\node at (1,1)[anchor=south west] {};
\end{scope}
\draw[color=black, line width=1pt, ->,>=stealth] (5.5,-11.5)--(4.5,-12.5);
\node at (5,-12)[anchor=north west] {$S_{\gamma^{L}}$};
\begin{scope}[shift={(0,-14)}]
\draw [color=black, fill=black] (-2,0) circle (.2); 
\draw [color=black, fill=black] (0,2) circle (.2); 
\draw [color=black, fill=black] (2,0) circle (.2); 
\draw [color=black, fill=black] (0,-2) circle (.2); 
\draw[color=blue, line width=1.5pt, <-,>=stealth] (-.2,-1.8)--node[below=5pt] {$\beta$}(-1.8,-.2);
\node at (-1,1) [anchor=south east] {};
\draw[color=violet, line width=1.5pt, <-,>=stealth] (-1.8,0.2)--node[left, above]{$\gamma$}(-0.2,1.8);
\node at (-1,-1) [anchor=north east] {};
\draw[color=red, line width=1.5pt, <-,>=stealth] (-1.8,0)--node[above]{$\alpha$}(1.8,0);
\node at (0,0) [anchor=south] {};
\end{scope}
\draw[color=black, line width=1pt, <-,>=stealth] (-5.5,-11.5)--(-4.5,-12.5);
\node at (-5,-12)[anchor=north east] {$S_{\beta^{L}}$};
\begin{scope}[shift={(-6,-8)}]
\draw [color=black, fill=black] (-2,0) circle (.2); 
\draw [color=black, fill=black] (0,2) circle (.2); 
\draw [color=black, fill=black] (2,0) circle (.2); 
\draw [color=black, fill=black] (0,-2) circle (.2); 
\draw[color=blue, line width=1.5pt, <-,>=stealth] (-.2,-1.8)--node[below=5pt] {$\beta$}(-1.8,-.2);
\node at (-1,1) [anchor=south east] {};
\draw[color=violet, line width=1.5pt, <-,>=stealth] (-1.8,.2)--node[left, above]{$\gamma$}(-.2,1.8);
\node at (1,-1) [anchor=north west] {};
\draw[color=red, line width=1.5pt, <-,>=stealth] (.2,-1.8)--node[below=5pt]{$\alpha$}(1.8,-.2);
\node at (-1,-1) [anchor=north east] {};
\end{scope}
\draw[color=black, line width=1pt, <-,>=stealth] (-6,-3)--node[left]{$S_{\beta^{L}}$}(-6,-4);
\begin{scope}[shift={(20,0)}]
\begin{scope}[shift={(-6,0)}]
\draw [color=black, fill=black] (-2,0) circle (.2); 
\draw [color=black, fill=black] (0,2) circle (.2); 
\draw [color=black, fill=black] (2,0) circle (.2); 
\draw [color=black, fill=black] (0,-2) circle (.2); 
\draw[color=blue, line width=1.5pt, <-,>=stealth] (-.2,-1.8)--node[left=5pt] {$\beta$}(-1.8,-.2);
\node at (-1,-1)[anchor=north east] {};
\draw[color=violet, line width=1.5pt, <-,>=stealth] (0,-1.8)--node[right=1.5pt]{$\gamma$}(0,1.8);
\node at (0,1)[anchor=west] {};
\draw[color=red, line width=1.5pt, <-,>=stealth] (.2,-1.8)--node[below=1.5pt]{$\alpha$}(1.8,-.2);
\node at (1,-1)[anchor=north west] {};
\end{scope}
\draw[color=black, line width=1pt, <-,>=stealth] (-1,0)--node[above=1.5pt]{$S_{\gamma^{R}}$}(1,0);
\begin{scope}[shift={(6,0)}]
\draw [color=black, fill=black] (-2,0) circle (.2); 
\draw [color=black, fill=black] (0,2) circle (.2); 
\draw [color=black, fill=black] (2,0) circle (.2); 
\draw [color=black, fill=black] (0,-2) circle (.2); 
\draw[color=blue, line width=1.5pt, ->,>=stealth] (-1.8,.2)--node[left] {$\beta$}(-.2,1.8);
\node at (-1,1) [anchor=south east] {};
\draw[color=violet, line width=1.5pt, <-,>=stealth] (0,-1.8)--node[right=1.5pt]{$\gamma$}(0,1.8);
\node at (0,1)[anchor=west] {};
\draw[color=red, line width=1.5pt, <-,>=stealth] (.2,-1.8)--node[below=1.5pt]{$\alpha$}(1.8,-.2);
\node at (1,-1)[anchor=north west] {};
\end{scope}
\draw[color=black, line width=1pt, <-,>=stealth] (6,-3)--node[right=1.5pt]{$S_{\alpha^{R}}$}(6,-4);
\begin{scope}[shift={(6,-8)}]
\draw [color=black, fill=black] (-2,0) circle (.2); 
\draw [color=black, fill=black] (0,2) circle (.2); 
\draw [color=black, fill=black] (2,0) circle (.2); 
\draw [color=black, fill=black] (0,-2) circle (.2); 
\draw[color=blue, line width=1.5pt, ->,>=stealth] (-1.8,.2)--node[left] {$\beta$}(-.2,1.8);
\node at (-1,1) [anchor=south east] {};
\draw[color=violet, line width=1.5pt, ->,>=stealth] (.2,1.8)--node[right=1.5pt]{$\gamma$}(1.8,.2);
\node at (1,1) [anchor=south west]{};
\draw[color=red, line width=1.5pt, <-,>=stealth] (.2,-1.8)--node[below=1.5pt]{$\alpha$}(1.8,-.2);
\node at (1,-1)[anchor=north west] {};
\end{scope}
\draw[color=black, line width=1pt, ->,>=stealth] (5.5,-11.5)--node[right=1.5pt]{$S_{\gamma^{R}}$}(4.5,-12.5);
\begin{scope}[shift={(0,-14)}]
\draw [color=black, fill=black] (-2,0) circle (.2); 
\draw [color=black, fill=black] (0,2) circle (.2); 
\draw [color=black, fill=black] (2,0) circle (.2); 
\draw [color=black, fill=black] (0,-2) circle (.2); 
\draw[color=blue, line width=1.5pt, ->,>=stealth] (-1.8,0)--node[above=-2pt] {$\beta$}(1.8,0);
\node at (0,0) [anchor=south]{};
\draw[color=violet, line width=1.5pt, ->,>=stealth] (.2,1.8)--node[right=1.5pt]{$\gamma$}(1.8,.2);
\node at (1,1) [anchor=south west] {};
\draw[color=red, line width=1.5pt, <-,>=stealth] (.2,-1.8)--node[below=1.5pt]{$\alpha$}(1.8,-.2);
\node at (1,-1)[anchor=north west] {};
\end{scope}
\draw[color=black, line width=1pt, <-,>=stealth] (-5.5,-11.5)--node[left=4pt]{$S_{\alpha^{R}}$}(-4.5,-12.5);
\begin{scope}[shift={(-6,-8)}]
\draw [color=black, fill=black] (-2,0) circle (.2); 
\draw [color=black, fill=black] (0,2) circle (.2); 
\draw [color=black, fill=black] (2,0) circle (.2); 
\draw [color=black, fill=black] (0,-2) circle (.2); 
\draw[color=blue, line width=1.5pt, ->,>=stealth] (-1.8,-.2)--node[left=5pt] {$\beta$}(-.2,-1.8);
\node at (-1,-1) [anchor=north east] {};
\draw[color=violet, line width=1.5pt, ->,>=stealth] (.2,1.8)--node[right=1.5pt]{$\gamma$}(1.8,.2);
\node at (1,1) [anchor=south west] {};
\draw[color=red, line width=1.5pt, <-,>=stealth] (0.2,-1.8)--node[below=1.5pt]{$\alpha$}(1.8,-.2);
\node at (1,-1)[anchor=north west] {};
\end{scope}
\draw[color=black, line width=1pt, <-,>=stealth] (-6,-3)--node[left=1.5pt]{$S_{\alpha^{R}}$}(-6,-4);
\end{scope}
\end{tikzpicture}
	\caption{The left and right pentagon relations for edge slides.}
	\label{fig:sliderel5}
\end{figure}

\begin{figure} 
\centering
\def\svgwidth{.55\columnwidth}
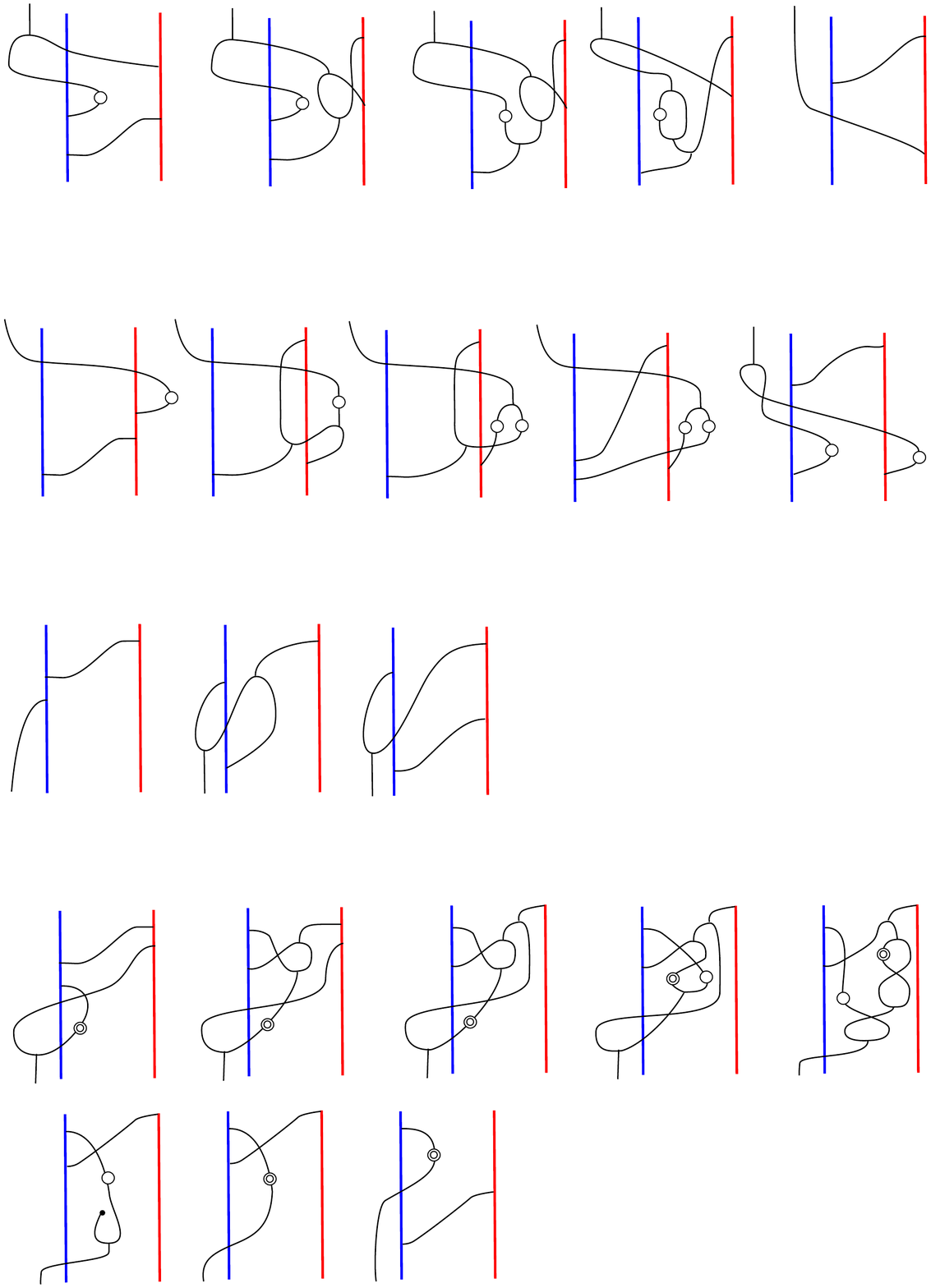
\caption{Diagrammatic proof of the identities \eqref{eq:vert1}, \eqref{eq:vert2}, \eqref{eq:vert4}, \eqref{eq:vert5}.}
\label{fig:compvertface}
\end{figure}

\begin{figure}
\begin{center}
\includegraphics[scale=.35]{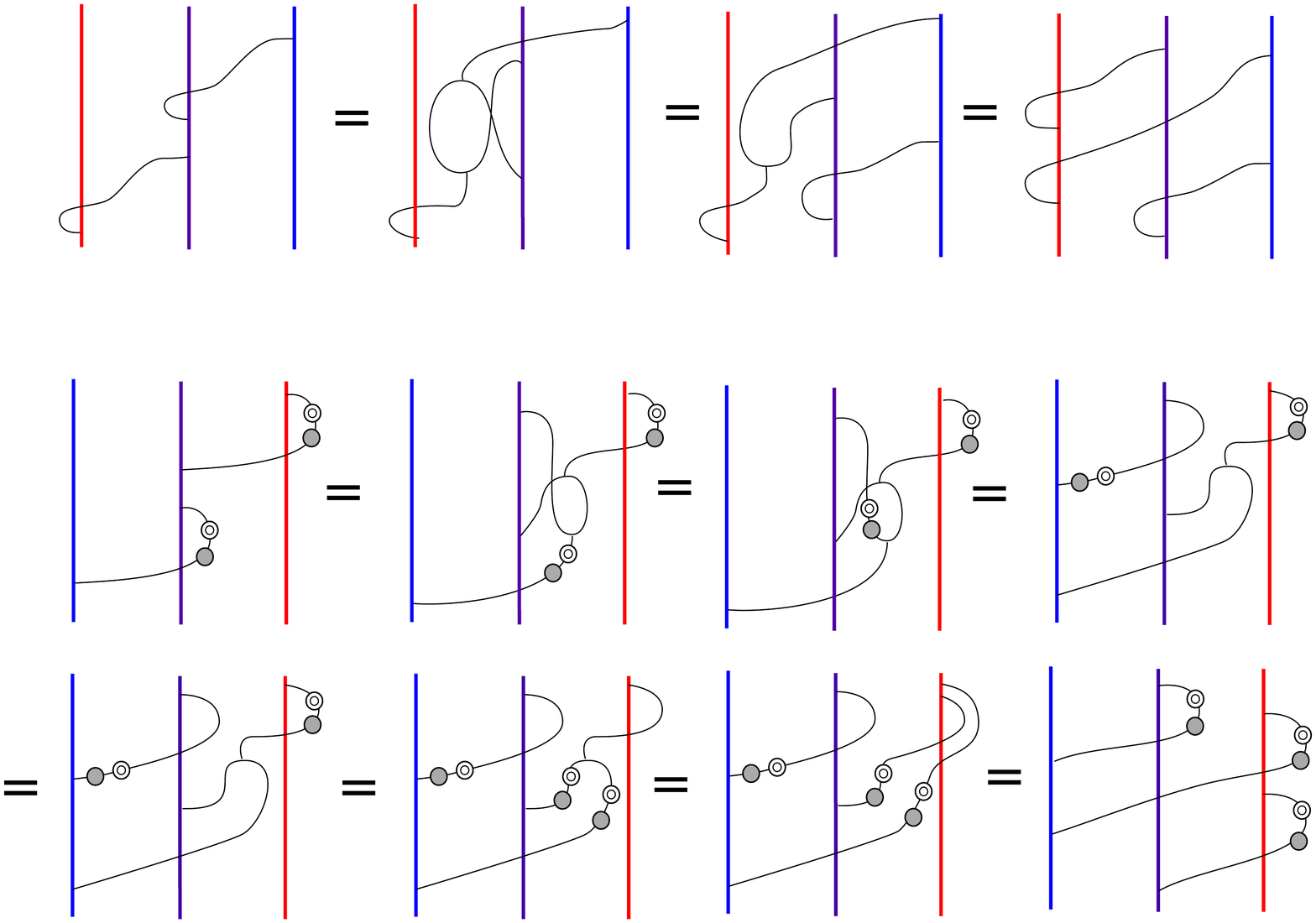}
\end{center}
\caption{Diagrammatic proof of the left  and  right pentagon relation in Figure~\ref{fig:sliderel5}.}
\label{fig:pent}
\end{figure}

\end{proof}

\begin{corollary}\label{cor:adjacentendslide} The edge slides from Definition \ref{def:edge slide} satisfy generalisations of the opposite end commutativity relation and the adjacent commutativity relation from Lemma \ref{lem:bene}.
\end{corollary}

Note that the  triangle and the pentagon relations in Proposition \ref{prop: relationsbene} have a different status. The diagrammatic proof of the pentagon relations in Figure \ref{fig:pent} generalises directly to tensor products of Hopf bimodules. The proof of the triangle relation relies on the fact that edges are decorated with copies of a Hopf monoid $H$ in $\mac$ and works only for trivial $H$-Hopf bimodules.

If we restrict attention to ribbon graphs $\Gamma$ with a single ciliated vertex, the edge slides from  \eqref{eq:slideedge} coincide with an oriented version of the chord slides in Section \ref{sec:chordslides}.  If additionally, the ribbon graph $\Gamma$ has only a single face, then the surface $\Sigma$ obtained by gluing an annulus to the face is an oriented surface with a single boundary component. 
By Theorem \ref{th:bene}, the mapping class group $\mathrm{Map}(\Sigma)$ is then presented by finite sequences of chord slides that preserve $\Gamma$ up to the cilium, subject to the relations in  Theorem \ref{th:bene}. As the edge slides from Definition \ref{def:edge slide} satisfy these relations by Proposition \ref{prop: relationsbene}, we obtain an action of the mapping class group $\mathrm{Map}(\Sigma)$. As no edge ends slide over cilia, Proposition \ref{prop:vertfacecomp} implies that it acts by automorphisms of Yetter-Drinfeld modules with respect to the Yetter-Drinfeld module structure associated with the cilium.

\begin{theorem} \label{the:edge1}Let $H$ be a pivotal Hopf monoid in a symmetric monoidal category $\mac$,  $\Gamma$ a directed ribbon graph with a single ciliated vertex  and a single face  and $\Sigma$ the oriented surface obtained by gluing an annulus to the face. 
 Then  the edge slides from Definition \ref{def:edge slide} define a group homomorphism  $\rho: \text{Map}(\Sigma)\to \text{Aut}_{YD}(H^{\oo E})$. 
\end{theorem}

\section{The torus and the one-holed torus}
\label{sec:torus}

In this section, we consider the simplest examples of  mapping class group actions by edge slides, namely the mapping class groups of the torus $T$ and of the torus $T^*$ with a disc removed. 
The mapping class group of $T$  is the modular group $\mathrm{SL}(2,\mathbb Z)$ \cite[Sec.~2.2.4]{FM}. A presentation of the mapping class group of $T^*$ is obtained from Theorem \ref{th:gervais}. It involves only two generators $\alpha_1$, $\delta_1$ and relation (iii), which yields 
 the braid group $B_3$ on three strands 
\begin{align}\label{eq:sl2rep}
&\mathrm{SL}(2,\mathbb Z)=\langle A,B\mid (BAB)^4=1, ABA=BAB\rangle\qquad\qquad
B_3=\langle A,B\mid  ABA=BAB\rangle.
\end{align}
The element $(BAB)^4$ is central in $B_3$.  
In both cases, the generators $A$, $B$ correspond to Dehn twists around the $a$- and $b$-cycle. For more details on these presentations, see for instance \cite[App.~A]{KT}.

To describe the tori $T$ and $T^*$  we
consider a ribbon graph $\Gamma$  that consists of  a single  ciliated vertex and two oriented loops representing the generators of the fundamental groups $\pi_1(T)\cong \mathbb Z\times\mathbb Z$ and $\pi_1(T^*)\cong F_2$. 
By cutting the  vertex at the cilium, we obtain the associated chord diagram. The ribbon graph $\Gamma$ and this  chord diagram each have a single face, as shown below.

\vspace{-2cm}
\begin{align}\label{eq:chordtorus}
\begin{tikzpicture}[scale=.5, baseline=(current  bounding  box.center)]
\begin{scope}[shift={(-5,0)}]
\draw [color=black,line width=1pt, fill=white] (0,0) circle (.2); 
\draw[color=red, line width=1.5pt, <-,>=stealth] (-.2,.2).. controls (-4,4)  and (4,4) ..(.2,.2);
\draw[color=blue, line width=1.5pt, ->,>=stealth] (.2,0).. controls (6,0)  and (0,6) ..(0,.2);
\draw[color=black, line width=1pt, style=dotted] (0,-.2)--(0,-.8);
\node at (0,3) [anchor=south, color=red]{$b$};
\node at (2.5,2.3) [anchor=south, color=blue]{$a$};
\end{scope}
\begin{scope}[shift={(1.5,0)}]
\draw[line width=1pt, color=black](-3,0)--(5,0);
\draw[color=red, line width=1.5pt, <-,>=stealth] (-2,0).. controls (-2,2) and (2,2)..(2,0);
\node at (0,1.5) [anchor=south, color=red]{$b$};
\draw[color=blue, line width=1.5pt, <-,>=stealth] (0,0).. controls (0,2) and (4,2)..(4,0);
\node at (2,1.5) [anchor=south, color=blue]{$a$};
\end{scope}
\begin{scope}[shift={(11,0)}]
\draw[line width=1pt, color=black](-3,0)--(5,0);
\draw[color=red, line width=1.5pt, <-,>=stealth] (-2,0).. controls (-2,2) and (2,2)..(2,0);
\node at (0,2) [anchor=south, color=red]{$b$};
\draw[color=blue, line width=1.5pt, <-,>=stealth] (0,0).. controls (0,2) and (4,2)..(4,0);
\node at (2,2) [anchor=south, color=blue]{$a$};
\draw[line width=.5pt, color=black,->,>=stealth] plot [smooth, tension=0.6] coordinates 
      {(-2.5,0)(-1.6,1.6)(0,2)(1.6,1.6)(2.5,.3)(3.5,.3)(3, 1)(2,1.3)(1,1)(0.5,.3)(1.5,.3)(1,1)(0,1.3)  (-1,1)(-1.5,.3)(-.5,.3)(.4,1.6)(2,2)(3.6,1.6)(4.5,.1)  };
\end{scope}
\end{tikzpicture}
\end{align}

Given a pivotal Hopf monoid $H$ in a symmetric monoidal category $\mac$, we associate to $\Gamma$ the object $H^{\oo 2}$ in $\mac$, with the first copy of $H$ assigned to  $a$ and the second to $b$.
The $H$-left module structure $\rhd_v: H\oo H^{\oo 2}\to H^{\oo 2}$ for the ciliated vertex and the $H$-left comodule structure $\delta_f: H^{\oo 2} \to H\oo H^{\oo 2}$ for the ciliated face of the ribbon graph from Definition \ref{def:vertexaction} define a left-left Yetter-Drinfeld module structure on $H^{\oo 2}$ by Lemma \ref{lem:facevertex}. They are given by
\begin{align}	
\label{facestruct}
&\rhd_v:\qquad \begin{tikzpicture}[scale=.5, baseline=(current  bounding  box.center)]
\begin{scope}[shift={(-5,0)}]
\draw[line width=1pt, color=black](-3,0)--(5,0);
\draw[color=red, line width=1.5pt, <-,>=stealth] (-2,0).. controls (-2,2) and (2,2)..(2,0);
\node at (0,1.5) [anchor=south, color=red]{$b$};
\draw[color=blue, line width=1.5pt, <-,>=stealth] (0,0).. controls (0,2) and (4,2)..(4,0);
\node at (2,1.5) [anchor=south, color=blue]{$a$};
\node at (1,0)[anchor=north] {$h$};
\end{scope}
\draw[line width=.5pt, color=black, ->,>=stealth] (0,1)--(2,1);
\begin{scope}[shift={(5,0)}]
\draw[line width=1pt, color=black](-3,0)--(5,0);
\draw[color=red, line width=1.5pt, <-,>=stealth] (-2,0).. controls (-2,2) and (2,2)..(2,0);
\node at (0,1.5) [anchor=south east, color=red]{$\low h 4 b S(\low h 2)$};
\draw[color=blue, line width=1.5pt, <-,>=stealth] (0,0).. controls (0,2) and (4,2)..(4,0);
\node at (2,1.5) [anchor=south west, color=blue]{$\low h 3 a S(\low h 1) $};
\end{scope}
\end{tikzpicture}\\
\nonumber
&\delta_f:\qquad \begin{tikzpicture}[scale=.5, baseline=(current  bounding  box.center)]
\begin{scope}[shift={(-5,0)}]
\draw[line width=1pt, color=black](-3,0)--(5,0);
\draw[color=red, line width=1.5pt, <-,>=stealth] (-2,0).. controls (-2,2) and (2,2)..(2,0);
\node at (0,1.5) [anchor=south, color=red]{$b$};
\draw[color=blue, line width=1.5pt, <-,>=stealth] (0,0).. controls (0,2) and (4,2)..(4,0);
\node at (2,1.5) [anchor=south, color=blue]{$a$};
\end{scope}
\draw[line width=.5pt, color=black, ->,>=stealth] (0,1)--(2,1);
\begin{scope}[shift={(5,0)}]
\draw[line width=1pt, color=black](-3,0)--(5,0);
\draw[color=red, line width=1.5pt, <-,>=stealth] (-2,0).. controls (-2,2) and (2,2)..(2,0);
\node at (0,1.5) [anchor=south, color=red]{$\low b 2$};
\draw[color=blue, line width=1.5pt, <-,>=stealth] (0,0).. controls (0,2) and (4,2)..(4,0);
\node at (2,1.5) [anchor=south, color=blue]{$\low a 2 $};
      \node at (1,0) [anchor=north]{$ T(\low a 3)\low b 1\low a 1 T(\low b 3) $};
\end{scope}
\end{tikzpicture}\nonumber
\end{align}
By Theorem \ref{the:edge1}, the  edge slides acting on this ribbon graph induce an action of the braid group $B_3$ on $H^{ \oo 2}$.
We will now show that  the generators $A$ and $B$ from \eqref{eq:sl2rep} can be identified with the  slides along the left side of the two edges to their targets. 
From  expressions \eqref{eq:slidedef2a} and \eqref{eq:slidedef1b}  one finds 
\begin{align}\label{eq:abslide}
&\begin{tikzpicture}[scale=.5]
\begin{scope}[shift={(-6,0)}]
\draw[line width=1pt, color=black](-3,0)--(5,0);
\draw[color=red, line width=1.5pt, <-,>=stealth] (-2,0).. controls (-2,2) and (2,2)..(2,0);
\node at (0,1.5) [anchor=south, color=red]{$b$};
\draw[color=blue, line width=1.5pt, <-,>=stealth] (0,0).. controls (0,2) and (4,2)..(4,0);
\node at (2,1.5) [anchor=south, color=blue]{$a$};
\end{scope}
\draw[line width=1pt, color=black, ->,>=stealth](0,1)--(2,1);
\node at (1,1) [anchor=south]{$D_b=S_{b^L}$};
\begin{scope}[shift={(6,0)}]
\draw[line width=1pt, color=black](-3,0)--(5,0);
\draw[color=red, line width=1.5pt, <-,>=stealth] (-2,0).. controls (-2,2) and (2,2)..(2,0);
\node at (0,1.5) [anchor=south, color=red]{$\low b 2$};
\draw[color=blue, line width=1.5pt, <-,>=stealth] (0,0).. controls (0,2) and (4,2)..(4,0);
\node at (2,1.5) [anchor=south west , color=blue]{$\low b 1 a$};
\end{scope}
\end{tikzpicture}\\
&\begin{tikzpicture}[scale=.5]
\begin{scope}[shift={(-6,0)}]
\draw[line width=1pt, color=black](-3,0)--(5,0);
\draw[color=red, line width=1.5pt, <-,>=stealth] (-2,0).. controls (-2,2) and (2,2)..(2,0);
\node at (0,1.5) [anchor=south, color=red]{$b$};
\draw[color=blue, line width=1.5pt, <-,>=stealth] (0,0).. controls (0,2) and (4,2)..(4,0);
\node at (2,1.5) [anchor=south, color=blue]{$a$};
\end{scope}
\draw[line width=1pt, color=black, ->,>=stealth](0,1)--(2,1);
\node at (1,1) [anchor=south]{$D_a=S_{a^L}$};
\begin{scope}[shift={(6,0)}]
\draw[line width=1pt, color=black](-3,0)--(5,0);
\draw[color=red, line width=1.5pt, <-,>=stealth] (-2,0).. controls (-2,2) and (2,2)..(2,0);
\node at (0,1.5) [anchor=south east, color=red]{$ b S(\low a 1) $};
\draw[color=blue, line width=1.5pt, <-,>=stealth] (0,0).. controls (0,2) and (4,2)..(4,0);
\node at (2,1.5) [anchor=south west , color=blue]{$\low a 2$};
\end{scope}
\end{tikzpicture}\nonumber
\end{align}

As they do not slide any edge ends over the cilium, Proposition   \ref{prop:vertfacecomp} implies that the  morphisms $D_a,D_b\in \mathrm{Aut}(H^{\oo 2})$  are automorphisms of Yetter-Drinfeld modules with respect to the Yetter-Drinfeld module structure \eqref{facestruct}. We will now show   they satisfy the braid relation and  hence define an action of the group $B_3$ on $H^{\oo 2}$ by automorphisms of Yetter-Drinfeld modules. 
We also show that whenever $\mac$ is finitely complete and cocomplete, this induces 
 an action of the modular group $\mathrm{SL}(2,\mathbb Z)$ on the object $H^{\oo 2}_{inv}$ from Definition \ref{def:biinv}.

\begin{theorem} \label{th:modular}Let $H$ be a pivotal Hopf monoid in a symmetric monoidal category $\mac$.
\begin{compactenum}
\item The slides  in \eqref{eq:abslide} define a group homomorphism $\rho: B_3\to \mathrm{Aut}_{YD}(H^{\oo 2})$.
	\item If $\mac$ is finitely complete and cocomplete, the slides in \eqref{eq:abslide}   induce a group homomorphism \linebreak
	$\rho:\mathrm{SL}(2,\mathbb Z)\to \mathrm{Aut}(H^{\oo 2}_{inv})$.
\end{compactenum}
\end{theorem}

\begin{proof}
We identify $A=D_a$ and $B=D_b$ and verify the relations in \eqref{eq:sl2rep}. A direct computation using expressions \eqref{eq:abslide} yields
\begin{align}
	\nonumber
&D_b \circ D_a \circ D_b=D_a\circ D_b\circ D_a:\\
&\begin{tikzpicture}[scale=.5]
\begin{scope}[shift={(-6,0)}]
\draw[line width=1pt, color=black](-3,0)--(5,0);
\draw[color=red, line width=1.5pt, <-,>=stealth] (-2,0).. controls (-2,2) and (2,2)..(2,0);
\node at (0,1.5) [anchor=south, color=red]{$b$};
\draw[color=blue, line width=1.5pt, <-,>=stealth] (0,0).. controls (0,2) and (4,2)..(4,0);
\node at (2,1.5) [anchor=south, color=blue]{$a$};
\end{scope}
\draw[line width=1pt, color=black, ->,>=stealth](0,1)--(2,1);
\begin{scope}[shift={(6,0)}]
\draw[line width=1pt, color=black](-3,0)--(5,0);
\draw[color=red, line width=1.5pt, <-,>=stealth] (-2,0).. controls (-2,2) and (2,2)..(2,0);
\node at (-1,1.5) [anchor=south, color=red]{$\low b 3 S(a) S(\low b 1)$};
\draw[color=blue, line width=1.5pt, <-,>=stealth] (0,0).. controls (0,2) and (4,2)..(4,0);
\node at (2,1.5) [anchor=south west , color=blue]{$\low b 2$};
\end{scope}
\end{tikzpicture}\nonumber
\end{align}
By applying this morphism four times we obtain
\begin{align}
&(D_b \circ D_a \circ D_b)^4=\rhd_v\circ (T\oo 1_{H^{\oo 2}})\circ\delta_f
:\\
&\begin{tikzpicture}[scale=.5]
\begin{scope}[shift={(-6,0)}]
\draw[line width=1pt, color=black](-3,0)--(5,0);
\draw[color=red, line width=1.5pt, <-,>=stealth] (-2,0).. controls (-2,2) and (2,2)..(2,0);
\node at (0,1.5) [anchor=south, color=red]{$b$};
\draw[color=blue, line width=1.5pt, <-,>=stealth] (0,0).. controls (0,2) and (4,2)..(4,0);
\node at (2,1.5) [anchor=south, color=blue]{$a$};
\end{scope}
\draw[line width=1pt, color=black, ->,>=stealth](0,1)--(2,1);
\begin{scope}[shift={(6,0)}]
\draw[line width=1pt, color=black](-3,0)--(5,0);
\draw[color=red, line width=1.5pt, <-,>=stealth] (-2,0).. controls (-2,2) and (2,2)..(2,0);
\node at (-2,1.5) [anchor=south, color=red]{$\low x 4 S^2( \low b 2) S(\low x 2)$};
\draw[color=blue, line width=1.5pt, <-,>=stealth] (0,0).. controls (0,2) and (4,2)..(4,0);
\node at (2,1.5) [anchor=south west , color=blue]{$\low x 3 S^2( \low a 2) S(\low x 1)$};
\node at (-3,-.5)[anchor=north west]{$x=\low b 3 S(\low a 1)S(\low b 1)S^2(\low a 3)$};
\end{scope}
\end{tikzpicture}\nonumber
\end{align}
By Proposition \ref{prop:vertfacecomp}, the automorphisms $D_a$ and $D_b$ are automorphisms of Yetter-Drinfeld modules with respect to \eqref{facestruct}, and by
Lemma \ref{lem:maclemma} they induce  automorphisms $D_a,D_b$ of $H^{\oo 2}_{inv}$. 
As the automorphism $(D_b \circ D_a \circ D_b)^4=\rhd_v\circ (T\oo 1_{H^{\oo 2}})\circ\delta_f$ 
induces
the identity morphisms on $H^{\oo 2}_{inv}$, we obtain an $\mathrm{SL}(2,\mathbb Z)$-action on $H^{\oo 2}_{inv}$.
\end{proof}

Note that the expressions for the Dehn twists in \eqref{eq:abslide} depend only on the Hopf monoid $H$ and not on the choice of the pivotal structure for $H$.  
It will become apparent in Section \ref{sec:gervaismap} that this is specific to the torus and not true for Dehn twists in higher genus mapping class groups.
Note also that the proof of Theorem \ref{th:modular} yields   not only an $\mathrm{SL}(2,\mathbb Z)$-action on the biinvariants of \eqref{eq:abslide}, but also an $\mathrm{SL}(2,\mathbb Z)$-action
on the invariants of the $H$-left module $(H^{\oo 2},\rhd_v)$ and on the coinvariants of the right $H$-comodule $(H^{\oo 2},\delta'_f)$ with $H$-right comodule structure $\delta'_f$ from \eqref{eq:rightdef}, whenever they are defined.
The restriction to the invariants, coinvariants or biinvariants is sufficient, but not  necessary  to obtain actions of the modular group $\mathrm{SL}(2,\mathbb Z)$.  This depends on  the Hopf monoid $H$. 

\begin{example} Let $\mac=\vect_\F$ and  $H\cong\mathbb F[G]$ for a group $G$ and field $\F$. Then $\F[G]\oo \F[G]\cong\F[G\times G]$, and the slides from \eqref{eq:abslide} are given by
\begin{align}
D_b: (a,b)\mapsto (ba, b)\qquad\qquad D_a:(a,b)\mapsto(a, ba^\inv),
\end{align}
 which yields
$$(D_bD_aD_b)^4(a,b)=([b,a^\inv] a[a^\inv,b], [b,a^\inv] b [a^\inv,b]).$$
This shows that  the modular group $\mathrm{SL}(2,\mathbb Z)$ acts  on the set
\begin{align*}
M=\mathrm{span}_\F\{(a,b)\in G\times G\mid [a,b]\in Z(G)\}
\end{align*}
If $G$ is nilpotent of nilpotence degree 2, then $[G,G]$ is central in $G$, and Theorem \ref{th:modular} defines an $\mathrm{SL}(2,\mathbb Z)$-action on $\mathbb F[G\times G]$.
\end{example}

\section{Slides and twists}
\label{sec:slidetwist}

In the previous section we gave  concrete expressions for  the action of the mapping class group of the torus and  one-holed torus  in terms of generating Dehn twists. In the remainder of this article, we  derive analogous descriptions for the mapping class groups of  surfaces  of genus $g\geq 1$ with $n\geq 0$ boundary components. 
 On one hand, this will extend Theorem \ref{the:edge1} and show that  edges slides associated with pivotal Hopf algebras define actions of the mapping class groups of  surfaces  with  more than one boundary component. 
It will also allow us to determine under which additional conditions they induce actions of mapping class groups of closed surfaces. Finally,   it is desirable, especially for geometric applications, 
to obtain concrete expressions for the action of the mapping class group in terms of generating Dehn twists and their relations.

For this,  we need to generalise the Dehn twists along the $a$- and $b$-cycles of the torus from Section \ref{sec:torus} to twists along more general closed paths in a graph. 
This requires a number of technical results on edge slides, which we derive in this section. The reader only interested in the results may skip these at first and proceed to Section \ref{sec:gervaismap}. 

To generalise the Dehn twists from Section \ref{sec:torus}, we proceed in three steps.
 In Section \ref{subsec:slide} we generalise the edge slides from Definition \ref{def:edge slide} to slides along face paths  (cf.~Definition \ref{def:face}).  In Section \ref{subsec:addedge}  we show how a slide along a face path can be described in terms of a slide along a single edge by adding and removing edges in the graph. In Section \ref{sec:twists} we then define Dehn twists along closed face paths  and investigate their interaction with edge slides. 
 In Section \ref{sec:gervaismap} we further generalise these Dehn twists to a number of closed paths that are not face paths. This yields a set of generating Dehn twists  that  satisfy the  relations   in Theorem \ref{th:gervais}.

Throughout this section, we assume that $H$ is a pivotal Hopf monoid in a symmetric monoidal category $\mac$ and $\Gamma$ a directed ribbon graph with edge set $E$.

\subsection{Slides along face paths}
\label{subsec:slide}

We start by generalising  the edge slides from Section \ref{sec:chordslideshopf}  to slides along face paths.
For this, note that the conditions on an edge slide after equation \eqref{eq:slideedge} allow one to slide an edge end $\beta$ along an entire face path $\gamma$ path via successive edge slides,  whenever it can slide along the first edge in  $\gamma$ and is not traversed by $\gamma$. This is illustrated in Figure \ref{fig:faceslide}.

\begin{definition} \label{def:slideface}Let
	$\gamma=\gamma_1^{\epsilon_1} \circ \ldots \circ\gamma_n^{\epsilon_n}:v\to w$ be a face path and denote by $\beta$ the edge  end directly after the starting end $\st(\gamma_{n}^{\epsilon_{n}})$ of $\gamma$ with respect to the cyclic ordering at $v$.
	If $\beta$ is not traversed by $\gamma$, 
the {\bf  slide} of $\beta$ along $\gamma$   is 
\begin{align}
	S_{\gamma}:= S_{(\gamma^{\epsilon_1}_1)^L} \cdots S_{(\gamma^{\epsilon_n}_{n})^L}\quad\text{with}\quad S_{(\gamma_i^\inv)^L}:=S_{\gamma_i^{-R}}.
	\nonumber
\end{align}
\end{definition}

Note that for a face path $\gamma=\gamma_1^{\epsilon_1}$ consisting of a single edge $\gamma_1\in E$, the slide along $\gamma$ reduces to the edge slides from Definition \ref{def:edge slide}, namely $S_{\gamma}=S_{\gamma_1^{L}}$  if $\epsilon_1=1$ and $S_{\gamma}=S_{\gamma_1^{-R}}$  if $\epsilon_1=-1$. 

Just as for slides along edges, there  is a  description of slides along face paths in terms of an $H$-comodule structure associated to face paths. From Definition \ref{def:slideface} and the expressions for edge slides in  \eqref{eq:slidedef2a} to \eqref{eq:slidedef4b}   one obtains the following alternative definition.

\begin{remark} \label{rem:otherslide}Associate to a face path $\gamma$ in $\Gamma$  an $H$-comodule structure $\delta_\gamma: H^{\oo E}\to H\oo H^{\oo E}$  as  in Definition  \ref{def:vertexface} and let $\beta$ be as in Definition \ref{def:slideface}. 
Then
 the slide along $\gamma$ is given by
$$
S_\gamma=\rhd_{\beta\pm}\circ\delta_\gamma
$$
where one takes $+$ if the sliding end of the edge $\beta$ is incoming and $-$ if it is outgoing.
\end{remark}

As slides along face paths are defined as composites of edge slides, we  obtain a direct generalisation of Proposition \ref{prop:vertfacecomp} 
for slides along face paths.

\begin{corollary}\label{cor:slidepath} Let  $\gamma$ be a face path satisfying the conditions in Definition \ref{def:slideface}. Then
	$S_{\gamma}$ is an isomorphism of  $H$-(co)modules 
	for any ciliated vertex $v$  (face $f$) whose cilium is not traversed by $\gamma$.
\end{corollary}

By considering  slides along  face paths, one obtains the left and right pentagon relation and the opposite end and adjacent commutativity relations as special cases of a  simpler and more general relation.  
For this, recall from the beginning of Section \ref{sec:chordslideshopf} that  edge slides induce isomorphisms between the path groupoids of directed  ribbon graphs. Because slides along face paths are composites of edge slides, this statement generalises to face paths. 

We now consider how a face path  $\rho$  in $\Gamma$   is transformed under the slide of an edge $\alpha$ along  a face path $\gamma$.
Then there are four possibilities for the relative positions of the  paths (cf.~Section \ref{subsec:paths}): 
\begin{compactenum}[(i)]
\item $\gamma$ and $\rho$ do not overlap,
\item $\gamma$ and $\rho$ overlap,  but neither is a subpath of the other, 
 \item  $\rho$ is a subpath of $\gamma$,
\item  $\gamma$ is a proper subpath of $\rho$.
 \end{compactenum}
 The transformation of $\rho$ under the slide $S_\gamma$  is depicted in Figure \ref{fig:FacePathTransformationCases}.
 In case (i), one has $S_\gamma(\rho)=\rho$,  unless $\rho$ traverses the edge $\alpha$. If $\rho$ traverses $\alpha$, then $S_\gamma(\rho)\neq \rho$, but $S_\gamma(\rho)$ is still a face path, as shown in Figure \ref{fig:FacePathTransformationCases} (i). In case (ii), $S_\gamma(\rho)$ is never a face path, as shown in Figure \ref{fig:FacePathTransformationCases} (ii), as face paths by definition  turn maximally left at every vertex.
 The same holds for case (iv) if $\gamma$ and $\rho$ have the same starting vertex. In case (iii)  $S_\gamma(\rho)$ is again a face path, as shown in Figure \ref{fig:FacePathTransformationCases} (iii). The same holds in case (iv) if $\gamma$ and $\rho$ do not start at the same vertex, as shown in Figure \ref{fig:FacePathTransformationCases} (iv).

\begin{figure}
\begin{tikzpicture}[scale=.5]
\begin{scope}[shift={(0,0)}]
\draw[line width=1.5pt, ->,>=stealth, color=blue] (0,2)--(0,.2);
\draw[line width=1.5pt, <-,>=stealth, color=blue] (.2,0)--(3.8,0);
\draw[line width=1.5pt, ->,>=stealth, color=blue] (4,0)--(4,1.8);
\draw[line width=1.5pt, ->,>=stealth, color=red] (2,2)--(.2,2);
\draw[line width=1pt, color=blue,->,>=stealth] plot [smooth, tension=0.6] coordinates 
      {(.5,1.5)(.5,.5)(3.5,.5)(3.5,1.5)};    
\draw[color=black, fill=black] (0,2) circle (.2);
\draw[color=black, fill=black] (0,0) circle (.2);
\draw[color=black, fill=black] (4,0) circle (.2);
\draw[color=black, fill=black] (4,2) circle (.2);
\draw[color=black, fill=black] (2,2) circle (.2);
\node at (2,.5)[color=blue, anchor=south]{$\gamma$};
\node at (0,1)[color=blue, anchor=east]{$\gamma_3$};
\node at (4,1)[color=blue, anchor=west]{$\gamma_1$};
\node at (2,0)[color=blue, anchor=north]{$\gamma_2$};
\end{scope}
\draw[line width=1pt, color=black, ->,>=stealth] (5.5,1)--(7.5,1);
\node at (6.5,1)[color=black, anchor=south] {$S_{\gamma_3^L}$};
\begin{scope}[shift={(9,0)}]
\draw[line width=1.5pt, ->,>=stealth, color=blue] (0,2)--(0,.2);
\draw[line width=1.5pt, <-,>=stealth, color=blue] (.2,0)--(3.8,0);
\draw[line width=1.5pt, ->,>=stealth, color=blue] (4,0)--(4,1.8);
\draw[line width=1.5pt, ->,>=stealth, color=red] (2,2)--(.2,.2);
\draw[color=black, fill=black] (0,2) circle (.2);
\draw[color=black, fill=black] (0,0) circle (.2);
\draw[color=black, fill=black] (4,0) circle (.2);
\draw[color=black, fill=black] (4,2) circle (.2);
\draw[color=black, fill=black] (2,2) circle (.2);
\node at (0,1)[color=blue, anchor=east]{$\gamma_3$};
\node at (4,1)[color=blue, anchor=west]{$\gamma_1$};
\node at (2,0)[color=blue, anchor=north]{$\gamma_2$};
\end{scope}
\draw[line width=1pt, color=black, ->,>=stealth] (14.5,1)--(16.5,1);
\node at (15.5,1)[color=black, anchor=south] {$S_{\gamma_2^{-R}}$};
\begin{scope}[shift={(18,0)}]
\draw[line width=1.5pt, ->,>=stealth, color=blue] (0,2)--(0,.2);
\draw[line width=1.5pt, <-,>=stealth, color=blue] (.2,0)--(3.8,0);
\draw[line width=1.5pt, ->,>=stealth, color=blue] (4,0)--(4,1.8);
\draw[line width=1.5pt, ->,>=stealth, color=red] (2,2)--(3.8,.2);
\draw[color=black, fill=black] (0,2) circle (.2);
\draw[color=black, fill=black] (0,0) circle (.2);
\draw[color=black, fill=black] (4,0) circle (.2);
\draw[color=black, fill=black] (4,2) circle (.2);
\draw[color=black, fill=black] (2,2) circle (.2);
\node at (0,1)[color=blue, anchor=east]{$\gamma_3$};
\node at (4,1)[color=blue, anchor=west]{$\gamma_1$};
\node at (2,0)[color=blue, anchor=north]{$\gamma_2$};
\end{scope}
\draw[line width=1pt, color=black, ->,>=stealth] (23.5,1)--(25.5,1);
\node at (24.5,1)[color=black, anchor=south] {$S_{\gamma_1^L}$};
\begin{scope}[shift={(27,0)}]
\draw[line width=1.5pt, ->,>=stealth, color=blue] (0,2)--(0,.2);
\draw[line width=1.5pt, <-,>=stealth, color=blue] (.2,0)--(3.8,0);
\draw[line width=1.5pt, ->,>=stealth, color=blue] (4,0)--(4,1.8);
\draw[line width=1.5pt, ->,>=stealth, color=red] (2,2)--(3.8,2);
\draw[color=black, fill=black] (0,2) circle (.2);
\draw[color=black, fill=black] (0,0) circle (.2);
\draw[color=black, fill=black] (4,0) circle (.2);
\draw[color=black, fill=black] (4,2) circle (.2);
\draw[color=black, fill=black] (2,2) circle (.2);
\node at (0,1)[color=blue, anchor=east]{$\gamma_3$};
\node at (4,1)[color=blue, anchor=west]{$\gamma_1$};
\node at (2,0)[color=blue, anchor=north]{$\gamma_2$};
\end{scope}
\end{tikzpicture}
\caption{Sliding an edge end along the face path $\gamma=\gamma_1\gamma_2^\inv\gamma_3$.}
\label{fig:faceslide}
\end{figure}
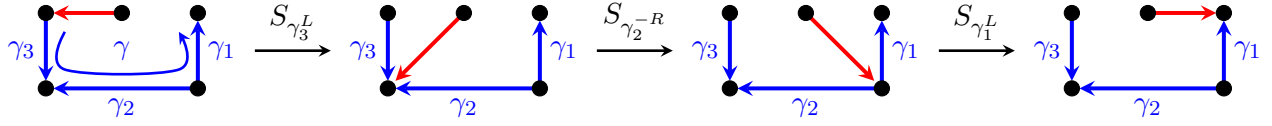

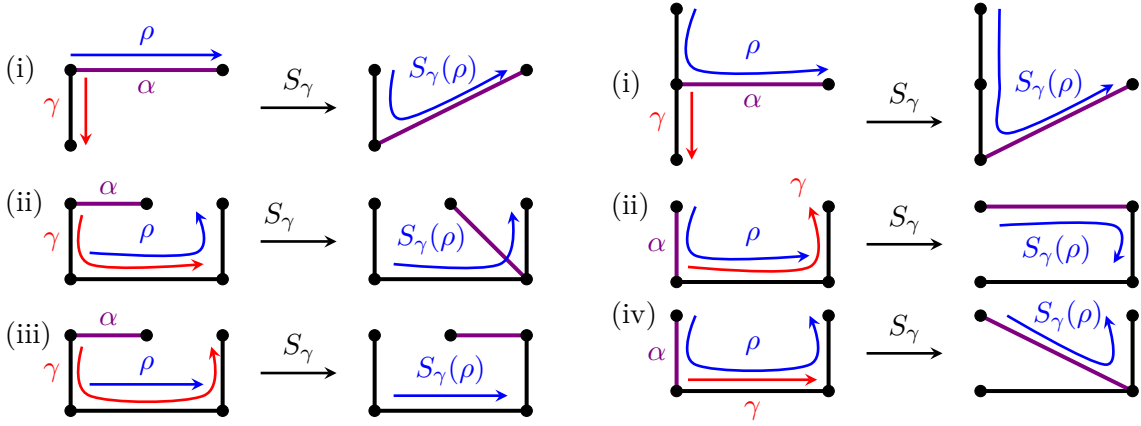
\begin{figure}
\begin{center}
\begin{tikzpicture}[scale=.5, baseline=(current bounding box.center)]
\begin{scope}[shift={(-4,0)}]
\node at (-4,2)[anchor=west]{(i)};
\draw[line width=1pt, color=violet, line width=1.5pt] (-2,2)--(2,2);
\draw[line width=1pt, color=black,  line width=1.5pt] (-2,0)--(-2,2);
\draw[fill=black, color=black,] (-2,0) circle (.15);
\draw[fill=black, color=black] (2,2) circle (.15);
\draw[fill=black, color=black] (-2,2) circle (.15);
\draw[line width=1pt, color=blue,->,>=stealth] plot [smooth, tension=0.6] coordinates 
      {(-2, 2.4)(2,2.4)};      
\draw[line width=1pt, color=red, ->,>=stealth] plot [smooth, tension=0.6] coordinates 
      {(-1.6,1.8)(-1.6,0)};
      \node at (-2,1)[anchor=east, color=red]{$\gamma$};
      \node at (0,2.4)[anchor=south, color=blue]{$\rho$}; 
       \node at (0,2)[anchor=north, color=violet]{$\alpha$};
       \end{scope}
       \draw[line width=1pt, ->,>=stealth](-1,1)--(1,1);
       \node at (0,1)[anchor=south]{$S_\gamma$};
       \begin{scope}[shift={(4,0)}]
\draw[line width=1pt, color=black,  line width=1.5pt] (-2,0)--(-2,2);
\draw[line width=1pt, color=violet,  line width=1.5pt] (-2,0)--(2,2);
\draw[fill=black, color=black,] (-2,0) circle (.15);
\draw[fill=black, color=black] (2,2) circle (.15);
\draw[fill=black, color=black] (-2,2) circle (.15);
\draw[line width=1pt, color=blue,->,>=stealth] plot [smooth, tension=0.6] coordinates 
      {(-1.5,2)(-1.3,.7)(1.5,2)};      
      \node at (1,2)[anchor=east, color=blue]{$S_\gamma(\rho)$}; 
       \end{scope}
\end{tikzpicture}
\qquad
\begin{tikzpicture}[scale=.5,baseline=(current bounding box.center)]
\begin{scope}[shift={(-4,0)}]
\node at (-4,2)[anchor=west]{(i)};
\draw[line width=1pt, color=violet, line width=1.5pt] (-2,2)--(2,2);
\draw[line width=1pt, color=black,  line width=1.5pt] (-2,0)--(-2,2);
\draw[line width=1pt, color=black,  line width=1.5pt] (-2,2)--(-2,4);
\draw[fill=black, color=black,] (-2,0) circle (.15);
\draw[fill=black, color=black] (2,2) circle (.15);
\draw[fill=black, color=black] (-2,4) circle (.15);
\draw[fill=black, color=black] (-2,2) circle (.15);
\draw[line width=1pt, color=red, ->,>=stealth] plot [smooth, tension=0.6] coordinates 
      {(-1.6,1.8)(-1.6,0)};
\draw[line width=1pt, color=blue,->,>=stealth] plot [smooth, tension=0.6] coordinates 
      {(-1.5, 4)(-1.5, 2.4)(2,2.4)};      
      \node at (-2,1)[anchor=east, color=red]{$\gamma$};
      \node at (0,2.4)[anchor=south, color=blue]{$\rho$}; 
       \node at (0,2)[anchor=north, color=violet]{$\alpha$};
       \end{scope}
       \draw[line width=1pt, ->,>=stealth](-1,1)--(1,1);
       \node at (0,1)[anchor=south]{$S_\gamma$};
       \begin{scope}[shift={(4,0)}]
\draw[line width=1pt, color=black,  line width=1.5pt] (-2,0)--(-2,2);
\draw[line width=1pt, color=violet,  line width=1.5pt] (-2,0)--(2,2);
\draw[line width=1pt, color=black,  line width=1.5pt] (-2,2)--(-2,4);
\draw[fill=black, color=black,] (-2,0) circle (.15);
\draw[fill=black, color=black] (2,2) circle (.15);
\draw[fill=black, color=black] (-2,4) circle (.15);
\draw[fill=black, color=black] (-2,2) circle (.15);
\draw[line width=1pt, color=blue,->,>=stealth] plot [smooth, tension=0.6] coordinates 
      {(-1.5,4)(-1.5,2)(-1.3,.7)(1.5,2)};      
      \node at (1,2)[anchor=east, color=blue]{$S_\gamma(\rho)$}; 
       \end{scope}
\end{tikzpicture}

\begin{tikzpicture}[scale=.5, baseline=(current bounding box.center)]
\begin{scope}[shift={(-4,0)}]
\node at (-4,2)[anchor=west]{(ii)};
\draw[line width=1pt, color=black,  line width=1.5pt] (-2,0)--(2,0);
\draw[line width=1pt, color=black, line width=1.5pt] (2,0)--(2,2);
\draw[line width=1pt, color=black,  line width=1.5pt] (-2,0)--(-2,2);
\draw[line width=1pt, color=violet,  line width=1.5pt] (0,2)--(-2,2);
\draw[fill=black, color=black,] (-2,0) circle (.15);
\draw[fill=black, color=black] (2,0) circle (.15);
\draw[fill=black, color=black] (2,2) circle (.15);
\draw[fill=black, color=black] (-2,2) circle (.15);
\draw[fill=black, color=black] (0,2) circle (.15);
\draw[line width=1pt, color=red, ->,>=stealth] plot [smooth, tension=0.6] coordinates 
      {(-1.7,1.7)(-1.5,.4)(1.5,.4)};
\draw[line width=1pt, color=blue,->,>=stealth] plot [smooth, tension=0.6] coordinates 
      {(-1.5,.7)(1.3,.7)(1.3,1.8)};      
      \node at (-2,1)[anchor=east, color=red]{$\gamma$};
      \node at (0,.7)[anchor=south, color=blue]{$\rho$}; 
       \node at (-1,2)[anchor=south, color=violet]{$\alpha$};
       \end{scope}
       \draw[line width=1pt, ->,>=stealth](-1,1)--(1,1);
       \node at (-.5,1)[anchor=south]{$S_\gamma$};
       \begin{scope}[shift={(4,0)}]
\draw[line width=1pt, color=black,  line width=1.5pt] (-2,0)--(2,0);
\draw[line width=1pt, color=black,  line width=1.5pt] (2,0)--(2,2);
\draw[line width=1pt, color=black,  line width=1.5pt] (-2,0)--(-2,2);
\draw[line width=1pt, color=violet,  line width=1.5pt] (0,2)--(2,0);
\draw[fill=black, color=black,] (-2,0) circle (.15);
\draw[fill=black, color=black] (2,0) circle (.15);
\draw[fill=black, color=black] (2,2) circle (.15);
\draw[fill=black, color=black] (-2,2) circle (.15);
\draw[fill=black, color=black] (0,2) circle (.15);
\draw[line width=1pt, color=blue,->,>=stealth] plot [smooth, tension=0.6] coordinates 
      {(-1.5,.4)(1.3,.4)(1.6,1.8)};      
            \node at (-.5,.4)[anchor=south, color=blue]{$S_\gamma(\rho)$}; 
      \end{scope}
\end{tikzpicture}
\qquad
\begin{tikzpicture}[scale=.5, baseline=(current bounding box.center)]
\begin{scope}[shift={(-4,0)}]
\node at (-4,2)[anchor=west]{(ii)};
\draw[line width=1pt, color=black,  line width=1.5pt] (-2,0)--(2,0);
\draw[line width=1pt, color=black, line width=1.5pt] (2,0)--(2,2);
\draw[line width=1pt, color=violet,  line width=1.5pt] (-2,0)--(-2,2);
\draw[fill=black, color=black,] (-2,0) circle (.15);
\draw[fill=black, color=black] (2,0) circle (.15);
\draw[fill=black, color=black] (2,2) circle (.15);
\draw[fill=black, color=black] (-2,2) circle (.15);
\draw[line width=1pt, color=red, ->,>=stealth] plot [smooth, tension=0.6] coordinates 
      {(-1.7,.4)(1.5,.4)(1.5,2)};
\draw[line width=1pt, color=blue,->,>=stealth] plot [smooth, tension=0.6] coordinates 
      {(-1.5,2)(-1.5,.7)(1.5,.7)};      
      \node at (1.7,2)[anchor=south east, color=red]{$\gamma$};
      \node at (0,.7)[anchor=south, color=blue]{$\rho$}; 
       \node at (-2,1)[anchor=east, color=violet]{$\alpha$};
       \end{scope}
       \draw[line width=1pt, ->,>=stealth](-1,1)--(1,1);
       \node at (0,1)[anchor=south]{$S_\gamma$};
       \begin{scope}[shift={(4,0)}]
\draw[line width=1pt, color=black,  line width=1.5pt] (-2,0)--(2,0);
\draw[line width=1pt, color=black, line width=1.5pt] (2,0)--(2,2);
\draw[line width=1pt, color=violet,  line width=1.5pt] (-2,2)--(2,2);
\draw[fill=black, color=black,] (-2,0) circle (.15);
\draw[fill=black, color=black] (2,0) circle (.15);
\draw[fill=black, color=black] (2,2) circle (.15);
\draw[fill=black, color=black] (-2,2) circle (.15);
\draw[line width=1pt, color=blue,->,>=stealth] plot [smooth, tension=0.6] coordinates 
      {(-1.5,1.5)(1.5,1.5)(1.5,.5)};      
      \node at (0,1.5)[anchor=north, color=blue]{$S_\gamma(\rho)$}; 
       \end{scope}
\end{tikzpicture}

\begin{tikzpicture}[scale=.5, baseline=(current bounding box.center)]
\begin{scope}[shift={(-4,0)}]
\node at (-4,2)[anchor=west]{(iii)};
\draw[line width=1pt, color=black,  line width=1.5pt] (-2,0)--(2,0);
\draw[line width=1pt, color=black, line width=1.5pt] (2,0)--(2,2);
\draw[line width=1pt, color=black,  line width=1.5pt] (-2,0)--(-2,2);
\draw[line width=1pt, color=violet,  line width=1.5pt] (0,2)--(-2,2);
\draw[fill=black, color=black,] (-2,0) circle (.15);
\draw[fill=black, color=black] (2,0) circle (.15);
\draw[fill=black, color=black] (2,2) circle (.15);
\draw[fill=black, color=black] (-2,2) circle (.15);
\draw[fill=black, color=black] (0,2) circle (.15);
\draw[line width=1pt, color=red, ->,>=stealth] plot [smooth, tension=0.6] coordinates 
      {(-1.7,1.7)(-1.5,.4)(1.5,.4)(1.7,1.7)};
\draw[line width=1pt, color=blue,->,>=stealth] plot [smooth, tension=0.6] coordinates 
      {(-1.5,.7)(1.5,.7)};      
      \node at (-2,1)[anchor=east, color=red]{$\gamma$};
      \node at (0,.7)[anchor=south, color=blue]{$\rho$}; 
       \node at (-1,2)[anchor=south, color=violet]{$\alpha$};
       \end{scope}
       \draw[line width=1pt, ->,>=stealth](-1,1)--(1,1);
       \node at (0,1)[anchor=south]{$S_\gamma$};
       \begin{scope}[shift={(4,0)}]
\draw[line width=1pt, color=black,  line width=1.5pt] (-2,0)--(2,0);
\draw[line width=1pt, color=black,  line width=1.5pt] (2,0)--(2,2);
\draw[line width=1pt, color=black,  line width=1.5pt] (-2,0)--(-2,2);
\draw[line width=1pt, color=violet,  line width=1.5pt] (0,2)--(2,2);
\draw[fill=black, color=black,] (-2,0) circle (.15);
\draw[fill=black, color=black] (2,0) circle (.15);
\draw[fill=black, color=black] (2,2) circle (.15);
\draw[fill=black, color=black] (-2,2) circle (.15);
\draw[fill=black, color=black] (0,2) circle (.15);
\draw[line width=1pt, color=blue,->,>=stealth] plot [smooth, tension=0.6] coordinates 
      {(-1.5,.4)(1.5,.4)};      
            \node at (0,.4)[anchor=south, color=blue]{$S_\gamma(\rho)$}; 
      \end{scope}
\end{tikzpicture}
\qquad
\begin{tikzpicture}[scale=.5, baseline=(current bounding box.center)]
\begin{scope}[shift={(-4,0)}]
\node at (-4,2)[anchor=west]{(iv)};
\draw[line width=1pt, color=black,  line width=1.5pt] (-2,0)--(2,0);
\draw[line width=1pt, color=black, line width=1.5pt] (2,0)--(2,2);
\draw[line width=1pt, color=violet,  line width=1.5pt] (-2,0)--(-2,2);
\draw[fill=black, color=black,] (-2,0) circle (.15);
\draw[fill=black, color=black] (2,0) circle (.15);
\draw[fill=black, color=black] (2,2) circle (.15);
\draw[fill=black, color=black] (-2,2) circle (.15);
\draw[line width=1pt, color=red, ->,>=stealth] plot [smooth, tension=0.6] coordinates 
      {(-1.7,.3)(1.7,.3)};
\draw[line width=1pt, color=blue,->,>=stealth] plot [smooth, tension=0.6] coordinates 
      {(-1.5,2)(-1.5,.7)(1.5,.7) (1.5,2)};      
      \node at (0,0)[anchor=north, color=red]{$\gamma$};
      \node at (0,.7)[anchor=south, color=blue]{$\rho$}; 
       \node at (-2,1)[anchor=east, color=violet]{$\alpha$};
       \end{scope}
       \draw[line width=1pt, ->,>=stealth](-1,1)--(1,1);
       \node at (0,1)[anchor=south]{$S_\gamma$};
       \begin{scope}[shift={(4,0)}]
\draw[line width=1pt, color=black,  line width=1.5pt] (-2,0)--(2,0);
\draw[line width=1pt, color=black, line width=1.5pt] (2,0)--(2,2);
\draw[line width=1pt, color=violet,  line width=1.5pt] (-2,2)--(2,0);
\draw[fill=black, color=black,] (-2,0) circle (.15);
\draw[fill=black, color=black] (2,0) circle (.15);
\draw[fill=black, color=black] (2,2) circle (.15);
\draw[fill=black, color=black] (-2,2) circle (.15);
\draw[line width=1pt, color=blue,->,>=stealth] plot [smooth, tension=0.6] coordinates 
      {(-1.3,2)(1.3,.7)(1.3,2)};      
      \node at (1.5,2)[anchor=east, color=blue]{$S_\gamma(\rho)$}; 
       \end{scope}
\end{tikzpicture}
\end{center}
\caption{Transformation of a face path $\rho$ under a slide along a face path $\gamma$.}
\label{fig:FacePathTransformationCases}
\end{figure}

\begin{lemma}[Commutativity of slides]
	\label{lemma:SlideCommutativityRelation}
	Let $\gamma$ and $\rho$  be  face paths such that $S_{\gamma}(\rho)$ and $S_{\rho}(\gamma)$ also are face paths and the slides along  them are defined. Then: 
	\begin{align}
		S_{S_{\rho}(\gamma)} S_{\rho} = S_{S_{\gamma}(\rho)} S_{\gamma}
		\label{eq:SlideCommutativityRelation}
	\end{align}
\end{lemma}
\begin{proof}
 We first prove the claim for paths of the form $\gamma=\gamma_{1}^{\nu_1}, \rho=\rho_{1}^{\epsilon_1}$ with  $\gamma_{1},\rho_{1} \in E$ and $\nu_1,\epsilon_1\in \{\pm 1\}$. 
	Then there are three cases:
(a) The paths  are equal: $\gamma = \rho$, (b)
 Either $S_{\gamma}$ or $S_{\rho}$ slides the edge traversed by the other path, 
(c)~$S_{\gamma}$ and $S_{\rho}$ slide different edges $\alpha \neq \beta$ or different ends of the same edge $\alpha =\beta$, and $\alpha$ and $\beta$ are distinct from $\rho_{1}$ and $\gamma_{1}$.

	In case (a) and case (c) we have $S_{\gamma}(\rho) =\rho$ and $S_{\rho}(\gamma)=\gamma$. In case (a) the claim is trivial, and in case (c) it follows directly from the adjacent commutativity, opposite end commutativity in Corollary \ref{cor:adjacentendslide} or from the commutativity relation in Proposition \ref{prop: relationsbene}. In case (b) we can suppose without loss of generality that
$S_\gamma=S_{\gamma_1^L}$ slides the target end of  $\rho_{1}$.   As we suppose that $S_\gamma(\rho)$ is a face path, this can happen only
in case (i) in Figure \ref{fig:FacePathTransformationCases}, and we have  $\rho=\rho_1^\inv$,  $S_\gamma(\rho)=\rho_1^\inv\gamma_1$, $S_\rho=S_{\rho_1^{-R}}$ and  $S_{\rho}(\gamma) = \gamma$. (Note that the slide along $S_\gamma(\rho)$ is not defined if the starting vertex of $\gamma$ is bivalent, since in this case $S_\gamma(\rho)$ starts at a univalent vertex).  We now consider the right pentagon relation from Figure~\ref{fig:sliderel5},  with $\alpha=\gamma_1^\inv$ and $\gamma=\rho_1$,  read against the orientation of its arrows from the top left picture to the middle picture on the right.  This yields 
	\begin{align*}
		S_{S_{\gamma}(\rho)} S_{\gamma} = S_{\rho_{1}^{-R}} S_{\gamma_{1}^{L}} S_{\gamma_{1}^{L}} =  S_{\gamma_{1}^{L}} S_{\rho_{1}^{-R}} = S_{S_{\rho}(\gamma)} S_{\rho}.
	\end{align*}
	This proves the claim for face paths $\gamma$ and $\rho$ consisting of one edge. For general  paths $\gamma = \gamma_{1}^{\nu_{1}} \cdots \gamma_{n}^{\nu_{n}}$ and $\rho= \rho_{1}^{\epsilon_{1}}\cdots \rho_{m}^{\epsilon_{m}}$ that satisfy the assumptions, it then follows by induction over $m$ and $n$.
\end{proof}

 Lemma \ref{lemma:SlideCommutativityRelation} allows one to generalise the pentagon equations  from an equation involving three edges to an identity involving three composable  face paths, which replace the three edges in Figure~\ref{fig:sliderel5}. The only condition  is that the  last edge of the first face path is not traversed by the second. 

\begin{corollary}[Generalised pentagon equation]
	\label{corollary:GeneralizedPentagon}
	Let $\gamma=\gamma_1\circ\gamma_2\circ\gamma_3$ be a face path with face paths $\gamma_1,\gamma_2$  and a non-trivial face path $\gamma_{3}$  such that the last edge of $\gamma_{3}$ is not traversed by $\gamma_{2}$. Then $S_{\gamma_2}(\gamma)=\gamma_1\circ \gamma_3$ and
	\begin{align}
		S_{\gamma_{2}} S_{\gamma} = S_{S_{\gamma_{2}}(\gamma)} S_{\gamma_{2}}.
		\label{eq:GeneralizedPentagon}
	\end{align}
\end{corollary}

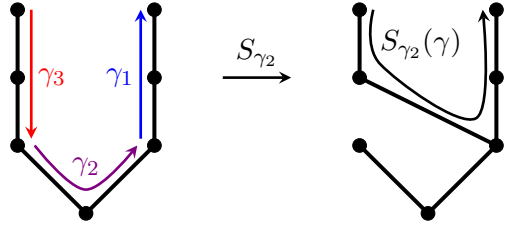
\begin{figure}
\centering
\begin{center}
\begin{tikzpicture}[scale=.45]
\begin{scope}[shift={(-5,0)}]
\draw[color=black, fill=black] (-2,4) circle (.2);
\draw[color=black, fill=black] (-2,2) circle (.2);
\draw[color=black, fill=black] (-2,0) circle (.2);
\draw[color=black, fill=black] (0,-2) circle (.2);
\draw[color=black, fill=black] (2,0) circle (.2);
\draw[color=black, fill=black] (2,2) circle (.2);
\draw[color=black, fill=black] (2,4) circle (.2);
\draw[line width=1.5pt, color=black] (-2,4)--(-2,2);
\draw[line width=1.5pt, color=black] (-2,0)--(-2,2);
\draw[line width=1.5pt, color=black] (0,-2)--(-2,0);
\draw[line width=1.5pt, color=black] (2,4)--(2,2);
\draw[line width=1.5pt, color=black] (2,0)--(2,2);
\draw[line width=1.5pt, color=black] (0,-2)--(2,0);
\draw[line width=1pt, color=red,->,>=stealth] plot [smooth, tension=0.6] coordinates 
      {(-1.6,4)(-1.6,.2)};
\draw[line width=1pt, color=violet,->,>=stealth] plot [smooth, tension=0.6] coordinates 
      {(-1.5,0)(0,-1.3)(1.5,0)};     
   \draw[line width=1pt, color=blue,->,>=stealth] plot [smooth, tension=0.6] coordinates 
      {(1.6,.2)(1.6,4)};   
      \node at (0,-1.2)[anchor=south, color=violet]{$\gamma_2$};
      \node at (-1.7,2)[anchor=west, color=red]{$\gamma_3$};      
            \node at (1.7,2)[anchor=east, color=blue]{$\gamma_1$}; 
            \end{scope}
            \draw[->,>=stealth, line width=1pt] (-1,2)--(1,2); 
            \node at (0,2)[anchor=south]{$S_{\gamma_2}$};
            \begin{scope}[shift={(5,0)}]
\draw[color=black, fill=black] (-2,4) circle (.2);
\draw[color=black, fill=black] (-2,2) circle (.2);
\draw[color=black, fill=black] (-2,0) circle (.2);
\draw[color=black, fill=black] (0,-2) circle (.2);
\draw[color=black, fill=black] (2,0) circle (.2);
\draw[color=black, fill=black] (2,2) circle (.2);
\draw[color=black, fill=black] (2,4) circle (.2);
\draw[line width=1.5pt, color=black] (-2,4)--(-2,2);
\draw[line width=1.5pt, color=black] (2,0)--(-2,2);
\draw[line width=1.5pt, color=black] (0,-2)--(-2,0);
\draw[line width=1.5pt, color=black] (2,4)--(2,2);
\draw[line width=1.5pt, color=black] (2,0)--(2,2);
\draw[line width=1.5pt, color=black] (0,-2)--(2,0);
\draw[line width=1pt, color=black,->,>=stealth] plot [smooth, tension=0.6] coordinates 
      {(-1.6,4)(-1.4,2.4)(1.4,.8) (1.6,4)};    
      \node at (-1.7,3)[anchor=west, color=black]{$S_{\gamma_2}(\gamma)$};      
            \end{scope}          
\end{tikzpicture}
\end{center}

\caption{The generalised pentagon relation from Corollary \ref{corollary:GeneralizedPentagon}.}
\label{fig:genpent}
\end{figure}

\subsection{Adding edges} 
\label{subsec:addedge}

In this section we introduce  three additional  transformations  of directed ribbon graphs that will be used to simplify slides along face paths. The first  is the
graph transformation $\epsilon_\alpha: \Gamma\to \Gamma'$ that deletes an edge $\alpha$ from $\Gamma$. If  $\alpha$ is incident at a univalent vertex $v$, then $v$ is deleted as well.

 The second is the  graph transformation $\eta_\alpha: \Gamma\to\Gamma'$ that adds an oriented edge $\alpha$ at a vertex $v$ of $\Gamma$. This edge $\alpha$ may either have a univalent vertex $v'$ at its other end or may be a loop based at $v$. If $\alpha$ has a univalent vertex at its other end, we orient $\alpha$ towards the univalent vertex. If 
  $\alpha$ is a loop based at $v$, we require
 that the target end  of $\alpha$ is directly before the starting end of $\alpha$ with respect to the cyclic ordering at $v$ and with respect to the linear ordering at $v$, if $v$ is ciliated.

 The third graph transformation is associated with a  face path $\gamma$ in $\Gamma$. It first adds a loop $\gamma'$ at the starting vertex of $\gamma$, such that the target end  of $\gamma'$ is directly after starting end of $\gamma$ with respect to the cyclic ordering. It then slides the target end of $\gamma'$ along $\gamma$, as shown in Figure \ref{fig:addingedge}. We denote it by $C_\gamma=S_\gamma\circ \eta_{\gamma'}: \Gamma\to \Gamma'$  and call it  {\em adding an edge} to  $\gamma$.

\begin{figure}
\centering
\begin{tikzpicture}[scale=.45]
\begin{scope}[shift={(-6,0)}]
\draw[fill=black] (-2,2) circle (.2);
\draw[fill=black] (0,-2) circle (.2);
\draw[fill=black] (-2,0) circle (.2);
\draw[fill=black] (2,0) circle (.2);
\draw[fill=black] (2,2) circle (.2);
\draw[line width=1pt, color=blue,->,>=stealth] plot [smooth, tension=0.6] coordinates 
      {(-1.5,0) (0,-1.5)(1.5,0)};         
\draw[line width=1.5pt, color=black] (-2,0)--(-2,2);
\draw[line width=1.5pt, color=black] (-2,0)--(0,-2);
\draw[line width=1.5pt, color=black,] (0,-2)--(2,0);
\draw[line width=1.5pt, color=black] (2,0)--(2,2);
\draw[line width=1pt, ->,>=stealth] (3,1)--(4,1);
\node at (3.5,1)[anchor=south]{$\eta_{\gamma'}$};
\node at (0,-1.5)[anchor=south, color=blue]{$\gamma$};
\end{scope}
\begin{scope}[shift={(1,0)}]
\draw[fill=black] (-2,2) circle (.2);
\draw[fill=black] (0,-2) circle (.2);
\draw[fill=black] (-2,0) circle (.2);
\draw[fill=black] (2,0) circle (.2);
\draw[fill=black] (2,2) circle (.2);
\draw[line width=1pt, color=blue,->,>=stealth] plot [smooth, tension=0.6] coordinates 
      {(-1.1,-.4) (0,-1.5)(1.5,0)};      
\draw[line width=1.5pt, color=red,->,>=stealth]  (-1.8,.2)..controls (0,1) and (0,-1)..(-1.8,0);  
\draw[line width=1.5pt, color=black] (-2,0)--(-2,2);
\draw[line width=1.5pt, color=black] (-2,0)--(0,-2);
\draw[line width=1.5pt, color=black,] (0,-2)--(2,0);
\draw[line width=1.5pt, color=black] (2,0)--(2,2);
\draw[line width=1pt, ->,>=stealth] (3,1)--(4,1);
\node at (3.5,1)[anchor=south]{$S_\gamma$};
\node at (0,-1.5)[anchor=south, color=blue]{$\gamma$};
\node at (-.5,0)[anchor=west, color=red]{$\gamma'$};
\end{scope}
\begin{scope}[shift={(8,0)}]
\draw[fill=black] (-2,2) circle (.2);
\draw[fill=black] (0,-2) circle (.2);
\draw[fill=black] (-2,0) circle (.2);
\draw[fill=black] (2,0) circle (.2);
\draw[fill=black] (2,2) circle (.2);
\draw[line width=1.5pt, color=red,->,>=stealth]  (-1.8,0)--(1.8,0);     
\draw[line width=1.5pt, color=black] (-2,0)--(-2,2);
\draw[line width=1.5pt, color=black] (-2,0)--(0,-2);
\draw[line width=1.5pt, color=black,] (0,-2)--(2,0);
\draw[line width=1.5pt, color=black] (2,0)--(2,2);
\node at (0,0)[anchor=south, color=red]{$\gamma'$};
\end{scope}
\end{tikzpicture}

\vspace{1cm}
\begin{tikzpicture}[scale=.45]
\begin{scope}[shift={(-6,0)}]
\draw[fill=black] (-2,2) circle (.2);
\draw[fill=black] (0,-2) circle (.2);
\draw[fill=black] (-2,0) circle (.2);
\draw[fill=black] (2,0) circle (.2);
\draw[fill=black] (2,2) circle (.2);
\draw[line width=1pt, color=violet,->,>=stealth] plot [smooth, tension=0.6] coordinates 
      {(-1.5,-.2) (0,-1.7)};      
      \draw[line width=1pt, color=black,->,>=stealth] plot [smooth, tension=0.6] coordinates 
      {(-1.3,2)(-1.1,0) (0,-1)(1.1,0)(1.3,2)};         
\draw[line width=1.5pt, color=black] (-2,0)--(-2,2);
\draw[line width=1.5pt, color=black] (-2,0)--(0,-2);
\draw[line width=1.5pt, color=black,] (0,-2)--(2,0);
\draw[line width=1.5pt, color=black] (2,0)--(2,2);
\draw[line width=1pt, ->,>=stealth] (3,1)--(4,1);
\node at (3.5,1)[anchor=south]{$C_{\gamma}$};
\node at (-1.3,-1.3)[anchor=east, color=violet]{$\sigma$};
\node at (1.3,2)[anchor=east, color=black]{$\rho$};
\end{scope}
\begin{scope}[shift={(1,0)}]
\draw[fill=black] (-2,2) circle (.2);
\draw[fill=black] (0,-2) circle (.2);
\draw[fill=black] (-2,0) circle (.2);
\draw[fill=black] (2,0) circle (.2);
\draw[fill=black] (2,2) circle (.2);
\draw[line width=1pt, color=violet,->,>=stealth] plot [smooth, tension=0.6] coordinates 
      {(-1.5,-.2) (0,-1.7)};   
      \draw[line width=1pt, color=black,->,>=stealth] plot [smooth, tension=0.6] coordinates 
      {(-1.5,2)(-1.5,.5)(1.5,.5)(1.5,2) };      
\draw[line width=1.5pt, color=red,->,>=stealth]  (-1.8,0)--(1.8,0);     
\draw[line width=1.5pt, color=black] (-2,0)--(-2,2);
\draw[line width=1.5pt, color=black] (-2,0)--(0,-2);
\draw[line width=1.5pt, color=black,] (0,-2)--(2,0);
\draw[line width=1.5pt, color=black] (2,0)--(2,2);
\node at (0,.5)[anchor=south, color=black]{$C_\gamma(\rho)$};
\node at (-1.3,-1.3)[anchor=east, color=violet]{$C_\gamma(\sigma)$};
\end{scope}
\end{tikzpicture}

\caption{Adding an edge $\gamma'$ to a face path $\gamma$ and the associated transformation of face paths.}
\label{fig:addingedge}
\end{figure}
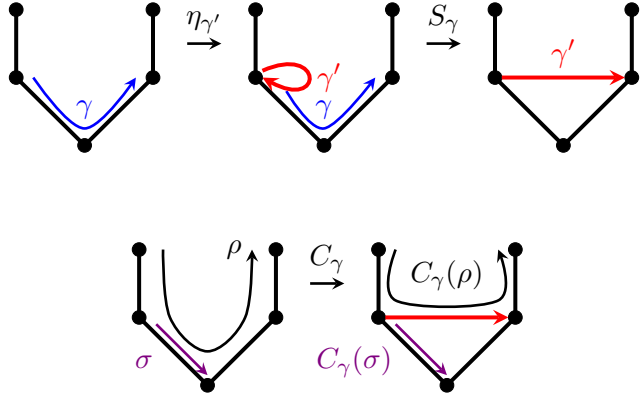

If $\Gamma'$ is obtained from $\Gamma$ by adding an edge $\gamma'$ to a face path $\gamma$, then  $\gamma'^\inv\circ\gamma$ is a ciliated face of $\Gamma'$.  More generally, every face path $\rho$ of $\Gamma$  that contains $\gamma$ as a subpath corresponds to a face path $C_\gamma(\rho)$ in $\Gamma'$ that is obtained  by replacing $\gamma$ by $\gamma'$ in the expression for $\rho$, as in Figure \ref{fig:addingedge}. Face paths of $\Gamma$ that are  proper subpaths of $\gamma$ also correspond canonically to face paths in $\Gamma'$, and the same holds for paths in $\Gamma$ that do not overlap with $\gamma$.

\begin{definition}
	\label{definition:FacePathTranslationComplementary}
	Let  $\gamma$, $\rho$  be face paths of $\Gamma$ such that either (i) $\rho$ is a proper subpath of $\gamma$,  (ii) $\rho$ and $\gamma$ do not overlap or (iii) $\gamma$ is a subpath   of $\rho$. Let $\Gamma'$ be the directed ribbon graph obtained by adding an edge $\gamma'$ to $\gamma$.
The path $C_{\gamma}(\rho)$ in $\Gamma'$ is defined as $C_\gamma(\rho)=\rho$ in cases (i),(ii)  and as  $C_\gamma(\rho)=\rho_1\circ\gamma'\circ\rho_2$ if $\rho=\rho_1\circ\gamma\circ\rho_2$ with possibly trivial face paths $\rho_1,\rho_2$.
\end{definition}

We now associate to each of these graph transformations a morphism in $\mac$. These morphisms are denoted by the same letters as the graph transformations and given as follows.

\begin{definition} \label{def:addingedge}Let $\Gamma$ be a directed ribbon graph and $H$ a pivotal Hopf monoid in $\mac$.
\begin{compactenum}
\item  {\bf Removing an edge} $\alpha\in E$ from $\Gamma$ corresponds to  the morphism
$$
\epsilon_\alpha: H^{\oo E}\to H^{\oo E\setminus\{\alpha\}},
$$
that is the counit on the copy of $H$ for $\alpha$ and  the identity morphism on all other components.

\item {\bf Adding an edge} $\alpha$ to $\Gamma$ corresponds to the morphism
$$\eta_\alpha: H^{\oo E}\to H^{\oo E\cup\{\alpha\}},$$
that is the unit on the copy of $H$ for $\alpha$ and the identity morphism on all other components.

\item {\bf Adding an edge $\gamma'$ to a face path} $\gamma$ in $\Gamma$ corresponds to the morphism
$$
C_\gamma=S_\gamma\circ \eta_{\gamma'}: H^{\oo E}\to H^{\oo E\cup\{\gamma'\}}.
$$
\end{compactenum}
\end{definition}

Note that adding an edge to a face path involves two choices. The first is the orientation of the added loop $\gamma'$, which can be 
reversed by applying the morphism $
T_{\gamma'}$. The second choice is to add the loop at the starting vertex of $\gamma$ and slide its target end to the target of $\gamma$. With expressions \eqref{eq:slidedef2a} to \eqref{eq:slidedef4b}  for the edge slides one finds  that 
adding the loop at the target vertex of $\gamma$ and sliding its starting end to the starting vertex of $\gamma$ yields the same morphism $C_\gamma$.
Equivalently, adding a loop at the starting vertex of a face path $\gamma$  and then sliding it to the target vertex gives the same morphism as adding the loop at the target vertex of $\gamma$.

The advantage of adding edges to face paths is that we can reduce slides along face paths to slides along edges. This is achieved by adding an edge to the face path, sliding along the added edge  and then deleting this edge. To show that this yields the same result as the slide along the  face path,  we  determine the properties of the associated morphism $C_\gamma$.

\begin{lemma}\label{lem:addingedge}  Let $\gamma,\rho$ be face paths  and $\alpha$ an edge in a directed ribbon graph $\Gamma$.
Then the morphism $C_\gamma: H^{\oo E}\to H^{\oo E\cup\{\gamma'\}}$  from Definition \ref{def:addingedge} has the following properties:
\begin{compactenum}
\item Removing the edge $\gamma'$ is left inverse to adding $\gamma'$: 
\begin{align}\label{eq:ComplementDelete}
\epsilon_{\gamma'}\circ C_\gamma=1_{H^{\oo E}}.
\end{align}
\item Removing $\alpha$ is left inverse to adding  an edge $\alpha'$ to $\alpha$ (to  $\alpha^\inv$ and  reversing orientation):
\begin{align}\label{eq:EdgeComplementInvolution}
\epsilon_\alpha\circ C_{\alpha}=1_{H^{\oo E}},\qquad \epsilon_\alpha\circ C_{\alpha^\inv}=T_{\alpha}.
\end{align}
\item Adding an edge $\gamma'$ to $\gamma$ is a morphism of $H$-modules and $H$-comodules for all $H$-(co)module structures at ciliated vertices $v$ (faces $f$) of $\Gamma$ whose cilium is not traversed by $\gamma$:
\begin{align}\label{eq:ComplementVertexFaceOperatorCompatibility}
\rhd_v\circ C_\gamma=C_\gamma\circ \rhd_v\qquad  \delta_{C_\rho(f)}\circ C_\gamma=C_\gamma\circ\delta_f.
\end{align}
\item
	If $\Gamma'$ is  obtained from $\Gamma$ by adding an edge $\gamma'$ to $\gamma$, then 
	$C_{\gamma}$ equalises the trivial comodule structure and 
	the $H$-right comodule structure \eqref{eq:rightdef} of any ciliated face $f$ of $\Gamma'$ that is a cyclic permutation of $\gamma'^\inv\circ\gamma$:
\begin{align}\label{eq:FlatFaceComplement}
\delta'_{f}\circ C_{\gamma}=\eta\oo C_\gamma.
\end{align}
\item Adding  edges commutes with slides along face paths: if $\rho,\gamma$ satisfy (i),(ii) or (iii) in Definition \ref{definition:FacePathTranslationComplementary} and $S_{\gamma}(\rho)$ is a face path, then
\begin{align}\label{eq:ComplementSlideGeneral}
S_{C_\rho(\gamma)}\circ C_\rho=C_{S_\gamma(\rho)}\circ S_\gamma.
\end{align}
In particular, whenever the slide along $\gamma$ is defined, one has
\begin{align}\label{eq:slidenotmatter}
S_{\gamma}= \epsilon_{\gamma'}\circ S_{\gamma'^{ L}}\circ C_\gamma.
\end{align}
\item Adding edges  is commutative: if $\rho,\gamma$  satisfy  (i), (ii) or (iii) in Definition \ref{definition:FacePathTranslationComplementary}, 
then
\begin{align}\label{eq:ComplementCommutativity}
C_{C_\gamma(\rho)}\circ C_\gamma=C_{C_\rho(\gamma)}\circ C_\rho.
\end{align}
\item If $f=\alpha\circ \rho$ is a ciliated face and $\rho$ does not traverse $\alpha$, then 
\begin{align}\label{eq:FlatFaceComplementt}
	T_{\rho'}\circ C_\rho\circ \epsilon_\alpha\circ\iota=\iota,\qquad\qquad  S_\rho\circ \epsilon_\alpha \circ\iota=\epsilon_\alpha\circ S_{\alpha^{-R}} \circ\iota,
\end{align}
where $\iota$ is the equaliser   for the   $H$-right comodule structure $\delta'_f$ from \eqref{eq:rightdef}  and  we identify the edges $\rho'$ and $\alpha$.
\end{compactenum}
\end{lemma}

\begin{proof}
1.~and 2.~The first and the second claim follow directly from the definition of $C_\gamma$, the expressions for the edge slides in  \eqref{eq:slidedef2a} to \eqref{eq:slidedef4b}    and from  the definition of a Hopf monoid in $\mac$.

3.~That $C_\gamma$ is a morphism of $H$-modules and $H$-comodules for the cilia that are  not traversed by $\gamma$  follows directly from the fact that this holds for the slide $S_\gamma$ by Corollary \ref{cor:slidepath} and it holds for adding a loop. The latter follows directly from Definition \ref{def:addingedge}  and from Definition \ref{def:vertexface} of the $H$-(co)module structures, together with the defining properties of a Hopf monoid.

4.~To prove \eqref{eq:FlatFaceComplement}, we first consider the ciliated face $f=\gamma'^\inv\circ\gamma$.  By Remark \ref{rem:otherslide}, we can describe the slide along the face path $\gamma$ in terms of a comodule structure associated with $\gamma$.  As adding a loop corresponds to the unit of $H$, the claim
 then follows  from the diagrammatic computation 
\begin{align}
\begin{tikzpicture}[scale=.3]
\begin{scope}[shift={(-8,0)}]
\draw[line width=1.5pt, color=blue] (0,3)--(0,-4);
\draw[line width=1pt, color=black] (0,2).. controls (-3,1.5) .. (-3,1);
\draw[line width=1pt, color=black] plot [smooth, tension=0.6] coordinates 
      {(-1,-4)(-4,0)  (-2,1) (-1,-1)(0,1)};   
      \draw[line width=1pt, color=black] plot [smooth, tension=0.6] coordinates 
      {(-1,-1) (-1.2,-2)(-2,-4)};   
      \draw[color=black, fill=white,  line width=1pt] (-1.5,0) circle (.4);
            \draw[color=black, fill=white,  line width=1pt] (-1.5,0) circle (.2);
            \end{scope}
            \node at (-6.5,-.5){$=$};
            \begin{scope}[shift={(0,0)}]
\draw[line width=1.5pt, color=blue] (0,3)--(0,-4);
\draw[line width=1pt, color=black] (0,2).. controls (-2.5,1.5) .. (-2.5,.4);
\draw[line width=1pt, color=black] (-4,-1).. controls (-4,1) and (-1,1).. (-1,-2);
\draw[line width=1pt, color=black] (-5,-2).. controls (-5,-.6) and (-3,-.6).. (-3,-2);
\draw[line width=1pt, color=black] (-3,-2).. controls (-3,-3) and (-1,-3).. (-1,-2);
     \draw[line width=1pt, color=black] plot [smooth, tension=0.6] coordinates 
      {(-1,-4)(-4,-3) (-5,-2)};   
           \draw[line width=1pt, color=black] plot [smooth, tension=0.6] coordinates 
      {(-2,-4)(-2,-2.7)};   
      \draw[color=black, fill=white,  line width=1pt] (-3,-2) circle (.4);
            \draw[color=black, fill=white,  line width=1pt] (-3,-2) circle (.2);
            \end{scope}
            \node at (1.5,-.5){$=$};
                        \begin{scope}[shift={(7,0)}]
\draw[line width=1.5pt, color=blue] (0,3)--(0,-4);
\draw[line width=1pt, color=black] (0,2).. controls (-3,1.5) .. (-3,.4);
\draw[line width=1pt, color=black] (-4,-1).. controls (-4,1) and (-2,1).. (-2,-1);
     \draw[line width=1pt, color=black] plot [smooth cycle, tension=0.6] coordinates 
      {(-2,-1)(-1,-2) (-2,-3)(-3,-2)};   
     \draw[line width=1pt, color=black] plot [smooth, tension=0.6] coordinates 
      {(-4,-1)(-3.5,-3) (-1,-4)};   
           \draw[line width=1pt, color=black] (-2,-3)--(-2,-4);
      \draw[color=black, fill=white,  line width=1pt] (-3,-2) circle (.4);
            \draw[color=black, fill=white,  line width=1pt] (-3,-2) circle (.2);
            \end{scope}
            \node at (8.5,-.5){$=$};
                        \begin{scope}[shift={(15,0)}]
\draw[line width=1.5pt, color=blue] (0,3)--(0,-4);
\draw[line width=1pt, color=black] (0,2).. controls (-3,1.5) .. (-3,.4);
\draw[line width=1pt, color=black] (-4,-1).. controls (-4,1) and (-2,1).. (-2,-1);
            \draw[color=black, fill=white,  line width=1pt] (-2,-1) circle (.2); 
     \draw[line width=1pt, color=black] plot [smooth, tension=0.6] coordinates 
      {(-4,-1)(-3.5,-3) (-1,-4)};   
           \draw[line width=1pt, color=black] (-2,-3)--(-2,-4);
            \draw[color=black, fill=black,  line width=1pt] (-2,-3) circle (.2);
            \node at (-1.8,-3)[anchor=west]{$p$};
            \end{scope}
            \node at (17.5,-.5){$=$};
                                    \begin{scope}[shift={(22,0)}]
\draw[line width=1.5pt, color=blue] (0,3)--(0,-4);
\draw[line width=1pt, color=black] (0,2).. controls (-1,1) .. (-1,-4);
           \draw[line width=1pt, color=black] (-2,-1)--(-2,-4);
            \draw[color=black, fill=black,  line width=1pt] (-2,-1) circle (.2);
            \node at (-2.2,-1)[anchor=east]{$p$};
            \end{scope}
\end{tikzpicture}
\tComma
\end{align}
in which the thick vertical line stands for the comodule structure associated with $\gamma$ and the first diagram on the left is the morphism 
$\delta_{f}\circ C_\gamma$. The claim for ciliated faces that are cyclic permutations of $\gamma'^\inv\circ\gamma$ then follows from
the diagrammatic computation
\begin{align*}
\begin{tikzpicture}[scale=.35]
\begin{scope}[shift={(0,0)}]
\draw[line width=1.5pt, color=blue] (0,3)--(0,-3);
\draw[line width=1.5pt, color=red] (1,3)--(1,-3);
\draw[line width=1.5pt, color=violet] (2,3)--(2,-3);
\draw[line width=1pt, color=black] (0,0)--(-3,-1);
\draw[line width=1pt, color=black] (1,2)--(-3,-1);
\draw[line width=1pt, color=black] (2,1.5)--(-3,-1);
\draw[line width=1pt, color=black] (-3,-1)--(-3,-3);
\draw[color=black, fill=white, line width=1pt] (-.5,3) rectangle (2.5,4);
\node at (1,3.5){$\iota$};
\draw[fill=white, line width=1pt] (-3,-2) circle (.3);
\draw[fill=white, line width=1pt] (-3,-2) circle (.14);
\end{scope}
\node at (3.5,0){$=$};
\begin{scope}[shift={(8,0)}]
\draw[line width=1.5pt, color=blue] (0,3)--(0,-3);
\draw[line width=1.5pt, color=red] (1,3)--(1,-3);
\draw[line width=1.5pt, color=violet] (2,3)--(2,-3);
\draw[line width=1pt, color=black] plot [smooth , tension=0.6] coordinates 
         {(0, 2.5)(-2.5,2)};
         \draw[line width=1pt, color=black] plot [smooth cycle, tension=0.6] coordinates 
         {(-2.5, 2)(-3.5,1.5)(-1.5,1)(-2.5,.5)(-3.5,1)(-1.5,1.5)};
         \draw[line width=1pt, color=black] (-2.5,.5)--(-3,-1);
         \draw[color=black, fill=gray, line width=1pt] (-1.5,1.5) circle (.3);
\draw[line width=1pt, color=black] (0,0)--(-3,-1);
\draw[line width=1pt, color=black] (1,2)--(-3,-1);
\draw[line width=1pt, color=black] (2,1.5)--(-3,-1);
\draw[line width=1pt, color=black] (-3,-1)--(-3,-3);
\draw[fill=white, line width=1pt] (-3,-2) circle (.3);
\draw[fill=white, line width=1pt] (-3,-2) circle (.15);
\draw[color=black, fill=white, line width=1pt] (-.5,3) rectangle (2.5,4);
\node at (1,3.5){$\iota$};
\end{scope}
\node at (11.5,0){$=$};
\begin{scope}[shift={(16.5,0)}]
\draw[line width=1.5pt, color=blue] (0,3)--(0,-3);
\draw[line width=1.5pt, color=red] (1,3)--(1,-3);
\draw[line width=1.5pt, color=violet] (2,3)--(2,-3);
\draw[line width=1pt, color=black] (0,2.5)--(-3,0);
\draw[line width=1pt, color=black] (1,2.5)--(-3,0);
\draw[line width=1pt, color=black] (2,2.5)--(-3,0);
\draw[line width=1pt, color=black] (-3,0)--(-3,-1.5);
\draw[line width=1pt, color=black] (0,-.5)--(-3,-1.5);
\draw[line width=1pt, color=black] plot [smooth , tension=0.6] coordinates 
        {(0, .5)(-4,-1)(-3,-1.5)};
          \draw[color=black, fill=gray, line width=1pt] (-1,.1) circle (.3);
          \draw[line width=1pt, color=black] (-3,-1)--(-3,-3);
          \draw[fill=white, line width=1pt] (-3,-2.5) circle (.3);
\draw[fill=white, line width=1pt] (-3,-2.5) circle (.15);
          \draw[color=black, fill=white, line width=1pt] (-.5,3) rectangle (2.5,4);
\node at (1,3.5){$\iota$};
\end{scope}
\node at (20,0){$=$};
\begin{scope}[shift={(25.5,0)}]
\draw[line width=1.5pt, color=blue] (0,3)--(0,-3);
\draw[line width=1.5pt, color=red] (1,3)--(1,-3);
\draw[line width=1.5pt, color=violet] (2,3)--(2,-3);
\node at (-3.2,0)[anchor=east]{$p$};
\draw[color=black, fill=black] (-3,0) circle (.15);
\draw[line width=1pt, color=black] (-3,0)--(-3,-1.5);
\draw[line width=1pt, color=black] (0,-.5)--(-3,-1.5);
\draw[line width=1pt, color=black] plot [smooth , tension=0.6] coordinates 
        {(0, .5)(-4,-1)(-3,-1.5)};
          \draw[color=black, fill=gray, line width=1pt] (-1,.1) circle (.3);
\draw[line width=1pt, color=black] (-3,-1)--(-3,-3);
\draw[fill=white, line width=1pt] (-3,-2.5) circle (.3);
\draw[fill=white, line width=1pt] (-3,-2.5) circle (.15);
\draw[color=black, fill=white, line width=1pt] (-.5,3) rectangle (2.5,4);
\node at (1,3.5){$\iota$};
\end{scope}
\node at (-4,-9){$=$};
\begin{scope}[shift={(0,-9)}]
\draw[line width=1.5pt, color=blue] (0,3)--(0,-3);
\draw[line width=1.5pt, color=red] (1,3)--(1,-3);
\draw[line width=1.5pt, color=violet] (2,3)--(2,-3);
\node at (-3,0)[anchor=south]{$p$};
\draw[color=black, fill=black] (-3,0) circle (.15);
\draw[line width=1pt, color=black] (0,-.5)--(-3,-1.5);
\draw[line width=1pt, color=black] plot [smooth , tension=0.6] coordinates 
        {(0, .5)(-3,-1.5)};
          \draw[color=black, fill=white, line width=1pt] (-1,-.2) circle (.3);
\draw[line width=1pt, color=black] (-3,0)--(-3,-3);
\draw[fill=white, line width=1pt] (-3,-2.5) circle (.3);
\draw[fill=white, line width=1pt] (-3,-2.5) circle (.15);
\draw[color=black, fill=white, line width=1pt] (-.5,3) rectangle (2.5,4);
\node at (1,3.5){$\iota$};
\end{scope}
\node at (3.5,-9){$=$};
\begin{scope}[shift={(8,-9)}]
\draw[line width=1.5pt, color=blue] (0,3)--(0,-3);
\draw[line width=1.5pt, color=red] (1,3)--(1,-3);
\draw[line width=1.5pt, color=violet] (2,3)--(2,-3);
\draw[line width=1pt, color=black] plot [smooth , tension=0.6] coordinates 
         {(0,2.5)(-1.5,2)(-2,1.5)};
         \draw[line width=1pt, color=black] plot [smooth cycle , tension=0.6] coordinates 
         {(-2,1.5)(-2.5,.5)(-2,-.5) (-1.5,.5)};
          \draw[color=black, fill=white, line width=1pt] (-2.5,.5) circle (.3);
         \draw[line width=1pt, color=black] (-2,-.5)--(-2,-3);      
\node at (-3,-1)[anchor=east]{$p$};
\draw[color=black, fill=black] (-3,-1) circle (.15);
\draw[line width=1pt, color=black] (-3,-1)--(-2,-1.5); 
\draw[line width=1pt, color=black] (-2,-2)--(-2,-1.5); 
\draw[fill=white, line width=1pt] (-2,-2) circle (.3);
\draw[fill=white, line width=1pt] (-2,-2) circle (.15);
\draw[color=black, fill=white, line width=1pt] (-.5,3) rectangle (2.5,4);
\node at (1,3.5){$\iota$};
\end{scope}
\node at (11.5,-9){$=$};
\begin{scope}[shift={(16.5,-9)}]
\draw[line width=1.5pt, color=blue] (0,3)--(0,-3);
\draw[line width=1.5pt, color=red] (1,3)--(1,-3);
\draw[line width=1.5pt, color=violet] (2,3)--(2,-3);
\draw[line width=1pt, color=black] (-2,1)--(-2,-3); 
\draw[color=black, fill=black] (-2,1) circle (.15);
\node at (-2,1)[anchor=south]{$p$};
\draw[color=black, fill=white, line width=1pt] (-.5,3) rectangle (2.5,4);
\node at (1,3.5){$\iota$};
\draw[fill=white, line width=1pt] (-2,-2) circle (.3);
\draw[fill=white, line width=1pt] (-2,-2) circle (.14);
\end{scope}
\node at (20,-9){$=$};
\begin{scope}[shift={(25.5,-9)}]
\draw[line width=1.5pt, color=blue] (0,3)--(0,-3);
\draw[line width=1.5pt, color=red] (1,3)--(1,-3);
\draw[line width=1.5pt, color=violet] (2,3)--(2,-3);
\draw[line width=1pt, color=black] (-2,1)--(-2,-3); 
\draw[color=black, fill=white] (-2,1) circle (.15);
\draw[color=black, fill=white, line width=1pt] (-.5,3) rectangle (2.5,4);
\node at (1,3.5){$\iota$};
\end{scope}
\end{tikzpicture}
\tComma
\end{align*}
where the coloured lines stand for the edges in $\gamma'^\inv\circ \gamma$ and $\iota$ for the coequaliser of $\delta'_f$. This shows that the claim is invariant under cyclic permutations of the edges in $f$.

5.~To prove \eqref{eq:ComplementSlideGeneral}, 
note first that if  $\gamma=\gamma_{1}\circ\gamma_{2}$ for a face path $\gamma_{2}$ and a non-trivial face path $\gamma_{1}$ and if $\eta_\alpha$ adds a loop directly before the starting end of $\gamma_{1}$, then one has 
	\begin{align}
		\eta_\alpha S_{\gamma_{1}\gamma_{2}} = S_{\gamma_{1}\alpha^{L}\gamma_{2}} \eta_\alpha.
		\label{eq:CoinvariantArcSlideCompatibility}
	\end{align}
Otherwise $S_{\gamma}$ and $\eta_{\alpha}$ simply commute. 
By decomposing   $C_\rho=S_\rho\circ\eta_{\rho'}$ and $C_{S_\gamma(\rho)}=S_{S_{\gamma}(\rho)}\circ \eta_{S_\gamma(\rho)'}$ , using the commutativity relation of slides in \eqref{eq:SlideCommutativityRelation} and identity \eqref{eq:CoinvariantArcSlideCompatibility}, one  obtains \eqref{eq:ComplementSlideGeneral}.

Identity \eqref{eq:slidenotmatter}  follows by setting $\gamma=\rho$ in  \eqref{eq:ComplementSlideGeneral}, for which $C_\gamma(\gamma)=\gamma'$ by Definition \ref{definition:FacePathTranslationComplementary}, composing both sides of  \eqref{eq:ComplementSlideGeneral} with $\epsilon_{\gamma'}$ and using \eqref{eq:ComplementDelete} .

6.~To prove Equation~\eqref{eq:ComplementCommutativity} we first treat the case where $\rho$ is a proper subpath of $\gamma = \gamma_{1}\circ \rho\circ \gamma_{2}$. Then equation~\eqref{eq:CoinvariantArcSlideCompatibility} holds for $\eta_{\alpha} = \eta_{\rho'}$ and we have $C_{\rho}(\gamma) = \gamma_{1} \rho' \gamma_{2}$ and $C_{\gamma} (\rho) = \rho$. We then obtain 
\begin{align}
	\nonumber
	C_{C_{\gamma}(\rho)} C_{\gamma}
	&= C_{\rho} C_{\gamma}
	= S_{\rho} \eta_{\rho'} C_{\gamma_{1}\rho\gamma_{2}}
	= S_{\rho} \eta_{\rho'} S_{\gamma_{1}\rho\gamma_{2}} \eta_{\gamma'}
	\overset{\text{\eqref{eq:CoinvariantArcSlideCompatibility}}}{=} S_{\rho} S_{\gamma_{1}\rho\rho'\gamma_{2}} \eta_{\rho'} \eta_{\gamma'}
	=S_{\rho} S_{\gamma_{1}\rho\rho'\gamma_{2}}\eta_{\gamma'} \eta_{\rho'}
	\\
	&= S_{\rho} C_{\gamma_{1}\rho\rho'\gamma_{2}} \eta_{\rho'}
	\overset{\text{\eqref{eq:ComplementSlideGeneral}}}{=} C_{S_{\rho}(\gamma_{1}\rho\rho'\gamma_{2})} S_{\rho} \eta_{\rho'}
	= C_{\gamma_{1}\rho'\gamma_{2}} S_{\rho}\eta_{\rho'}
	= C_{C_{\rho}(\gamma)} C_{\rho}.
		\nonumber
\end{align}
Because of the symmetry of~\eqref{eq:ComplementCommutativity} the same holds if $\gamma$ is a proper subpath of $\rho$, while for $\gamma=\rho$ the claim is tautological. For non-overlapping $\rho$ and $\gamma$ we have $C_{\gamma}(\rho)=\rho$, $C_{\rho}(\gamma)= \gamma$ and $\eta_{\rho'}$ commutes with $C_{\gamma}$. We then obtain
\begin{align*}
	C_{C_{\gamma}(\rho)} \circ C_{\gamma} 
	= C_{\rho} \circ C_{\gamma}
	= S_{\rho} \circ \eta_{\rho'} \circ C_{\gamma}
	= S_{\rho} \circ C_{\gamma} \circ \eta_{\rho'}
	\overset{\text{\eqref{eq:ComplementSlideGeneral}}}{=} C_{S_{\rho}(\gamma)} \circ S_{\rho} \circ \eta_{\rho'}
	= C_{\gamma} \circ C_{\rho}
	= C_{C_{\rho}(\gamma)} \circ C_{\rho}.
\end{align*} 
7.~The first identity in \eqref{eq:FlatFaceComplementt} follows from Remark \ref{rem:otherslide} and the identity
\begin{align}\label{eq:edgeexpress}
\begin{tikzpicture}[scale=.4]
\begin{scope}[shift={(0,0)}]
\draw[color=black, line width=1pt, fill=white] (-.5,3) rectangle (2.5,4.5);
\node at (1,3)[anchor=south]{$\iota$};
\draw[line width=1.5pt, color=blue] (0,3)--(0,-3);
\draw[line width=1.5pt, color=red] (1,3)--(1,-3);
\draw[line width=1.5pt, color=violet] (2,3)--(2,-3);
\end{scope}  
\node at (2.8,0){$=$};
\begin{scope}[shift={(5,0)}]
\draw[color=black, line width=1pt, fill=white] (-.5,3) rectangle (2.5,4.5);
\node at (1,3)[anchor=south]{$\iota$};
\draw[line width=1.5pt, color=blue] (0,3)--(0,-3);
\draw[line width=1.5pt, color=red] (1,3)--(1,-3);
\draw[line width=1.5pt, color=violet] (2,3)--(2,-3);
\draw[line width=1pt, color=black] (0,2.5)--(-1,.5);
\draw[line width=1pt, color=black] (1,2)--(-1,.5);
\draw[line width=1pt, color=black] (2,1.5)--(-1,.5);
\draw[line width=1pt, color=black] plot [smooth , tension=0.6] coordinates 
         {(-1,.5)(-1,-2)(0,-2.5)};
\draw[color=black, fill=white] (-1.1,-1) circle (.3);
\draw[color=black, fill=white] (-1.1,-1) circle (.15);
\end{scope}
\node at (8,0){$=$};
\begin{scope}[shift={(11,0)}]
\draw[color=black, line width=1pt, fill=white] (-.5,3) rectangle (2.5,4.5);
\node at (1,3)[anchor=south]{$\iota$};
\draw[line width=1.5pt, color=blue] (0,3)--(0,-3);
\draw[line width=1.5pt, color=red] (1,3)--(1,-3);
\draw[line width=1.5pt, color=violet] (2,3)--(2,-3);
\draw[line width=1pt, color=black] (0,2.5)--(-1,1.5);
\draw[line width=1pt, color=black] (1,2.5)--(-2,.5);
\draw[line width=1pt, color=black] (2,2)--(-2,.5);
\draw[line width=1pt, color=black] (-2,.5)--(-2,-.5);
\draw[line width=1pt, color=black] (-1,1.5)--(-1,-.5);
\draw[line width=1pt, color=black] plot [smooth , tension=0.6] coordinates 
         {(-2,-.5)(-1.5,-1)(-1,-.5)};
\draw[line width=1pt, color=black] plot [smooth , tension=0.6] coordinates 
         {(-1.5,-1)(-1,-2)(0,-2.5)};
\draw[color=black, fill=white] (-1,0) circle (.3);
\draw[color=black, fill=white] (-2,0) circle (.3);
\draw[color=black, fill=white] (-2,0) circle (.15);
\end{scope}
\node at (13.8,0){$=$};
\begin{scope}[shift={(16,0)}]
\draw[color=black, line width=1pt, fill=white] (-.5,3) rectangle (2.5,4.5);
\node at (1,3)[anchor=south]{$\iota$};
\draw[line width=1.5pt, color=blue] (0,3)--(0,-3);
\draw[line width=1.5pt, color=red] (1,3)--(1,-3);
\draw[line width=1.5pt, color=violet] (2,3)--(2,-3);
\draw[line width=1pt, color=black] plot [smooth , tension=0.6] coordinates 
         {(0,2.5)(-.5,2)(-.5,1)(0,.5)};
         \draw[color=black, fill=white] (-.5,1.5) circle (.3);
  \draw[line width=1pt, color=black] (1,1)--(-1,-1);
    \draw[line width=1pt, color=black] (2,.5)--(-1,-1);
    \draw[line width=1pt, color=black] plot [smooth , tension=0.6] coordinates 
         {(-1,-1)(-1,-2)(0,-2.5)};
         \draw[color=black, fill=white] (-1,-1.5) circle (.3);
                  \draw[color=black, fill=white] (-1,-1.5) circle (.15);
                  \end{scope}
 \node at (19,0) {$=$};        
 \begin{scope}[shift={(21,0)}]
 \draw[color=black, line width=1pt, fill=white] (-.5,3) rectangle (2.5,4.5);
\node at (1,3)[anchor=south]{$\iota$};
\draw[line width=1.5pt, color=blue] (0,3)--(0,1);
\draw[line width=1.5pt, color=red] (1,3)--(1,-3);
\draw[line width=1.5pt, color=violet] (2,3)--(2,-3);
  \draw[color=black, fill=white] (0,1) circle (.15);
  \draw[line width=1pt, color=black] (1,1)--(0,-1);
    \draw[line width=1pt, color=black] (2,.5)--(0,-1);
    \draw[line width=1pt, color=black] (0,-1)--(0,-3);
         \draw[color=black, fill=white] (0,-1.5) circle (.3);
                  \draw[color=black, fill=white] (0,-1.5) circle (.15);
                  \end{scope}         
\end{tikzpicture}
\tComma
\end{align}
where $\iota$ stands for  the inclusion morphism of the coinvariants from Definition \ref{def:invcoinv},
the left line for the edge $\alpha$ 
and the other vertical lines for other edges in the face. 
The second identity in  \eqref{eq:FlatFaceComplementt} then follows from the first, together with
~\eqref{eq:ComplementDelete}, \eqref{eq:ComplementSlideGeneral} and the Definition of the edge slides: 
	\begin{align*}
		S_{\rho} \epsilon_{\alpha} 
		\stackrel{\eqref{eq:ComplementDelete}}= \epsilon_{\rho'} C_{\rho} S_{\rho} \epsilon_{\alpha}
		\stackrel{\eqref{eq:ComplementSlideGeneral}}= \epsilon_{\rho'} S_{\rho'} C_{\rho} \epsilon_{\alpha}\stackrel{\eqref{eq:FlatFaceComplementt}.1}= \epsilon_{\alpha}S_{\alpha}T_{\alpha}
		\stackrel{\text{Def.}\ref{def:edge slide}}= \epsilon_{\alpha} S_{\alpha^{-R}}.
	\end{align*}
\end{proof}

Equation \eqref{eq:slidenotmatter} allows us to replace slides along face paths or their inverses by slides along edges. It also extends Definition \ref{def:slideface} to slides of edge ends that are traversed by the face path $\gamma$. Equation \eqref{eq:ComplementSlideGeneral} states that sliding along a face path $\gamma$ commutes with the addition of an edge to another face path $\rho$, whenever $\gamma$ remains a face path.   Equations \eqref{eq:EdgeComplementInvolution}, \eqref{eq:FlatFaceComplement} and \eqref{eq:FlatFaceComplementt}  ensure that we can add multiple edges to face paths and delete any of the added edges without changing the slide. Equation \eqref{eq:ComplementCommutativity} states that the order in which edges are added to different face paths is irrelevant, as long as face paths remain face paths.
Together, these identities allow us to replace slides  along  face paths  by  slides along  edges that are added to these face paths.   

Similarly, adding edges  to a vertex $v$ allows one to replace slides of multiple edge ends incident at  $v$  by a single slide. This is achieved by adding an edge $\alpha$ with a univalent target vertex  to the starting vertex of a face path $\gamma$, sliding the edge ends in question along $\alpha$ to the target of $\alpha$, then sliding $\alpha$ along $\gamma$ and finally sliding the edge ends back to the starting vertex of $\alpha$, as shown in Figure \ref{fig:multiplide}.

\begin{lemma}\label{rem:shorthand} Let $\gamma$ be a face path in $\Gamma$ and suppose that the first $n+m$ edge ends at $s(\gamma)$ after the starting end  of $\gamma$ are not traversed by $\gamma$.
 Insert an edge   $\alpha$ with a univalent target vertex  at  $s(\gamma)$ such that there are $m$ edge ends  between the starting end of $\gamma$ and of $\alpha$, as shown in Figure \ref{fig:multiplide}.  Denote by $S=S_{\alpha_L}^n S^m_{\alpha^R}$ the slide that slides these $m+n$ edge ends  to the  target of $\alpha$. Then 
\begin{align*}
S_{\gamma}^{m+n}=\epsilon_\alpha\circ S^\inv\circ S_{\gamma} \circ S\circ \eta_\alpha.
\end{align*}
\end{lemma}

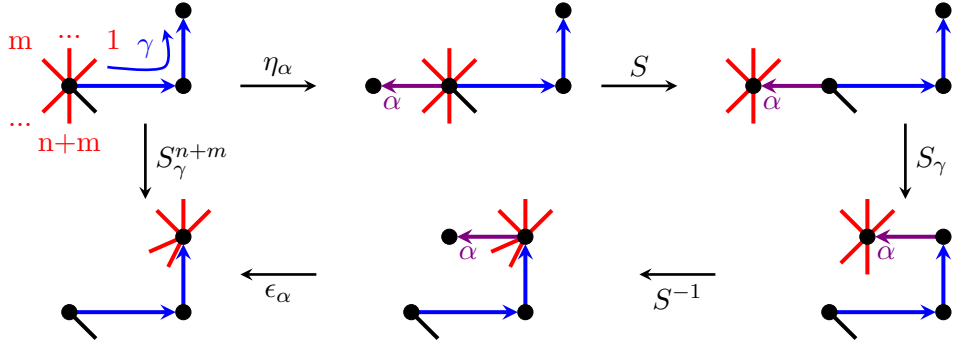
\begin{figure}
\begin{center}
\begin{tikzpicture}[scale=.5]
\begin{scope}[shift={(0,0)}]
\draw[color=blue, line width=1.5pt, ->,>=stealth] (0,0)--(2.8,0);
\draw[color=blue, line width=1.5pt, ->,>=stealth] (3,0)--(3,1.8);
\draw[color=red, line width=1.5pt] (0,0)--(.7,.7);
\draw[color=red, line width=1.5pt] (0,0)--(0,1);
\draw[color=red, line width=1.5pt] (0,0)--(-.7,.7);
\draw[color=red, line width=1.5pt] (0,0)--(-.7,-.7);
\draw[color=red, line width=1.5pt] (0,0)--(0,-1);
\draw[color=black, line width=1.5pt] (0,0)--(.7,-.7);
\draw[color=black, fill=black] (0,0) circle (.2);
\draw[color=black, fill=black] (3,0) circle (.2);
\draw[color=black, fill=black] (3,2) circle (.2);
\node at (.7,.7)[anchor=south west, color=red]{1};
\node at (-.7,.7)[anchor=south east, color=red]{m};
\node at (0,1)[anchor=south, color=red]{...};
\node at (-.7,-.7)[anchor=north east, color=red]{...};
\node at (0,-1)[anchor=north, color=red]{n+m};
    \draw[line width=1pt, color=blue, ->,>=stealth] plot [smooth , tension=0.6] coordinates 
         {(1,.5)(2.5,.5)(2.5,1.5)};
\node at (2.5,1)[anchor=east, color=blue]{$\gamma$};
\end{scope}
\draw[line width=1pt, ->,>=stealth] (2,-1)--(2,-3);
\node at (2,-2)[anchor=west]{$S_\gamma^{n+m}$};
\draw[line width=1pt, ->,>=stealth] (4.5,0)--(6.5,0);
\node at (5.5,0)[anchor=south]{$\eta_\alpha$};
\begin{scope}[shift={(10,0)}]
\draw[color=blue, line width=1.5pt, ->,>=stealth] (0,0)--(2.8,0);
\draw[color=blue, line width=1.5pt, ->,>=stealth] (3,0)--(3,1.8);
\draw[color=red, line width=1.5pt] (0,0)--(.7,.7);
\draw[color=red, line width=1.5pt] (0,0)--(0,1);
\draw[color=red, line width=1.5pt] (0,0)--(-.7,.7);
\draw[color=red, line width=1.5pt] (0,0)--(-.7,-.7);
\draw[color=red, line width=1.5pt] (0,0)--(0,-1);
\draw[color=black, line width=1.5pt] (0,0)--(.7,-.7);
\draw[color=violet, line width=1.5pt, ->,>=stealth] (0,0)--(-1.8,0);
\node at (-1.5,0)[anchor=north, color=violet]{$\alpha$};
\draw[color=black, fill=black] (-2,0) circle (.2);
\draw[color=black, fill=black] (0,0) circle (.2);
\draw[color=black, fill=black] (3,0) circle (.2);
\draw[color=black, fill=black] (3,2) circle (.2);
\end{scope}
\draw[line width=1pt, ->,>=stealth] (14,0)--(16,0);
\node at (15,0)[anchor=south]{$S$};
\begin{scope}[shift={(20,0)}]
\draw[color=blue, line width=1.5pt, ->,>=stealth] (0,0)--(2.8,0);
\draw[color=blue, line width=1.5pt, ->,>=stealth] (3,0)--(3,1.8);
\draw[color=red, line width=1.5pt] (-2,0)--(-1.3,.7);
\draw[color=red, line width=1.5pt] (-2,0)--(-2,1);
\draw[color=red, line width=1.5pt] (-2,0)--(-2.7,.7);
\draw[color=red, line width=1.5pt] (-2,0)--(-2.7,-.7);
\draw[color=red, line width=1.5pt] (-2,0)--(-2,-1);
\draw[color=black, line width=1.5pt] (0,0)--(.7,-.7);
\draw[color=violet, line width=1.5pt, ->,>=stealth] (0,0)--(-1.8,0);
\node at (-1.5,0)[anchor=north, color=violet]{$\alpha$};
\draw[color=black, fill=black] (-2,0) circle (.2);
\draw[color=black, fill=black] (0,0) circle (.2);
\draw[color=black, fill=black] (3,0) circle (.2);
\draw[color=black, fill=black] (3,2) circle (.2);
\end{scope}
\draw[line width=1pt, ->,>=stealth] (22,-1)--(22,-3);
\node at (22,-2)[anchor=west]{$S_\gamma$};
\begin{scope}[shift={(20,-6)}]
\draw[color=blue, line width=1.5pt, ->,>=stealth] (0,0)--(2.8,0);
\draw[color=blue, line width=1.5pt, ->,>=stealth] (3,0)--(3,1.8);
\draw[color=red, line width=1.5pt] (1,2)--(1.7,2.7);
\draw[color=red, line width=1.5pt] (1,2)--(1,3);
\draw[color=red, line width=1.5pt] (1,2)--(.3,2.7);
\draw[color=red, line width=1.5pt] (1,2)--(.3,1.3);
\draw[color=red, line width=1.5pt] (1,2)--(1,1 );
\draw[color=black, line width=1.5pt] (0,0)--(.7,-.7);
\draw[color=violet, line width=1.5pt, ->,>=stealth] (3,2)--(1.2,2);
\node at (1.5,2)[anchor=north, color=violet]{$\alpha$};
\draw[color=black, fill=black] (1,2) circle (.2);
\draw[color=black, fill=black] (0,0) circle (.2);
\draw[color=black, fill=black] (3,0) circle (.2);
\draw[color=black, fill=black] (3,2) circle (.2);
\end{scope}
\draw[line width=1pt, <-,>=stealth] (15,-5)--(17,-5);
\node at (16,-5)[anchor=north]{$S^\inv$};
\begin{scope}[shift={(9,-6)}]
\draw[color=blue, line width=1.5pt, ->,>=stealth] (0,0)--(2.8,0);
\draw[color=blue, line width=1.5pt, ->,>=stealth] (3,0)--(3,1.8);
\draw[color=red, line width=1.5pt] (3,2)--(3.7,2.7);
\draw[color=red, line width=1.5pt] (3,2)--(3,3);
\draw[color=red, line width=1.5pt] (3,2)--(2.3,2.7);
\draw[color=red, line width=1.5pt] (3,2)--(2.1,1.6);
\draw[color=red, line width=1.5pt] (3,2)--(2.6, 1.2 );
\draw[color=black, line width=1.5pt] (0,0)--(.7,-.7);
\draw[color=violet, line width=1.5pt, ->,>=stealth] (3,2)--(1.2,2);
\node at (1.5,2)[anchor=north, color=violet]{$\alpha$};
\draw[color=black, fill=black] (1,2) circle (.2);
\draw[color=black, fill=black] (0,0) circle (.2);
\draw[color=black, fill=black] (3,0) circle (.2);
\draw[color=black, fill=black] (3,2) circle (.2);
\end{scope}
\begin{scope}[shift={(0,-6)}]
\draw[color=blue, line width=1.5pt, ->,>=stealth] (0,0)--(2.8,0);
\draw[color=blue, line width=1.5pt, ->,>=stealth] (3,0)--(3,1.8);
\draw[color=red, line width=1.5pt] (3,2)--(3.7,2.7);
\draw[color=red, line width=1.5pt] (3,2)--(3,3);
\draw[color=red, line width=1.5pt] (3,2)--(2.3,2.7);
\draw[color=red, line width=1.5pt] (3,2)--(2.1,1.6);
\draw[color=red, line width=1.5pt] (3,2)--(2.6, 1.2 );
\draw[color=black, line width=1.5pt] (0,0)--(.7,-.7);
\draw[color=black, fill=black] (0,0) circle (.2);
\draw[color=black, fill=black] (3,0) circle (.2);
\draw[color=black, fill=black] (3,2) circle (.2);
\end{scope}
\draw[line width=1pt, <-,>=stealth] (4.5,-5)--(6.5,-5);
\node at (5.5,-5)[anchor=north]{$\epsilon_\alpha$};
\end{tikzpicture}
\end{center}
\caption{Replacing  slides of multiple edge ends along a face path by a single slide.}
\label{fig:multiplide}
\end{figure}

\begin{proof}  By the generalised pentagon relation~\eqref{eq:GeneralizedPentagon}, sliding the $m+n$ edge ends to the target of $\alpha$,  sliding $\alpha$ along $\gamma$ and then sliding the edge ends back to the start of $\alpha$ gives the same result as sliding the first $m$ edge ends, then $\alpha$ and then the last $n$ edge ends along $\gamma$. 
Adding the edge $\alpha$ and deleting the edge $\alpha$ commute with the slides of the  $m+n$   edge ends along $\gamma$.
Expressions \eqref{eq:slidedef2a} to \eqref{eq:slidedef4b} for the slides 
show that  adding an edge $\alpha$ by applying the unit of $H$, sliding it along $\gamma$ and then deleting $\alpha$ gives the identity morphism. Thus we have 
\begin{align*}
\epsilon_\alpha \circ S^\inv \circ S_\gamma\circ S\circ \eta_{\alpha}\stackrel{\eqref{eq:GeneralizedPentagon}} =\epsilon_\alpha\circ S_\gamma^{n+m+1}\circ \eta_\alpha=
S^{n}_\gamma\circ \epsilon_\alpha\circ S_\gamma\circ \eta_\alpha\circ S_\gamma^m
\stackrel{\eqref{eq:slidedef2a}-\eqref{eq:slidedef4b}}=S^n_\gamma\circ S^m_\gamma=S^{m+n}_\gamma.
\end{align*}
\end{proof}

\subsection{Twists along closed face paths}
\label{sec:twists}

As already apparent in Section \ref{sec:torus}, slides along loops or closed face paths $\gamma$ do not change the underlying ribbon graph,  if  they are applied to all edge ends between the starting and target  end of $\gamma$. In fact, such slides correspond to Dehn twists, and we will will relate them to the  Dehn twists from Theorem \ref{th:gervais}
in  Section \ref{sec:gervaismap}.

We introduce these twists in three stages. We first define twists along loops in a directed ribbon graph $\Gamma$ that are based at a ciliated vertex. We then extend this definition to twists along closed face paths,  by adding an edge to the face path, twisting around the edge, and then removing it. Finally,  in Section \ref{sec:gervaismap}  we  define twists along certain closed paths that are not face paths.

\begin{definition} \label{def:dtnull}  Let  $\beta$ be a loop  at a ciliated vertex of  $\Gamma$. 
 The {\bf twist}  along $\beta$ is the counterclockwise slide along $\beta$, applied to each edge end  between the  ends of $\beta$ once. 
 \begin{align*}
 D_\beta=\begin{cases} S_{\beta^L}^n & \st(\beta)< \ta(\beta),\\
 S_{\beta^{-R}}^n & \st(\beta)>\ta(\beta),
 \end{cases}
 \end{align*} 
 where $n$ is the number of edge ends between the starting end $\st(\beta)$ and the target end $\ta(\beta)$ of  $\beta$.
\end{definition}

It is clear from Definition \ref{def:dtnull} and  expressions \eqref{eq:slidedef2a} to \eqref{eq:slidedef4b}  for the edge slides that the twist $D_\beta$ does not depend  
 on the orientation of $\beta$.  Reversing the  orientation of $\beta$  and applying the involution $T$ to the associated copy of $H$ yields the same automorphism of $H^{\oo E}$. 

 Note, however,  that the twist along a loop $\beta$ depends on the choice of the cilium at the starting and target vertex of $\beta$, that is the {\em linear} ordering of the edge ends at this vertex. This dependence is investigated in more depth at the end of Section \ref{sec:gervaismap}, where we show that different choices of cilia lead to mapping class group actions related by  conjugation.

The twist along a loop $\beta$ acts only on those copies of $H$ that belong to edges with  at least one end  incident between the starting and the target end of $\beta$. With Lemma \ref{rem:shorthand}, we can  restrict attention to the case where there is only a single such edge end.  More precisely,  if  $\beta$ is a loop  based at a vertex $v$ of $\Gamma$, we can add
an edge $\alpha$ at $v$,  with the starting end between $\st(\beta)$ and $\ta(\beta)$ and a univalent vertex $t(\alpha)$. Then by Lemma \ref{rem:shorthand}, we have
$
D_\beta=\epsilon_\alpha \circ S^\inv \circ D_\beta\circ S\circ \eta_{\alpha}
$, 
where  $S$  is the composite slide  that slides all other edge ends incident between $s(\beta)$ and $t(\beta)$ along $\alpha$ to its target. 
This allows us to replace some or all edge ends between the ends of a loop by a single edge end in diagrammatic computations, as shown in Figure \ref{fig:multiplide}.

We now extend the definition of the twist from   loops to  closed face paths in a  ribbon graph. This is achieved  by adding an edge to the face path, twisting along the added edge and then removing it. This reduces to the original  definition whenever the closed face path is a single loop.

\begin{definition}\label{def:dtnotface} Let $\phi$ be a closed face path based at a ciliated vertex  in a directed ribbon graph $\Gamma$. The  {\bf  twist} $D_{\phi}$ along $\phi$  is  the composite $D_\phi=\epsilon_{\phi'}\circ D_{\phi'}\circ C_\phi: H^{\oo E}\to H^{\oo E}$ obtained by
\begin{compactenum}
\item Adding an edge $\phi'$  to $\phi$ as in Definition~\ref{def:addingedge},
\item Performing a   twist along  the edge $\phi'$,
\item Removing the edge $\phi'$ as in Definition~\ref{def:addingedge}.
\end{compactenum}
\end{definition}

\begin{remark}
	\label{remark:TwistArcFaceCoincide}
Definition~\ref{def:dtnotface}  extends Definition~\ref{def:dtnull} of twists along loops. If $\phi=\beta^{\pm 1}$ for a loop $\beta$, then $D_\phi=D_\beta$.  This follows with Lemma \ref{rem:shorthand} and the diagrammatic computation

\begin{align}
		\begin{tikzpicture}[scale=.5]
			\begin{scope}[shift={(0,0)}]
				\draw[line width=1pt, color=black] (-1,0)--(4,0);
				\draw[color=red, line width=1.5 pt, <-,>=stealth] (0,0).. controls (0,2) and (3,2) .. (3,0);
				\draw[color=blue, line width=1.5 pt, <-,>=stealth] (1.5,0)--(1.5,3);
				\node at (2.5,1.2) [color=red, anchor=south]{$b$};
				\node at (1.5,3)[color=blue, anchor=south] {$a$};
			\end{scope}
			\draw[color=black, line width=1.5pt,->,>=stealth] (4.5,2) -- (6.5,2);
			\node at (5.5,2)[anchor=south]{$C_{\phi}$};
			\draw[color=black, line width=1pt,->,>=stealth] (1.5,-0.5) -- (1.5,-2.5);
			\node at (1.6,-1.5) [anchor=west, color=black] {$D_{\beta}$};
			\begin{scope}[shift={(8,0)}]
				\draw[line width=1pt, color=black] (-1,0)--(6,0);
				\draw[color=red, line width=1.5 pt, <-,>=stealth,style=dashed] (-0.5,0).. controls (-0.5,3.5) and (5.5,3.5) .. (5.5,0);
				\draw[color=red, line width=1.5 pt, <-,>=stealth] (1,0).. controls (1,2) and (4,2) .. (4,0);
				\draw[color=blue, line width=1.5 pt, <-,>=stealth] (2.5,0)--(2.5,3);
				\node at (4.6,2) [color=red, anchor=south]{$b_{(2)}$};
				\node at (3.5,1.2) [color=red, anchor=south]{$b_{(1)}$};
				\node at (2.5,3)[color=blue, anchor=south] {$1$};
			\end{scope}
			\draw[color=black, line width=1pt,->,>=stealth] (10.5,-0.5) -- (10.5,-2.5);
			\node at (10.5,-1.5)[anchor=west]{$D_{\phi'}$};
			\begin{scope}[shift={(8,-7)}]
				\draw[line width=1pt, color=black] (-1,0)--(6,0);
				\draw[color=red, line width=1.5 pt, <-,>=stealth,style=dashed] (-0.5,0).. controls (-0.5,3.5) and (5.5,3.5) .. (5.5,0);
				\draw[color=red, line width=1.5 pt, <-,>=stealth] (1,0).. controls (1,2) and (4,2) .. (4,0);
				\draw[color=blue, line width=1.5 pt, <-,>=stealth] (2.5,0)--(2.5,3);
				\node at (4.6,2) [color=red, anchor=south]{$b_{(5)}$};
				\node at (3.5,-0.1) [color=red, anchor=north]{$b_{(4)}b_{(1)}S(b_{(2)})$};
				\node at (2.5,3)[color=blue, anchor=south] {$b_{(3)}a$};
			\end{scope}
			\draw[color=black, line width=1pt,<-,>=stealth] (4.5,-5) -- (6.5,-5);
			\node at (5.5,-5)[anchor=north]{$\epsilon_{\phi'}$};
			\begin{scope}[shift={(0,-7)}]
				\draw[line width=1pt, color=black] (-1,0)--(4,0);
				\draw[color=red, line width=1.5 pt, <-,>=stealth] (0,0).. controls (0,2) and (3,2) .. (3,0);
				\draw[color=blue, line width=1.5 pt, <-,>=stealth] (1.5,0)--(1.5,3);
				\node at (2.5,1.2) [color=red, anchor=south]{$b_{(2)}$};
				\node at (1.5,3)[color=blue, anchor=south] {$b_{(1)}a$};
			\end{scope}
		\end{tikzpicture}
\end{align}
\end{remark}

Definition \ref{def:dtnotface} extends Definition \ref{def:dtnull} from edges to  face paths. In the following, we will also need to consider twists along certain closed paths that are not face paths.  These will be defined in Section \ref{sec:gervaismap}. The basic idea is to slide the edge ends that prevent a closed path from being a face path out of the way, either along  the path or along other edges, twist along the resulting face path, and then slide the edge ends back in their original position. 
For this, we require a number of technical results on  the interaction between twists, slides and adding edges.

\begin{lemma} \label{lem:slideDT}  Let  $\alpha$ be a loop based at a ciliated vertex and  $\beta$ an edge whose starting and target vertex are ciliated, such that the following  are defined. Then:
\begin{compactenum}
\item Sliding both  ends of  $\alpha$ and all edge ends between them along  $\beta$ commutes with $D_\alpha$.
\item Sliding both ends of  $\alpha$ in the same direction along different sides of  $\beta$  commutes with  $D_\alpha$.
\item Sliding an end of an edge $\gamma$  into  or out of a loop $\alpha$ along an edge $\delta$  commutes with  $D_\alpha$.
\end{compactenum}
\end{lemma}
\begin{proof}
	Using Lemma \ref{rem:shorthand} we replace the edge ends between the ends of $\alpha$ with a single edge labelled with $c$.  We then verify 1.~by a direct diagrammatic computation, 
\begin{align*}
\begin{tikzpicture}[scale=.6]
\begin{scope}
\draw[line width=1pt] (-5,0)--(4,0);
\draw[color=red, line width=1.5 pt, ->,>=stealth] (-4,0).. controls (-4,2) and (-1,2) .. (-1,0);
\draw[color=blue, line width=1.5 pt, ->,>=stealth] (0,0).. controls (0,2) and (3,2) .. (3,0);
\draw[color=violet, line width=1.5 pt, ->,>=stealth] (1.5,0)-- (1.5,2.5);
\node at (-2.5,1.5)[anchor=south, color=red]{$b$};
\node at (1.7,1.5)[anchor=south west, color=blue]{$a$};
\node at (1.5,2.5)[anchor=south, color=violet]{$c$};
\end{scope}
\draw[line width=1pt, ->,>=stealth] (-.5,-1)--node[right]{$S_{\beta^{L}}^{-3}$}(-.5,-2.5);
\draw[line width=1pt, ->,>=stealth] (5,1)--(7,1);
\node at (6,1)[anchor=south]{$D_\alpha$};
\draw[line width=1pt, ->,>=stealth] (12.5,-1)--node[right]{$S_{\beta^{L}}^{-3}$}(12.5,-2.5);
\begin{scope}[shift={(13,0)}]
\draw[line width=1pt] (-5,0)--(4,0);
\draw[color=red, line width=1.5 pt, ->,>=stealth] (-4,0).. controls (-4,2) and (-1,2) .. (-1,0);
\draw[color=blue, line width=1.5 pt, ->,>=stealth] (0,0).. controls (0,2) and (3,2) .. (3,0);
\draw[color=violet, line width=1.5 pt, ->,>=stealth] (1.5,0)-- (1.5,2.5);
\node at (-2.5,1.5)[anchor=south, color=red]{$b$};
\node at (1.7,1.5)[anchor=south west, color=blue]{$\low a 1$};
\node at (1.5,2.5)[anchor=south, color=violet]{$c ST( \low a 2)$};
\end{scope}
\begin{scope}[shift={(0,-5)}]
\draw[line width=1pt] (-5,0)--(4,0);
\draw[color=blue, line width=1.5 pt, ->,>=stealth] (-4,0).. controls (-4,2) and (-1,2) .. (-1,0);
\draw[color=red, line width=1.5 pt, ->,>=stealth] (0,0).. controls (0,2) and (3,2) .. (3,0);
\draw[color=violet, line width=1.5 pt, ->,>=stealth] (-2.5,0)-- (-2.5,2.5);
\node at (-2.5,0)[anchor=north, color=blue]{$ S^\inv(\low b 3) a  \low b 1$};
\node at (1.7,1.5)[anchor=south west, color=red]{$\low b 4 $};
\node at (-2.5,2.5)[anchor=south, color=violet]{$c  \low b 2$};
\end{scope}
\draw[line width=1pt, ->,>=stealth] (5,-4)--(7,-4);
\node at (6,-4)[anchor=south]{$D_{\alpha}$};
\begin{scope}[shift={(13,-5)}]
\draw[line width=1pt] (-5,0)--(4,0);
\draw[color=blue, line width=1.5 pt, ->,>=stealth] (-4,0).. controls (-4,2) and (-1,2) .. (-1,0);
\draw[color=red, line width=1.5 pt, ->,>=stealth] (0,0).. controls (0,2) and (3,2) .. (3,0);
\draw[color=violet, line width=1.5 pt, ->,>=stealth] (-2.5,0)-- (-2.5,2.5);
\node at (-2.5,0)[anchor=north, color=blue]{$ S^\inv(\low b 3) \low a 1  \low b 1$};
\node at (1.7,1.5)[anchor=south west, color=red]{$\low b 4 $};
\node at (-2.8,2.5)[anchor=south, color=violet]{$c ST(\low a 2) \low b 2$};
\end{scope}
\end{tikzpicture}
\end{align*}
in which 
$b$ is drawn as a loop for notational convenience. 
This proves 1.~for the given edge orientations. The claim for other edge orientations follows by applying the involution $T$.
We verify claim 2. by the following computation, where the loop $\alpha$ carries the label $a$,  the edge $\beta$  the label $b$ and $\beta$ is drawn as a loop for convenience. Note that in this situation, there is only one end $\beta$ between the starting and target end of $\alpha$ by assumption.
\begin{align*}
&\begin{tikzpicture}[scale=.6]
\begin{scope}[shift={(-5,0)}]
	\draw[line width=1pt, color=black](-3,0)--(5,0);
	\draw[color=red, line width=1.5pt, <-,>=stealth] (-2,0).. controls (-2,2) and (2,2)..(2,0);
	\node at (0,1.5) [anchor=south, color=red]{$b$};
	\draw[color=blue, line width=1.5pt, <-,>=stealth] (0,0).. controls (0,2) and (4,2)..(4,0);
	\node at (2,1.5) [anchor=south, color=blue]{$a$};
\end{scope}
\draw[line width=1pt, color=black, ->,>=stealth](0,1)--(2,1);
\node at (1,1) [anchor=south]{$D_{\alpha}$};
\begin{scope}[shift={(5,0)}]
	\draw[line width=1pt, color=black](-3,0)--(5,0);
	\draw[color=red, line width=1.5pt, <-,>=stealth] (-2,0).. controls (-2,2) and (2,2)..(2,0);
	\node at (0,1.5) [anchor=south, color=red]{$bS(a_{(1)})$};
	\draw[color=blue, line width=1.5pt, <-,>=stealth] (0,0).. controls (0,2) and (4,2)..(4,0);
	\node at (2,1.5) [anchor=south west , color=blue]{$a_{(2)}$};
\end{scope}
\draw[line width=1pt, color=black, ->,>=stealth](-4,-1)--(-4,-2);
\node at (-3.9,-1.5) [anchor=west]{$S_{\beta^{L}}S_{\beta^{R}}$};
\draw[line width=1pt, color=black, ->,>=stealth](6,-1)--(6,-2);
\node at (6.1,-1.5) [anchor=west]{$S_{\beta^{L}}S_{\beta^{R}}$};
\begin{scope}[shift={(-5,-4)}]
	\draw[line width=1pt, color=black](-3,0)--(5,0);
	\draw[color=blue, line width=1.5pt, ->,>=stealth] (-2,0).. controls (-2,2) and (2,2)..(2,0);
	\node at (-.5,-.1) [anchor=north, color=blue]{$b_{(1)}aT(b_{(3)})$};
	\draw[color=red, line width=1.5pt, <-,>=stealth] (0,0).. controls (0,2) and (4,2)..(4,0);
	\node at (4,-.1) [anchor=north, color=red]{$b_{(2)}$};
\end{scope}
\draw[line width=1pt, color=black, ->,>=stealth](0,-3)--(2,-3);
\node at (1,-3) [anchor=south]{$D_{\alpha}$};
\begin{scope}[shift={(5,-4)}]
	\draw[line width=1pt, color=black](-3,0)--(5,0);
	\draw[color=blue, line width=1.5pt, ->,>=stealth] (-2,0).. controls (-2,2) and (2,2)..(2,0);
	\node at (-.5,-.1) [anchor=north, color=blue]{$b_{(1)}a_{(1)}T(b_{(3)})$};
	\draw[color=red, line width=1.5pt, <-,>=stealth] (0,0).. controls (0,2) and (4,2)..(4,0);
	\node at (4,-.1) [anchor=north, color=red]{$b_{(2)}S(a_{(2)})$};
\end{scope}
\end{tikzpicture}
\end{align*}
Claim 3. is the pentagon relation from Figure~\ref{fig:sliderel5}  applied to the twist along $\alpha$, which moves  the ends  of $\gamma$ and $\delta$, and the slide along $\delta$.
\end{proof}

The third relation in Lemma \ref{lem:slideDT} implies a more general relation between  slides and twists.

\begin{lemma}
	\label{lemma:SlideDehnCompatibilityGeneral}
	 Let $\gamma$ be a face path and $\rho$ a closed face path based at a ciliated vertex,  satisfying one of the conditions of Definition~\ref{definition:FacePathTranslationComplementary} and   such that $S_{\gamma}(\rho)$ is a face path. Then one has
	\begin{align}
		S_{\gamma} D_{\rho} = D_{S_{\gamma}(\rho)} S_{\gamma}.
		\label{eq:SlideDehnCompatibilityGeneral}
	\end{align}
\end{lemma}
\begin{proof}
1.~We first prove \eqref{eq:SlideDehnCompatibilityGeneral} under the  assumptions that
		(i) $\rho$ and $\gamma$ each involve only a single edge,
		(ii) $\st(\rho)< \ta(\rho)$, (iii) the slide is along the left  of $\gamma$ to the target, (iv)
		the sliding edge end  is not an end of $\rho$. 
		Note that (ii) and (iii) do not restrict generality as twists and slides are invariant under edge orientation reversal. 
	As a consequence of (i)-(iv) we have $S_{\gamma^{L}}(\rho) = \rho$ and $D_{\rho}=S_{\rho^{L}}^{n}$, where $n$ is the number of edge ends between $\st(\rho)$ and $\ta(\rho)$.

	If neither the starting nor the target end of $\gamma$ is between the ends of $\rho$,  it follows directly that  $S_{\gamma^{L}}$ commutes with $D_{\rho}$.
	If only one of the ends of $\gamma$ is between the ends of $\rho$, then $S_{\gamma^{L}}$  moves an edge end that does not belong to $\rho$ into or out of the loop $\rho$. The slide $S_{\gamma^L}$ then commutes with $D_{\rho}$ by Lemma~\ref{lem:slideDT}, 3.
	If both end points of $\gamma$ are between $s(\rho)$ and $t(\rho)$, we verify  that $S_{\gamma^L}$ and $D_{\rho}$ commute by the following direct computation: 	
\begin{align*}
\begin{tikzpicture}[scale=.6]
\begin{scope}
\draw[line width=1pt] (-5,0)--(4,0);
\draw[color=red, line width=1.5 pt, <-,>=stealth] (-4,0).. controls (-4,3) and (3,3) .. (3,0);
\draw[color=blue, line width=1.5 pt, <-,>=stealth] (-2,0).. controls (-2,2) and (1,2) .. (1,0);
\draw[color=violet, line width=1.5 pt, ->,>=stealth] (-.5,0)-- (-.5,3);
\node at (-4,-1)[anchor=south, color=red]{$b$};
\node at (1,-1)[anchor=south, color=blue]{$c$};
\node at (0,3)[anchor=west, color=violet]{$a$};
\end{scope}
\draw[line width=1pt, ->,>=stealth] (-.5,-1)--(-.5,-2.5);
\draw[line width=1pt, ->,>=stealth] (5,1)--(7,1);
\node at (6,1)[anchor=south]{$D_\rho$};
\node at (-.5,-1.5)[anchor=east]{$S_{\gamma}$};
\node at (12.5,-1.5)[anchor=east]{$S_{\gamma}$};
\draw[line width=1pt, ->,>=stealth] (12.5,-1)--(12.5,-2.5);
\begin{scope}[shift={(13,0)}]
\draw[line width=1pt] (-5,0)--(4,0);
\draw[color=red, line width=1.5 pt, <-,>=stealth] (-4,0).. controls (-4,3) and (3,3) .. (3,0);
\draw[color=blue, line width=1.5 pt, <-,>=stealth] (-2,0).. controls (-2,2) and (1,2) .. (1,0);
\draw[color=violet, line width=1.5 pt, ->,>=stealth] (-.5,0)-- (-.5,3);
\node at (-4,-.2)[anchor=north, color=red]{$b_{(4)}$};
\node at (2,-.1)[anchor=north, color=blue]{$b_{(3)}cS(b_{(1)})$};
\node at (0,3)[anchor=west, color=violet]{$aS(b_{(2)})$};
\end{scope}
\begin{scope}[shift={(0,-6)}]
\draw[line width=1pt] (-5,0)--(4,0);
\draw[color=red, line width=1.5 pt, <-,>=stealth] (-4,0).. controls (-4,3) and (3,3) .. (3,0);
\draw[color=blue, line width=1.5 pt, <-,>=stealth] (-2,0).. controls (-2,2) and (1,2) .. (1,0);
\draw[color=violet, line width=1.5 pt, ->,>=stealth] (-.5,0)-- (-.5,3);
\node at (-4,-.1)[anchor=north, color=red]{$b$};
\node at (1,-.1)[anchor=north, color=blue]{$c_{(2)}$};
\node at (0,3)[anchor=west, color=violet]{$aS(c_{(1)})$};
\end{scope}
\draw[line width=1pt, ->,>=stealth] (5,-4)--(7,-4);
\node at (6,-4)[anchor=south]{$D_{S_\gamma(\rho)}$};
\begin{scope}[shift={(13,-6)}]
\draw[line width=1pt] (-5,0)--(4,0);
\draw[color=red, line width=1.5 pt, <-,>=stealth] (-4,0).. controls (-4,3) and (3,3) .. (3,0);
\draw[color=blue, line width=1.5 pt, <-,>=stealth] (-2,0).. controls (-2,2) and (1,2) .. (1,0);
\draw[color=violet, line width=1.5 pt, ->,>=stealth] (-.5,0)-- (-.5,3);
\node at (-4,-.1)[anchor=north, color=red]{$b_{(4)}$};
\node at (1,-.1)[anchor=north, color=blue]{$b_{(3)}c_{(2)}S(b_{(1)})$};
\node at (0,3)[anchor=west, color=violet]{$aS(c_{(1)})S(b_{(2)})$};
\end{scope}
\end{tikzpicture}
\end{align*}
This proves the claim under assumption (i)-(iv) for $\st(\gamma) < \ta(\gamma)$, and an analogous computation proves it  for $\ta(\gamma) < \st(\gamma)$.

2.~If $\rho$ traverses a single edge and $\gamma$ is a face path such that $S_{\gamma}$ does not slide an end of $\rho$, then  $S_{\gamma}(\rho) =\rho$ and 
\eqref{eq:SlideDehnCompatibilityGeneral} follows from 1.~by factoring $S_{\gamma}$ into slides along edges as in Definition~\ref{def:slideface}.

3.~Suppose now that $\gamma$ and $\rho$ are face paths satisfying the conditions of Lemma~\ref{lemma:SlideDehnCompatibilityGeneral}.
	By Definition \ref{def:dtnotface} the twist $D_{\rho}$ can be decomposed as $D_\rho=\epsilon_{\rho'} \circ D_{\rho'} \circ C_{\rho}$. We then obtain
	\begin{align*}
		D_{S_{\gamma}(\rho)}S_{\gamma} 
		=
		\epsilon_{S_{\gamma}(\rho)'}D_{S_{\gamma}(\rho)'}C_{S_{\gamma}(\rho)} S_{\gamma}
		\overset{\text{\eqref{eq:ComplementSlideGeneral}}}{=}
		\epsilon_{\rho'} D_{\rho'} S_{C_{\rho}(\gamma)} C_{\rho}
		\overset{\text{2.}}{=} 
		\epsilon_{\rho'} S_{C_{\rho}(\gamma)} D_{\rho'} C_{\rho}
		=
		S_{\gamma} \epsilon_{\rho'} D_{\rho'}C_{\rho}
		\stackrel{\ref{def:dtnotface}}=
		S_{\gamma} D_{\rho}.
	\end{align*}
Note that we can apply 2.~because $\rho'$ is an edge and $S_{C_{\rho}(\gamma)}$ does not move an end of $\rho'$.
The fourth step is immediate if $\gamma$ does not traverse $\rho'$. Otherwise, it follows by decomposing $S_{\gamma} = S_{\gamma_{1}} S_{(\rho')^{L}}S_{\gamma_{2}}$ and applying the second equation in~\eqref{eq:FlatFaceComplementt} to the face $(\rho')^{-1} \circ\rho$, for which \eqref{eq:FlatFaceComplement} holds.
\end{proof}

\begin{corollary} 
	\label{corollary:SlideIntoFacePathDehnCommute}
	Let $\alpha:v\to w,\beta:w\to v$ be face paths  with $v$ ciliated  and
	 $n$ edge ends between the target end of $\beta$ and the starting end of $\alpha$, such that sliding them along $\alpha$ or $\beta$ yields a closed face path  $\gamma=\alpha\circ\beta$. Then
\begin{align}
	S_{\alpha}^{-n} \circ D_{\gamma}\circ S_{\alpha}^{n} = S_{\beta}^{n} \circ D_{\gamma}\circ S_{\beta}^{-n}.
	\label{eq:SimplePathDehnTwistCoherency}
\end{align}	
\end{corollary}
\begin{proof} Identity \eqref{eq:SimplePathDehnTwistCoherency} is equivalent to the identity $S_\gamma\circ D_\gamma=D_\gamma\circ S_\gamma$ for the face path $\gamma$, which follows directly from 
Lemma~\ref{lemma:SlideDehnCompatibilityGeneral}.
\end{proof}

\begin{corollary}\label{cor:addingtwist}
	 Let $\rho$ be a face path and $\gamma$ a closed face path  based at ciliated vertex such that   one of the conditions of Definition~\ref{definition:FacePathTranslationComplementary} is satisfied. Then  adding an edge to $\rho$ commutes with $D_{\gamma}$:
	\begin{align}
		C_{\rho}\circ D_{\gamma} = D_{C_{\rho}(\gamma)} \circ C_{\rho}.
		\nonumber
	\end{align}
\end{corollary}
\begin{proof}
	This follows from the identity $C_{\rho} = S_{\rho} \circ\eta_{\rho'}$, because   adding a loop at the starting end of a face path $\rho$  and sliding both ends of the loop to its target  is  the same as adding the loop at the target of $\rho$ (see the discussion after Definition \ref{def:addingedge}) and from  identity  \eqref{eq:SlideDehnCompatibilityGeneral}.
\end{proof}

\section{Mapping class group action by Dehn twists}
\label{sec:gervaismap}

In this section we give explicit descriptions of  mapping class group actions for surfaces with and without boundaries in terms of generating Dehn twists. We prove that these Dehn twists satisfy the defining relations of Theorem \ref{th:gervais} for mapping class groups of surfaces with $n+1\geq 1$ boundary components.
We then  show that for  $\mac$ finitely complete and cocomplete and involutive pivotal structures they  induce actions of mapping class groups of closed surfaces.

We consider an oriented surface $\Sigma$ of genus $g\geq 1$ with $n+1\geq 1$ boundary components, as shown in Figure \ref{fig:pi1gens}.
We work with the set of generators of the fundamental group 
$$
 \pi_1(\Sigma)=\langle \mu_1,...,\mu_n,\alpha_1,\beta_1,...,\alpha_g,\beta_g\rangle=F_{n+2g},
$$
depicted in Figure \ref{fig:pi1gens}.
Viewed as a directed ribbon graph, this set of generators of $\pi_1(\Sigma)$ has a single vertex  and  $n+1$ faces, of which $n$ correspond to the loops $\mu_i$ and the last one to   the path $f=\mu_1^\inv\cdots\mu_n^\inv \alpha_1^\inv\beta_1\alpha_1\beta_1^\inv\cdots  \alpha_g^\inv\beta_g\alpha_g\beta_g^\inv$ 
  in Figure \ref{fig:pi1gens}.

\begin{figure}
\centering
\def\svgwidth{.9\columnwidth}
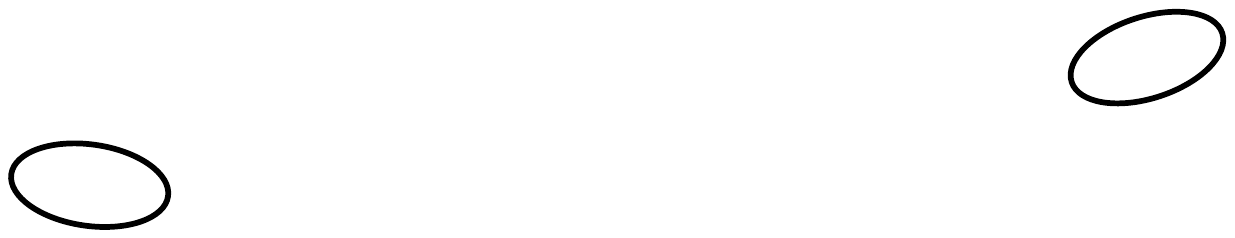
\caption{The generators of the fundamental group $\pi_1(\Sigma)$.}
\label{fig:pi1gens}
\end{figure}

We add a cilium to the vertex as in Figure \ref{fig:pi1gens} and  $n$ edges $\nu_{1},\dots,\nu_{n}$ that connect this vertex to the  boundary components for the loops $\mu_i$  and carry ciliated univalent vertices at their other ends. This yields  a  ribbon graph $\Gamma$ with $n+1$ ciliated vertices and $2(n+g)$ edges given by the following chord diagram  
\begin{align}\label{eq:standardchord}
\begin{tikzpicture}[scale=.6]
\draw[line width=1pt, color=black] (-10,0)--(14,0);
\draw[line width=1.5pt, color=red,<-,>=stealth] (-9,0).. controls (-9,2) and (-7,2)..(-7,0);
\draw[line width=1.5pt, color=blue,<-,>=stealth] (-8,0).. controls (-8,2) and (-6,2)..(-6,0);
\node at (-8,1.5)[anchor=south, color=red]{$\beta_g$};
\node at (-7,1.5)[anchor=south, color=blue]{$\alpha_g$};
\draw[line width=1.5pt, color=red,<-,>=stealth] (-5,0).. controls (-5,2) and (-3,2)..(-3,0);
\draw[line width=1.5pt, color=blue,<-,>=stealth] (-4,0).. controls (-4,2) and (-2,2)..(-2,0);
\node at (-4.3,1.5)[anchor=south, color=red]{$\beta_{g-1}$};
\node at (-2.7,1.5)[anchor=south, color=blue]{$\alpha_{g-1}$};
\node at (-1,1)[anchor=south]{$\ldots$};
\draw[line width=1.5pt, color=red,<-,>=stealth] (0,0).. controls (0,2) and (2,2)..(2,0);
\draw[line width=1.5pt, color=blue,<-,>=stealth] (1,0).. controls (1,2) and (3,2)..(3,0);
\node at (1,1.5)[anchor=south, color=red]{$\beta_{1}$};
\node at (2,1.5)[anchor=south, color=blue]{$\alpha_{1}$};
\draw[line width=1.5pt, color=black,<-,>=stealth] (4,0).. controls (4,2) and (6,2)..(6,0);
\draw[color=violet, line width=1.5pt,<-,>=stealth] (5,0) -- (5,1.2);
\node at (5,0) [anchor=north, color=violet] {$\nu_{n}$};
\node at (5,1.5)[anchor=south, color=black]{$\mu_n$};
\draw[line width=1.5pt, color=black,<-,>=stealth] (7,0).. controls (7,2) and (9,2)..(9,0);
\node at (8,1.5)[anchor=south, color=black]{$\mu_{n-1}$};
\draw[color=violet, line width=1.5pt,<-,>=stealth] (8,0) -- (8,1.2);
\node at (8,0) [anchor=north, color=violet] {$\nu_{n-1}$};
\node at (10,1)[anchor=south]{$\ldots$};
\draw[line width=1.5pt, color=black,<-,>=stealth] (11,0).. controls (11,2) and (13,2)..(13,0);
\node at (12,1.5)[anchor=south, color=black]{$\mu_{1}$};
\draw[color=violet, line width=1.5pt,<-,>=stealth] (12,0) -- (12,1.2);
\node at (12,0) [anchor=north, color=violet] {$\nu_{1}$};
\end{tikzpicture}
\end{align}

Every face and vertex of this graph is assigned to a unique boundary component of $\Sigma$ and to a unique cilium. 
The solid boundary component in Figure \ref{fig:pi1gens} corresponds to the baseline, while the other boundary components correspond to the cilia at the univalent vertices. 
By Lemma \ref{lem:facevertex}, this defines $n+1$ commuting Yetter-Drinfeld module structures on $H^{\oo E}=H^{\oo 2(n+g)}$. 

We now consider the curves $\alpha_i$, $\delta_j$ and $\gamma_{i,j}$ for the generating Dehn twists from Theorem \ref{th:gervais} and choose representatives of these curves in $\pi_1(\Sigma)$. As $\Sigma$ has $n+1$ boundary components, we index the curves $\delta_j$  by $0\leq j\leq n+2g-2$ and the curves $\gamma_{k,l}$ by $0\leq k,l\leq n+2g-2$, with the 0th boundary component  associated with the  baseline in \eqref{eq:standardchord}.
By comparing Figure \ref{fig:pi1gens} with Figures \ref{fig:gervais} and \ref{fig:gervais2}, we 
find that the
curves  from Theorem \ref{th:gervais}  are represented by the following elements of $\pi_1(\Sigma)$
\begingroup
\allowdisplaybreaks
\begin{align}\label{eq:pathexpress}
&\alpha_i & &i=1,...,g\\
&\delta_{0} = \beta_{g}^{-1} \nonumber\\
	&\delta_{i}=\mu_i^{-1}\cdots\mu_n^{-1}  \alpha_1^{-1}\beta_1\alpha_1\beta_1^{-1}\cdots  \alpha^{-1}_{g-1}\beta_{g-1}\alpha_{g-1}\beta^{-1}_{g-1}
	\alpha_g^{-1} \beta_g \alpha_g
& &1\leq i\leq n\nonumber\\
	&\delta_{n+2j-1}= \alpha_j^{-1}\beta_j\alpha_j\beta^{-1}_j\cdots  \alpha^{-1}_{g-1}\beta_{g-1}\alpha_{g-1}\beta^{-1}_{g-1}
	\alpha^{-1}_g \beta_g \alpha_g
	& &1\leq j\leq g-1\nonumber\\
	&\delta_{n+2j}=\beta_j^{-1} \alpha_{j+1}^{-1}\beta_{j+1}\alpha_{j+1}\beta_{j+1}^{-1}\cdots  \alpha_{g-1}^{-1}\beta_{g-1}\alpha_{g-1}\beta_{g-1}^{-1}
	\alpha_g^{-1} \beta_g \alpha_g
	& &1\leq j\leq g-1\nonumber\\
	&\gamma_{i,0}=\delta_i \beta_{g}^{-1} & & 1\leq  i \leq n+2g-2\nonumber\\
	&\gamma_{i,j}=\delta_j \delta_i^{-1} & & 1\leq j < i \leq n+2g-2\nonumber\\
	&\gamma_{0,j}=\delta_j \alpha_{g}^{-1}\beta_{g}^{-1}\alpha_{g} & & 1\leq j  \leq n+2g-2\nonumber\\
	&\gamma_{i,j}=\delta_i\alpha_g  \delta_j^{-1} \alpha_g^{-1}& &1\leq i< j\leq n+2g-2.\nonumber
\end{align}
\endgroup
Paths in the chord diagram \eqref{eq:standardchord} that represent   $\delta_i$  and  $\gamma_{k,l}$ are drawn  in Figures \ref{fig:deltafig} and  \ref{fig:delta}.
Note that  the paths $\alpha_i$ for $1\leq i\leq g$, the paths   $\delta_j$ for $0\leq j\leq n+2g-2$,    the paths  $\gamma_{k,l}$  for $1\leq l<k \leq n+2g-2$ and the paths $\gamma_{0,j}$ and $\gamma_{j,0}$ for $1\leq j\leq n+2g-2$ in \eqref{eq:pathexpress} are face paths. However, this is not true for the paths $\gamma_{k,l}$ with $1\leq k<l \leq n+2g-2$.

\begin{figure}
\begin{center}

\begin{tikzpicture}[scale=.4]
\draw[line width=1pt, color=black](-10,0)--(16,0);
\draw[color=red, line width=1.5pt,<-,>=stealth] (-9,0) .. controls (-9,2) and (-5,2).. (-5,0);
\node at (-7,2)[color=red, anchor=south]{$\beta_g$};
\draw[color=blue, line width=1.5pt,<-,>=stealth] (-7,0) .. controls (-7,2) and (-3,2).. (-3,0);
\node at (-5,2)[color=blue, anchor=south]{$\alpha_g$};
\draw[color=red, line width=1.5pt,<-,>=stealth] (-1,0) .. controls (-1,2) and (3,2).. (3,0);
\node at (1,2)[color=red, anchor=south]{$\beta_{g-1}$};
\draw[color=blue, line width=1.5pt,<-,>=stealth] (1,0) .. controls (1,2) and (5,2).. (5,0);
\node at (3,2)[color=blue, anchor=south]{$\alpha_{g-1}$};
\node at (7,1){$\ldots$};
\draw[color=black, line width=1.5pt,<-,>=stealth] (9,0) .. controls (9,2) and (11,2).. (11,0);
\node at (10,2)[color=black, anchor=south]{$\mu_{i+1}$};
\draw[color=violet, line width=1.5pt,<-,>=stealth] (10,0) -- (10,1.2);
\node at (10,0) [anchor=north, color=violet] {$\nu_{i+1}$};
\draw[color=black, line width=1.5pt,<-,>=stealth] (12,0) .. controls (12,2) and (14,2).. (14,0);
\node at (13,2)[color=black, anchor=south]{$\mu_{i}$};
\draw[color=violet, line width=1.5pt,<-,>=stealth] (13,0) -- (13,1.2);
\node at (13,0) [anchor=north, color=violet] {$\nu_{i}$};
\draw[line width=.5pt, ->,>=stealth] (-3.5,-.5)--(-3.5,.5);
\draw[color=black, line width=.5pt,->,>=stealth] (-3.5,.5) .. controls (-3.5,1.5) and (-6.5,1.5).. (-6.5,.5);
\draw[line width=.5pt, ->,>=stealth] (-6.5,.5)--(-5.5,.5);
\draw[color=black, line width=.5pt,->,>=stealth] (-5.5,.5) .. controls (-5.5,1.5) and (-8.5,1.5).. (-8.5,.5);
\draw[line width=.5pt, ->,>=stealth] (-8.5,.5)--(-7.5,.5);
\draw[color=black, line width=.5pt,->,>=stealth] (-7.5,.5) .. controls (-7.5,2.5) and (-2.5,2.5).. (-2.5,.5);
\draw[line width=.5pt, ->,>=stealth] (-2.5,.5)--(-1.5,.5);
\draw[color=black, line width=.5pt,->,>=stealth] (-1.5,.5) .. controls (-1.5,2.5) and (3.5,2.5).. (3.5,.5);
\draw[line width=.5pt, ->,>=stealth] (3.5,.5)--(4.5,.5);
\draw[color=black, line width=.5pt,->,>=stealth] (4.5,.5) .. controls (4.5,1.5) and (1.5,1.5).. (1.5,.5);
\draw[line width=.5pt, ->,>=stealth] (1.5,.5)--(2.5,.5);
\draw[color=black, line width=.5pt,->,>=stealth] (2.5,.5) .. controls (2.5,1.5) and (-.5,1.5).. (-.5,.5);
\draw[line width=.5pt, ->,>=stealth] (-.5,.5)--(.5,.5);
\draw[color=black, line width=.5pt,->,>=stealth] (.5,.5) .. controls (.5,2.5) and (5.5,2.5).. (5.5,.5);
\draw[line width=.5pt, ->,>=stealth] (5.5,.5)--(6.5,.5);
\draw[line width=.5pt, ->,>=stealth] (7.5,.5)--(8.5,.5);
\draw[color=black, line width=.5pt,->,>=stealth] (8.5,.5) .. controls (8.5,2.5) and (11.5,2.5).. (11.5,.5);
\draw[color=black, line width=.5pt,->,>=stealth] (11.5,.5) .. controls (11.5,2.5) and (14.5,2.5).. (14.5,.5);
\draw[line width=.5pt, ->,>=stealth] (14.5,.5)--(14.5,-.5);
\node at (16,1)[anchor=west]{$\ldots$};
\node at (14.5,0)[anchor=north west]{$\delta_{i}$};
\end{tikzpicture}

\begin{tikzpicture}[scale=.4]
\draw[line width=1pt, color=black](-10,0)--(16,0);
\draw[color=red, line width=1.5pt,<-,>=stealth] (-9,0) .. controls (-9,2) and (-5,2).. (-5,0);
\node at (-7,2)[color=red, anchor=south]{$\beta_g$};
\draw[color=blue, line width=1.5pt,<-,>=stealth] (-7,0) .. controls (-7,2) and (-3,2).. (-3,0);
\node at (-5,2)[color=blue, anchor=south]{$\alpha_g$};
\draw[color=red, line width=1.5pt,<-,>=stealth] (-1,0) .. controls (-1,2) and (3,2).. (3,0);
\node at (1,2)[color=red, anchor=south]{$\beta_{g-1}$};
\draw[color=blue, line width=1.5pt,<-,>=stealth] (1,0) .. controls (1,2) and (5,2).. (5,0);
\node at (3,2)[color=blue, anchor=south]{$\alpha_{g-1}$};
\node at (7,1){$\ldots$};
\draw[color=red, line width=1.5pt,<-,>=stealth] (9,0) .. controls (9,2) and (13,2).. (13,0);
\node at (11,2)[color=red, anchor=south]{$\beta_j$};
\draw[color=blue, line width=1.5pt,<-,>=stealth] (11,0) .. controls (11,2) and (15,2).. (15,0);
\node at (13,2)[color=blue, anchor=south]{$\alpha_{j}$};
\draw[line width=.5pt, ->,>=stealth] (-3.5,-.5)--(-3.5,.5);
\draw[color=black, line width=.5pt,->,>=stealth] (-3.5,.5) .. controls (-3.5,1.5) and (-6.5,1.5).. (-6.5,.5);
\draw[line width=.5pt, ->,>=stealth] (-6.5,.5)--(-5.5,.5);
\draw[color=black, line width=.5pt,->,>=stealth] (-5.5,.5) .. controls (-5.5,1.5) and (-8.5,1.5).. (-8.5,.5);
\draw[line width=.5pt, ->,>=stealth] (-8.5,.5)--(-7.5,.5);
\draw[color=black, line width=.5pt,->,>=stealth] (-7.5,.5) .. controls (-7.5,2.5) and (-2.5,2.5).. (-2.5,.5);
\draw[line width=.5pt, ->,>=stealth] (-2.5,.5)--(-1.5,.5);
\draw[color=black, line width=.5pt,->,>=stealth] (-1.5,.5) .. controls (-1.5,2.5) and (3.5,2.5).. (3.5,.5);
\draw[line width=.5pt, ->,>=stealth] (3.5,.5)--(4.5,.5);
\draw[color=black, line width=.5pt,->,>=stealth] (4.5,.5) .. controls (4.5,1.5) and (1.5,1.5).. (1.5,.5);
\draw[line width=.5pt, ->,>=stealth] (1.5,.5)--(2.5,.5);
\draw[color=black, line width=.5pt,->,>=stealth] (2.5,.5) .. controls (2.5,1.5) and (-.5,1.5).. (-.5,.5);
\draw[line width=.5pt, ->,>=stealth] (-.5,.5)--(.5,.5);
\draw[color=black, line width=.5pt,->,>=stealth] (.5,.5) .. controls (.5,2.5) and (5.5,2.5).. (5.5,.5);
\draw[line width=.5pt, ->,>=stealth] (5.5,.5)--(6.5,.5);
\draw[line width=.5pt, ->,>=stealth] (7.5,.5)--(8.5,.5);
\draw[color=black, line width=.5pt,->,>=stealth] (8.5,.5) .. controls (8.5,2.5) and (13.5,2.5).. (13.5,.5);
\draw[line width=.5pt, ->,>=stealth] (13.5,.5)--(14.5,.5);
\draw[color=black, line width=.5pt,->,>=stealth] (14.5,.5) .. controls (14.5,1.5) and (11.5,1.5).. (11.5,.5);
\draw[line width=.5pt, ->,>=stealth] (11.5,.5)--(12.5,.5);
\draw[color=black, line width=.5pt,->,>=stealth] (12.5,.5) .. controls (12.5,1.5) and (9.5,1.5).. (9.5,.5);
\draw[line width=.5pt, ->,>=stealth] (9.5,.5)--(10.5,.5);
\draw[color=black, line width=.5pt,->,>=stealth] (10.5,.5) .. controls (10.5,2.5) and (15.5,2.5).. (15.5,.5);
\draw[line width=.5pt, ->,>=stealth] (15.5,.5)--(15.5,-.5);
\node at (16,1)[anchor=west]{$\ldots$};
\node at (15.5,0)[anchor=north west]{$\delta_{n+2j-1}$};
\end{tikzpicture}

\begin{tikzpicture}[scale=.4]
\draw[line width=1pt, color=black](-10,0)--(16,0);
\draw[color=red, line width=1.5pt,<-,>=stealth] (-9,0) .. controls (-9,2) and (-5,2).. (-5,0);
\node at (-7,2)[color=red, anchor=south]{$\beta_g$};
\draw[color=blue, line width=1.5pt,<-,>=stealth] (-7,0) .. controls (-7,2) and (-3,2).. (-3,0);
\node at (-5,2)[color=blue, anchor=south]{$\alpha_g$};
\draw[color=red, line width=1.5pt,<-,>=stealth] (-1,0) .. controls (-1,2) and (3,2).. (3,0);
\node at (1,2)[color=red, anchor=south]{$\beta_{g-1}$};
\draw[color=blue, line width=1.5pt,<-,>=stealth] (1,0) .. controls (1,2) and (5,2).. (5,0);
\node at (3,2)[color=blue, anchor=south]{$\alpha_{g-1}$};
\node at (7,1){$\ldots$};
\draw[color=red, line width=1.5pt,<-,>=stealth] (9,0) .. controls (9,2) and (13,2).. (13,0);
\node at (11,2)[color=red, anchor=south]{$\beta_j$};
\draw[color=blue, line width=1.5pt,<-,>=stealth] (11,0) .. controls (11,2) and (15,2).. (15,0);
\node at (15,1)[color=blue, anchor=south]{$\alpha_{j}$};
\draw[line width=.5pt, ->,>=stealth] (-3.5,-.5)--(-3.5,.5);
\draw[color=black, line width=.5pt,->,>=stealth] (-3.5,.5) .. controls (-3.5,1.5) and (-6.5,1.5).. (-6.5,.5);
\draw[line width=.5pt, ->,>=stealth] (-6.5,.5)--(-5.5,.5);
\draw[color=black, line width=.5pt,->,>=stealth] (-5.5,.5) .. controls (-5.5,1.5) and (-8.5,1.5).. (-8.5,.5);
\draw[line width=.5pt, ->,>=stealth] (-8.5,.5)--(-7.5,.5);
\draw[color=black, line width=.5pt,->,>=stealth] (-7.5,.5) .. controls (-7.5,2.5) and (-2.5,2.5).. (-2.5,.5);
\draw[line width=.5pt, ->,>=stealth] (-2.5,.5)--(-1.5,.5);
\draw[color=black, line width=.5pt,->,>=stealth] (-1.5,.5) .. controls (-1.5,2.5) and (3.5,2.5).. (3.5,.5);
\draw[line width=.5pt, ->,>=stealth] (3.5,.5)--(4.5,.5);
\draw[color=black, line width=.5pt,->,>=stealth] (4.5,.5) .. controls (4.5,1.5) and (1.5,1.5).. (1.5,.5);
\draw[line width=.5pt, ->,>=stealth] (1.5,.5)--(2.5,.5);
\draw[color=black, line width=.5pt,->,>=stealth] (2.5,.5) .. controls (2.5,1.5) and (-.5,1.5).. (-.5,.5);
\draw[line width=.5pt, ->,>=stealth] (-.5,.5)--(.5,.5);
\draw[color=black, line width=.5pt,->,>=stealth] (.5,.5) .. controls (.5,2.5) and (5.5,2.5).. (5.5,.5);
\draw[line width=.5pt, ->,>=stealth] (5.5,.5)--(6.5,.5);
\draw[line width=.5pt, ->,>=stealth] (7.5,.5)--(8.5,.5);
\draw[color=black, line width=.5pt,->,>=stealth] (8.5,.5) .. controls (8.5,2.5) and (13.5,2.5).. (13.5,.5);
\draw[line width=.5pt, ->,>=stealth] (13.5,.5)--(13.5,-.5);
\node at (16,1)[anchor=west]{$\ldots$};
\node at (13.5,0)[anchor=north west]{$\delta_{n+2j}$};
\end{tikzpicture}
\end{center}
\vspace{-.5cm}
\caption{Paths in the chord diagram \eqref{eq:standardchord} representing the paths  $\delta_i$  from \eqref{eq:pathexpress}.}
\label{fig:deltafig}
\end{figure}
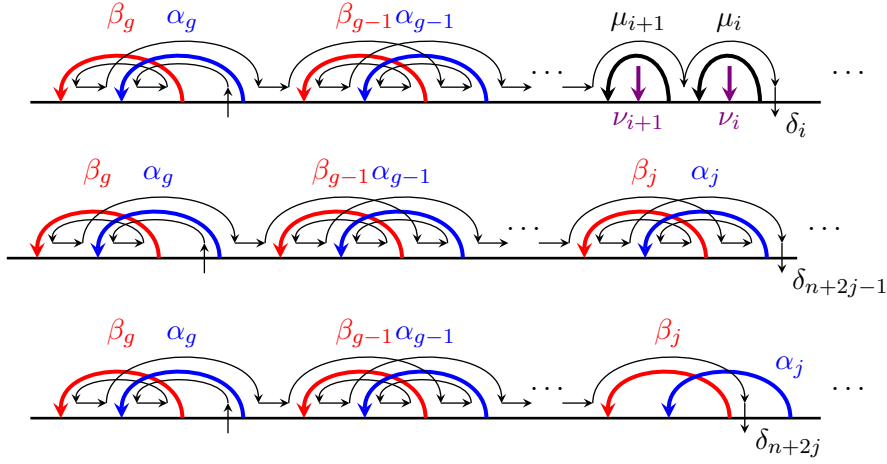

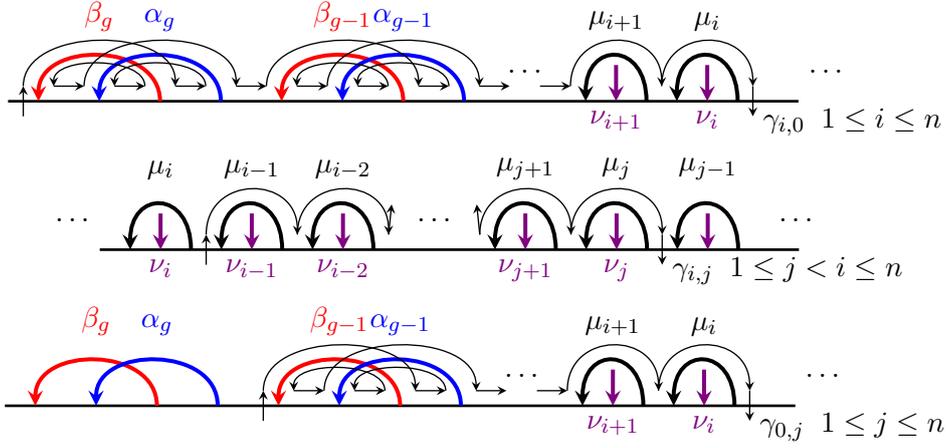
\begin{figure}
\begin{center}
\begin{tikzpicture}[scale=.4]
\draw[line width=1pt, color=black](-10,0)--(16,0);
\draw[color=red, line width=1.5pt,<-,>=stealth] (-9,0) .. controls (-9,2) and (-5,2).. (-5,0);
\node at (-7,2)[color=red, anchor=south]{$\beta_g$};
\draw[color=blue, line width=1.5pt,<-,>=stealth] (-7,0) .. controls (-7,2) and (-3,2).. (-3,0);
\node at (-5,2)[color=blue, anchor=south]{$\alpha_g$};
\draw[color=red, line width=1.5pt,<-,>=stealth] (-1,0) .. controls (-1,2) and (3,2).. (3,0);
\node at (1,2)[color=red, anchor=south]{$\beta_{g-1}$};
\draw[color=blue, line width=1.5pt,<-,>=stealth] (1,0) .. controls (1,2) and (5,2).. (5,0);
\node at (3,2)[color=blue, anchor=south]{$\alpha_{g-1}$};
\node at (7,1){$\ldots$};
\draw[color=black, line width=1.5pt,<-,>=stealth] (9,0) .. controls (9,2) and (11,2).. (11,0);
\node at (10,2)[color=black, anchor=south]{$\mu_{i+1}$};
\draw[color=violet, line width=1.5pt,<-,>=stealth] (10,0) -- (10,1.2);
\node at (10,0) [anchor=north, color=violet] {$\nu_{i+1}$};
\draw[color=black, line width=1.5pt,<-,>=stealth] (12,0) .. controls (12,2) and (14,2).. (14,0);
\node at (13,2)[color=black, anchor=south]{$\mu_{i}$};
\draw[color=violet, line width=1.5pt,<-,>=stealth] (13,0) -- (13,1.2);
\node at (13,0) [anchor=north, color=violet] {$\nu_{i}$};
\draw[line width=.5pt, ->,>=stealth] (-9.5,-.5)--(-9.5,.5);
\draw[color=black, line width=.5pt,->,>=stealth] (-9.5,.5) .. controls (-9.5,2.5) and (-4.5,2.5).. (-4.5,.5);
\draw[line width=.5pt, ->,>=stealth] (-4.5,.5)--(-3.5,.5);
\draw[color=black, line width=.5pt,->,>=stealth] (-3.5,.5) .. controls (-3.5,1.5) and (-6.5,1.5).. (-6.5,.5);
\draw[line width=.5pt, ->,>=stealth] (-6.5,.5)--(-5.5,.5);
\draw[color=black, line width=.5pt,->,>=stealth] (-5.5,.5) .. controls (-5.5,1.5) and (-8.5,1.5).. (-8.5,.5);
\draw[line width=.5pt, ->,>=stealth] (-8.5,.5)--(-7.5,.5);
\draw[color=black, line width=.5pt,->,>=stealth] (-7.5,.5) .. controls (-7.5,2.5) and (-2.5,2.5).. (-2.5,.5);
\draw[line width=.5pt, ->,>=stealth] (-2.5,.5)--(-1.5,.5);
\draw[color=black, line width=.5pt,->,>=stealth] (-1.5,.5) .. controls (-1.5,2.5) and (3.5,2.5).. (3.5,.5);
\draw[line width=.5pt, ->,>=stealth] (3.5,.5)--(4.5,.5);
\draw[color=black, line width=.5pt,->,>=stealth] (4.5,.5) .. controls (4.5,1.5) and (1.5,1.5).. (1.5,.5);
\draw[line width=.5pt, ->,>=stealth] (1.5,.5)--(2.5,.5);
\draw[color=black, line width=.5pt,->,>=stealth] (2.5,.5) .. controls (2.5,1.5) and (-.5,1.5).. (-.5,.5);
\draw[line width=.5pt, ->,>=stealth] (-.5,.5)--(.5,.5);
\draw[color=black, line width=.5pt,->,>=stealth] (.5,.5) .. controls (.5,2.5) and (5.5,2.5).. (5.5,.5);
\draw[line width=.5pt, ->,>=stealth] (5.5,.5)--(6.5,.5);
\draw[line width=.5pt, ->,>=stealth] (7.5,.5)--(8.5,.5);
\draw[color=black, line width=.5pt,->,>=stealth] (8.5,.5) .. controls (8.5,2.5) and (11.5,2.5).. (11.5,.5);
\draw[color=black, line width=.5pt,->,>=stealth] (11.5,.5) .. controls (11.5,2.5) and (14.5,2.5).. (14.5,.5);
\draw[line width=.5pt, ->,>=stealth] (14.5,.5)--(14.5,-.5);
\node at (16,1)[anchor=west]{$\ldots$};
\node at (14.5,-.7)[anchor=west]{$\gamma_{i,0} \;\; 1\leq i\leq n$};
\end{tikzpicture}
\begin{tikzpicture}[scale=.4]
\draw[line width=1pt, color=black](-13,0)--(10,0);
\node at (-13,1)[anchor=east]{$\ldots$};
\draw[color=black, line width=1.5pt,<-,>=stealth] (-12,0) .. controls (-12,2) and (-10,2).. (-10,0);
\node at (-11,2)[color=black, anchor=south]{$\mu_{i}$};
\draw[color=violet, line width=1.5pt,<-,>=stealth] (-11,0) -- (-11,1.2);
\node at (-11,0) [anchor=north, color=violet] {$\nu_{i}$};
\draw[color=black, line width=1.5pt,<-,>=stealth] (-9,0) .. controls (-9,2) and (-7,2).. (-7,0);
\node at (-8,2)[color=black, anchor=south]{$\mu_{i-1}$};
\draw[color=violet, line width=1.5pt,<-,>=stealth] (-8,0) -- (-8,1.2);
\node at (-8,0) [anchor=north, color=violet] {$\nu_{i-1}$};
\draw[color=black, line width=1.5pt,<-,>=stealth] (-6,0) .. controls (-6,2) and (-4,2).. (-4,0);
\node at (-5,2)[color=black, anchor=south]{$\mu_{i-2}$};
\draw[color=violet, line width=1.5pt,<-,>=stealth] (-5,0) -- (-5,1.2);
\node at (-5,0) [anchor=north, color=violet] {$\nu_{i-2}$};
\node at (-2,1){$\ldots$};
\draw[color=black, line width=1.5pt,<-,>=stealth] (0,0) .. controls (0,2) and (2,2).. (2,0);
\node at (1,2)[color=black, anchor=south]{$\mu_{j+1}$};
\draw[color=violet, line width=1.5pt,<-,>=stealth] (1,0) -- (1,1.2);
\node at (1,0) [anchor=north, color=violet] {$\nu_{j+1}$};
\draw[color=black, line width=1.5pt,<-,>=stealth] (3,0) .. controls (3,2) and (5,2).. (5,0);
\node at (4,2)[color=black, anchor=south]{$\mu_{j}$};
\draw[color=violet, line width=1.5pt,<-,>=stealth] (4,0) -- (4,1.2);
\node at (4,0) [anchor=north, color=violet] {$\nu_{j}$};
\draw[color=black, line width=1.5pt,<-,>=stealth] (6,0) .. controls (6,2) and (8,2).. (8,0);
\node at (7,2)[color=black, anchor=south]{$\mu_{j-1}$};
\draw[color=violet, line width=1.5pt,<-,>=stealth] (7,0) -- (7,1.2);
\draw[line width=.5pt, ->,>=stealth] (-9.5,-.5)--(-9.5,.5);
\draw[color=black, line width=.5pt,->,>=stealth] (-9.5,.5) .. controls (-9.5,2.5) and (-6.5,2.5).. (-6.5,.5);
\draw[color=black, line width=.5pt,->,>=stealth] (-6.5,.5) .. controls (-6.5,2.5) and (-3.5,2.5).. (-3.5,.5);
\draw[line width=.5pt, ->,>=stealth] (-3.5,.5)--(-3.4,1.5);
\draw[line width=.5pt, ->,>=stealth] (-.5,.5)--(-.6,1.5);
\draw[color=black, line width=.5pt,->,>=stealth] (-.5,.5) .. controls (-.5,2.5) and (2.5,2.5).. (2.5,.5);
\draw[color=black, line width=.5pt,->,>=stealth] (2.5,.5) .. controls (2.5,2.5) and (5.5,2.5).. (5.5,.5);
\draw[line width=.5pt, ->,>=stealth] (5.5,.5)--(5.5,-.5);
\node at (9,1)[anchor=west]{$\ldots$};
\node at (5.5,-.7)[anchor=west]{$\gamma_{i,j}\;\; 1\leq j<i\leq n$};
\end{tikzpicture}

\begin{tikzpicture}[scale=.4]
\draw[line width=1pt, color=black](-10,0)--(16,0);
\draw[color=red, line width=1.5pt,<-,>=stealth] (-9,0) .. controls (-9,2) and (-5,2).. (-5,0);
\node at (-7,2)[color=red, anchor=south]{$\beta_g$};
\draw[color=blue, line width=1.5pt,<-,>=stealth] (-7,0) .. controls (-7,2) and (-3,2).. (-3,0);
\node at (-5,2)[color=blue, anchor=south]{$\alpha_g$};
\draw[color=red, line width=1.5pt,<-,>=stealth] (-1,0) .. controls (-1,2) and (3,2).. (3,0);
\node at (1,2)[color=red, anchor=south]{$\beta_{g-1}$};
\draw[color=blue, line width=1.5pt,<-,>=stealth] (1,0) .. controls (1,2) and (5,2).. (5,0);
\node at (3,2)[color=blue, anchor=south]{$\alpha_{g-1}$};
\node at (7,1){$\ldots$};
\draw[color=black, line width=1.5pt,<-,>=stealth] (9,0) .. controls (9,2) and (11,2).. (11,0);
\node at (10,2)[color=black, anchor=south]{$\mu_{i+1}$};
\draw[color=violet, line width=1.5pt,<-,>=stealth] (10,0) -- (10,1.2);
\node at (10,0) [anchor=north, color=violet] {$\nu_{i+1}$};
\draw[color=black, line width=1.5pt,<-,>=stealth] (12,0) .. controls (12,2) and (14,2).. (14,0);
\node at (13,2)[color=black, anchor=south]{$\mu_{i}$};
\draw[color=violet, line width=1.5pt,<-,>=stealth] (13,0) -- (13,1.2);
\node at (13,0) [anchor=north, color=violet] {$\nu_{i}$};
\draw[line width=.5pt, ->,>=stealth] (-1.5,-.5)--(-1.5,.5);
\draw[color=black, line width=.5pt,->,>=stealth] (-1.5,.5) .. controls (-1.5,2.5) and (3.5,2.5).. (3.5,.5);
\draw[line width=.5pt, ->,>=stealth] (3.5,.5)--(4.5,.5);
\draw[color=black, line width=.5pt,->,>=stealth] (4.5,.5) .. controls (4.5,1.5) and (1.5,1.5).. (1.5,.5);
\draw[line width=.5pt, ->,>=stealth] (1.5,.5)--(2.5,.5);
\draw[color=black, line width=.5pt,->,>=stealth] (2.5,.5) .. controls (2.5,1.5) and (-.5,1.5).. (-.5,.5);
\draw[line width=.5pt, ->,>=stealth] (-.5,.5)--(.5,.5);
\draw[color=black, line width=.5pt,->,>=stealth] (.5,.5) .. controls (.5,2.5) and (5.5,2.5).. (5.5,.5);
\draw[line width=.5pt, ->,>=stealth] (5.5,.5)--(6.5,.5);
\draw[line width=.5pt, ->,>=stealth] (7.5,.5)--(8.5,.5);
\draw[color=black, line width=.5pt,->,>=stealth] (8.5,.5) .. controls (8.5,2.5) and (11.5,2.5).. (11.5,.5);
\draw[color=black, line width=.5pt,->,>=stealth] (11.5,.5) .. controls (11.5,2.5) and (14.5,2.5).. (14.5,.5);
\draw[line width=.5pt, ->,>=stealth] (14.5,.5)--(14.5,-.5);
\node at (16,1)[anchor=west]{$\ldots$};
\node at (14.5,-.7)[anchor=west]{$\gamma_{0,j}\;\; 1\leq j\leq n$};
\end{tikzpicture}
\end{center}
\vspace{-.5cm}
\caption{Paths in the chord diagram \eqref{eq:standardchord} representing  the paths   $\gamma_{i,j}$  from \eqref{eq:pathexpress}.}
\label{fig:delta}
\end{figure}

\begin{figure}
\begin{center}
\begin{tikzpicture}[scale=.4]
\draw[line width=1pt, color=black] (-5,0)--(12,0);
\draw[color=blue, line width=1.5 pt, <-,>=stealth] (0,0).. controls (0,3) and (5,3) .. (5,0);
\draw[color=red, line width=1.5 pt, ->,>=stealth] (1,0).. controls (1,3) and (-4,3)..(-4,0);
\draw[color=violet, line width=1.5 pt, ->,>=stealth] (2,0).. controls (2,4) and (11,4) .. (11,0);
\draw[color=orange, line width=1.5 pt, ->,>=stealth] (4,0).. controls (4,3) and (9,3) .. (9,0);
\draw[color=gray, line width=1.5 pt, <-,>=stealth] (7,0)--(7,4);
\draw[color=olive, line width=1.5 pt, <-,>=stealth] (10,0)--(10,4);
\node at (8.5,3.3) [color=violet, anchor=south]{$\delta'_i$};
\node at (8.5,1.5) [color=orange, anchor=south]{$\delta'_j$};
\node at (2.5,2.8) [color=blue, anchor=south]{$\alpha_g$};
\node at (-1.5,2.3)[color=red, anchor=south] {$\beta_g$};
\node at (7,4)[color=gray, anchor=south] {$\nu$};
\node at (10,4)[color=olive, anchor=south] {$\mu$};
\draw[line width=.5pt, color=black,->,>=stealth] (-.5,.3).. controls (-.5,3.5) and (5.5,3.5).. (5.5,.5);
\draw[line width=.5pt, color=black,->,>=stealth] (5.5,.5)--(8.5,.5);     
\draw[line width=.5pt, color=black,] (8.5,.5).. controls (8.5,2.5) and (4.5,2.5) .. (4.5,.3);
\draw[line width=.5pt, color=black,->,>=stealth] (4.5,.5).. controls (4.5,2.2) and (.5,2.2) .. (.5,.5);
\draw[line width=.5pt, color=black,->,>=stealth] (.5,.5)--(1.5,.5); 
\draw[line width=.5pt, color=black,->,>=stealth] (1.5,.5).. controls (1.5,4.5) and (11.5, 4.5) .. (11.5,.5);
\node at (11.5,.5)[anchor=west]{$\gamma$};
\end{tikzpicture}
\end{center}
\caption{The path $\gamma'_{i,j}=\delta'_i\circ \alpha_g\circ \delta'^\inv_j\circ\alpha_g^\inv$ in Definition \ref{definition:DehnTwistGammaNonFace}}
\label{fig:gammaijpath}
\end{figure}
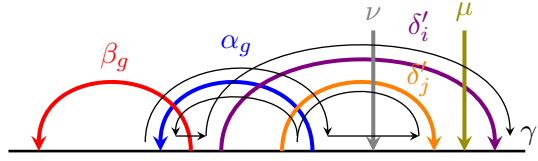

Given a pivotal Hopf monoid $H$ in a symmetric monoidal category $\mac$, we associate to each of the paths $\rho$  in \eqref{eq:pathexpress} an isomorphism $D_\rho: H^{\oo 2(n+g)}\to H^{\oo 2(n+g)}$. 
For the  paths $\rho$  in \eqref{eq:pathexpress} that are face paths, this isomorphism $D_\rho$  is the twist along $\rho$  from Definition \ref{def:dtnotface}. 
For the remaining paths $\gamma_{i,j}=\delta_i\circ\alpha_g\circ \delta_j^\inv\circ \alpha_g^\inv$ with $1\leq i<j \leq n+2g-2$  we define the twist as follows.

\begin{definition}
	\label{definition:DehnTwistGammaNonFace} For $1\leq i<j\leq n+2g-2$ the twist 
along  $\gamma_{i,j}$   is the  morphism $D_{\gamma_{i,j}}: H^{\oo E}\to H^{\oo E}$ defined as the following composite:
\begin{compactenum}
	\item Add edges $\delta'_{i}$ and $\delta'_{j}$ to the face paths $\delta_{i}$ and $\delta_{j}$, as shown in Figure \ref{fig:gammaijpath}.
	\item Slide the edge ends between the starting end of $\alpha_{g}$ and the target  end of $\delta'_{j}$ along  $\alpha_{g}$ and slide the starting end of $\beta_{g}$ along $\delta'_{i}$. This yields a face path  $$\gamma^*_{i,j}:= \delta'_{i}\circ \alpha_{g}\circ  \delta'^\inv_{j}\circ \alpha_{g}^{-1}.$$
	\item Perform a twist along the face path $\gamma^*_{i,j}$.
	\item Reverse the slides of step 2., moving back all edge ends into their original position.
	\item Remove the edges $\delta'_{i}$ and $\delta'_{j}$.
\end{compactenum}
\end{definition}

Note that the choices in Definition \ref{definition:DehnTwistGammaNonFace} do not affect the result. Instead of sliding the edge ends between the starting end of $\alpha_g$ and the target end of $\delta'_j$ along  the right of $\alpha_g$, we could also slide them along the face path $\alpha_g\circ \delta'^\inv_j$ and then along $\delta'_i$ together with the starting end of $\beta_g$. Similarly, instead of sliding the starting end of $\beta_g$ along $\delta'_i$, we could  slide it along 
 the face path $\alpha_g\circ\delta'_j$ and then along $\alpha_g$ together with the edge ends between the starting end of $\alpha_g$ and the target end of $\delta'_j$. Corollary \ref{corollary:SlideIntoFacePathDehnCommute} ensures that this yields the same morphisms in $\mac$.

Note also that neither the slides along the face paths in \eqref{eq:pathexpress} nor the slides along the paths $\gamma_{i,j}$ in Definition \ref{definition:DehnTwistGammaNonFace} slide  edge ends over the cilia. Proposition \ref{prop:vertfacecomp} then implies that  the associated automorphisms $D_\rho$  of $H^{\oo 2(n+g)}$ are automorphisms of Yetter-Drinfeld modules with respect to the Yetter-Drinfeld module structures from Lemma \ref{lem:facevertex} for each cilium. 

\begin{theorem}\label{th:maptheorem} Let $H$ be a pivotal Hopf monoid in a symmetric monoidal category $\mac$, 
$\Sigma$ a surface of genus $g\geq 1$ with $n+1\geq 1$ boundary components
and  $\Gamma$ the embedded  graph from \eqref{eq:standardchord}.

Then the  twists along the paths 
$\alpha_i$, $\delta_j$, $\gamma_{k,l}$ 
from \eqref{eq:pathexpress} for $i\in\{1,\ldots, g\}$, $ j\in\{0,\ldots, n+2g-2\}$ and $ k\neq l\in\{0,\ldots, n+2g-2\}$
satisfy the relations in Theorem \ref{th:gervais} and 
define a group homomorphism
$$\rho:\mathrm{Map}(\Sigma)\to \mathrm{Aut}_{YD}(H^{\oo 2(n+g)}).$$
\end{theorem}

\begin{proof} 
We use the
 results on twists from Section \ref{sec:twists} to verify the relations in Theorem \ref{th:gervais} by explicit computations and structural arguments.

{\bf Relation (i):}
For $1\leq j\leq g-1$ we have from \eqref{eq:pathexpress}  
$$\gamma_{n+2j+1, n+2j}=\delta_{n+2j}\circ \delta_{n+2j+1}^{-1}=\beta_j^\inv\qquad 
\gamma_{n+2j, n+2j-1}=\delta_{n+2j-1}\circ \delta_{n+2j}^\inv=\alpha_j^\inv\beta_j\alpha_j,$$
where we set $\gamma_{n+2g-1, n+2g-2}:=\gamma_{0,n+2g-2}=\beta_{g-1}^\inv$, as in Theorem \ref{th:gervais}. 
We consider the action of  the associated   twists for the diagram \eqref{eq:standardchord}. They affect only the labels of the  chords $\alpha_j$ and $\beta_j$, and their action on these edge labels is given by 
\begin{align*}
&\begin{tikzpicture}[scale=.5]
\node at (-14.5,1)[anchor=west]{$D_{\gamma_{n+2j, n+2j-1}}:$};
\draw[line width=1pt, color=black] (-10,0)--(-2,0);
\draw[line width=1.5pt, color=red,<-,>=stealth] (-9,0).. controls (-9,2) and (-5,2)..(-5,0);
\draw[line width=1.5pt, color=blue,<-,>=stealth] (-7,0).. controls (-7,2) and (-3,2)..(-3,0);
\node at (-7,1.5)[anchor=south, color=red]{$b$};
\node at (-5,1.5)[anchor=south, color=blue]{$a$};
\draw[line width=1pt, ->,>=stealth] (-1.5,1)--(-.5,1);
\begin{scope}[shift={(10,0)}]
\draw[line width=1pt, color=black] (-10,0)--(-2,0);
\draw[line width=1.5pt, color=red,<-,>=stealth] (-9,0).. controls (-9,2) and (-5,2)..(-5,0);
\draw[line width=1.5pt, color=blue,<-,>=stealth] (-7,0).. controls (-7,2) and (-3,2)..(-3,0);
\node at (-7,1.5)[anchor=south, color=red]{$\low b 2$};
\node at (-5,1.5)[anchor=south, color=blue]{$\low b 1 a$};
\end{scope}
\end{tikzpicture}\\
&\begin{tikzpicture}[scale=.5]
\node at (-14.5,1)[anchor=west]{$D_{\gamma_{n+2j+1, n+2j}}:$};
\draw[line width=1pt, color=black] (-10,0)--(-2,0);
\draw[line width=1.5pt, color=red,<-,>=stealth] (-9,0).. controls (-9,2) and (-5,2)..(-5,0);
\draw[line width=1.5pt, color=blue,<-,>=stealth] (-7,0).. controls (-7,2) and (-3,2)..(-3,0);
\node at (-7,1.5)[anchor=south, color=red]{$b$};
\node at (-5,1.5)[anchor=south, color=blue]{$a$};
\draw[line width=1pt, ->,>=stealth] (-1.5,1)--(-.5,1);
\begin{scope}[shift={(10,0)}]
\draw[line width=1pt, color=black] (-10,0)--(-1,0);
\draw[line width=1.5pt, color=red,<-,>=stealth] (-9,0).. controls (-9,2) and (-5,2)..(-5,0);
\draw[line width=1.5pt, color=blue,<-,>=stealth] (-7,0).. controls (-7,2) and (-3,2)..(-3,0);
\draw[line width=1.5pt, color=black, ->,>=stealth, style=dashed] (-4,0).. controls (-4,2) and (-2,2) .. (-2,0);
\node at (-2,1.5)[anchor=south]{$pS(\low a 3) \low b 1\low a 1$};
\node at (-7,1.5)[anchor=south, color=red]{$\low b 2$};
\node at (-3,0)[anchor=north, color=blue]{$\low a 2$};
\end{scope}
\begin{scope}[shift={(0,-3)}]
\draw[line width=1pt, ->,>=stealth] (-11.5,1)--(-10.5,1);
\draw[line width=1pt, color=black] (-10,0)--(-1,0);
\draw[line width=1.5pt, color=red,<-,>=stealth] (-9,0).. controls (-9,2) and (-5,2)..(-5,0);
\draw[line width=1.5pt, color=blue,<-,>=stealth] (-7,0).. controls (-7,2) and (-3,2)..(-3,0);
\draw[line width=1.5pt, color=black, ->,>=stealth, style=dashed] (-4,0).. controls (-4,2) and (-2,2) .. (-2,0);
\node at (-2,1.5)[anchor=south]{$pS(\low a 3) \low b 1\low a 1$};
\node at (-7,1.5)[anchor=south, color=red]{$\low b 3$};
\node at (-3,0)[anchor=north, color=blue]{$\low b 2\low a 2$};
\end{scope}
\begin{scope}[shift={(10,-3)}]
\draw[line width=1pt, ->,>=stealth] (-11.5,1)--(-10.5,1);
\draw[line width=1pt, color=black] (-10,0)--(-2,0);
\draw[line width=1.5pt, color=red,<-,>=stealth] (-9,0).. controls (-9,2) and (-5,2)..(-5,0);
\draw[line width=1.5pt, color=blue,<-,>=stealth] (-7,0).. controls (-7,2) and (-3,2)..(-3,0);
\node at (-7,1.5)[anchor=south, color=red]{$\low b 2$};
\node at (-5,1.5)[anchor=south, color=blue]{$\low b 1 a$};
\end{scope}
\end{tikzpicture}
\end{align*}
{\bf  Relation  (ii): }  For curves  $\alpha,\beta$ with $|\alpha\cap \beta|=0$ that  are represented by face paths, Definition \ref{def:dtnotface} and Lemma \ref{rem:shorthand} allow us to assume without loss of generality that the  chord diagram that defines the   twists $D_\alpha$ and $D_\beta$ takes the form
\begin{align}\label{eq:commutenonint}
\begin{tikzpicture}[scale=.3]
\draw[line width=1pt, color=black] (-6,0)--(8,0);
\draw[color=red, line width=1.5pt, <-,>=stealth] (-4,0) .. controls (-4,2) and (0,2) .. (0,0);
\draw[color=blue, line width=1.5pt, <-,>=stealth] (2,0) .. controls (2,2) and (6,2) .. (6,0);
\draw[color=olive, line width=1.5pt,<-,>=stealth] (-5,0)--(-5,4);
\draw[color=orange, line width=1.5pt,<-,>=stealth] (-2,0)--(-2,4);
\draw[color=cyan, line width=1.5pt,<-,>=stealth] (4,0)--(4,4);
\draw[color=violet, line width=1.5pt,<-,>=stealth] (1,0)--(1,4);
\draw[color=brown, line width=1.5pt,<-,>=stealth] (7,0)--(7,4);
\node at (-1,1.5)[anchor=south,color=red]{$\beta'$};
\node at (5,1.5)[anchor=south,color=blue]{$\alpha'$};
\end{tikzpicture}
\qquad or \qquad
\begin{tikzpicture}[scale=.3]
\draw[line width=1pt, color=black] (-6,0)--(8,0);
\draw[color=red, line width=1.5pt, <-,>=stealth] (-4,0) .. controls (-4,4) and (6,4) .. (6,0);
\draw[color=blue, line width=1.5pt, <-,>=stealth] (-1,0) .. controls (-1,2) and (3,2) .. (3,0);
\draw[color=olive, line width=1.5pt,<-,>=stealth] (-5,0)--(-5,4);
\draw[color=orange, line width=1.5pt,<-,>=stealth] (-2,0)--(-2,4);
\draw[color=cyan, line width=1.5pt,<-,>=stealth] (4,0)--(4,4);
\draw[color=violet, line width=1.5pt,<-,>=stealth] (1,0)--(1,4);
\draw[color=brown, line width=1.5pt,<-,>=stealth] (7,0)--(7,4);
\node at (5,2.5)[anchor=south,color=red]{$\beta'$};
\node at (2,1.5)[anchor=south,color=blue]{$\alpha'$};
\end{tikzpicture}
\end{align}

where the edges $\alpha'$ and $\beta'$ are obtained by adding an edge to the face paths $\alpha$ and $\beta$ as in Definition \ref{def:dtnotface} and the other edge ends stand for multiple incoming or outgoing edge ends  at these positions, as in Lemma \ref{rem:shorthand}. 
 For the diagram on the left, the twists along $\alpha'$ and $\beta'$ commute, since they act on different copies of the Hopf monoid $H$.  
For the diagram on the right, the claim follows by applying Lemma~\ref{lem:slideDT}, 1. to  $D_{\alpha'}$ and the slide $S_{\beta'^L}$.

It remains to prove relation (ii) for the cases where $\alpha$ is not a face path, that is $\alpha=\gamma_{i,j}$ with $1\leq i<j\leq n+2g-2$ and $\beta\in G$ with $|\alpha\cap\beta|=0$. This involves the following cases:
\begin{compactenum}[(a)]
\item $\beta=\alpha_g$, 
\item  $\beta=\alpha_k$ with $n+2k\neq j$ and $n+2k\neq i$, 
 \item    $\beta=\gamma_{k,l}$ with $l<k\leq i$ or $j\leq l<k$ 
	or $i\leq l < k \leq j$,
	
  \item  $\beta=\delta_k$ with $i\leq k \leq j$, 
  \item  $ \beta = \gamma_{k,l}$ with $i \leq k < l \leq j$ or $1\leq k \leq i<j\leq l$. As this is symmetric in $(i,j)$ and $(k,l)$ we can restrict attention to $i \leq k < l \leq j$.
  \item $\beta = \gamma_{0,k}$ with $j \leq k$.
\end{compactenum}
To prove  relation (ii) for these cases, we use 
Definition \ref{definition:DehnTwistGammaNonFace} and   Figure \ref{fig:gammaijpath} for the twist along $\alpha'=\gamma_{i,j}$.

$\bullet$ Case (a): 
Sliding  the starting end of $\beta_g$  in Figure \ref{fig:gammaijpath}  along  $\delta'_i$ commutes with the  twist $D_{\alpha_g}$ by  Lemma \ref{lem:slideDT}, 3.
Sliding the edge ends between the starting end of $\alpha_g$ and the target of $\delta'_j$ along $\alpha_g$ commutes with the twist $D_{\alpha_g}$ by adjacent commutativity from Corollary \ref{cor:adjacentendslide}.
The   twist $D_{\alpha_g}$ also commutes with the   twist  along the face path $\gamma^*_{i,j}$ from Definition \ref{definition:DehnTwistGammaNonFace}
because adding an edge to $\gamma^*_{i,j}$ yields the second diagram in \eqref{eq:commutenonint}.  This shows that $D_{\alpha_g}\circ D_{\gamma_{i,j}}= D_{\gamma_{i,j}}\circ D_{\alpha_g}$ for $i<j$.

$\bullet$ Cases (b), (c) and (f):  We  use relation (iv), which we prove below without making use of relation (ii), to express the   twist $D_{\gamma_{i,j}}$ for $i<j$ as the product 
 $D_{\gamma_{i,j}}=  (D_{\delta_j}\circ D_{\delta_i}^2 \circ D_{\alpha_g})^3 \circ D_{\gamma_{j,i}}^\inv$. As $\beta$ is a {\em face path} that  intersects neither  $\alpha_g$ nor  the {\em face paths} $\delta_i$, $\delta_j$ and $\gamma_{i,j}$, it follows that $D_\beta$ commutes with  $D_{\alpha_g}$, $D_{\delta_i}$, $D_{\delta_j}$ and $D_{\gamma_{i,j}}$. Hence,  $D_\beta$ commutes with $D_{\gamma_{i,j}}= (D_{\delta_j}\circ D_{\delta_i}^2 \circ D_{\alpha_g})^3 \circ D_{\gamma_{j,i}}^\inv$.

$\bullet$ Case (d): To verify relation (ii) for $\beta=\delta_k$ with $i\leq k\leq j$, we note that adding edges to the face paths $\delta_i$, $\delta_j$ and $\delta_k$ yields the  chord diagram
\begin{center}
\begin{tikzpicture}[scale=.4]
\draw[line width=1pt, color=black] (-13,0)--(4,0);
\draw[line width=1.5pt, color=red, <-,>=stealth] (-12,0).. controls (-12,2) and (-8,2).. (-8,0);
\draw[line width=1.5pt, color=blue, <-,>=stealth] (-9,0).. controls (-9,2) and (-3,2).. (-3,0);
\draw[line width=1.5pt, color=violet, ->,>=stealth] (-7,0).. controls (-7,4) and (3,4).. (3,0);
\draw[line width=1.5pt, color=black, ->,>=stealth] (-6,0).. controls (-6,3) and (1,3).. (1,0);
\draw[line width=1.5pt, color=orange, ->,>=stealth] (-5,0).. controls (-5,2) and (-1,2).. (-1,0);
\draw[line width=1.5pt, color=gray, <-,>=stealth] (-2,0)--(-2,4);
\draw[line width=1.5pt, color=olive, <-,>=stealth] (2,0)--(2,4);
\draw[line width=1.5pt, color=brown, <-,>=stealth] (0,0)--(0,4);
\node at (-10,1.5)[anchor=south, color=red]{$\beta_g$};
\node at (-7,1.5)[anchor=south, color=blue]{$\alpha_g$};
\node at (-1,0)[anchor=north, color=orange]{$\delta'_j$};
\node at (3,0)[anchor=north, color=violet]{$\delta'_i$};
\node at (1,0)[anchor=north, color=black]{$\delta'_k$};
\node at (-2,4)[anchor=south, color=gray]{$\nu$};
\node at (2,4)[anchor=south, color=olive]{$\mu$};
\node at (0,4)[anchor=south, color=brown]{$\sigma$};
\end{tikzpicture}
\end{center}
where $\mu,\nu,\sigma$ stand for multiple edge ends incident at these positions, as in Lemma \ref{rem:shorthand}.
The   twist along $\gamma_{i,j}$ is again given by Definition \ref{definition:DehnTwistGammaNonFace} and Figure \ref{fig:gammaijpath}. Sliding the starting end of  $\beta_g$ along $\delta'_i$  commutes with the twist $D_{\delta'_k}$ by the commutativity relation.  Siding the edge end $\nu$ along $\alpha_g$ commutes with $D_{\delta'_k}$
by  by Lemma \ref{lem:slideDT}, 3. The twist $D_{\gamma_{i,j}^*}$ commutes with the twist $D_{\delta'_k}$ because they are both face paths as in the second diagram in  \eqref{eq:commutenonint}. This shows that $D_{\gamma_{i,j}}$ and $D_{\delta_k}$ commute.

$\bullet$  Case (e): for $\alpha=\gamma_{i,j}$ and $\beta=\gamma_{k,l}$ with $i\leq k<l\leq j$ we use relation (iv) to express $D_{\gamma_{k,l}}$ as  $D_{\gamma_{k,l}}=  (D_{\delta_k}\circ D_{\delta_l}^2 \circ D_{\alpha_g})^3 \circ D_{\gamma_{l,k}}^\inv $. The right hand side involves only twists from cases (a) to (d) and hence commutes with $D_{\gamma_{i,j}}$.

{\bf Relation (iii):}
If $\alpha,\beta \in G$ with $|\alpha\cap \beta|=1$, then because of~\eqref{eq:intnumbers} we assume w.~l.~o.~g.~that $\alpha=\alpha_k$ with $1\leq k\leq g$ and $\beta=\delta_i$ or $\beta=\gamma_{i,j}$ for suitable $i,j\in\{0,...,n+2g-2\}$.  If $\beta$ is a face path, then adding an edge to $\beta$ commutes with $D_{\alpha}$ by Corollary \ref{cor:addingtwist}.
We can therefore restrict attention to 
\begin{center}
\begin{tikzpicture}[scale=.4]
\draw[line width=1pt, color=black] (-6,0)--(6,0);
\draw[line width=1.5pt, color=red, <-,>=stealth] (-4,0).. controls (-4,2) and (1,2).. (1,0);
\draw[line width=1.5pt, color=blue, <-,>=stealth] (-1,0).. controls (-1,2) and (4,2).. (4,0);
\draw[line width=1.5pt, color=gray, <-,>=stealth] (-2.5,0)--(-2.5,3);
\draw[line width=1.5pt, color=brown, <-,>=stealth] (0,0)--(0,3);
\draw[line width=1.5pt, color=olive, <-,>=stealth] (2.5,0)--(2.5,3);
\node at (-2.5,3)[anchor=south, color=gray] {$\nu$};
\node at (0,3)[anchor=south, color=brown] {$\sigma$};
\node at (2.5,3)[anchor=south, color=olive] {$\mu$};
\node at (4,0)[anchor=north, color=blue] {$\alpha$};
\node at (-4,0)[anchor=north, color=red] {$\beta'$};
\end{tikzpicture}
\end{center}
where $\beta'$ is obtained by adding an edge to the face path $\beta$  and   $\mu,\nu,\sigma$ stand for multiple  edge ends  at these positions, as in Lemma \ref{rem:shorthand}. We can  slide them out of the loops $\alpha$ and $\beta'$ by first sliding $\nu$ and $\sigma$ to the starting end of $\alpha$, then $\mu$ and $\sigma$ to the target end of $\beta'$. 
By Lemma~\ref{lem:slideDT}, 3.~and by the adjacent commutativity relation, this commutes with the twists $D_\alpha$ and $D_{\beta'}$, and it reduces the chord diagram to the one in \eqref{eq:chordtorus}, for which relation (iii) was shown in Theorem~\ref{th:modular}.

It remains to consider the cases, where $\beta$ is not a face path, namely $\beta=\gamma_{i,j}$ with $1\leq i<j\leq n+2g-2$. In this case  $|\alpha_k\cap \beta|=1$ if and only if $i=n+2k$ or $j=n+2k$, see Figures \ref{fig:gervais} and \ref{fig:gervais2} and the intersection numbers in \eqref{eq:intnumbers}.

$\bullet$ For $\alpha=\alpha_k$ and $\beta = \gamma_{i, j}$ with $i=n+2k<j$ we consider the  chord diagram
\begin{center}
\begin{tikzpicture}[scale=.4]
\draw[line width=1pt, color=black] (-11,0)--(11,0);
\draw[line width=1.5pt, color=red,  <-,>=stealth] (-10,0).. controls (-10,2) and (-6,2).. (-6,0);
\draw[line width=1.5pt, color=blue,  <-,>=stealth] (-7,0).. controls (-7,2) and (-3,2).. (-3,0);
\draw[line width=1.5pt, color=violet,  ->,>=stealth] (-5,0).. controls (-5,4) and (7,4).. (7,0);
\draw[line width=1.5pt, color=orange,  ->,>=stealth] (-4,0).. controls (-4,2) and (0,2).. (0,0);
\draw[line width=1.5pt, color=brown,  <-,>=stealth] (2,0).. controls (2,2) and (6,2).. (6,0);
\draw[line width=1.5pt, color=black,  <-,>=stealth] (4,0).. controls (4,2) and (8,2).. (8,0);
\draw[line width=1.5pt, color=gray,  <-,>=stealth] (-2,0)--(-2,4);
\draw[line width=1.5pt, color=olive,  <-,>=stealth] (1,0)--(1,4);
\node at (-2,4)[anchor=south, color=gray]{$\nu$};
\node at (1,4)[anchor=south, color=olive]{$\mu$};
\node at (6,0)[anchor=north, color=brown]{$\beta_k$};
\node at (8,0)[anchor=north, color=black]{$\alpha_k$};
\node at (-8,1.5)[anchor=south, color=red]{$\beta_g$};
\node at (-5,1.5)[anchor=south, color=blue]{$\alpha_g$};
\node at (0,0)[anchor=north, color=orange]{$\delta'_j$};
\node at (6,2.3)[anchor=south, color=violet]{$\delta'_{n+2k}$};
\end{tikzpicture}
\end{center}
where $\delta'_j$ and $\delta'_{n+2k}$ are obtained by adding edges to the face paths $\delta_j$ and $\delta_{n+2k}$ and the edge ends $\mu$ and $\nu$ stand for multiple edge ends incident at these positions, as in Lemma \ref{rem:shorthand}.
The   twist along $\gamma_{n+2k,j}$ is given by the sequence of slides in Definition \ref{definition:DehnTwistGammaNonFace}. 
Sliding the edge end $\nu$ along $\alpha_g$ commutes with $D_{\alpha_k}$ by the  commutativity relation. Sliding the starting end of $\beta_g$ along $\delta'_{n+2k}$ commutes with $D_{\alpha_k}$ by Lemma \ref{lem:slideDT}, 3. 
The   twists $D_{\gamma^*_{n+2k,j}}$ and $D_{\alpha_k}$ satisfy  relation (iii), since they are  face paths that intersect in a single point. Hence, $D_{\gamma_{n+2k,j}}$ and $ D_{\alpha_k}$  satisfy relation (iii) as well. 

$\bullet$ For $\alpha=\alpha_k$ and $\beta=\gamma_{i, j}$ with $i<j=n+2k$ we consider the chord diagram
\begin{center}
\begin{tikzpicture}[scale=.4]
\draw[line width=1pt, color=black] (-11,0)--(8,0);
\draw[line width=1.5pt, color=red,  <-,>=stealth] (-10,0).. controls (-10,2) and (-6,2).. (-6,0);
\draw[line width=1.5pt, color=blue,  <-,>=stealth] (-7,0).. controls (-7,2) and (-3,2).. (-3,0);
\draw[line width=1.5pt, color=violet,  ->,>=stealth] (-5,0).. controls (-5,4) and (7,4).. (7,0);
\draw[line width=1.5pt, color=orange,  ->,>=stealth] (-4,0).. controls (-4,3) and (4,3).. (4,0);
\draw[line width=1.5pt, color=black,  <-,>=stealth] (1,0).. controls (1,2) and (5,2).. (5,0);
\draw[line width=1.5pt, color=brown,  <-,>=stealth] (-1,0).. controls (-1,2) and (3,2).. (3,0);
\draw[line width=1.5pt, color=gray,  <-,>=stealth] (-2,0)--(-2,4);
\draw[line width=1.5pt, color=olive,  <-,>=stealth] (6,0)--(6,4);
\node at (.5,0)[anchor=north west, color=black]{$\alpha_k$};
\node at (-8,1.5)[anchor=south, color=red]{$\beta_g$};
\node at (-5,1.5)[anchor=south, color=blue]{$\alpha_g$};
\node at (4.5,0)[anchor=north, color=orange]{$\delta'_{n+2k}$};
\node at (7,0)[anchor=north, color=violet]{$\delta'_i$};
\node at (-2,4)[anchor=south, color=gray]{$\nu$};
\node at (6,4)[anchor=south, color=olive]{$\mu$};
\node at (-1,0)[anchor=north, color=brown]{$\beta_k$};
\end{tikzpicture}
\end{center}
where $\delta'_i$ and $\delta'_{n+2k}$ are obtained by adding edges to the face paths $\delta_i$ and $\delta_{n+2k}$, and the edge ends labeled $\mu$ and $\nu$ stand again for multiple edge ends incident at these positions. 

The twist along $\gamma_{i,n+2k}$ is given 
by the sequence of edge slides in Definition \ref{definition:DehnTwistGammaNonFace}. 
By Corollary \ref{corollary:SlideIntoFacePathDehnCommute} and the discussion after Definition \ref{definition:DehnTwistGammaNonFace}, sliding the starting end of $\beta_k$ along $\alpha_g\circ \delta'^\inv_{n+2k}$ instead of $\alpha_g$ and then along $\delta'_i$ together with 
the starting end of $\beta+g$ yields the same twist. By Lemma~\ref{lem:slideDT}, 3.~the slide along $\alpha_g\circ \delta'^\inv_{n+2k}$ commutes with the twist $D_{\alpha_k}$. Sliding the edge end $\nu$  and the target end of $\beta_k$ along $\alpha_g$ and sliding the starting end of $\beta_g$ along $\delta'_{i}$  commute with $D_{\alpha_k}$ as well by the commutativity relation and opposite end commutativity.  We can thus omit  $\beta_g$, $\beta_k$ and  $\nu$ and  prove relation (iii) for the diagram
\begin{align}
\begin{tikzpicture}[scale=.55]
\draw[line width=1pt, color=black] (0,0)--(10,0);
		\draw[line width=1.5pt, color=blue,  <-,>=stealth] (1,0).. controls (1,2) and (5,2).. (5,0);
		\draw[line width=1.5pt, color=violet,  ->,>=stealth] (2.5,0).. controls (2.5,4) and (9,4).. (9,0);
		\draw[line width=1.5pt, color=orange,  ->,>=stealth] (3.5,0).. controls (3.5,2) and (7,2).. (7,0);
		\draw[line width=1.5pt, color=black,  <-,>=stealth] (6,0).. controls (6,2) and (8,2).. (8,0);
		\draw[line width=1.5pt, color=olive,  <-,>=stealth] (8.5,0)--(8.5,4);
		\node at (8.5,4)[anchor=south, color=olive]{$\mu$};
		\node at (5.5,0)[anchor=north west, color=black]{$\alpha_k$};
		\node at (1,0)[anchor=north west, color=blue]{$\alpha_g$};
		\node at (3.5,0)[anchor=north, color=orange]{$\delta'_{n+2k}$};
		\node at (9,0)[anchor=north, color=violet]{$\delta'_i$};
		\draw[color=black, line width=.5pt,->,>=stealth] (0.7,-.5) -- (0.7,.5);
		\draw[color=black, line width=.5pt,->,>=stealth] (0.7,.5) .. controls (0.7,2.2) and (5.2,2.2).. (5.2,.5);
		\draw[color=black, line width=.5pt,->,>=stealth] (5.2,.5) -- (6.5,.5);
		\draw[color=black, line width=.5pt,->,>=stealth] (6.5,.5) .. controls (6.5,1.5) and (4,1.5).. (4,.5);
		\draw[color=black, line width=.5pt,->,>=stealth] (4,.5) -- (4.5,.5);
		\draw[color=black, line width=.5pt,->,>=stealth] (4.5,.5) .. controls (4.5,1.5) and (1.5,1.5).. (1.5,.5);
		\draw[color=black, line width=.5pt,->,>=stealth] (1.5,.5) -- (2,.5);
		\draw[color=black, line width=.5pt,->,>=stealth] (2,.5) .. controls (2,4.5) and (9.5,4.5).. (9.5,.5);
		\draw[color=black, line width=.5pt] (9.5,.5) -- (9.5,-.5);
		\node at (2.5,2.5)[anchor=east, color=black]{${\gamma_{i,n+2k}}$};
\end{tikzpicture}
\end{align}

The twist along $\gamma_{i,n+2k}$ is  then given by  sliding the target end of $\alpha_k$ along $\alpha_g$,  twisting along the face path $\gamma^*_{i,n+2k}=S_{\alpha_g^R}(\gamma_{i,n+2k})$ and  sliding the target end of $\alpha_k$ back along $\alpha_g$ in its original position. 
By Lemma~\ref{lemma:SlideDehnCompatibilityGeneral},  this is identical to 
$$D_{\gamma_{i, n+2k}}=S^3_{\alpha_g^{-R}}\circ D_{\gamma''}\circ S^3_{\alpha_g^R}\quad\text{with}\quad 
\gamma''=S^3_{\alpha_g^R}(\gamma_{i,n+2k}).$$ 
The face path $\gamma''$ is the image of $\gamma_{i,n+2k}$ under the slide  that slides the starting and target end of $\alpha_k$ and the target end of $\delta'_{n+2k}$ along $\alpha_g$ and is given as follows 

\begin{align}\label{eq:simppp3}
\begin{tikzpicture}[scale=.6]
	\begin{scope}[shift={(0,0)}]
		\draw[line width=1pt, color=black] (0,0)--(10,0);
		\draw[line width=1.5pt, color=blue,  <-,>=stealth] (1,0).. controls (1,2) and (5,2).. (5,0);
		\draw[line width=1.5pt, color=violet,  ->,>=stealth] (2.5,0).. controls (2.5,4) and (9,4).. (9,0);
		\draw[line width=1.5pt, color=orange,  ->,>=stealth] (3.5,0).. controls (3.5,2) and (7,2).. (7,0);
		\draw[line width=1.5pt, color=black,  <-,>=stealth] (6,0).. controls (6,2) and (8,2).. (8,0);
		\draw[line width=1.5pt, color=olive,  <-,>=stealth] (8.5,0)--(8.5,4);
		\node at (8.5,4)[anchor=south, color=olive]{$\mu$};
		\node at (5.5,0)[anchor=north west, color=black]{$\alpha_k$};
		\node at (1,0)[anchor=north west, color=blue]{$\alpha_g$};
		\node at (3.5,0)[anchor=north, color=orange]{$\delta'_{n+2k}$};
		\node at (9,0)[anchor=north, color=violet]{$\delta'_i$};
		\draw[color=black, line width=.5pt,->,>=stealth] (0.7,-.5) -- (0.7,.5);
		\draw[color=black, line width=.5pt,->,>=stealth] (0.7,.5) .. controls (0.7,2.2) and (5.2,2.2).. (5.2,.5);
		\draw[color=black, line width=.5pt,->,>=stealth] (5.2,.5) -- (6.5,.5);
		\draw[color=black, line width=.5pt,->,>=stealth] (6.5,.5) .. controls (6.5,1.5) and (4,1.5).. (4,.5);
		\draw[color=black, line width=.5pt,->,>=stealth] (4,.5) -- (4.5,.5);
		\draw[color=black, line width=.5pt,->,>=stealth] (4.5,.5) .. controls (4.5,1.5) and (1.5,1.5).. (1.5,.5);
		\draw[color=black, line width=.5pt,->,>=stealth] (1.5,.5) -- (2,.5);
		\draw[color=black, line width=.5pt,->,>=stealth] (2,.5) .. controls (2,4.5) and (9.5,4.5).. (9.5,.5);
		\draw[color=black, line width=.5pt] (9.5,.5) -- (9.5,-.5);
		\node at (2.5,2.5)[anchor=east, color=black]{${\gamma_{i,n+2k}}$};
	\end{scope}
	\draw[line width=1pt, color=black,->,>=stealth] (10.5,1)-- node[above] {$S^3_{\alpha_g^R}$}(12.5,1);
		\begin{scope}[shift={(16,0)}]
		\draw[line width=1pt, color=black] (-3,0)--(7,0);
		\draw[line width=1.5pt, color=blue,  <-,>=stealth] (1,0).. controls (1,2) and (5,2).. (5,0);
		\draw[line width=1.5pt, color=violet,  ->,>=stealth] (2.5,0).. controls (2.5,2) and (6,2).. (6,0);
		\draw[line width=1.5pt, color=orange,  ->,>=stealth] (3.5,0).. controls (3.5,2) and (-1,2).. (-1,0);
		\draw[line width=1.5pt, color=black,  <-,>=stealth] (-2,0).. controls (-2,2) and (0,2).. (0,0);
		\draw[line width=1.5pt, color=olive,  <-,>=stealth] (5.5,0)--(5.5,3);
		\node at (5.5,3)[anchor=south, color=olive]{$\mu$};
		\node at (-2,0)[anchor=north, color=black]{$\alpha_k$};
		\node at (1,0)[anchor=north west, color=blue]{$\alpha_g$};
		\node at (3.5,0)[anchor=north, color=orange]{$\delta'_{n+2k}$};
		\node at (6,0)[anchor=north, color=violet]{$\delta'_i$};
		\draw[color=black, line width=.5pt,->,>=stealth] (-1.3,-.5) -- (-1.3,.5);
		\draw[color=black, line width=.5pt,->,>=stealth] (-1.3,.5) .. controls (-1.3,2.2) and (3.8,2.2).. (3.8,.5);
		\draw[color=black, line width=.5pt,->,>=stealth] (3.8,.5) -- (4.5,.5);
		\draw[color=black, line width=.5pt,->,>=stealth] (4.5,.5) .. controls (4.5,1.5) and (1.5,1.5).. (1.5,.5);
		\draw[color=black, line width=.5pt,->,>=stealth] (1.5,.5) -- (2,.5);
		\draw[color=black, line width=.5pt,->,>=stealth] (2,.5) .. controls (2,2.5) and (6.5,2.5).. (6.5,.5);
		\draw[color=black, line width=.5pt] (6.5,.5) -- (6.5,-.5);
		\node at (5.5,2.2)[anchor=west, color=black]{$\gamma''$};
	\end{scope}
\end{tikzpicture}
\end{align}
As the slide $S_{\alpha^R}^3$ commutes with the twist $D_{\alpha_k}$ by Lemma~\ref{lem:slideDT},1.~it is now sufficient
to show that relation (iii) holds for the twists along $\alpha_k$ and $\gamma''$ in \eqref{eq:simppp3}. This follows directly because $\gamma''$ is a face path that intersects $\alpha_k$ in a single point.

{\bf Relation (iv):}
Due to the symmetries of  relation (iv), it is sufficient to consider the cases $j<i<k$, $j=i<k$ and $j<i=k$. We  treat the cases $j=0$ and $j \neq 0$ separately.  

$\bullet$ For the case $1 \leq j<i<k$, depicted  in Figure \ref{fig:gervais2}, we note that all of the paths in relation (iv) except $\gamma_{j,k}$ are face paths. The twists along these paths are thus obtained by adding edges 
$\delta'_i$, $\delta'_j$, $\delta'_k$, $\gamma'_{k,i}$ and $\gamma'_{i,j}$ to them, twisting along these added edges and then deleting them.
Note that adding these edges commutes with 
the twists 
along the face paths
in Relation (iv) 
by Corollary~\ref{cor:addingtwist}.
We can apply the same argument to the twist along the non-face path $\gamma_{j,k}$ by also noting that the slides in step~2. and step~4. of Definition~\ref{definition:DehnTwistGammaNonFace} commute with adding the edges by~\eqref{eq:ComplementSlideGeneral}.
The order in which these complements are added does not affect the result 
by \eqref{eq:ComplementCommutativity}. 
From this,  we also obtain that deleting any of the added edges commutes with the twists under consideration.

The twist along $\gamma_{j,k}$ is given by Definition \ref{definition:DehnTwistGammaNonFace} and computed from the  chord diagram
\begin{center}
\begin{tikzpicture}[scale=.4]
\draw[line width=1pt, color=black] (-2,0)--(18,0);
\draw[color=blue, line width=1.5 pt,  <-,>=stealth] (0,0).. controls (0,3) and (5,3) .. (5,0);
\draw[color=red, line width=1.5 pt,  <-,>=stealth] (-1,0)..controls(-1,2) and (1,2)..(1,0);
\draw[color=black, line width=1.5 pt,  ->,>=stealth] (2,0).. controls (2,5) and (17,5) .. (17,0);
\draw[color=magenta, line width=1.5 pt,  ->,>=stealth] (3,0).. controls (3,4) and (13,4) .. (13,0);
\draw[color=orange, line width=1.5 pt,  ->,>=stealth] (4,0).. controls (4,3) and (9,3) .. (9,0);
\draw[color=gray, line width=1.5 pt,  <-,>=stealth] (7,0)--(7,5);
\draw[color=olive, line width=1.5 pt,  <-,>=stealth] (11,0)--(11,5);
\draw[color=brown, line width=1.5 pt,  <-,>=stealth] (15,0)--(15,5);
\draw[color=violet, line width=1.5 pt,  ->,>=stealth] (10,0).. controls (10,2) and (12,2) .. (12,0);
\draw[color=purple, line width=1.5 pt,  ->,>=stealth] (14,0).. controls (14,2) and (16,2) .. (16,0);
\draw[line width=.5pt, color=black,->,>=stealth] (-.5,.3).. controls (-.5,3.5) and (5.5,3.5).. (5.5,.5);
\draw[line width=.5pt, color=black,->,>=stealth] (5.5,.5)--(8.5,.5);     
\draw[line width=.5pt, color=black,] (8.5,.5).. controls (8.5,2.5) and (4.5,2.5) .. (4.5,.3);
\draw[line width=.5pt, color=black,->,>=stealth] (4.5,.5).. controls (4.5,2.2) and (.5,2.2) .. (.5,.5);
\draw[line width=.5pt, color=black,->,>=stealth] (.5,.5)--(1.5,.5); 
\draw[line width=.5pt, color=black,->,>=stealth] (1.5,.5).. controls (1.5,5.5) and (17.5, 5.5) .. (17.5,.5);
\node at (7,5) [color=gray, anchor=south]{$\nu$};
\node at (11,5) [color=olive, anchor=south]{$\sigma$};
\node at (15,5) [color=brown, anchor=south]{$\mu$};
\node at (0,0)[anchor=north, color=blue]{$\alpha_g$};
\node at (17,0)[anchor=north, color=black]{$\delta'_j$};
\node at (13,0)[anchor=north, color=magenta]{$\delta'_i$};
\node at (9,0)[anchor=north, color=orange]{$\delta'_k$};
\node at (-1,0)[anchor=north, color=red]{$\beta_g$};
\node at (10.5,1) [color=violet, anchor=south east]{$\gamma'_{k,i}$};
\node at (14.5,1) [color=purple, anchor=south east]{$\gamma'_{i,j}$};
\node at (17.5,1)[anchor=west]{$\gamma_{j,k}$};
\end{tikzpicture}
\end{center}
 where the edges $\mu,\nu,\sigma$ stand again for multiple edge ends incident at this position, as in  Lemma \ref{rem:shorthand}. 
 The twist $D_{\gamma_{j,k}}$ is obtained by first sliding the starting end of $\beta_g$ along $\delta'_j$ and the edge end $\nu$ along $\alpha_g$, which yields the face path $\gamma^*_{j,k}=\delta'_j\circ\alpha_g\circ \delta'^\inv_k\circ \alpha_g^\inv$, then twisting along $\gamma_{j,k}^*$ and finally sliding the edge ends back. 
 Sliding the starting end of $\beta_g$ along $\delta'_j$ commutes with the twists $D_{\delta'_i}$, $D_{\delta'_k}$, $D_{\gamma'_{i,j}}$ and $D_{\gamma'_{k,i}}$ by the commutativity relation, with $D_{\delta'_j}$ by the adjacent commutativity relation in Corollary \ref{cor:adjacentendslide} and with $D_{\alpha_g}$ by Lemma \ref{lem:slideDT}, 3.  Sliding the edge end $\nu$ along $\alpha_g$ commutes with  $D_{\gamma'_{i,j}}$ and $D_{\gamma'_{k,i}}$ by the commutativity relation, with $D_{\alpha_g}$  by the adjacent commutativity relation  and with $D_{\delta'_i}$, $D_{\delta'_j}$ and $D_{\delta'_k}$ by Lemma \ref{lem:slideDT}, 3.  It is therefore sufficient to consider  the following chord diagram,
obtained by omitting $\beta_g$ and $\nu$ and adding an edge $\gamma'_{j,k}$ to $\gamma^*_{j,k}$.

 \begin{align}
\begin{tikzpicture}[scale=.4]
\draw[line width=1pt, color=black] (-2,0)--(19,0);
\draw[color=blue, line width=1.5 pt,  <-,>=stealth] (0,0).. controls (0,3) and (5,3) .. (5,0);
\draw[color=black, line width=1.5 pt,  ->,>=stealth] (2,0).. controls (2,5) and (17,5) .. (17,0);
\draw[color=magenta, line width=1.5 pt,  ->,>=stealth] (3,0).. controls (3,4) and (13,4) .. (13,0);
\draw[color=orange, line width=1.5 pt,  ->,>=stealth] (4,0).. controls (4,3) and (9,3) .. (9,0);
\draw[color=olive, line width=1.5 pt,  <-,>=stealth] (11,0)--(11,5);
\draw[color=brown, line width=1.5 pt,  <-,>=stealth] (15,0)--(15,5);
\draw[color=violet, line width=1.5 pt,  ->,>=stealth] (10,0).. controls (10,2) and (12,2) .. (12,0);
\draw[color=purple, line width=1.5 pt,  ->,>=stealth] (14,0).. controls (14,2) and (16,2) .. (16,0);
\draw[color=red, line width=1.5 pt,  ->,>=stealth] (-1,0).. controls (-1,6) and (18,6) .. (18,0);
\node at (0,0)[anchor=north, color=blue]{$\alpha_g$};
\node at (17,0)[anchor=north, color=black]{$\delta'_j$};
\node at (13,0)[anchor=north, color=magenta]{$\delta'_i$};
\node at (9,0)[anchor=north, color=orange]{$\delta'_k$};
\node at (10.5,1) [color=violet, anchor=south east]{$\gamma'_{k,i}$};
\node at (14.5,1) [color=purple, anchor=south east]{$\gamma'_{i,j}$};
\node at (18,1)[anchor=west, color=red]{$\gamma'_{j,k}$};
\node at (11,5) [color=olive, anchor=south]{$\sigma$};
\node at (15,5) [color=brown, anchor=south]{$\mu$};
\end{tikzpicture}
\label{eq:RelationIVGraph}
\end{align}

We have to show that for this diagram
$\epsilon\circ  (D_{\delta'_k}\circ D_{\delta'_i}\circ D_{\delta'_j}\circ D_{\alpha_g})^3=\epsilon\circ D_{\gamma'_{i,j}}\circ D_{\gamma'_{j,k}}\circ D_{\gamma'_{k,i}}$, 
where $\epsilon$ 
deletes the added edges $\delta'_i$, $\delta'_j$, 
$\gamma'_{k,i}$, $\gamma'_{i,j}$ and $\gamma'_{j,k}$.

The paths $
\gamma'_{k,i}\circ \delta'_k\circ \delta'^\inv_i$,
$\gamma_{j,k}'^{-1}\circ\delta_{j}'\circ\alpha_{g}\circ\delta'^{-1}_{k}\circ\alpha_{g}^{-1}$ and 
$\gamma'_{i,j}\circ \delta'_i\circ\delta'^\inv_j$
are ciliated faces of the chord diagram, and because they were obtained by adding edges to face paths, 
the $H$-right comodule structures associated to them are trivial by \eqref{eq:FlatFaceComplement}. This allows us to use identity~\eqref{eq:edgeexpress} to express the label of one edge of every face in terms of the other labels of that face.
				
To simplify the expressions in the following computations, we also reverse the orientation of the edges $\gamma'_{i,j}$, $\gamma'_{k,i}$, $\gamma'_{j,k}$, $\delta'_i$, $\delta'_j$, $\delta'_k$.
Assigning the labels $a,c,d,e$ to the edges $\alpha_{g}, \delta'_{k}, \gamma'_{k,i}, \gamma'_{i,j}$ and labels $y,z$ to the edge ends $\sigma,\mu$, we then obtain the labels of the edges $\delta'_{i},\delta'_{j}$ and $\gamma'_{j,k}$ from identity \eqref{eq:edgeexpress}.
We then 
can delete  $\gamma'_{j,k}$ and $\gamma'_{k,i}$ before computing $(D'_{\delta_k}\circ D'_{\delta_i}\circ D'_{\delta_j}\circ D_{\alpha_g})^3$, and $\delta'_{i}$ and $\delta'_{j}$ before computing $D_{\gamma'_{j,k}}\circ D_{\gamma'_{k,i}}\circ D_{\gamma'_{i,j}}$   
and  obtain the following two $H$-labelled chord diagrams:

 \begin{align}\label{eq:chord1}
 &\begin{tikzpicture}[scale=.4, baseline=(current  bounding  box.center)]
\draw[line width=1pt, color=black] (-2,0)--(17,0);
\draw[color=blue, line width=1.5 pt,  <-,>=stealth] (0,0).. controls (0,3) and (5,3) .. (5,0);
\draw[color=orange, line width=1.5 pt,  <-,>=stealth] (4,0).. controls (4,3) and (9,3) .. (9,0);
\draw[color=violet, line width=1.5 pt,  <-,>=stealth] (10,0).. controls (10,2) and (12,2) .. (12,0);
\draw[color=purple, line width=1.5 pt,  <-,>=stealth] (13,0).. controls (13,2) and (15,2) .. (15,0);
\draw[color=red, line width=1.5 pt,  <-,>=stealth] (-1,0).. controls (-1,5.5) and (16,5.5) .. (16,0);
\draw[color=olive, line width=1.5 pt,  <-,>=stealth] (11,0)--(11,5);
\draw[color=brown, line width=1.5 pt,  <-,>=stealth] (14,0)--(14,5);
\node at (8,2.1) [color=orange, anchor=south]{$c_{(2)}$};
\node at (4,2.1) [color=blue, anchor=south]{$a_{(2)}$};
\node at (11.2,1.3) [color=violet, anchor=south east]{$d_{(1)}$};
\node at (14.2,1.3) [color=purple, anchor=south east]{$e_{(1)}$};
\node at (11,5)[color=olive, anchor=south] {$y$};
\node at (14,5)[color=brown, anchor=south] {$z$};
\node at (-2.5,0)[anchor=north west, color=red]{$a_{(3)} S(c_{(1)})S(a_{(1)}) p^{-1} c_{(3)}p^{-1}d_{(2)}p^{-1}e_{(2)}$};
\end{tikzpicture}
\qquad\begin{tikzpicture}[scale=.4, baseline=(current  bounding  box.center)]
\draw[line width=1pt, color=black] (-2,0)--(18,0);
\draw[color=blue, line width=1.5 pt,  <-,>=stealth] (0,0).. controls (0,3) and (5,3) .. (5,0);
\draw[color=black, line width=1.5 pt,  <-,>=stealth] (2,0).. controls (2,5) and (17,5) .. (17,0);
\draw[color=magenta, line width=1.5 pt,  <-,>=stealth] (3,0).. controls (3,4) and (13,4) .. (13,0);
\draw[color=orange, line width=1.5 pt,  <-,>=stealth] (4,0).. controls (4,3) and (9,3) .. (9,0);
\draw[color=olive, line width=1.5 pt,  <-,>=stealth] (11,0)--(11,5);
\draw[color=brown, line width=1.5 pt,  <-,>=stealth] (15,0)--(15,5);
\node at (9,4.7) [color=black, anchor=east]{$\low c 3 p^\inv \low d 2p^\inv e$};
\node at (13,0) [color=magenta, anchor=north]{$\low c 2 p^\inv \low d 1$};
\node at (8.5,1.5) [color=orange, anchor=south]{$\low c 1$};
\node at (2,2.3) [color=blue, anchor=south]{$a$};
\node at (11,5)[color=olive, anchor=south] {$y$};
\node at (15,5)[color=brown, anchor=south] {$z$};
\end{tikzpicture}
\end{align}

Using the second diagram in \eqref{eq:chord1}, we compute  the morphism $D_{\delta'_k}\circ D_{\delta'_i}\circ D_{\delta'_j}\circ D_{\alpha_g}$. 

\begin{tikzpicture}[scale=.6]
\node at (-3,5)[anchor=south west] {$D_{\delta'_k}\circ D_{\delta'_i}\circ D_{\delta'_j}\circ D_{\alpha_g}:$};
\draw[line width=1pt, color=black] (-2,0)--(18,0);
\draw[color=blue, line width=1.5 pt,   <-,>=stealth] (0,0).. controls (0,3) and (5,3) .. (5,0);
\draw[color=black, line width=1.5 pt,  <-,>=stealth] (2,0).. controls (2,5) and (17,5) .. (17,0);
\draw[color=magenta, line width=1.5 pt,  <-,>=stealth] (3,0).. controls (3,4) and (13,4) .. (13,0);
\draw[color=orange, line width=1.5 pt,  <-,>=stealth] (4,0).. controls (4,3) and (9,3) .. (9,0);
\draw[color=olive, line width=1.5 pt,  <-,>=stealth] (11,0)--(11,5);
\draw[color=brown, line width=1.5 pt,  <-,>=stealth] (15,0)--(15,5);
\node at (15.5,2.5) [color=black, anchor=south west]{$\low c 3 p^\inv \low d 2 p^\inv e$};
\node at (13,2) [color=magenta, anchor=south]{$\low c 2 p^\inv \low d 1$};
\node at (8.5,1.7) [color=orange, anchor=south]{$\low c 1 $};
\node at (2,2.3) [color=blue, anchor=south]{$a$};
\node at (11,5)[color=olive, anchor=south] {$y$};
\node at (15,5)[color=brown, anchor=south] {$z$};
\end{tikzpicture}

\begin{tikzpicture}[scale=.6]
\draw[line width=1pt, ->,>=stealth, color=black](-3,2)--(-1,2);
\node at (-2,2)[anchor=south]{$D_{\alpha_g}$};
\draw[line width=1pt, color=black] (-2,0)--(18,0);
\draw[color=blue, line width=1.5 pt,  <-,>=stealth] (0,0).. controls (0,3) and (5,3) .. (5,0);
\draw[color=black, line width=1.5 pt,  <-,>=stealth] (2,0).. controls (2,5) and (17,5) .. (17,0);
\draw[color=magenta, line width=1.5 pt,  <-,>=stealth] (3,0).. controls (3,4) and (13,4) .. (13,0);
\draw[color=orange, line width=1.5 pt,  <-,>=stealth] (4,0).. controls (4,3) and (9,3) .. (9,0);
\draw[color=olive, line width=1.5 pt,  <-,>=stealth] (11,0)--(11,5);
\draw[color=brown, line width=1.5 pt,  <-,>=stealth] (15,0)--(15,5);
\node at (15.5,2.5) [color=black, anchor=south west]{$t$};
\node at (12.5,2) [color=magenta, anchor=south]{$u$};
\node at (8.5,1.7) [color=orange, anchor=south]{$\low r 1$};
\node at (2,2.3) [color=blue, anchor=south]{$\low a 2$};
\node at (11,5)[color=olive, anchor=south] {$y$};
\node at (15,5)[color=brown, anchor=south] {$z$};
\node at (18.5,4) [color=black, anchor=west]{$r=\low a 1 c$};
\node at (18.5,3) [color=black, anchor=west]{$u=\low r 2 p^\inv \low d 1$};
\node at (18.5,2) [color=black, anchor=west]{$t=\low r 3 p^\inv \low d 2 p^\inv e$};
\end{tikzpicture}

\begin{tikzpicture}[scale=.6]
\draw[line width=1pt, ->,>=stealth, color=black](-3,2)--(-1,2);
\node at (-2,2)[anchor=south]{$D_{\delta'_j}$};
\draw[line width=1pt, color=black] (-2,0)--(18,0);
\draw[color=blue, line width=1.5 pt,  <-,>=stealth] (0,0).. controls (0,3) and (5,3) .. (5,0);
\draw[color=black, line width=1.5 pt,  <-,>=stealth] (2,0).. controls (2,5) and (17,5) .. (17,0);
\draw[color=magenta, line width=1.5 pt,  <-,>=stealth] (3,0).. controls (3,4) and (13,4) .. (13,0);
\draw[color=orange, line width=1.5 pt,  <-,>=stealth] (4,0).. controls (4,3) and (9,3) .. (9,0);
\draw[color=olive, line width=1.5 pt,  <-,>=stealth] (11,0)--(11,5);
\draw[color=brown, line width=1.5 pt,  <-,>=stealth] (15,0)--(15,5);
\node at (15.5,2.5) [color=black, anchor=south west]{$\low t 8$};
\node at (13,0) [color=magenta, anchor=north]{$\low t 7 u S(\low t 2)$};
\node at (8,0) [color=orange, anchor=north]{$\low t 6 \low r 1 S(\low t 4)$};
\node at (2.7,2.2) [color=blue, anchor=south east]{$\low a 2 S(\low t 5)$};
\node at (11,5)[color=olive, anchor=south] {$\low t 3 y$};
\node at (15,5)[color=brown, anchor=south] {$\low t 1 z$};
\node at (18.5,4) [color=black, anchor=west]{$r=\low a 1 c$};
\node at (18.5,3) [color=black, anchor=west]{$u=\low r 2 p^\inv \low d 1$};
\node at (18.5,2) [color=black, anchor=west]{$t=\low r 3 p^\inv \low d 2p^\inv e$};
\end{tikzpicture}

\begin{tikzpicture}[scale=.6]
\draw[line width=1pt, ->,>=stealth, color=black](-4,2)--(-2,2);
\node at (-3,2)[anchor=south]{$D_{\delta'_i}$};
\draw[line width=1pt, color=black] (-2,0)--(18,0);
\draw[color=blue, line width=1.5 pt,  <-,>=stealth] (0,0).. controls (0,3) and (5,3) .. (5,0);
\draw[color=black, line width=1.5 pt,  <-,>=stealth] (2,0).. controls (2,5) and (17,5) .. (17,0);
\draw[color=magenta, line width=1.5 pt,  <-,>=stealth] (3,0).. controls (3,4) and (13,4) .. (13,0);
\draw[color=orange, line width=1.5 pt,  <-,>=stealth] (4,0).. controls (4,3) and (9,3) .. (9,0);
\draw[color=olive, line width=1.5 pt,  <-,>=stealth] (11,0)--(11,5);
\draw[color=brown, line width=1.5 pt,  <-,>=stealth] (15,0)--(15,5);
\node at (15.5,2.5) [color=black, anchor=south west]{$\low t 8$};
\node at (12,0) [color=magenta, anchor=north]{$\low t 7 \low u 5 S(\low t 2)$};
\node at (7,0) [color=orange, anchor=north]{$\low t 6 \low u 4 \low r 1 S(\low t 4\low u 2)$};
\node at (3.5, 2.2) [color=blue, anchor=south east]{$\low a 2 S(\low t 5\low u 3) $};
\node at (11,5)[color=olive, anchor=south] {$\low t 3 \low u 1 y$};
\node at (15,5)[color=brown, anchor=south] {$\low t 1 z$};
\node at (18.5,4) [color=black, anchor=west]{$r=\low a 1 c$};
\node at (18.5,3) [color=black, anchor=west]{$u=\low r 2 p^\inv \low d 1$};
\node at (18.5,2) [color=black, anchor=west]{$t=\low r 3 p^\inv \low d 2p^\inv e$};
\end{tikzpicture}

\begin{tikzpicture}[scale=.6]
\draw[line width=1pt, ->,>=stealth, color=black](-4,2)--(-2,2);
\node at (-3,2)[anchor=south]{$D_{\delta'_k}$};
\draw[line width=1pt, color=black] (-2,0)--(18,0);
\draw[color=blue, line width=1.5 pt,  <-,>=stealth] (0,0).. controls (0,3) and (5,3) .. (5,0);
\draw[color=black, line width=1.5 pt,  <-,>=stealth] (2,0).. controls (2,5) and (17,5) .. (17,0);
\draw[color=magenta, line width=1.5 pt,  <-,>=stealth] (3,0).. controls (3,4) and (13,4) .. (13,0);
\draw[color=orange, line width=1.5 pt,  <-,>=stealth] (4,0).. controls (4,3) and (9,3) .. (9,0);
\draw[color=olive, line width=1.5 pt,  <-,>=stealth] (11,0)--(11,5);
\draw[color=brown, line width=1.5 pt,  <-,>=stealth] (15,0)--(15,5);
\node at (15.5,2.5) [color=black, anchor=south west]{$\low t 8$};
\node at (13.5,0) [color=magenta, anchor=north]{$\low t 7 \low u 5 S(\low t 2)$};
\node at (8.3,0) [color=orange, anchor=north]{$\low t 6\low u 4  \low  r 2 S(\low t 4\low u 2)$};
\node at (3.5,2.2) [color=blue, anchor=south east]{$\low a 2 S(\low t 5\low u 3 \low r 1)$};
\node at (11.5,5)[color=olive, anchor=south] {$\low t 3 \low u 1 y$};
\node at (15,5)[color=brown, anchor=south] {$\low t 1 z$};
\node at (18.5,4) [color=black, anchor=west]{$r=\low a 1 c$};
\node at (18.5,3) [color=black, anchor=west]{$u=\low r 3 p^\inv \low d 1$};
\node at (18.5,2) [color=black, anchor=west]{$t=\low r 4 p^\inv \low d 2p^\inv e$};
\end{tikzpicture}

Applying this transformation three times, we get

\begin{align}
&\begin{tikzpicture}[scale=.6]
\node at (-3,4)[anchor=south west] {$(D_{\delta'_k}\circ D_{\delta'_i}\circ D_{\delta'_j}\circ D_{\alpha_g})^3:$};
\draw[line width=1pt, color=black] (-2,0)--(18,0);
\draw[color=blue, line width=1.5 pt,   <-,>=stealth] (0,0).. controls (0,3) and (5,3) .. (5,0);
\draw[color=black, line width=1.5 pt,  <-,>=stealth] (2,0).. controls (2,5) and (17,5) .. (17,0);
\draw[color=magenta, line width=1.5 pt,  <-,>=stealth] (3,0).. controls (3,4) and (13,4) .. (13,0);
\draw[color=orange, line width=1.5 pt,  <-,>=stealth] (4,0).. controls (4,3) and (9,3) .. (9,0);
\draw[color=olive, line width=1.5 pt,  <-,>=stealth] (11,0)--(11,5);
\draw[color=brown, line width=1.5 pt,  <-,>=stealth] (15,0)--(15,5);
\node at (15.5,2.5) [color=black, anchor=south west]{$\low c 3 p^\inv \low d 2 p^\inv e$};
\node at (13,2) [color=magenta, anchor=south]{$\low c 2 p^\inv \low d 1$};
\node at (8.5,1.7) [color=orange, anchor=south]{$\low c 1 $};
\node at (2,2.3) [color=blue, anchor=south]{$a$};
\node at (11,5)[color=olive, anchor=south] {$y$};
\node at (15,5)[color=brown, anchor=south] {$z$};
\end{tikzpicture}\nonumber\\
&\begin{tikzpicture}[scale=.6]
\draw[line width=1pt, ->,>=stealth] (-3,2)--(-2,2);
\draw[line width=1pt, color=black] (-2,0)--(18,0);
\draw[color=blue, line width=1.5 pt,   <-,>=stealth] (0,0).. controls (0,3) and (5,3) .. (5,0);
\draw[color=black, line width=1.5 pt,  <-,>=stealth] (2,0).. controls (2,5) and (17,5) .. (17,0);
\draw[color=magenta, line width=1.5 pt,  <-,>=stealth] (3,0).. controls (3,4) and (13,4) .. (13,0);
\draw[color=orange, line width=1.5 pt,  <-,>=stealth] (4,0).. controls (4,3) and (9,3) .. (9,0);
\draw[color=olive, line width=1.5 pt,  <-,>=stealth] (11,0)--(11,5);
\draw[color=brown, line width=1.5 pt,  <-,>=stealth] (15,0)--(15,5);
\node at (15.5,2.5) [color=black, anchor=south west]{$e'$};
\node at (12.5,2) [color=magenta, anchor=south]{$d'$};
\node at (8.5,1.7) [color=orange, anchor=south]{$c'$};
\node at (2,2.3) [color=blue, anchor=south]{$a'$};
\node at (11,5)[color=olive, anchor=south] {$y'$};
\node at (15,5)[color=brown, anchor=south] {$z'$};
\end{tikzpicture}
\label{eq:rel41}
\end{align}
where
	 \begin{align*}
		&v = a_{(5)}S(c_{(1)}) S(a_{(1)})p^{-1}c_{(5)}p^{-1} d_{(3)}p^{-1}e_{(2)}
		 &&a'= v_{(9)} a_{(3)}S(v_{(5)})\\
		 &y'= v_{(3)} d_{(1)}y
		 &&z'= v_{(1)} e_{(1)} z\\
		 &c'= v_{(6)} c_{(3)} S(v_{(4)})
		&&d'= v_{(7)} c_{(4)}p^{-1}d_{(2)}S(v_{(2)})\\
		 &e'= v_{(8)} a_{(2)}c_{(2)}p S(a_{(4)}).
	 \end{align*}
	 
The morphism $ D_{\gamma'_{j,k}}\circ D_{\gamma'_{k,i}}\circ D_{\gamma'_{i,j}}$ is obtained from the first diagram in \eqref{eq:chord1} and given by
 
\begin{align}
&\begin{tikzpicture}[scale=.6]
\node at (-3,2)[anchor=west]{$\quad$};
\draw[line width=1pt, color=black] (-2,0)--(18,0);
\draw[color=blue, line width=1.5 pt,  <-,>=stealth] (0,0).. controls (0,3) and (5,3) .. (5,0);
\draw[color=red, line width=1.5 pt,  <-,>=stealth] (-1,0).. controls (-1,5) and (17,5) .. (17,0);
\draw[color=orange, line width=1.5 pt,  <-,>=stealth] (4,0).. controls (4,3) and (9,3) .. (9,0);
\draw[color=violet, line width=1.5 pt,  <-,>=stealth] (10,0).. controls (10,2) and (12,2) .. (12,0);
\draw[color=purple, line width=1.5 pt,  <-,>=stealth] (14,0).. controls (14,2) and (16,2) .. (16,0);
\draw[color=olive, line width=1.5 pt,  <-,>=stealth] (11,0)--(11,5);
\draw[color=brown, line width=1.5 pt,  <-,>=stealth] (15,0)--(15,5);
\node at (-2,3.7) [color=red, anchor=south west]{$a_{(3)}S(c_{(1)}) S(a_{(1)})p^{-1} c_{(3)}p^{-1} d_{(2)}p^{-1} e_{(2)}$};
\node at (8.5,1.7) [color=orange, anchor=south]{$\low c 2$};
\node at (4.5,1.7) [color=blue, anchor=south]{$a_{(2)}$};
\node at (10.5,1) [color=violet, anchor=south east]{$\low d 1$};
\node at (14.5,1) [color=purple, anchor=south east]{$\low e 1$};
\node at (11,5)[anchor=south, color=olive] {$y$};
\node at (15,5)[anchor=south, color=brown] {$z$};
\end{tikzpicture}\nonumber
\intertext{}\nonumber\\
 &\begin{tikzpicture}[scale=.6]
 \draw[color=black, line width=1pt, ->,>=stealth] (-3,2)--(-1,2);
 \node at (-2,2)[anchor=south]{$D_{\gamma'_{i,j}}$};
\draw[line width=1pt, color=black] (-2,0)--(18,0);
\draw[color=blue, line width=1.5 pt,  <-,>=stealth] (0,0).. controls (0,3) and (5,3) .. (5,0);
\draw[color=orange, line width=1.5 pt,  <-,>=stealth] (4,0).. controls (4,3) and (9,3) .. (9,0);
\draw[color=violet, line width=1.5 pt,  <-,>=stealth] (10,0).. controls (10,2) and (12,2) .. (12,0);
\draw[color=purple, line width=1.5 pt,  <-,>=stealth] (14,0).. controls (14,2) and (16,2) .. (16,0);
\draw[color=red, line width=1.5 pt,  <-,>=stealth] (-1,0).. controls (-1,5) and (17,5) .. (17,0);
\draw[color=olive, line width=1.5 pt,  <-,>=stealth] (11,0)--(11,5);
\draw[color=brown, line width=1.5 pt,  <-,>=stealth] (15,0)--(15,5);
\node at (-2,3.7) [color=red, anchor=south west]{$a_{(3)}S(c_{(1)}) S(a_{(1)})p^{-1} c_{(3)}p^{-1} d_{(2)}p^{-1} e_{(3)}$};
\node at (8.5,1.7) [color=orange, anchor=south]{$\low c 2$};
\node at (4.5,1.7) [color=blue, anchor=south]{$a_{(2)}$};
\node at (10.5,1) [color=violet, anchor=south east]{$\low d 1$};
\node at (14.5,1) [color=purple, anchor=south east]{$\low e 2$};
\node at (11,5)[anchor=south, color=olive] {$y$};
\node at (15,5)[anchor=south, color=brown] {$\low e 1 z$};
\end{tikzpicture}
 \nonumber
\intertext{}\nonumber\\
  &\begin{tikzpicture}[scale=.6]
 \draw[color=black, line width=1pt, ->,>=stealth] (-3,2)--(-1,2);
 \node at (-2,2)[anchor=south]{$D_{\gamma'_{k,i}}$};
\draw[line width=1pt, color=black] (-2,0)--(18,0);
\draw[color=blue, line width=1.5 pt,  <-,>=stealth] (0,0).. controls (0,3) and (5,3) .. (5,0);
\draw[color=red, line width=1.5 pt,  <-,>=stealth] (-1,0).. controls (-1,5) and (17,5) .. (17,0);
\draw[color=orange, line width=1.5 pt,  <-,>=stealth] (4,0).. controls (4,3) and (9,3) .. (9,0);
\draw[color=violet, line width=1.5 pt,  <-,>=stealth] (10,0).. controls (10,2) and (12,2) .. (12,0);
\draw[color=purple, line width=1.5 pt,  <-,>=stealth] (14,0).. controls (14,2) and (16,2) .. (16,0);
\draw[color=olive, line width=1.5 pt,  <-,>=stealth] (11,0)--(11,5);
\draw[color=brown, line width=1.5 pt,  <-,>=stealth] (15,0)--(15,5);
\node at (-2,3.7) [color=red, anchor=south west]{$a_{(3)}S(c_{(1)}) S(a_{(1)})p^{-1} c_{(3)}p^{-1} d_{(3)}p^{-1} e_{(3)}$};
\node at (8.5,1.7) [color=orange, anchor=south]{$\low c 2$};
\node at (4.5,1.7) [color=blue, anchor=south]{$a_{(2)}$};
\node at (10.5,1) [color=violet, anchor=south east]{$\low d 2$};
\node at (14.5,1) [color=purple, anchor=south east]{$\low e 2$};
\node at (11,5)[anchor=south, color=olive] {$\low d 1 y $};
\node at (15,5)[anchor=south, color=brown] {$\low e 1 z$};
\end{tikzpicture}
 \nonumber
\intertext{}\nonumber\\
&\begin{tikzpicture}[scale=.6]
        \node at (-4, 5)[anchor=south west]{$v=\low a 3 S(\low c 1) S(\low a 1) p^\inv \low c 4 p^\inv \low d 3 p^\inv \low e 3$};
 \draw[color=black, line width=1pt,  ->,>=stealth] (-3,2)--(-1,2);
 \node at (-2,2)[anchor=south]{$D_{\gamma'_{j,k}}$};
\draw[line width=1pt, color=black] (-2,0)--(18,0);
\draw[color=blue, line width=1.5 pt,  <-,>=stealth] (0,0).. controls (0,3) and (5,3) .. (5,0);
\draw[color=orange, line width=1.5 pt,  <-,>=stealth] (4,0).. controls (4,3) and (9,3) .. (9,0);
\draw[color=red, line width=1.5 pt,  <-,>=stealth] (-1,0).. controls (-1,5) and (17,5) .. (17,0);
\draw[color=olive, line width=1.5 pt,  <-,>=stealth] (11,0)--(11,5);
\draw[color=brown, line width=1.5 pt,  <-,>=stealth] (15,0)--(15,5);
\draw[color=violet, line width=1.5 pt,  <-,>=stealth] (10,0).. controls (10,2) and (12,2) .. (12,0);
\draw[color=purple, line width=1.5 pt,  <-,>=stealth] (14,0).. controls (14,2) and (16,2) .. (16,0);
\node at (9,2.2) [color=orange, anchor=south]{$\low v 9 \low c 2 S(\low v 7)$};
\node at (2.5,2.2) [color=blue, anchor=south west]{$\low v {10} \low a 2 S(\low v 8)$};
\node at (0,2.7) [color=red, anchor=south west]{$v_{(11)}$};
\node at (11,-1.2) [color=violet, anchor=south]{$v_{(6)}\low d 2 S(v_{(4)})$};
\node at (15,-1.2) [color=purple, anchor=south]{$v_{(3)} \low e 2 S(v_{(1)})$};
\node at (11,5)[anchor=south, color=olive] {$\low v 5 \low d 1 y$};
\node at (15,5)[anchor=south, color=brown] {$\low v 2 \low e 1 z$};
\end{tikzpicture}
\label{eq:rel4pppp}
\end{align}
After applying
$\epsilon_{\gamma'_{i,j}}, \epsilon_{\gamma'_{k,i}}$ and $\epsilon_{\gamma'_{j,k}}$ 
to diagram \eqref{eq:rel4pppp} to eliminate the edges  $\gamma'_{i,j}$, $\gamma'_{k,i}$ and $\gamma'_{j,k}$ 
and applying $\epsilon_{\delta'_i}$ and $\epsilon_{\delta'_j}$ 
to \eqref{eq:rel41} to eliminate the edges  $\delta'_i$ and $\delta'_j$, 
one finds that the resulting transformations agree. This proves relation (iv) for $1 \leq j<i<k$.

$\bullet$  The proof for the case $1 \leq j=i<k$ is obtained from the same computation by removing the  edge labeled $z$ and setting $e=1$.
Relation (iv) for  $1 \leq j<i=k$ follows similarly by  removing the edge  labeled $y$ and setting $d=1$. 

$\bullet$ For the case $0 = j \leq i \leq k$ all of the paths representing the curves in relation (iv) are face paths. The twists along these face paths are given by adding edges  to the face paths $\gamma_{i,0}$, $\gamma_{k,i}$, $\gamma_{0,k}$, $\delta_i$ and $\delta_k$, which yields the following diagram

\begin{center}
\begin{tikzpicture}[scale=.4]
\draw[line width=1pt, color=black] (-3,0)--(19,0);
\draw[color=blue, line width=1.5 pt,  <-,>=stealth] (0,0).. controls (0,3) and (5,3) .. (5,0);
\draw[color=red, line width=1.5 pt,  <-,>=stealth] (-1,0)..controls(-1,2) and (1,2)..(1,0);
\draw[color=cyan, line width=1.5 pt,  ->,>=stealth] (2,0).. controls (2,5) and (17,5) .. (17,0);
\draw[color=black, line width=1.5 pt,  ->,>=stealth] (-2,0).. controls (-2,6) and (18,6) .. (18,0);
\draw[color=magenta, line width=1.5 pt,  ->,>=stealth] (3,0).. controls (3,4) and (13,4) .. (13,0);
\draw[color=olive, line width=1.5 pt,  <-,>=stealth] (11,0)--(11,5);
\draw[color=brown, line width=1.5 pt,  <-,>=stealth] (15,0)--(15,5);
\draw[color=violet, line width=1.5 pt,  ->,>=stealth] (6,0).. controls (6,2) and (12,2) .. (12,0);
\draw[color=purple, line width=1.5 pt,  ->,>=stealth] (14,0).. controls (14,2) and (16,2) .. (16,0);
\node at (5,0)[anchor=north, color=blue]{$\alpha_g$};
\node at (17,0)[anchor=north, color=cyan]{$\delta'_i$};
\node at (13,0)[anchor=north, color=magenta]{$\delta'_k$};
\node at (-0.5,0)[anchor=north, color=red]{$\beta_g=\delta_{0}$};
\node at (10.5,1.3) [color=violet, anchor=south east]{$\gamma'_{0,k}$};
\node at (15,1.2) [color=purple, anchor=south east]{$\gamma'_{k,i}$};
\node at (18,1)[anchor=west]{$\gamma'_{i,0}$};
\end{tikzpicture}
\end{center}
Sliding the starting and target end of $\beta_{g}$ along the edge $\alpha_{g}$  commutes with the  twist along $\delta_0=\beta_{g}$ by Lemma~\ref{lem:slideDT}, 2, with the twists along $\delta'_i$ and $\delta'_k$ by Lemma~\ref{lem:slideDT}, 3.~and with the twists along $\gamma'_{i,0}$, $\gamma'_{0,k}$ and $\gamma'_{k,i}$ by the commutativity relation.
 It is therefore sufficient to prove that relation (iv) holds for the diagram obtained by this slide:

\begin{center}
\begin{tikzpicture}[scale=.4]
\draw[line width=1pt, color=black] (-2,0)--(19,0);
\draw[color=blue, line width=1.5 pt,  <-,>=stealth] (0,0).. controls (0,3) and (5,3) .. (5,0);
\draw[color=cyan, line width=1.5 pt,  ->,>=stealth] (2,0).. controls (2,5) and (17,5) .. (17,0);
\draw[color=magenta, line width=1.5 pt,  ->,>=stealth] (3,0).. controls (3,4) and (13,4) .. (13,0);
\draw[color=red, line width=1.5 pt,  ->,>=stealth] (4,0).. controls (4,3) and (9,3) .. (9,0);
\draw[color=olive, line width=1.5 pt,  <-,>=stealth] (11,0)--(11,5);
\draw[color=brown, line width=1.5 pt,  <-,>=stealth] (15,0)--(15,5);
\draw[color=violet, line width=1.5 pt, ->,>=stealth] (10,0).. controls (10,2) and (12,2) .. (12,0);
\draw[color=purple, line width=1.5 pt,  ->,>=stealth] (14,0).. controls (14,2) and (16,2) .. (16,0);
\draw[color=black, line width=1.5 pt,  ->,>=stealth] (-1,0).. controls (-1,6) and (18,6) .. (18,0);
\node at (0,0)[anchor=north, color=blue]{$\alpha_g$};
\node at (17,0)[anchor=north, color=cyan]{$\delta'_i$};
\node at (13,0)[anchor=north, color=magenta]{$\delta'_k$};
\node at (9,0)[anchor=north, color=red]{$\beta_{g}=\delta_0$};
\node at (10.8,1) [color=violet, anchor=south east]{$\gamma'_{0,k}$};
\node at (15,1.2) [color=purple, anchor=south east]{$\gamma'_{k,i}$};
\node at (18,1)[anchor=west]{$\gamma'_{i,0}$};
\end{tikzpicture}
\end{center}
Up to the labelling of the edges, this coincides with diagram~\eqref{eq:RelationIVGraph}, for which we have already proven the claim. This concludes the proof.
\end{proof}

Theorem \ref{th:maptheorem} extends  Theorem \ref{the:edge1} to surfaces  of genus $g\geq 1$ with more than one boundary component and  gives an explicit description of the mapping class group action in terms of generating Dehn twists. 
It will also allow us to determine sufficient conditions under which the edge slides  and associated twists define mapping class group actions for  closed surfaces.

For this, note that  a  surface $\Sigma'$ of genus $g\geq 1$ with $n\geq 0$ boundary components is obtained from the surface $\Sigma$ of genus $g$ with $n+1$ boundary components in Figure \ref{fig:pi1gens} by attaching a disc to the coloured boundary component.  If we work with the graph $\Gamma$ from \eqref{eq:standardchord} and Figure \ref{fig:pi1gens}, this amounts to attaching a 
disc to the path $\gamma_{1,0}=f$ in \eqref{eq:pathexpress} and in Figure \ref{fig:pi1gens}. By Theorem \ref{th:gervais} the  presentation of the mapping class group $\mathrm{Map}(\Sigma')$ 
is then obtained from the presentation of  $\mathrm{Map}(\Sigma)$ used in Theorem \ref{th:maptheorem} 
by imposing the additional relations  
\begin{align}\label{eq:addrel}
	D_{\gamma_{1,0}}=1,  \quad D_{\delta_0}=D_{\delta_1}, \quad 
	D_{\gamma_{0,1}}=(D_{\delta_{0}}^3D_{\alpha_{g}})^3,
	\quad   D_{\gamma_{0,k}}=D_{\gamma_{1,k}},\quad 
D_{\gamma_{k,0}}=D_{\gamma_{k,1}}
\end{align}
for $1< k\leq n+2g-2$, see \cite{G01}. In general, they do not hold for the action of $\mathrm{Map}(\Sigma)$ on $H^{\oo 2(n+g)}$ from Theorem \ref{th:maptheorem} and neither on the invariants of the $H$-left module structure nor on the coinvariants of the $H$-left comodule structure associated with the baseline of diagram \eqref{eq:standardchord} (cf.~Example \ref{ex:groupex}). Instead, we have to impose both invariance and coinvariance and work with the biinvariants of the Yetter-Drinfeld module structure associated with the baseline of 
\eqref{eq:standardchord}.

We first show that the group homomorphism $\rho$ from Theorem \ref{th:maptheorem} induces an action of the mapping class group $\mathrm{Map}(\Sigma)$ on the biinvariants of this Yetter-Drinfeld module. 
For this, recall that the generating Dehn twists from Theorem \ref{th:maptheorem} do not slide edge ends over the associated cilium and   are automorphisms 
 of this Yetter-Drinfeld module  by Proposition \ref{prop:vertfacecomp}.  From Lemma \ref{lem:maclemma} we then obtain

\begin{corollary} \label{cor:maccor} Let $\mac$ be a finitely complete and cocomplete symmetric monoidal category  and  $H$ a pivotal  Hopf monoid in $\mac$. Then  $\rho$ from Theorem \ref{th:maptheorem} induces a group homomorphism
\begin{align}\label{eq:mapinv}
\rho_{inv}:\mathrm{Map}(\Sigma)\to \mathrm{Aut}(H^{\oo 2(n+g)}_{inv})
\end{align}
\end{corollary}
\begin{proof}
By Lemma \ref{lem:maclemma} every element  $\phi\in \mathrm{Map}(\Sigma)$ induces a unique automorphism $\rho_{inv}(\phi)$ of $H^{\oo 2(n+g)}_{inv}$ with $I\circ \rho_{inv}(\phi)\circ P=\pi\circ \rho(\phi)\circ \iota$. This defines a group homomorphism as in \eqref{eq:mapinv}, because $\rho$ is a group homomorphism, $I$ is a monomorphism and $P$ an epimorphism.
\end{proof}

We now show that under suitable assumptions  the group homomorphism $\rho$ from Theorem \ref{th:maptheorem} not only induces an action of the mapping class group $\mathrm{Map}(\Sigma)$ on the biinvariants of this Yetter-Drinfeld module structure but an action of the mapping class group $\mathrm{Map}(\Sigma')$, where $\Sigma'$
is the  surface obtained by attaching a disc to the face $f$  in Figure \ref{fig:pi1gens}. For this, we require the following technical lemma.

\begin{lemma}\label{lem:facelemma}  Let $\mac$ be a finitely complete and cocomplete symmetric monoidal category  and  $H$ a pivotal  Hopf monoid in $\mac$ with $m\circ (p\oo p)=\eta$. Let $\Gamma$ be a directed ribbon graph and $\gamma_1,\gamma_2$ closed face paths  in $\Gamma$ such that  $\gamma=\gamma_{1}\gamma_{2}$ is a ciliated face.
Then we have for the biinvariants of the Yetter-Drinfeld module structure at this cilium
	\begin{align}
		D^{inv}_{\gamma_{1}}=D^{inv}_{\gamma_{2}},\qquad D^{inv}_{\gamma} = 1_{H^{\oo E}_{inv}}.
		\nonumber
	\end{align}
	\label{lemma:InvCoinvFacePathtwists}\end{lemma}
\begin{proof}   The second claim follows from the first by setting $\gamma_1=\gamma$ and taking the trivial path for $\gamma_2$. 
To prove the first claim, note
that these  twists slide no edge ends over the cilium and hence  the  automorphisms  $D_{\gamma_i}: H^{\oo E}\to H^{\oo E}$  are automorphisms  of Yetter-Drinfeld modules by Proposition \ref{prop:vertfacecomp}.  By Lemma \ref{lem:maclemma}  they induce unique automorphisms
 $D^{inv}_{\gamma_i}: H^{\oo E}_{inv}\to H^{\oo E}_{inv}$ with
 $
 I\circ D^{inv}_{\gamma_i}\circ P=\pi\circ D_{\gamma_i}\circ \iota
 $. As $I$ is a monomorphism and $P$ an epimorphism,  it is sufficient to show that $\pi\circ D_{\gamma_1}\circ \iota=\pi\circ D_{\gamma_2}\circ \iota$. As the twist along $\gamma_i$ is obtained by adding an edge $\gamma'_i$ to $\gamma_i$ and twisting along $\gamma'_i$ , it is sufficient to verify this for the 
 chord diagram	
\begin{align*}
\begin{tikzpicture}[scale=.5]
\begin{scope}
\draw[line width=1pt] (-5,0)--(4,0);
\draw[color=red, line width=1.5 pt, ->,>=stealth] (-4,0).. controls (-4,2) and (-1,2) .. (-1,0);
\draw[color=blue, line width=1.5 pt, ->,>=stealth] (0,0).. controls (0,2) and (3,2) .. (3,0);
\draw[color=black, line width=1.5 pt, ->,>=stealth] (-2.5,2.5)-- (-2.5,0);
\draw[color=violet, line width=1.5 pt, ->,>=stealth] (1.5,2.5)-- (1.5,0);
\node at (-2.3,1.5)[anchor=south west, color=red]{$\gamma_{2}'$};
\node at (1.7,1.5)[anchor=south west, color=blue]{$\gamma_{1}'$};
\node at (-2.5,2.5)[anchor=south, color=black]{$\mu$};
\node at (1.5,2.5)[anchor=south, color=violet]{$\nu$};
\end{scope}
\end{tikzpicture}
\tComma
\end{align*}
	in which we replaced the edge ends between the starting and target ends of $\gamma_1',\gamma_2'$ by single edges $\mu,\nu$, as in  Lemma \ref{rem:shorthand}.
  The Yetter-Drinfeld module structure for the cilium is given by
  \begingroup
\allowdisplaybreaks
  \begin{align*}
  &\begin{tikzpicture}[scale=.55]
\begin{scope}
\draw[line width=1pt] (-5,0)--(4,0);
\draw[color=red, line width=1.5 pt, ->,>=stealth] (-4,0).. controls (-4,2) and (-1,2) .. (-1,0);
\draw[color=blue, line width=1.5 pt, ->,>=stealth] (0,0).. controls (0,2) and (3,2) .. (3,0);
\draw[color=black, line width=1.5 pt, ->,>=stealth] (-2.5,2.5)-- (-2.5,0);
\draw[color=violet, line width=1.5 pt, ->,>=stealth] (1.5,2.5)-- (1.5,0);
\node at (-2.3,1.5)[anchor=south west, color=red]{$b$};
\node at (1.7,1.5)[anchor=south west, color=blue]{$a$};
\node at (-2.5,2.5)[anchor=south, color=black]{$c$};
\node at (1.5,2.5)[anchor=south, color=violet]{$d$};
\node at (0,-.5)[anchor=north]{$h$};
\end{scope}
\draw[line width=1pt, ->,>=stealth](4,1)--(6,1);
\node at (5,1)[anchor=south]{$\rhd$}; 
\begin{scope}[shift={(11,0)}]
\draw[line width=1pt] (-5,0)--(4,0);
\draw[color=red, line width=1.5 pt, ->,>=stealth] (-4,0).. controls (-4,2) and (-1,2) .. (-1,0);
\draw[color=blue, line width=1.5 pt, ->,>=stealth] (0,0).. controls (0,2) and (3,2) .. (3,0);
\draw[color=black, line width=1.5 pt, ->,>=stealth] (-2.5,2.5)-- (-2.5,0);
\draw[color=violet, line width=1.5 pt, ->,>=stealth] (1.5,2.5)-- (1.5,0);
\node at (-2.5,1.5)[anchor=south west, color=red]{$\low h 4 b S(\low h 6)$};
\node at (1.7,1.5)[anchor=south west, color=blue]{$\low h 1 a S(\low h 3)$};
\node at (-2.5,2.5)[anchor=south, color=black]{$\low h 5 c$};
\node at (1.5,2.5)[anchor=south, color=violet]{$\low h 2d$};
\end{scope}
\end{tikzpicture}\\
&\begin{tikzpicture}[scale=.55]
\begin{scope}
\draw[line width=1pt] (-5,0)--(4,0);
\draw[color=red, line width=1.5 pt, ->,>=stealth] (-4,0).. controls (-4,2) and (-1,2) .. (-1,0);
\draw[color=blue, line width=1.5 pt, ->,>=stealth] (0,0).. controls (0,2) and (3,2) .. (3,0);
\draw[color=black, line width=1.5 pt, ->,>=stealth] (-2.5,2.5)-- (-2.5,0);
\draw[color=violet, line width=1.5 pt, ->,>=stealth] (1.5,2.5)-- (1.5,0);
\node at (-2.3,1.5)[anchor=south west, color=red]{$b$};
\node at (1.7,1.5)[anchor=south west, color=blue]{$a$};
\node at (-2.5,2.5)[anchor=south, color=black]{$c$};
\node at (1.5,2.5)[anchor=south, color=violet]{$d$};
\end{scope}
\draw[line width=1pt, ->,>=stealth](4,1)--(6,1);
\node at (5,1)[anchor=south]{$\delta$}; 
\begin{scope}[shift={(11,0)}]
\draw[line width=1pt] (-5,0)--(4,0);
\draw[color=red, line width=1.5 pt, ->,>=stealth] (-4,0).. controls (-4,2) and (-1,2) .. (-1,0);
\draw[color=blue, line width=1.5 pt, ->,>=stealth] (0,0).. controls (0,2) and (3,2) .. (3,0);
\draw[color=black, line width=1.5 pt, ->,>=stealth] (-2.5,2.5)-- (-2.5,0);
\draw[color=violet, line width=1.5 pt, ->,>=stealth] (1.5,2.5)-- (1.5,0);
\node at (-2.3,1.5)[anchor=south west, color=red]{$\low b 2$};
\node at (1.7,1.5)[anchor=south west, color=blue]{$\low a 2$};
\node at (-2.5,2.5)[anchor=south, color=black]{$c$};
\node at (1.5,2.5)[anchor=south, color=violet]{$d$};
\node at (0,-.5)[anchor=north]{$\low a 1 \low b 1$};
\end{scope}
\end{tikzpicture}
\end{align*}
With the identities $\delta\circ \iota=(\eta\oo 1_{H^{\oo 4}})\circ \iota$, $\pi\circ \rhd=\pi\circ (\epsilon\oo 1_{H^{\oo 4}})$  and $m\circ (p\oo p)=\eta$ we then obtain 
 \begin{align*}
 &\begin{tikzpicture}[scale=.55]
\begin{scope}
\draw[line width=1pt] (-5,0)--(4,0);
\draw[color=red, line width=1.5 pt, ->,>=stealth] (-4,0).. controls (-4,2) and (-1,2) .. (-1,0);
\draw[color=blue, line width=1.5 pt, ->,>=stealth] (0,0).. controls (0,2) and (3,2) .. (3,0);
\draw[color=black, line width=1.5 pt, ->,>=stealth] (-2.5,2.5)-- (-2.5,0);
\draw[color=violet, line width=1.5 pt, ->,>=stealth] (1.5,2.5)-- (1.5,0);
\node at (-2.3,1.5)[anchor=south west, color=red]{$b$};
\node at (1.7,1.5)[anchor=south west, color=blue]{$a$};
\node at (-2.5,2.5)[anchor=south, color=black]{$c$};
\node at (1.5,2.5)[anchor=south, color=violet]{$d$};
\end{scope}
\begin{scope}[shift={(11,0)}]
\node at (-6,1)[anchor=south]{$\overset{\iota}{\equiv}$};
\draw[line width=1pt] (-5,0)--(4,0);
\draw[color=red, line width=1.5 pt, ->,>=stealth] (-4,0).. controls (-4,2) and (-1,2) .. (-1,0);
\draw[color=blue, line width=1.5 pt, ->,>=stealth] (0,0).. controls (0,2) and (3,2) .. (3,0);
\draw[color=black, line width=1.5 pt, ->,>=stealth] (-2.5,2.5)-- (-2.5,0);
\draw[color=violet, line width=1.5 pt, ->,>=stealth] (1.5,2.5)-- (1.5,0);
\node at (-2.3,1.5)[anchor=south west, color=red]{$b_{(2)}$};
\node at (1.7,1.5)[anchor=south west, color=blue]{$\epsilon(a)S(b_{(1)})$};
\node at (-2.5,2.5)[anchor=south, color=black]{$c$};
\node at (1.5,2.5)[anchor=south, color=violet]{$d$};
\end{scope}
\end{tikzpicture}
\\
 &\begin{tikzpicture}[scale=.55]
 \begin{scope}[shift={(0,0)}]
 \draw[line width=1pt, ->,>=stealth](-7,1)--(-5,1);
\node at (-6,1)[anchor=south]{$D_{\gamma'_1}$}; 
\draw[line width=1pt] (-5,0)--(4,0);
\draw[color=red, line width=1.5 pt, ->,>=stealth] (-4,0).. controls (-4,2) and (-1,2) .. (-1,0);
\draw[color=blue, line width=1.5 pt, ->,>=stealth] (0,0).. controls (0,2) and (3,2) .. (3,0);
\draw[color=black, line width=1.5 pt, ->,>=stealth] (-2.5,2.5)-- (-2.5,0);
\draw[color=violet, line width=1.5 pt, ->,>=stealth] (1.5,2.5)-- (1.5,0);
\node at (-2.3,1.5)[anchor=south west, color=red]{$b_{(3)}$};
\node at (1,0)[anchor=north west, color=blue]{$\epsilon(a)S(b_{(2)})$};
\node at (-2.5,2.5)[anchor=south, color=black]{$c$};
\node at (1.5,2.5)[anchor=south, color=violet]{$b_{(1)}pd$};
\end{scope}
\begin{scope}[shift={(11,0)}]
\draw[line width=1pt, ->,>=stealth](-7,1)--(-5,1);
\node at (-6,1)[anchor=south]{$D_{\gamma'_2}^\inv$}; 
\draw[line width=1pt] (-5,0)--(4,0);
\draw[color=red, line width=1.5 pt, ->,>=stealth] (-4,0).. controls (-4,2) and (-1,2) .. (-1,0);
\draw[color=blue, line width=1.5 pt, ->,>=stealth] (0,0).. controls (0,2) and (3,2) .. (3,0);
\draw[color=black, line width=1.5 pt, ->,>=stealth] (-2.5,2.5)-- (-2.5,0);
\draw[color=violet, line width=1.5 pt, ->,>=stealth] (1.5,2.5)-- (1.5,0);
\node at (-2.3,1.5)[anchor=south west, color=red]{$\low b 3$};
\node at (1,0)[anchor=north west, color=blue]{$\epsilon(a)S(b_{(2)})$};
\node at (-2.5,2.5)[anchor=south, color=black]{$\low b 4 p c$};
\node at (1.5,2.5)[anchor=south, color=violet]{$b_{(1)}pd$};
\end{scope}
\end{tikzpicture}
\\
&\begin{tikzpicture}[scale=.55]
\begin{scope}[shift={(0,0)}]
\node at (-5.5,1)[anchor=south]{$=$};
\draw[line width=1pt] (-5,0)--(4,0);
\draw[color=red, line width=1.5 pt, ->,>=stealth] (-4,0).. controls (-4,2) and (-1,2) .. (-1,0);
\draw[color=blue, line width=1.5 pt, ->,>=stealth] (0,0).. controls (0,2) and (3,2) .. (3,0);
\draw[color=black, line width=1.5 pt, ->,>=stealth] (-2.5,2.5)-- (-2.5,0);
\draw[color=violet, line width=1.5 pt, ->,>=stealth] (1.5,2.5)-- (1.5,0);
\node at (-2.3,1.5)[anchor=south west, color=red]{$\low b 5$};
\node at (-1.5,0)[anchor=north west, color=blue]{$\epsilon(a) \low b 2 T(b_{(1)})    S(\low b 4 p)$};
\node at (-2.5,2.5)[anchor=south, color=black]{$\low b 6 p c$};
\node at (1.5,2.5)[anchor=south, color=violet]{$\low b 3 pd$};
\end{scope}
\begin{scope}[shift={(11,0)}]
\node at (-6,1)[anchor=south]{$=$};
\draw[line width=1pt] (-5,0)--(4,0);
\draw[color=red, line width=1.5 pt, ->,>=stealth] (-4,0).. controls (-4,2) and (-1,2) .. (-1,0);
\draw[color=blue, line width=1.5 pt, ->,>=stealth] (0,0).. controls (0,2) and (3,2) .. (3,0);
\draw[color=black, line width=1.5 pt, ->,>=stealth] (-2.5,2.5)-- (-2.5,0);
\draw[color=violet, line width=1.5 pt, ->,>=stealth] (1.5,2.5)-- (1.5,0);
\node at (-5.5,0)[anchor=north west, color=red]{$\low b 5 p  \low b 8 S( \low b 7 p)$};
\node at (0,0)[anchor=north west, color=blue]{$\epsilon(a) \low b 2 T(b_{(1)})   S(\low b 4 p)$};
\node at (-2.5,2.5)[anchor=south, color=black]{$\low b 6 p c$};
\node at (1.5,2.5)[anchor=south, color=violet]{$\low b 3 pd$};
\end{scope}
\end{tikzpicture}\\
&\begin{tikzpicture}[scale=.55]
\begin{scope}[shift={(0,0)}]
	\node at (-5.5,0)[anchor=south]{$\overset{\pi}{\equiv}$};
\draw[line width=1pt] (-5,0)--(4,0);
\draw[color=red, line width=1.5 pt, ->,>=stealth] (-4,0).. controls (-4,2) and (-1,2) .. (-1,0);
\draw[color=blue, line width=1.5 pt, ->,>=stealth] (0,0).. controls (0,2) and (3,2) .. (3,0);
\draw[color=black, line width=1.5 pt, ->,>=stealth] (-2.5,2.5)-- (-2.5,0);
\draw[color=violet, line width=1.5 pt, ->,>=stealth] (1.5,2.5)-- (1.5,0);
\node at (-2.3,1.5)[anchor=south west, color=red]{$  \low b 2 $};
\node at (1.7,1.5)[anchor=south west, color=blue]{$ \epsilon(a) S(b_{(1)})   $};
\node at (-2.5,2.5)[anchor=south, color=black]{$ c$};
\node at (1.5,2.5)[anchor=south, color=violet]{$d$};
\end{scope}
\begin{scope}[shift={(11,0)}]
\node at (-6,0)[anchor=south]{$\overset{\iota}{\equiv}$}; 
\draw[line width=1pt] (-5,0)--(4,0);
\draw[color=red, line width=1.5 pt, ->,>=stealth] (-4,0).. controls (-4,2) and (-1,2) .. (-1,0);
\draw[color=blue, line width=1.5 pt, ->,>=stealth] (0,0).. controls (0,2) and (3,2) .. (3,0);
\draw[color=black, line width=1.5 pt, ->,>=stealth] (-2.5,2.5)-- (-2.5,0);
\draw[color=violet, line width=1.5 pt, ->,>=stealth] (1.5,2.5)-- (1.5,0);
\node at (-2.3,1.5)[anchor=south west, color=red]{$b$};
\node at (1.7,1.5)[anchor=south west, color=blue]{$a$};
\node at (-2.5,2.5)[anchor=south, color=black]{$c$};
\node at (1.5,2.5)[anchor=south, color=violet]{$d$};
\end{scope}
\end{tikzpicture}
 \end{align*}
   \endgroup
 where we used
 first the identity $\delta\circ \iota=(\eta\oo 1_{H^{\oo 4}})\circ \iota$, 
 then the definition of the twists, 
 then the  properties of a Hopf monoid,
 then the identity $\pi\circ \rhd=\pi\circ (\epsilon\oo 1_{H^{\oo 4}})$ to pass to the fourth line and the identity $\delta\circ \iota=(\eta\oo 1_{H^{\oo 4}})\circ \iota$ in the final step. The notations $\overset{\iota}{\equiv}$
and $\overset{\pi}{\equiv}$ mean that the morphisms associated to the two labelled diagrams are coequalised by $\iota$ and equalised by $\pi$, respectively.
\end{proof}

Note that the condition $m\circ (p\oo p)=\eta$ is required only for the first relation in Lemma \ref{lem:facelemma}. The second relation holds without any assumptions on the pivotal structure, but we will need both relations in the following. Note also that this condition imposes a restriction on the antipode of the Hopf monoid, namely the condition $S^4=1_H$.

We now consider  the oriented surface $\Sigma$ of genus $g\geq 1$ with $n+1\geq 1$ boundary components from Figure \ref{fig:pi1gens} and the  surface $\Sigma'$  with $n\geq 0$ boundary components obtained by attaching a disc to the face $f$ in Figure \ref{fig:pi1gens}. Let $\Gamma$ be the  associated 
 ribbon graph  from  \eqref{eq:standardchord},  consider the Yetter-Drinfeld module structure for the baseline in   \eqref{eq:standardchord} and  the associated object $H^{\oo 2(n+g)}_{inv}$ from Definition \ref{def:biinv}.

\begin{theorem}\label{th:close} Let $\mac$ be a finitely complete and cocomplete symmetric monoidal category  and  $H$ a pivotal  Hopf monoid in $\mac$ with $m\circ (p\oo p)=\eta$. 
 The group homomorphism $\rho$ from Theorem \ref{th:maptheorem} induces a group homomorphism
$$\rho_{inv}:\mathrm{Map}(\Sigma')\to \mathrm{Aut}(H^{\oo 2(n+g)}_{inv}).$$
\end{theorem}

\begin{proof}We  show that the images of the generating Dehn twists under the group homomorphism $\rho_{inv}$ from Corollary \ref{cor:maccor}
 satisfy the additional relations in \eqref{eq:addrel}. 
As the face of the ribbon graph in \eqref{eq:standardchord} is given in terms of the paths in \eqref{eq:pathexpress}  by $f=\gamma_{1,0}=\delta_1\circ \delta_0$, we have by 
 Lemma \ref{lem:facelemma}
\begin{align}\label{eq:rhoinv}
\rho_{inv}(D_{\gamma_{1,0}})=1_{H^{\oo 2(n+g)}_{inv}},\qquad \rho_{inv}(D_{\delta_0})=\rho_{inv}(D_{\delta_1}).
\end{align}
This proves the first two relations in \eqref{eq:addrel}. Combining \eqref{eq:rhoinv} with  relation (iv) in Theorem \ref{th:gervais}  for the mapping class group $\mathrm{Map}(\Sigma)$ for $(j,i,k)=(0,0,1),(0,0,k), (1,1,k)$ and $(0,1,k)$  
yields for $1<k\leq n+2g-2$
\begin{align*}
&\rho_{inv}(D_{\gamma_{0,1}})\stackrel{(iv)}= 
\rho_{inv}(D_{\gamma_{1,0}}^{-1}(D_{\delta_{0}}^{2}D_{\delta_{1}}D_{\alpha_{g}})^{3}),
\stackrel{\eqref{eq:rhoinv}}=\rho_{inv}(  (D_{\delta_0}^3D_{\alpha_g})^3)\\
%
&\rho_{inv}(   (D_{\delta_0}^2  D_{\delta_k} D_{\alpha_g})^3)
\stackrel{(iv)}=\rho_{inv}(D_{\gamma_{0,k}}D_{\gamma_{k,0}})
=\rho_{inv}(D_{\gamma_{1,k}}D_{\gamma_{k,1}})= \rho_{inv}(D_{\gamma_{0,k}} D_{\gamma_{k,1}}). 
\end{align*} 
 The first line is the third relation in \eqref{eq:addrel}. The second line implies the last two relations in \eqref{eq:addrel}. \end{proof}

In particular, Theorem \ref{th:close}  defines actions of mapping class groups of {\em closed} surfaces of genus $g\geq 1$ on   $H^{\oo 2g}_{inv}$. That the condition on the pivotal element is necessary is already apparent in the simplest example for a surface of genus $g=2$.

\begin{example} \label{ex:groupex}Let $\Sigma$ be a surface of genus $g=2$ with one boundary component,  $\Sigma'$ the closed surface obtained by attaching a disc to the boundary  of $\Sigma$ and $\Gamma$ the  ribbon graph from Figure \ref{fig:pi1gens} and \eqref{eq:standardchord}. Let $G$ be a finite group, $H=\C[G]$ and $p\in Z(G)$ a central element.  

The left-left Yetter-Drinfeld module structure on $\C[G^{\times 4}]$ associated with the cilium in Figure \ref{fig:pi1gens} and the baseline of the chord diagram  \eqref{eq:standardchord} is given by
\begin{align}\label{ydact}
&\rhd: \C[G^{\times 5}]\to \C[G^{\times 4}],  & &(h,a_1,b_1,a_2,b_2)\mapsto (ha_1 h^\inv, h b_1 h^\inv, ha_2 h^\inv, h b_2 h^\inv),\\
&\delta: \C[G^{\times 4}]\to \C[G^{\times 5}], & &(a_1,b_1,a_2,b_2)\mapsto ([a_1^\inv, b_1][a_2^\inv, b_2], a_1,b_1,a_2,b_2),\nonumber
\end{align}
where the arguments $a_i$ and $b_i$ are associated with the loops $\alpha_i$ and $\beta_i$, respectively. The linear subspace of coinvariants of this Yetter-Drinfeld module structure is
\begin{align*}
M^{coH}=\mathrm{span}_\C\{(a_1,b_1,a_2,b_2)\mid [a_1^\inv, b_1][a_2^\inv, b_2]=1\}\subset \C[G^{\times 4}],
\end{align*}
and the subspace $M_{inv}\subset \C[G^{\times 4}]$ of elements that are both invariants and coinvariants with respect to \eqref{ydact} is the image of $M^{coH}$ under the projector
$$
\pi: M'\to M', \quad x\mapsto \tfrac 1 {|G|}\Sigma_{h\in G}\, h\rhd x.
$$
By Theorem \ref{th:gervais}, the action of  $\mathrm{Map}(\Sigma)$ on $\C[G^{\times 4}]$ via the group homomorphism $\rho$ in Theorem \ref{th:maptheorem}  is generated by the Dehn twists along $\alpha_1$, $\alpha_2$ and along the following paths from \eqref{eq:pathexpress}
\begin{align*}
&\delta_0=\beta_2^\inv, \quad \delta_1=[\alpha_1^\inv,\beta_1]\alpha_2^\inv\beta_2\alpha_2, \quad \delta_2=\beta_1^\inv\alpha_2^\inv\beta_2\alpha_2,
\quad 
\gamma_{0,1}
=[\alpha_1^\inv,\beta_1] ,  \quad \gamma_{0,2}=\beta_1^\inv,
\\
&\gamma_{1,0}
=[\alpha_1^\inv,\beta_1][\alpha_2^\inv,\beta_2],  \quad 
\gamma_{2,0}
=\beta_1^\inv[\alpha_2^\inv,\beta_2],\quad 
\gamma_{2,1}
=\alpha_1^\inv\beta_1\alpha_1, \quad
 \gamma_{1,2}
=[\alpha_1^\inv,\beta_1]
[\alpha_2^\inv,\beta_2] \alpha_2\beta_1\alpha_2^\inv.
\end{align*}
From Definitions \ref{def:dtnull} and \ref{definition:DehnTwistGammaNonFace} and formulas \eqref{eq:slidedef2a} to \eqref{eq:slidedef4b} for the slides, one computes the action of the associated  twists on $\C[G^{\times 4}]$. Listing only those arguments that transform nontrivially, we obtain
\begingroup
\allowdisplaybreaks
\begin{align}\label{eq:rho1gens}
&D_{\alpha_1}: b_1\mapsto b_1 a_1^\inv,\\
&D_{\alpha_2}: b_2\mapsto b_2 a_2^\inv,\nonumber\\
&D_{\delta_0}: a_2\mapsto b_2 a_2,\nonumber\\
&D_{\delta_1}: x\mapsto a_2^\inv b_2^\inv a_2[b_1, a_1^\inv] x [a_1^\inv, b_1]a_2^\inv b_2 a_2,  & &x\in\{a_1,b_1\},\nonumber\\
&\qquad\; a_2\mapsto p^2 a_2   [a_1^\inv, b_1]a_2^\inv b_2 a_2,\nonumber\\
&D_{\delta_2}:  a_1\mapsto p^\inv a_2^\inv b_2^\inv a_2 b_1 a_1,\nonumber\\
&\qquad\; b_1\mapsto a_2^\inv b_2^\inv a_2 b_1 a_2^\inv b_2 a_2,\nonumber\\
&\qquad\; a_2\mapsto p a_2 b_1^\inv a_2^\inv b_2 a_2,\nonumber\\
&D_{\gamma_{0,1}}: x\mapsto [b_1, a_1^\inv] x[a_1^\inv, b_1], & &x\in\{a_1,b_1\}, \nonumber\\
&D_{\gamma_{0,2}}: a_1\mapsto b_1 a_1,\nonumber\\
&D_{\gamma_{1,0}}: x\mapsto [b_2, a_2^\inv][b_1,a_1^\inv] x[a_1^\inv, b_1][a_2^\inv, b_2], & &x\in \{a_1,b_1,a_2,b_2\},\nonumber\\ 
&D_{\gamma_{2,0}}: a_1\mapsto p^{-2}[b_2, a_2^\inv] b_1 a_1,\nonumber\\
&\qquad\;\;\; x\mapsto [b_2, a_2^\inv] b_1 x b_1^\inv[a_2^\inv, b_2], & &x\in\{b_1, a_2,b_2\},\nonumber\\
&D_{\gamma_{2,1}}: a_1\mapsto b_1 a_1,\nonumber\\
&D_{\gamma_{1,2}}: a_1\mapsto p^2 [w,a_2^\inv] a_1 w^\inv, & &w=a_2 b_1^\inv a_2^\inv [b_2, a_2^\inv][b_1, a_1^\inv],\nonumber\\
&\qquad\;\;\; b_1\mapsto [w, a_2^\inv] b_1 [a_2^\inv, w], & &u= a_2^{-1}b_{2}^{-1}a_{2}[b_1, a_1^\inv],\nonumber\\
&\qquad\;\;\; a_2\mapsto wa_2 w^\inv,\nonumber\\
&\qquad\;\;\; b_2\mapsto b_2[u, w].\nonumber
\end{align}
\endgroup
By a direct but  lengthy computation, one can verify that these generating Dehn twists indeed satisfy the relations from Theorem \ref{th:gervais} for the mapping class group $\mathrm{Map}(\Sigma)$. 
The mapping class group $\mathrm{Map}(\Sigma')$ is obtained by imposing the additional relations from \eqref{eq:addrel}
\begin{align}\label{eq:addrel2}
D_{\gamma_{1,0}}=\id_{\C[G^{\times 4}]},  \quad D_{\delta_0}=D_{\delta_1}, \quad D_{\gamma_{0,1}}=(D_{\delta_0}^3D_{\alpha_2})^3,\quad   D_{\gamma_{0,2}}=D_{\gamma_{1,2}},\quad 
D_{\gamma_{2,0}}=D_{\gamma_{2,1}},
\end{align}
which follow from the first two relations in \eqref{eq:addrel2}. From the expressions for the Dehn twists in \eqref{eq:rho1gens} it is apparent that
the first relation holds both on the invariants and on the coinvariants of the Yetter-Drinfeld module structure \eqref{ydact} and for all elements $p\in Z(G)$.
In general, the relation $D_{\delta_0}=D_{\delta_1}$ holds only if one imposes both, invariance and coinvariance and $p^2=1$.
\end{example}

\begin{example} 
Let $\Sigma$,  $\Sigma'$  and $\Gamma$ be as in Example \ref{ex:groupex}.
Let $G$ be a group, viewed as a Hopf monoid in the cartesian monoidal category $\mac=\mathrm{Set}$, and $p\in Z(G)$.

Then the Yetter-Drinfeld module structure on $G^{\times 4}$ for the cilium in Figure \ref{fig:pi1gens} and the baseline of  \eqref{eq:standardchord} is again  given by \eqref{ydact}. Its  biinvariants are the moduli space of flat $G$-connections on $\Sigma'$
\begin{align*}
G^{\times 4}_{inv}=\{(a_1,b_1,a_2,b_2)\in G^{\times 4}\mid [a_1^\inv, b_1][a_2^\inv, b_2]=1\}/G=\mathrm{Hom}(\pi_1(\Sigma'), G)/G.
\end{align*}
 The generating Dehn twists of the mapping class group $\mathrm{Map}(\Sigma)$ are again given by \eqref{eq:rho1gens}, and one finds that they induce an action of  $\mathrm{Map}(\Sigma')$ on $M_{inv}$ if $p^2=1$.
 For $p=1$, this gives the  action of the mapping class group $\mathrm{Map}(\Sigma')$ on $\mathrm{Hom}(\pi_1(\Sigma'), G)/G$.
\end{example}

 Theorems \ref{th:maptheorem} and \ref{th:close} associate mapping class group actions to specific ribbon graphs equipped with additional structures, namely choices of cilia at each vertex and face. In the remainder of this section we show that these mapping class group actions are not limited to the specific graphs in Figure \ref{fig:pi1gens} and in~\eqref{eq:standardchord} and only depend mildly on these choices. Similar actions of the mapping class groups
  $\mathrm{Map}(\Sigma)$ and $\mathrm{Map}(\Sigma')$  are obtained for any  ribbon graph $\Gamma$,  as long as each vertex and each face of $\Gamma'$ carries exactly one cilium.  This follows because any such graph can be transformed into the graph from \eqref{eq:standardchord} by a sequence of edge slides.

\begin{figure}
\begin{center}
\begin{tikzpicture}[scale=.5]
\node at (7,1)[anchor=south]{(a)};
\begin{scope}[shift={(0,0)}]
\draw[line width=1pt, color=black](-3,0)--(5,0);
\draw[color=red, line width=1.5pt, <-,>=stealth] (-2,0).. controls (-2,2) and (2,2)..(2,0);
\draw[line width=1.5pt,->,>=stealth] (-1,.8)--(-1,0);
\node at (-1,0)[anchor=north] {$\mu$};
\draw[line width=1.5pt,->,>=stealth] (1,.8)--(1,0);
\node at (1,0)[anchor=north] {$\nu$};
\draw[line width=1.5pt,->,>=stealth] (3,.8)--(3,0);
\node at (3,0)[anchor=north] {$\gamma$};
\node at (0,2) [anchor=south, color=red]{$\beta$};
\draw[color=blue, line width=1.5pt, <-,>=stealth] (0,0).. controls (0,2) and (4,2)..(4,0);
\node at (2,2) [anchor=south, color=blue]{$\alpha$};
\draw[line width=.5pt, color=black,->,>=stealth] plot [smooth, tension=0.6] coordinates 
      {(-2.5,0)(-1.6,1.6)(0,2)(1.6,1.6)(2.5,.3)(3.5,.3)(3, 1)(2,1.3)(1,1)(0.5,.3)(1.5,.3)(1,1)(0,1.3)  (-1,1)(-1.5,.3)(-.5,.3)(.4,1.6)(2,2)(3.6,1.6)(4.5,.1)  };
\end{scope}
\draw[line width=1pt, ->,>=stealth] (6,1)--(8,1);
\begin{scope}[shift={(12,0)}]
\draw[line width=1pt, color=black](-3,0)--(8,0);
\draw[color=red, line width=1.5pt, <-,>=stealth] (-2,0).. controls (-2,2) and (2,2)..(2,0);
\draw[line width=1.5pt,->,>=stealth] (7,.8)--(7,0);
\node at (7,.8)[anchor=south] {$\mu$};
\draw[line width=1.5pt,->,>=stealth] (6,.8)--(6,0);
\node at (6,.8)[anchor=south] {$\nu$};
\draw[line width=1.5pt,->,>=stealth] (5,.8)--(5,0);
\node at (5,.8)[anchor=south] {$\gamma$};
\node at (0,2) [anchor=south, color=red]{$\beta$};
\draw[color=blue, line width=1.5pt, <-,>=stealth] (0,0).. controls (0,2) and (4,2)..(4,0);
\node at (2,2) [anchor=south, color=blue]{$\alpha$};
\draw[line width=.5pt, color=black,->,>=stealth] plot [smooth, tension=0.6] coordinates 
      {(-2.5,0)(-1.6,1.6)(0,2)(1.6,1.6)(2.5,.3)(3.5,.3)(3, 1)(2,1.3)(1,1)(0.5,.3)(1.5,.3)(1,1)(0,1.3)  (-1,1)(-1.5,.3)(-.5,.3)(.4,1.6)(2,2)(3.6,1.6)(4.5,.1)  };
\end{scope}
\end{tikzpicture}

\begin{tikzpicture}[scale=.5]
\begin{scope}[shift={(0,0)}]
\node at (7,1)[anchor=south]{(b)};
\draw[line width=1pt, color=black](-5,0)--(5,0);
\draw[line width=1.5 pt,<-,>=stealth] (1,0) .. controls (1,1.5) and (3,1.5) .. (3,0);
\draw[line width=1.5 pt,->,>=stealth] (-1,0) .. controls (-1,1.5) and (-3,1.5) .. (-3,0);
\draw[line width=1.5pt, ->, >=stealth] (0,2)--(0,0);
\draw[line width=1.5 pt,<-,>=stealth] (-4,0) .. controls (-4,4) and (4,4) .. (4,0);
\node at (0,3)[anchor=south]{$\mu$};
\node at (0,2)[anchor=north west]{$\nu$};
\draw[line width=.5pt] (-4,0) circle (.3);
\draw[line width=.5pt] (-3.7,.15) -- (-3.3,.15);
\draw[line width=.5 pt,->,>=stealth] (-3.3,.15) .. controls (-3.3,1.7) and (-.7,1.7) .. (-.7,.15);
\draw[line width=.5pt] (3.7,.15) -- (3.3,.15);
\draw[line width=.5 pt,->,>=stealth] (3.3,.15) .. controls (3.3,1.7) and (.7,1.7) .. (.7,.15);
\draw[line width=.5pt] (4,0) circle (.3);
\end{scope}
\draw[line width=1pt, ->,>=stealth] (6,1)--(8,1);
\begin{scope}[shift={(15,0)}]
\draw[line width=1pt, color=black](-6,0)--(6,0);
\draw[line width=1.5 pt,<-,>=stealth] (-5,0) .. controls (-5,1.5) and (-3,1.5) .. (-3,0);
\draw[line width=1.5 pt,->,>=stealth] (5,0) .. controls (5,1.5) and (3,1.5) .. (3,0);
\draw[line width=1.5pt, ->, >=stealth] (0,2)--(0,0);
\draw[line width=1.5 pt,<-,>=stealth] (-2,0) .. controls (-2,4) and (2,4) .. (2,0);
\node at (0,3)[anchor=south]{$\mu$};
\node at (0,2)[anchor=north west]{$\nu$};
\end{scope}
\end{tikzpicture}
\end{center}
\caption{Transforming a  ribbon graph into the graph from Figure \ref{fig:pi1gens}.}
\label{fig:trafo}
\end{figure}
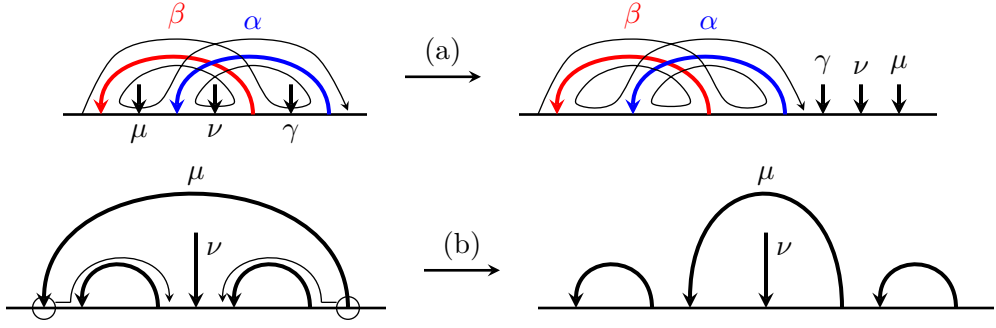

\begin{lemma}\label{lem:graphred}
	Let $\Gamma$ be a  ribbon graph with $n+1$ vertices and faces, such that every vertex and face carries exactly one cilium. 
	Then there is a  sequence of slides that do not slide edges over cilia and of edge reversals that  transforms $\Gamma$ into the graph from~\eqref{eq:standardchord} with $g=\frac{|E|}{2}-n$.
\end{lemma}

\begin{proof}
	Note that the edge slides preserve the bijection between the faces and vertices defined by the cilia, as long as no edges  slide over cilia.
	As a first step we transform $\Gamma$ into a chord diagram.
First we choose a vertex $x$. For every neighbouring vertex $w$, we choose an edge $\nu_w$ connecting $x$ and $w$, orient $\nu_w$ towards $x$, and slide every other edge end incident at $w$ along $\nu_w$ to  $x$, such that no edge end slides over the cilium at $w$.
We repeat this until every vertex except $x$ is univalent.  We then orient every loop $\gamma$ at $x$ such that $\ta(\gamma)> \st(\gamma)$ for the  ordering at $x$ and obtain a chord diagram $\Gamma'$.

As a second step we form the pairs of loops labeled $\alpha_i$ and $\beta_i$  in~\eqref{eq:standardchord}. For any pair of edges $\alpha$ and $\beta$ in $\Gamma'$ with $\ta(\beta)> \ta(\alpha)>\st(\beta)>\st(\alpha)$,  we slide all other edge ends between their starting and target end  along $\alpha$ and $\beta$ to the right, as shown in Figure \ref{fig:trafo} (a). 
This yields a chord diagram $\Gamma''$.

As a third step we form the loops labeled $\mu_i$ in \eqref{eq:standardchord}. Suppose $w$ is a univalent vertex of $\Gamma''$ and $\nu$ the unique edge that connects $w$ and $x$. 
Because $w$ carries a cilium that also belongs to a unique face of $\Gamma''$, there is a loop $\mu$  with $\ta(\mu)<\ta(\nu)<\st(\mu)$, as shown in Figure \ref{fig:trafo} (b). 
We then slide both edge ends of $\mu$ next to the target vertex of $\nu$, as shown in Figure \ref{fig:trafo} (b). 
We repeat this step until all univalent vertices are associated with an edge $\nu_i$ to the baseline and a loop $\mu_i$ as in \eqref{eq:standardchord}. 
Finally we slide all these pairs $(\mu_i,\nu_i)$ to the right of all pairs $(\alpha_j,\beta_j)$  and  obtain  the graph from  \eqref{eq:standardchord}. The identity for $g$ follows from the assumptions and the formula for the Euler characteristic.
\end{proof}

Suppose now that $\Gamma$ is a directed ribbon graph as in Lemma \ref{lem:graphred}. Then attaching annuli to the faces yields a surface $\Sigma$  with $n+1$ boundary components. By Lemma \ref{lem:graphred} for any vertex $x$ of $\Gamma$ there is a sequence of edge slides and edge reversals that transforms $\Gamma$ into the graph \eqref{eq:standardchord} and the cilium at $x$ to its baseline. As no edges slide over cilia and orientation reversal is compatible with the Yetter-Drinfeld module structure by definition, this defines an 
 automorphism  $\phi: H^{\oo 2(n+g)}\to H^{\oo 2(n+g)}$ of Yetter-Drinfeld modules.
By conjugating the generating Dehn twists from Theorem \ref{th:maptheorem} with $\phi$, we then obtain another set of automorphisms
 for each curve $\sigma$ in \eqref{eq:pathexpress}  that satisfy the relations from Theorem \ref{th:gervais} and define an action of  $\text{Map}(\Sigma)$  by automorphisms of Yetter-Drinfeld modules. 

If $\mac$ is finitely complete and cocomplete and $m\circ (p\oo p)=\eta$, 
by Lemma \ref{lem:maclemma} the automorphism $\phi$ induces an isomorphism $\phi_{inv}$ between the biinvariants associated with the Yetter-Drinfeld module structure at $x$ and with the baseline of \eqref{eq:standardchord}. Conjugating the action of the mapping class group $\mathrm{Map}(\Sigma')$ from Theorem \ref{th:close} on the biinvariants  of \eqref{eq:standardchord} with $\phi_{inv}$  then defines an action of $\mathrm{Map}(\Sigma')$ on the biinvariants at $x$. 
In particular, this implies that the choice of the cilium at the multivalent vertex in Figure \ref{fig:pi1gens} affects the associated mapping class group actions from Theorem \ref{th:maptheorem} and Theorem \ref{th:close} only by conjugation with an automorphism.

\subsection*{Concluding Remarks}

The mapping class group actions constructed in this article are distinct from the ones obtained by Lyubashenko \cite{Ly95a,Ly95b,Ly96}, from the mapping class group actions  \cite{AS,Fa18a,Fa18b} in the context of Chern-Simons theory, as they are based on different  data. 

As our mapping class group actions are obtained from a generalisation of Kitaev lattice models \cite{Ki,BMCA} which were in turn related to Turaev-Viro TQFTs   in \cite{BK,KKR}, it is plausible that they  should be related  to mapping class group actions in Turaev-Viro TQFTs  when our pivotal Hopf monoid $H$ is a finite-dimensional semisimple Hopf algebra over $\C$. 
The relations between Turaev-Viro TQFTs and Reshetikhin-Turaev TQFTs  \cite{Ba,KB,TV} and  between  Kitaev lattice models and the quantum moduli algebra  \cite{Me} then suggest  a relation
 to the mapping class group actions from \cite{AS,Fa18a,Fa18b,Ly95a,Ly95b,Ly96}
 for the Drinfeld double $D(H)$. However, this question is beyond the scope of this article.
  It would also be interesting to see if our mapping class group actions generalise to Kitaev models with defects such as the ones considered in \cite{Ko}.

If $H$ is a finite-dimensional semisimple Hopf algebra over $\C$, the results in Section \ref{sec:chordslideshopf} are also related to an earlier observation by Kashaev \cite{Ka}.
It is shown in \cite{Ka} that  the action of the universal $R$-matrix of the Drinfeld double $D(H)$ on three modules over its Heisenberg double $\mathcal H(H)$ satisfies the pentagon relation, which is related to the mapping class group action  on a triangulation via flip or Whitehead moves \cite{P87}.
Under these assumptions, modules over the Heisenberg double $\mathcal H(H)$ correspond bijectively to $H$-Hopf modules. A direct computation shows that the action of the universal $R$-matrix of $D(H)$ on   $\mathcal H(H)\oo \mathcal H(H)$ then gives the expressions for the edge slides in \eqref{eq:slidedef2a} to \eqref{eq:slidedef4b} for the pivotal element $\eta$.
The  pentagon relations from Figure~\ref{fig:sliderel5} for the edge slides and the diagrammatic proof in Figure \ref{fig:pent} can thus be viewed as a  generalisation of \cite{Ka}.

\subsection*{Acknowledgements}
This research was supported in part by Perimeter Institute for Theoretical Physics. Research at Perimeter Institute is supported by the Government of Canada through the Department of Innovation, Science and Economic Development and by the Province of Ontario through the Ministry of Research and Innovation.


\end{document}

%% file: gervais.pdf_tex
\begingroup%
  \makeatletter%
  \providecommand\color[2][]{%
    \errmessage{(Inkscape) Color is used for the text in Inkscape, but the package 'color.sty' is not loaded}%
    \renewcommand\color[2][]{}%
  }%
  \providecommand\transparent[1]{%
    \errmessage{(Inkscape) Transparency is used (non-zero) for the text in Inkscape, but the package 'transparent.sty' is not loaded}%
    \renewcommand\transparent[1]{}%
  }%
  \providecommand\rotatebox[2]{#2}%
  \ifx\svgwidth\undefined%
    \setlength{\unitlength}{841.88976378bp}%
    \ifx\svgscale\undefined%
      \relax%
    \else%
      \setlength{\unitlength}{\unitlength * \real{\svgscale}}%
    \fi%
  \else%
    \setlength{\unitlength}{\svgwidth}%
  \fi%
  \global\let\svgwidth\undefined%
  \global\let\svgscale\undefined%
  \makeatother%
  \begin{picture}(1,0.70707071)%
    \put(0,0){\includegraphics[width=\unitlength,page=1]{gervais.pdf}}%
    \put(0.60514853,0.49253004){\color[rgb]{0,0,0}\rotatebox{-10.40000028}{\makebox(0,0)[lb]{\smash{...}}}}%
    \put(0,0){\includegraphics[width=\unitlength,page=2]{gervais.pdf}}%
    \put(0.49177695,0.11380132){\color[rgb]{0,0,0}\rotatebox{-0.40000056}{\makebox(0,0)[lb]{\smash{...}}}}%
    \put(0,0){\includegraphics[width=\unitlength,page=3]{gervais.pdf}}%
    \put(0.0913111,0.22750392){\color[rgb]{0,0,0}\makebox(0,0)[lb]{\smash{1}}}%
    \put(0.22440865,0.10988286){\color[rgb]{0,0,0}\makebox(0,0)[lb]{\smash{2}}}%
    \put(0.36833972,0.06809643){\color[rgb]{0,0,0}\makebox(0,0)[lb]{\smash{3}}}%
    \put(0.60540227,0.0830837){\color[rgb]{0,0,0}\makebox(0,0)[lb]{\smash{$n-2$}}}%
    \put(0.74335919,0.11403659){\color[rgb]{0,0,0}\makebox(0,0)[lb]{\smash{$n-1$}}}%
    \put(0.85894343,0.20428924){\color[rgb]{0,0,0}\makebox(0,0)[lb]{\smash{$n$}}}%
    \put(0.88397812,0.33632825){\color[rgb]{0,0,0}\makebox(0,0)[lb]{\smash{$n+1$}}}%
    \put(0.7058629,0.42250727){\color[rgb]{0,0,0}\makebox(0,0)[lb]{\smash{$n+2$}}}%
    \put(0.47453179,0.44991159){\color[rgb]{0,0,0}\makebox(0,0)[lb]{\smash{$n\!+\!2g\!-\!5$}}}%
    \put(0.10068028,0.36551322){\color[rgb]{0,0,0}\makebox(0,0)[lb]{\smash{}}}%
    \put(0.36619105,0.44881722){\color[rgb]{0,0,0}\makebox(0,0)[lb]{\smash{$n\!+\!2g\!-\!4$}}}%
    \put(0.25992204,0.41489236){\color[rgb]{0,0,0}\makebox(0,0)[lb]{\smash{$n\!+\!2g\!-\!3$}}}%
    \put(0.07728243,0.35689175){\color[rgb]{0,0,0}\makebox(0,0)[lb]{\smash{$n\!+\!2g\!-\!2$}}}%
    \put(0.43989131,0.71659873){\color[rgb]{0,0,0}\makebox(0,0)[lb]{\smash{{\color{blue} $\alpha_{g-2}$}}}}%
    \put(0.04879021,0.62682192){\color[rgb]{0,0,0}\makebox(0,0)[lb]{\smash{{\color{blue} $\alpha_{g-1}$}}}}%
    \put(0.93377027,0.59434024){\color[rgb]{0,0,0}\makebox(0,0)[lb]{\smash{{\color{blue} $\alpha_{1}$}}}}%
    \put(0.47587143,0.34230751){\color[rgb]{0,0,0}\makebox(0,0)[lb]{\smash{{\color{blue} $\alpha_{g}$}}}}%
    \put(0.01038893,0.32764152){\color[rgb]{0,0,0}\makebox(0,0)[lb]{\smash{{\color{red} $\delta_{1}$}}}}%
    \put(0.07040462,0.13552127){\color[rgb]{0,0,0}\makebox(0,0)[lb]{\smash{{\color{red} $\delta_{2}$}}}}%
    \put(0.71380024,0.0317948){\color[rgb]{0,0,0}\makebox(0,0)[lb]{\smash{{\color{red} $\delta_{n-1}$}}}}%
    \put(0.27386657,0.03064186){\color[rgb]{0,0,0}\makebox(0,0)[lb]{\smash{{\color{red} $\delta_{3}$}}}}%
    \put(0.42046405,0.00563682){\color[rgb]{0,0,0}\makebox(0,0)[lb]{\smash{{\color{red} $\delta_{4}$}}}}%
    \put(0.97040103,0.28263077){\color[rgb]{0,0,0}\makebox(0,0)[lb]{\smash{{\color{red} $\delta_{n+1}$}}}}%
    \put(0.55840242,0.00933129){\color[rgb]{0,0,0}\makebox(0,0)[lb]{\smash{{\color{red} $\delta_{n-2}$}}}}%
    \put(0.88307667,0.11807996){\color[rgb]{0,0,0}\makebox(0,0)[lb]{\smash{{\color{red} $\delta_{n}$}}}}%
    \put(0.850887,0.46619349){\color[rgb]{0,0,0}\makebox(0,0)[lb]{\smash{{\color{red} $\delta_{n+2}$}}}}%
    \put(0.68178052,0.54060933){\color[rgb]{0,0,0}\makebox(0,0)[lb]{\smash{{\color{red} $\delta_{n+3}$}}}}%
    \put(0.11871912,0.48660453){\color[rgb]{0,0,0}\makebox(0,0)[lb]{\smash{{\color{red} $\delta_{n+2g-2}$}}}}%
    \put(0.25773469,0.54196692){\color[rgb]{0,0,0}\makebox(0,0)[lb]{\smash{{\color{red} $\delta_{n+2g-3}$}}}}%
    \put(0.41868987,0.57327269){\color[rgb]{0,0,0}\makebox(0,0)[lb]{\smash{{\color{red} $\delta_{n+2g-4}$}}}}%
    \put(0.557978,0.55899103){\color[rgb]{0,0,0}\makebox(0,0)[lb]{\smash{{\color{red} $\delta_{n+2g-5}$}}}}%
    \put(0,0){\includegraphics[width=\unitlength,page=4]{gervais.pdf}}%
  \end{picture}%
\endgroup%

%% file: gervais2.pdf_tex
\begingroup%
  \makeatletter%
  \providecommand\color[2][]{%
    \errmessage{(Inkscape) Color is used for the text in Inkscape, but the package 'color.sty' is not loaded}%
    \renewcommand\color[2][]{}%
  }%
  \providecommand\transparent[1]{%
    \errmessage{(Inkscape) Transparency is used (non-zero) for the text in Inkscape, but the package 'transparent.sty' is not loaded}%
    \renewcommand\transparent[1]{}%
  }%
  \providecommand\rotatebox[2]{#2}%
  \ifx\svgwidth\undefined%
    \setlength{\unitlength}{841.88976378bp}%
    \ifx\svgscale\undefined%
      \relax%
    \else%
      \setlength{\unitlength}{\unitlength * \real{\svgscale}}%
    \fi%
  \else%
    \setlength{\unitlength}{\svgwidth}%
  \fi%
  \global\let\svgwidth\undefined%
  \global\let\svgscale\undefined%
  \makeatother%
  \begin{picture}(1,0.70707071)%
    \put(0,0){\includegraphics[width=\unitlength,page=1]{gervais2.pdf}}%
    \put(0.63919673,0.50723267){\color[rgb]{0,0,0}\rotatebox{-10.40000028}{\makebox(0,0)[lb]{\smash{...}}}}%
    \put(0,0){\includegraphics[width=\unitlength,page=2]{gervais2.pdf}}%
    \put(0.49177681,0.0928958){\color[rgb]{0,0,0}\rotatebox{-0.40000056}{\makebox(0,0)[lb]{\smash{...}}}}%
    \put(0,0){\includegraphics[width=\unitlength,page=3]{gervais2.pdf}}%
    \put(0.70700581,0.09483795){\color[rgb]{0,0,0}\makebox(0,0)[lb]{\smash{$i-1$}}}%
    \put(0.84329197,0.2337215){\color[rgb]{0,0,0}\makebox(0,0)[lb]{\smash{$i$}}}%
    \put(0.50131405,0.45837819){\color[rgb]{0,0,0}\makebox(0,0)[lb]{\smash{$k-1$}}}%
    \put(0.10068028,0.36551322){\color[rgb]{0,0,0}\makebox(0,0)[lb]{\smash{}}}%
    \put(0.37494585,0.45757201){\color[rgb]{0,0,0}\makebox(0,0)[lb]{\smash{$k$}}}%
    \put(0.43423585,0.56144735){\color[rgb]{0,0,0}\makebox(0,0)[lb]{\smash{{\color{red} $\delta_{k}$}}}}%
    \put(0,0){\includegraphics[width=\unitlength,page=4]{gervais2.pdf}}%
    \put(0.30653818,0.1492248){\color[rgb]{0,0,0}\makebox(0,0)[lb]{\smash{$j$}}}%
    \put(0.15606493,0.27245174){\color[rgb]{0,0,0}\makebox(0,0)[lb]{\smash{$j-1$}}}%
    \put(0.07989589,0.12779616){\color[rgb]{0,0,0}\makebox(0,0)[lb]{\smash{{\color{red} $\delta_{j}$}}}}%
    \put(0.84985952,0.09905482){\color[rgb]{0,0,0}\makebox(0,0)[lb]{\smash{{\color{red} $\delta_{i}$}}}}%
    \put(0.62469842,0.33798863){\color[rgb]{0,0,0}\makebox(0,0)[lb]{\smash{{\color{violet} $\gamma_{k,i}$}}}}%
    \put(0,0){\includegraphics[width=\unitlength,page=5]{gervais2.pdf}}%
    \put(0.45617316,0.20902635){\color[rgb]{0,0,0}\makebox(0,0)[lb]{\smash{{\color{violet} $\gamma_{i,j}$}}}}%
    \put(0.28153645,0.32975176){\color[rgb]{0,0,0}\makebox(0,0)[lb]{\smash{{\color{violet} $\gamma_{j,k}$}}}}%
    \put(0,0){\includegraphics[width=\unitlength,page=6]{gervais2.pdf}}%
    \put(0.48679955,0.36918643){\color[rgb]{0,0,0}\makebox(0,0)[lb]{\smash{{\color{blue} $\alpha_{g}$}}}}%
    \put(0.27868148,0.48974126){\color[rgb]{0,0,0}\rotatebox{14.59999955}{\makebox(0,0)[lb]{\smash{...}}}}%
    \put(0.08930145,0.30543864){\color[rgb]{0,0,0}\rotatebox{74.59999924}{\makebox(0,0)[lb]{\smash{...}}}}%
    \put(0.90933075,0.32829547){\color[rgb]{0,0,0}\rotatebox{-70.39999993}{\makebox(0,0)[lb]{\smash{...}}}}%
  \end{picture}%
\endgroup%

%% file: facevertex.pdf_tex
\begingroup%
  \makeatletter%
  \providecommand\color[2][]{%
    \errmessage{(Inkscape) Color is used for the text in Inkscape, but the package 'color.sty' is not loaded}%
    \renewcommand\color[2][]{}%
  }%
  \providecommand\transparent[1]{%
    \errmessage{(Inkscape) Transparency is used (non-zero) for the text in Inkscape, but the package 'transparent.sty' is not loaded}%
    \renewcommand\transparent[1]{}%
  }%
  \providecommand\rotatebox[2]{#2}%
  \ifx\svgwidth\undefined%
    \setlength{\unitlength}{841.88976378bp}%
    \ifx\svgscale\undefined%
      \relax%
    \else%
      \setlength{\unitlength}{\unitlength * \real{\svgscale}}%
    \fi%
  \else%
    \setlength{\unitlength}{\svgwidth}%
  \fi%
  \global\let\svgwidth\undefined%
  \global\let\svgscale\undefined%
  \makeatother%
  \begin{picture}(1,0.70707071)%
    \put(0,0){\includegraphics[width=\unitlength,page=1]{facevertex.pdf}}%
    \put(0.14034664,0.57043166){\color[rgb]{0,0,0}\makebox(0,0)[lb]{\smash{=}}}%
    \put(0.31368299,0.57043166){\color[rgb]{0,0,0}\makebox(0,0)[lb]{\smash{=}}}%
    \put(0.47618581,0.56888401){\color[rgb]{0,0,0}\makebox(0,0)[lb]{\smash{=}}}%
    \put(0.65106977,0.57197931){\color[rgb]{0,0,0}\makebox(0,0)[lb]{\smash{=}}}%
    \put(0.8352396,0.57352695){\color[rgb]{0,0,0}\makebox(0,0)[lb]{\smash{=}}}%
    \put(0.33198218,0.34794779){\color[rgb]{0,0,0}\makebox(0,0)[lb]{\smash{=}}}%
    \put(0.66627353,0.34794779){\color[rgb]{0,0,0}\makebox(0,0)[lb]{\smash{=}}}%
    \put(0.04919609,0.09961863){\color[rgb]{0,0,0}\makebox(0,0)[lb]{\smash{=}}}%
    \put(0.37110646,0.1150951){\color[rgb]{0,0,0}\makebox(0,0)[lb]{\smash{=}}}%
    \put(0.00773823,0.67632124){\color[rgb]{0,0,0}\makebox(0,0)[lb]{\smash{(a)}}}%
    \put(0.00689626,0.42763938){\color[rgb]{0,0,0}\makebox(0,0)[lb]{\smash{(b)}}}%
  \end{picture}%
\endgroup%

%% file: compvertface2.pdf_tex
\begingroup%
  \makeatletter%
  \providecommand\color[2][]{%
    \errmessage{(Inkscape) Color is used for the text in Inkscape, but the package 'color.sty' is not loaded}%
    \renewcommand\color[2][]{}%
  }%
  \providecommand\transparent[1]{%
    \errmessage{(Inkscape) Transparency is used (non-zero) for the text in Inkscape, but the package 'transparent.sty' is not loaded}%
    \renewcommand\transparent[1]{}%
  }%
  \providecommand\rotatebox[2]{#2}%
  \ifx\svgwidth\undefined%
    \setlength{\unitlength}{595.27559055bp}%
    \ifx\svgscale\undefined%
      \relax%
    \else%
      \setlength{\unitlength}{\unitlength * \real{\svgscale}}%
    \fi%
  \else%
    \setlength{\unitlength}{\svgwidth}%
  \fi%
  \global\let\svgwidth\undefined%
  \global\let\svgscale\undefined%
  \makeatother%
  \begin{picture}(1,1.41428571)%
    \put(0,0){\includegraphics[width=\unitlength,page=1]{compvertface2.pdf}}%
    \put(0.2218022,1.28056618){\color[rgb]{0,0,0}\makebox(0,0)[lb]{\smash{=}}}%
    \put(0.42181339,1.28712925){\color[rgb]{0,0,0}\makebox(0,0)[lb]{\smash{=}}}%
    \put(0.63890862,1.28750139){\color[rgb]{0,0,0}\makebox(0,0)[lb]{\smash{=}}}%
    \put(0.8047807,1.29710079){\color[rgb]{0,0,0}\makebox(0,0)[lb]{\smash{=}}}%
    \put(0.75682383,0.8959038){\color[rgb]{0,0,0}\makebox(0,0)[lb]{\smash{=}}}%
    \put(0.55791258,0.88869009){\color[rgb]{0,0,0}\makebox(0,0)[lb]{\smash{=}}}%
    \put(0.35627229,0.8862079){\color[rgb]{0,0,0}\makebox(0,0)[lb]{\smash{=}}}%
    \put(0.17315534,0.89187107){\color[rgb]{0,0,0}\makebox(0,0)[lb]{\smash{=}}}%
    \put(0.17572962,0.68024865){\color[rgb]{0,0,0}\makebox(0,0)[lb]{\smash{=}}}%
    \put(0.35996617,0.68353421){\color[rgb]{0,0,0}\makebox(0,0)[lb]{\smash{=}}}%
    \put(0.21005729,0.35413977){\color[rgb]{0,0,0}\makebox(0,0)[lb]{\smash{=}}}%
    \put(0.43314755,0.35413977){\color[rgb]{0,0,0}\makebox(0,0)[lb]{\smash{=}}}%
    \put(0.63473506,0.35413977){\color[rgb]{0,0,0}\makebox(0,0)[lb]{\smash{=}}}%
    \put(0.82825747,0.35413977){\color[rgb]{0,0,0}\makebox(0,0)[lb]{\smash{=}}}%
    \put(0.01985272,0.13008782){\color[rgb]{0,0,0}\makebox(0,0)[lb]{\smash{=}}}%
    \put(0.1918718,0.13008782){\color[rgb]{0,0,0}\makebox(0,0)[lb]{\smash{=}}}%
    \put(0.36926715,0.13008782){\color[rgb]{0,0,0}\makebox(0,0)[lb]{\smash{=}}}%
  \end{picture}%
\endgroup%

%% file: pi1gens2.pdf_tex
\begingroup%
  \makeatletter%
  \providecommand\color[2][]{%
    \errmessage{(Inkscape) Color is used for the text in Inkscape, but the package 'color.sty' is not loaded}%
    \renewcommand\color[2][]{}%
  }%
  \providecommand\transparent[1]{%
    \errmessage{(Inkscape) Transparency is used (non-zero) for the text in Inkscape, but the package 'transparent.sty' is not loaded}%
    \renewcommand\transparent[1]{}%
  }%
  \providecommand\rotatebox[2]{#2}%
  \ifx\svgwidth\undefined%
    \setlength{\unitlength}{841.88976378bp}%
    \ifx\svgscale\undefined%
      \relax%
    \else%
      \setlength{\unitlength}{\unitlength * \real{\svgscale}}%
    \fi%
  \else%
    \setlength{\unitlength}{\svgwidth}%
  \fi%
  \global\let\svgwidth\undefined%
  \global\let\svgscale\undefined%
  \makeatother%
  \begin{picture}(1,0.70707071)%
    \put(0,0){\includegraphics[width=\unitlength,page=1]{pi1gens2.pdf}}%
    \put(0.60514853,0.49253004){\color[rgb]{0,0,0}\rotatebox{-10.40000028}{\makebox(0,0)[lb]{\smash{...}}}}%
    \put(0,0){\includegraphics[width=\unitlength,page=2]{pi1gens2.pdf}}%
    \put(0.49177695,0.11380132){\color[rgb]{0,0,0}\rotatebox{-0.40000056}{\makebox(0,0)[lb]{\smash{...}}}}%
    \put(0,0){\includegraphics[width=\unitlength,page=3]{pi1gens2.pdf}}%
    \put(0.18501201,0.16897779){\color[rgb]{0,0,0}\makebox(0,0)[lb]{\smash{1}}}%
    \put(0.36833972,0.06049458){\color[rgb]{0,0,0}\makebox(0,0)[lb]{\smash{2}}}%
    \put(0.60540227,0.06978045){\color[rgb]{0,0,0}\makebox(0,0)[lb]{\smash{$n-2$}}}%
    \put(0.74335919,0.10453427){\color[rgb]{0,0,0}\makebox(0,0)[lb]{\smash{$\mu_{n-1}$}}}%
    \put(0.86550954,0.27761076){\color[rgb]{0,0,0}\makebox(0,0)[lb]{\smash{$\mu_n$}}}%
    \put(0.10068028,0.36551322){\color[rgb]{0,0,0}\makebox(0,0)[lb]{\smash{}}}%
    \put(0.05686849,0.61482409){\color[rgb]{0,0,0}\makebox(0,0)[lb]{\smash{{\color{blue} $\alpha_{g-1}$}}}}%
    \put(0.4143034,0.71218129){\color[rgb]{0,0,0}\makebox(0,0)[lb]{\smash{{\color{blue} $\alpha_{g-2}$}}}}%
    \put(0.69573609,0.43685117){\color[rgb]{0,0,0}\makebox(0,0)[lb]{\smash{{\color{red} $\beta_{1}$}}}}%
    \put(0,0){\includegraphics[width=\unitlength,page=4]{pi1gens2.pdf}}%
    \put(0.19332595,0.11988589){\color[rgb]{0,0,0}\makebox(0,0)[lb]{\smash{$\mu_{1}$}}}%
    \put(0.9253653,0.59289704){\color[rgb]{0,0,0}\makebox(0,0)[lb]{\smash{{\color{blue} $\alpha_{1}$}}}}%
    \put(0.19959466,0.30408373){\color[rgb]{0,0,0}\makebox(0,0)[lb]{\smash{{\color{red} $\beta_{g}$}}}}%
    \put(0.29261338,0.31007229){\color[rgb]{0,0,0}\makebox(0,0)[lb]{\smash{{\color{blue} $\alpha_{g}$}}}}%
    \put(0.32389122,0.4666241){\color[rgb]{0,0,0}\makebox(0,0)[lb]{\smash{{\color{red} $\beta_{g-2}$}}}}%
    \put(0.07627969,0.35995417){\color[rgb]{0,0,0}\makebox(0,0)[lb]{\smash{{\color{red} $\beta_{g-1}$}}}}%
    \put(0,0){\includegraphics[width=\unitlength,page=5]{pi1gens2.pdf}}%
    \put(0.05307602,0.24787045){\color[rgb]{0,0,0}\makebox(0,0)[lb]{\smash{$f$}}}%
  \end{picture}%
\endgroup%